\documentclass[reqno,11pt]{amsart}

\usepackage{amssymb, amsmath,latexsym,amsfonts,amsbsy, amsthm,mathtools,graphicx,,color}
\usepackage{float}
\usepackage{tikz}
\usetikzlibrary{arrows, decorations.markings,matrix,trees}
\allowdisplaybreaks
\usepackage[colorlinks=true]{hyperref}
\hypersetup{
   colorlinks=true,
   citecolor=black,
   filecolor=black,
   linkcolor=black,
   urlcolor=black
}
\def\Xint#1{\mathchoice
  {\XXint\displaystyle\textstyle{#1}}%
  {\XXint\textstyle\scriptstyle{#1}}%
  {\XXint\scriptstyle\scriptscriptstyle{#1}}%
  {\XXint\scriptscriptstyle\scriptscriptstyle{#1}}%
  \!\int}
\def\XXint#1#2#3{{\setbox0=\hbox{$#1{#2#3}{\int}$}
  \vcenter{\hbox{$#2#3$}}\kern-.5\wd0}}

\def\dashint{\Xint-}
\newcommand{\al}{\alpha}            
\newcommand{\lda}{\lambda}
\newcommand{\om}{\Omega}            
\newcommand{\pa}{\partial}
\newcommand{\va}{\varepsilon}       
\newcommand{\ud}{\mathrm{d}}
\newcommand{\be}{\begin{equation}} 
\newcommand{\ee}{\end{equation}}
\newcommand{\w}{\omega}      
\newcommand{\Lda}{\Lambda}    

\newcommand{\A}{\mathbf{A}}
\newcommand{\cA}{\mathcal{A}}

\newcommand{\B}{\mathbf{B}}

\newcommand{\cB}{\mathcal{B}}

\newcommand{\CC}{\mathbf{C}}
\newcommand{\cC}{\mathcal{C}}

\newcommand{\E}{\mathbb{E}}
\newcommand{\cE}{\mathcal{E}}
\newcommand{\e}{\mathbf{e}}

\newcommand{\fF}{\mathbf{F}}
\newcommand{\cF}{\mathcal{F}}

\newcommand{\bG}{\mathbb{G}}
\newcommand{\fG}{\mathbf{G}}
\newcommand{\cG}{\mathcal{G}}

\newcommand{\rh}{\mathrm{h}}
\newcommand{\rH}{\mathrm{H}}
\newcommand{\bH}{\mathbb{H}}

\newcommand{\I}{\mathbf{I}}  
\newcommand{\ii}{\mathrm{i}}

\newcommand{\jj}{\mathrm{j}}

\newcommand{\kk}{\mathrm{k}}

\newcommand{\LL}{\mathbf{L}} 

\newcommand{\Z}{\mathbb{Z}}

\newcommand{\M}{\mathcal{M}}
\newcommand{\MM}{\mathbb{M}}
\newcommand{\m}{\mathbf{m}}

\newcommand{\cN}{\mathcal{N}}

\newcommand{\n}{\mathbf{n}}

\newcommand{\cP}{\mathcal{P}}
\newcommand{\PP}{\mathbf{P}}
\newcommand{\p}{\mathbf{p}}

\newcommand{\Q}{\mathbf{Q}}  

\newcommand{\R}{\mathbb{R}}

\newcommand{\cS}{\mathcal{S}} 
\newcommand{\Ss}{\mathbb{S}}
\newcommand{\fS}{\mathbf{S}}

\newcommand{\uu}{\mathbf{u}}
\newcommand{\U}{\mathbf{U}}

\newcommand{\vv}{\mathbf{v}}
\newcommand{\V}{\mathbf{V}}

\newcommand{\W}{\mathbf{W}}

\newcommand{\X}{\mathbf{X}}
\newcommand{\cX}{\mathcal{X}}

\newcommand{\Y}{\mathbf{Y}}
\newcommand{\cY}{\mathcal{Y}}

\newcommand{\wc}{\rightharpoonup}        

\newcommand{\HH}{\mathcal{H}}
\newcommand{\FF}{\mathcal{F}}

\newcommand{\vp}{\varphi}
\newcommand{\T}{\mathrm{T}}
\newcommand{\ga}{\gamma}
\newcommand{\Ga}{\Gamma}

\newcommand{\sg}{\sigma} 
\newcommand{\ift}{\infty} 
\newcommand{\wt}{\widetilde}
\newcommand{\wh}{\widehat}
\newcommand{\f}{\frac}

\newcommand{\ol}{\overline}
\newcommand{\Ra}{\Rightarrow}
\newcommand{\op}{\operatorname}
\newcommand{\nn}{\nonumber}
\newcommand{\na}{\nabla}

\DeclareMathOperator{\dist}{dist}

\DeclareMathOperator{\diam}{diam}

\DeclareMathOperator{\sgn}{sgn}
\DeclareMathOperator{\supp}{supp}

\DeclareMathOperator{\pts}{pts}

\DeclareMathOperator{\tr}{tr}
\DeclareMathOperator*{\osc}{osc}

\DeclareMathOperator{\diag}{diag}

\DeclareMathOperator{\loc}{loc}

\DeclareMathOperator{\ad}{ad}

\DeclareMathOperator*{\essinf}{ess\,inf}

\def\<{\langle}\def\>{\rangle}
\def\({\left(}\def\){\right)}
\def\[{\left[}\def\]{\right]}
\numberwithin{equation}{section}
\theoremstyle{plain}
\newtheorem{thm}{Theorem}[section]
\newtheorem{cor}[thm]{Corollary}

\newtheorem{lem}[thm]{Lemma}
\newtheorem{prop}[thm]{Proposition}

\theoremstyle{definition}
\newtheorem{defn}[thm]{Definition}

\newtheorem{rem}[thm]{Remark}
\newtheorem{q}[thm]{Question}

\title[Landau-de Gennes model with sextic potentials]{ Landau-de Gennes model with sextic potentials: asymptotic behavior of minimizers}  

\author[W. Wang]{Wei Wang}
\address{School of Mathematical Sciences, Peking University, Beijing 100871, China}
\email{2201110024@stu.pku.edu.cn}

\author[Z. Zhang]{Zhifei Zhang}
\address{School of Mathematical Sciences, Peking University, Beijing 100871, China}
\email{zfzhang@math.pku.edu.cn}

\date{\today}

\begin{document}

\maketitle

\begin{abstract}
We study a class of Landau-de Gennes energy functionals with a sextic bulk energy density in a three-dimensional domain.  We examine the asymptotic behavior of uniformly bounded minimizers in two distinct scenarios: one where their energy remains uniformly bounded, and another where it logarithmically diverges as a function of the elastic constant. In the first case, we show that up to a subsequence, the minimizers converge to a locally minimizing harmonic map  in both the $ H_{\loc}^1 $ and $ C_{\loc}^j, j\in \Z_+ $ norms within compact subsets that are distant from the singularities of the limit. For the second case, we establish the existence of a closed set denoted as $ \cS_{\op{line}} $. This set has finite length and consists of finite segments of lines locally such that the energy of minimizers are locally uniformly bounded away from it. 
This work solves an open question raised by Canevari ({\it ARMA, 223 (2017), 591-676}), specifically concerning point and line defects in the Landau-de Gennes model with sextic potentials.
\end{abstract}

\tableofcontents

\section{Introduction}

\subsection{Background and main results}
Liquid crystals are basically anisotropic fluids, where the anisotropy arises from the directional nature of the molecular geometry, physical, or chemical properties. To characterize this anisotropy, mathematical models are introduced to describe the orientation of the molecules. These models are referred to as order parameters in physics. Depending on the choice of order parameter, existing mathematical models for liquid crystals can be broadly classified into three models. The first is to use a probability distribution function, denoted as $ f(x,\m) $, to represent the probability of the liquid crystal molecules at a fluid point $ x $ being oriented in the direction $ \m $. This model was initially established by Onsager in \cite{O49}. The second model, referred to as the vector model or the Oseen-Frank model, uses a unit vector field denoted as $ \n(x) $ to characterize the average orientation of liquid crystal molecules at each point $ x $. A comprehensive review of the research findings on the Oseen-Frank model and its unresolved issues was provided by \cite{LL01}. The third one, known as the Landau-de Gennes model, describes the local configurations of the medium using $ \Q $-tensors, which are traceless $ 3\times 3 $ matrices represented by
$$
\Ss_0:=\{\Q\in\mathbb{M}^{3\times 3}:\Q^{\T}=\Q,\,\,\tr\Q=0\}.
$$
Each of these models has its own advantages and disadvantages in practical applications, stemming from their different starting points. In this paper, our focus is specifically on the Landau-de Gennes model.

Let $ \om\subset\R^3 $ be a bounded domain. The general Landau-de Gennes energy functional is defined as
\be
\cF(\Q,\om)=\int_{\om}\cF(\Q)\ud x,\label{GeneralLdG}
\ee
where the total energy density $ \cF(\Q)=\cF_e(\Q)+\cF_b(\Q) $ consists of two parts. The elastic part $ \cF_e(\Q) $ is given by
\begin{align*}
\cF_e(\Q)&=\f{L_1}{2}\pa_k\Q_{ij}\pa_k\Q_{ij}+\f{L_2}{2}\pa_j\Q_{ij}\pa_k\Q_{ik}+\f{L_3}{2}\pa_k\Q_{ij}\pa_j\Q_{ik}+\f{L_4}{2}\Q_{ij}\pa_i\Q_{k\ell}\pa_j\Q_{k\ell},
\end{align*}
where $ \{L_i\}_{i=1}^4 $ are non-negative constants and we have used the Einstein summation convention throughout this paper. The bulk energy density $ \cF_b $ is given by the following sextic form
\begin{align*}
\cF_b(\Q)&=a_1-\f{a_2}{2}\tr \Q^2+\f{a_3}{3}\tr\Q^3+\f{a_4}{4}(\tr\Q^2)^2\\
&\quad\quad\,\,+\f{a_5}{5}(\tr\Q^2)(\tr\Q^3)+\f{a_6}{6}(\tr\Q^2)^3+\f{a_6'}{6}(\tr\Q^3)^2.
\end{align*}
Here, $ \{a_i\}_{i=1}^6 $ and $ a_6' $ are non-negative constants. In the context of the Landau-de Gennes theory, the elastic part of the energy functional describes the distortion of the liquid crystal director field, while the bulk part captures the interactions among the molecules. 

To simplify the analysis, many researchers make certain assumptions: $ L_1>0 $, $ L_2=L_3=L_4=0 $, and $ a_5=a_6=a_6'=0 $. Additionally, they investigate the asymptotic behavior of minimizers in $ H^1(\om,\Ss_0) $, subject to appropriate boundary conditions, for the functional \eqref{GeneralLdG} as $ L_1\to 0^+ $. This convergence result is closely related to the minimizers of the energy functional
\be
E_{\va}^{(4)}(\Q,\om)=\int_{\om}\(\f{1}{2}|\na\Q|^2+\f{1}{\va^2}\FF_b^{(4)}(\Q)\)\ud x,\label{quarticbulk}
\ee
where the quartic bulk energy density $ \FF_b^{(4)} $ is defined by 
$$
\FF_b^{(4)}(\Q)=a_1-\f{a_2}{2}\tr\Q^2+\f{a_3}{3}\tr\Q^3+\f{a_3}{4}(\tr\Q^2)^2
$$
for $ a_2,a_3,a_4>0 $. Here, the constant $ a_1 $ is chosen such that $ \min_{\Q\in\Ss_0}\FF_b^{(4)}(\Q)=0 $. From an intuitive perspective, as $ \va\to 0^+ $, the term $ \va^{-2}\cF_b^{(4)}(\Q) $ in \eqref{quarticbulk} forces the minimizers to take their value in $ \cN_u=(\cF_b^{(4)})^{-1}(0) $, defined as
$$
\cN_u=\left\{s_0\(\n\n-\f{1}{3}\I\):\n\in\Ss^2\right\},\quad s_0=\f{a_3+\sqrt{a_2^2+24a_2a_3}}{4a_4},
$$
where $ \n\n=\n\otimes\n $ is the matrix defined by $ (\n\n)_{ij}=\n_i\n_j $. Here, $ \cN_u $ corresponds the unaxial nematic phase, and extensive research has been conducted in this problem setting. In their work \cite{MZ10}, Majumdar and Zarnescu established $ H^1 $ and uniform convergence for a sequence of minimizers with prescribed boundary conditions as $ \va\to 0^+ $. Specifically, they proved that if $ \om $ is a bounded smooth domain, $ \Q_b\in C^{\ift}(\pa\om,\cN_u) $, and $ \{\Q_{\va}\}_{0<\va<1} $ is a sequence of minimizers for the minimizing problem
$$
\min\{E_{\va}^{(4)}(\Q,\om):\Q\in H^1(\om,\Ss_0),\,\,\Q|_{\pa\om}=\Q_b\},
$$
then up to a subsequence, $ \Q_{\va} $ converges in $ H^1(\om,\Ss_0) $ to some $ \Q_0\in H^1(\om,\cN_u) $. Furthermore, $ \Q_0 $ is a minimizer for
$$
\min\left\{\int_{\om}|\na\Q|^2\ud x:\Q\in H^1(\om,\cN),\,\,\Q|_{\pa\om}=\Q_b\right\}
$$
and $ \Q_{\va} $ converges to $ \Q_0 $ uniformly in every compact subset of $ \om $ containing no singularities of $ \Q_0 $. In \cite{NZ13}, Nguyen and Zarnescu improved the uniform convergence result to $ C^j $ convergence with $ j\in\Z_+ $. Indeed, results in \cite{MZ10,NZ13} are all given in the regime where the sequence of minimizers has uniformly bounded energy, that is, $ E_{\va}^{(4)}(\Q_{\va},\om)\leq C $ for some $ C>0 $ independent of $ \va $. A recent notable contribution to this model was presented by Canevari \cite{C17}. In this paper, the author examined minimizers whose energy diverge logarithmically as the elastic constant approaches zero. Specifically, the minimizers $ \{\Q_{\va}\}_{0<\va<1} $ satisfy the energy bound
$$ 
E_{\va}^{(4)}(\Q_{\va},\om)\leq C\(\log\f{1}{\va}+1\),\quad 0<\va<1.
$$
In \cite{C17}, the author established the clearing-out lemma, which states that there exists a sufficiently small $ \eta>0 $ and a $ \va $-independent constant $ C>0 $, such that 
$$
E_{\va}^{(4)}(\Q_{\va},B_r)\leq\eta\log\f{1}{\va}\quad\Ra\quad E_{\va}^{(4)}(\Q_{\va},B_{r/2})\leq C,
$$
for $ B_r(x)\subset\subset\om $. By using this lemma, the existence of line defects can be inferred by studying the convergence of the measure
$$
\(\f{1}{2}|\na\Q_{\va}|^2+\f{1}{\va^2}\FF_b^{(4)}(\Q_{\va})\)\f{\ud x}{\log\f{1}{\va}}
$$
as $ \va\to 0^+ $. Furthermore, it can be proved that the set of line defects is a $ 1 $-dimensional stationary varifold, and standard results imply that it consists of finite segments of lines within the interior of $ \om $. For further investigations on \eqref{quarticbulk}, interested readers can refer to \cite{BZ11,DMP21,GZ23,MP21}. In \cite{C17}, the author also proposed an open question as follows.

\begin{q}\label{openq}
Are the results on the model \eqref{quarticbulk} can be generalized to the model with sextic potentials such that $ a_6,a_6'>0 $? 
\end{q}

When both $ a_6 $ and $ a_6' $ are positive, the situation becomes considerably more complex. It is crucial to note that if $ a_6'>0 $, the set $ \FF_b^{-1}(0) $ can be characterized as
$$
\cN_b=\left\{s_1\(\n\n-\f{1}{3}\I\)+s_1r_1\(\m\m-\f{1}{3}\I\):\n,\m\in\Ss^2,\,\,\n\cdot\m=0\right\},
$$
where $ s_1>0 $ and $ 0<r_1<1 $. Here, $ \cN_b $ represents the biaxial nematic phase. These types of nematic phases have been investigated in physics, for example, in \cite{AL08,SS04}, but they lack thorough mathematical analysis. The geometric structure of $ \cN_b $ is significantly more intricate compared to that of $ \cN_u $, requiring a more sophisticated analysis to explore this model. In their work \cite{DEG98}, Davis, Eugene, and Gartland established the existence of minimizers for this model. Additionally, in \cite{HL22}, Huang and Lin examined its orientability and dynamic properties. Furthermore, the study conducted by Monteil, Rodiac, and Schaftingen in \cite{MRS21,MRS22} considered a more general scenario in a two-dimensional domain.

The goal of this paper is to answer {\bf Question \ref{openq}} in some context. More precisely, we generalize the results presented in \cite{MZ10,NZ13,C17} to the model with sextic bulk energy density. In \cite{HL22}, the authors simplified their analysis by assuming that $ a_3=a_5=0 $. In other words, they considered the following energy functional
\be
E_{\va}(\Q,\om)=\int_{\om}\(\f{1}{2}|\na\Q|^2+\f{1}{\va^2}f_b(\Q)\)\ud x,\label{energy}
\ee
where $ f_b $ represents a sextic bulk energy density defined as
\be
f_b(\Q):=a_1-\f{a_2}{2}\tr \Q^2+\f{a_4}{4}(\tr\Q^2)^2+\f{a_6}{6}(\tr\Q^2)^3+\f{a_6'}{6}(\tr\Q^3)^2\label{fbQa}
\ee
with $ a_2,a_4,a_6,a_6'>0 $. Here, $ a_1 $ is chosen such that $ \min_{\Q\in\Ss_0}f_b(\Q)=0 $. In this paper, we also investigate this model, and for  more general case when $ a_3 $ and $ a_5 $ are non-zero, one can use almost the same arguments.

To determine the minimum of $ f_b(\Q) $ within $ \Ss_0 $, following the calculations in \cite{HL22}, we introduce the variables $ \mu=\tr\Q^2 $ and $ \nu=\tr\Q^3 $. We define
$$
f_b(\Q)=f_0(\mu,\nu):=a_1-\f{a_2}{2}\mu+\f{a_4}{4}\mu^2+\f{a_6}{6}\mu^3+\f{a_6'}{6}\nu^2.
$$
To determine the minimum of $ f_0 $, we consider the equations $ \pa_{\mu}f_0(\mu,\nu)=0 $ and $ \pa_{\nu}f_0(\mu,\nu)=0 $. These equations imply that $ -a_2+a_4\mu+a_6\mu^2=0 $ and $ a_6'\nu=0 $. Consequently,
$$
\mu=\f{-a_4+\sqrt{a_4^2+4a_2a_6}}{2a_6},\quad\nu=0,
$$
where we have eliminated the negative root of $ \mu $ due to the fact that $ \mu=\tr\Q^2\geq 0 $. Thus, the set of eigenvalues of $ \Q $ is $ \{r_*,-r_*,0\} $, where $ r_*>0 $ satisfies 
\be
4a_6r_*^4+2a_4r_*^2-a_2=0.\label{rstar}
\ee
As a result, $ f_b(\Q)=0 $ if and only if
$$
\Q\in\cN:=\{r_*(\n\n-\m\m):(\n,\m)\in\M\},
$$
where $ \M $ is defined by
$$
\M:=\{(\n,\m)\in\Ss^2\times\Ss^2:\n\cdot\m=0\}.
$$
For $ 0<\va<1 $, there are two types of minimizing problems defined as follows.

\begin{defn}[Local minimizers]
Let $ 0<\va<1 $. Assume that $ \om $ is a bounded domain. A map $ \Q_{\va}\in H_{\loc}^1(\om,\Ss_0) $ is a local minimizer of the energy \eqref{energy} if for any ball $ B\subset\subset\om $ and every map $ \PP\in H^1(B,\Ss_0) $ with $ \PP=\Q_{\va}$ on $\partial B$, there holds $ E_{\va}(\Q_{\va},B)\leq E_{\va}(\PP,B) $.
\end{defn}

\begin{defn}[Global minimizers]
Let $ 0<\va<1 $. Assume that $ \om $ is a bounded Lipschitz domain. A map $ \Q_{\va}\in H^1(\om,\Ss_0) $ is a global minimizer of the energy \eqref{energy} with the boundary condition $ \Q_{b,\va}\in H^{1/2}(\pa\om,\Ss_0) $ if for any $ \PP\in H^1(\om,\Ss_0;\Q_{b,\va}) $, there holds $
E_{\va}(\Q_{\va},\om)\leq E_{\va}(\PP,\om) $, where
$$
H^1(\om,\Ss_0;\Q_{b,\va})=\{\Q\in H^1(\om,\Ss_0):\Q|_{\pa\om}=\Q_{b,\va}\}.
$$
We also call that $ \Q_{\va} $ is a global minimizer of \eqref{energy} in the space $ H^1(\om,\Ss_0;\Q_{b,\va}) $.
\end{defn}

For the sake of simplicity, we define $ \cA:=\{a_2,a_4,a_6,a_6'\} $ as the set of coefficients of the bulk energy defined in \eqref{fbQa}. We call a constant $ C>0 $ depends on $ \cA $, if the positive constant $ C $ depends on $ a_2,a_4,a_6 $, and $ a_6' $. Our main results are stated as follows.

\begin{thm}\label{main1}
Let $ \om\subset\R^3 $ be a bounded domain. Assume that $ \{\Q_{\va}\}_{0<\va<1} $ is a sequence of local minimizers of \eqref{energy}, satisfying
\be
E_{\va}(\Q_{\va},\om)\leq M\text{ and }\|\Q_{\va}\|_{L^{\ift}(\om)}\leq M\label{assumptionbound1}
\ee
for some $ M>0 $. There exist a subsequence $ \va_n\to 0^+ $ and $ \Q_0\in H^1(\om,\cN) $ such that the following properties hold.
\begin{enumerate}
\item $ \Q_{\va_n}\to\Q_0 $ strongly in $ H_{\loc}^1(\om,\Ss_0) $ and $ \va_n^{-2}f_b(\Q_{\va_n})\to 0 $ in $ L_{\loc}^1(\om) $.
\item $ \Q_0 $ is locally minimizing harmonic in $ \om $, that is, for every ball $ B\subset\subset\om $ and any $ \PP\in H^1(B,\cN) $ with $ \PP=\Q_0 $ on $ \pa B $, there holds
$$
\f{1}{2}\int_B|\na\Q_0|^2\ud x\leq\f{1}{2} \int_B|\na\PP|^2\ud x.
$$
Moreover, $ \Q_0 $ is a weak solution of
\be
\Delta\Q_0=-\f{1}{2r_*^2}|\na\Q_0|^2\Q_0-\f{3}{r_*^4}\tr(\na\Q_0\na\Q_0\Q_0)\(\Q_0^2-\f{2r_*^2}{3}\I\).\label{weakharmoniceq}
\ee
\item There exists a locally finite set $ \cS_{\pts}\subset\om $, such that $ \Q_0 $ is smooth in $ \om\backslash\cS_{\pts} $.
\item For any $ j\in\Z_{\geq 0} $, $ \Q_{\va_n}\to\Q_{0} $ in $ C_{\loc}^j(\om\backslash\cS_{\pts},\Ss_0) $. In particular, for any $ B_r(x_0)\subset\subset\om\backslash\cS_{\pts} $, with $ r>0 $, $ \Q_{\va_n} $ is a classical solution for
\be
\Delta\Q_{\va_n}=-\f{1}{2r_*^2}|\na\Q_{\va_n}|^2\Q_{\va_n}-\f{3}{r_*^4}\tr(\na\Q_{\va_n}\na\Q_{\va_n}\Q_{\va_n})\(\Q_{\va_n}^2-\f{2r_*^2}{3}\I\)+\mathbf{R}_n.\label{Reesti}
\ee
in $ B_{r/2}(x_0) $. Here, $ \mathbf{R}_n $ is a remainder satisfying the estimate
$$
\|D^j\mathbf{R}_n\|_{L^{\ift}(B_{r/2}(x_0))}\leq C\va_n^2r^{-j-2},
$$
where $ C>0 $ depends only on $ \cA $ and $ M $.
\end{enumerate}
\end{thm}

This theorem provides a generalization of the $ H^1 $ and uniform convergence results presented in \cite{MZ10}, as well as the $ C^j $ convergence result discussed in \cite{NZ13}. In our specific context, where we consider local minimizers, the $ H^1 $ convergence can be established using a Luckhaus-type lemma. Furthermore, we provide detailed estimates for the remainder terms in the Euler-Lagrange equation. These results heavily rely on the specific structure of the Euler-Lagrange equation and necessitate the use of bootstrap arguments.

Furthermore, when considering global minimizers in a smooth domain with prescribed boundary values that are sufficiently smooth themselves, we can obtain stronger regularity properties for the map $ \Q_{\va} $. These enhanced regularity results will be presented in detail in Theorem \ref{expansion}.

\begin{thm}\label{main}
Let $ \om\subset\R^3 $ be a bounded Lipschitz domain. Assume that $ \{\Q_{\va}\}_{0<\va<1} $ is a sequence of local minimizers of \eqref{energy}, satisfying
\be
E_{\va}(\Q_{\va},\om)\leq M\(\log\f{1}{\va}+1\)\text{ and }\|\Q_{\va}\|_{L^{\ift}(\om)}\leq M\label{assumptionbound}
\ee
for some $ M>0 $. There exists a subsequence $ \va_n\to 0^+ $ and a closed set $ \cS_{\op{line}}\subset\ol{\om} $ such that
$$
\(\f{1}{2}|\na\Q_{\va_n}|^2+\f{1}{\va_n^2}f_b(\Q_{\va_n})\)\f{\ud x}{\log\f{1}{\va_n}}\wc^*\mu_0\text{ in }(C(\ol{\om}))'
$$
as $ n\to+\ift $ and the following properties hold.
\begin{enumerate}
\item $ \supp(\mu_0)=\cS_{\op{line}} $. 
\item $ \om\cap\cS_{\op{line}} $ is a countably $ \HH^1 $-rectifiable set, and $ \HH^1(\om\cap\cS_{\op{line}})<+\ift $.
\item If a subdomain $ U $ of $ \om $ satisfies $ U\subset\subset\om\backslash\cS_{\op{line}} $, then
$$
E_{\va_n}(\Q_{\va_n},U)\leq C,
$$
where $ C>0 $ depends only on $ \cA,M $ and $ U $. In such $ U $, one can apply Theorem \ref{main1} to get further results.
\item For $ \HH^1 $-a.e. $ x\in\cS_{\op{line}}\cap\om $,
$$
\lim_{r\to 0^+}\f{\mu_0(\ol{B_r(x)})}{2r}\in\{\kappa_*,2\kappa_*\},
$$
where $ \kappa_*=\pi r_*^2/2 $.
\item $ \mu_0\left\llcorner\right.\om $ is associated with a $ 1 $-dimensional stationary varifold (see subsection \ref{sectionvarifold} for the definition). 
\item For any open set $ K\subset\subset\om $, there holds
$$
\cS_{\op{line}}\cap\ol{K}=\{L_1,L_2,...,L_p\},\,\,p\in\Z_+,
$$
where $ \{L_i\}_{i=1}^p $ are closed straight line segments such that for $ i\neq j $, $ L_i $ and $ L_j $ are disjoint or they intersect at a common endpoint. Moreover, the assertions as follows are true.
\begin{enumerate}
\item If $ \ol{B_r^2(x)}\subset K $ satisfies $ B_r^2(x)\cap\cS_{\op{line}}=\{x\} $, and $ x $ is not an endpoint for any $ L_i $, then the free homotopy class (see Definition \ref{freehomodef}) of $ \Q_0|_{\pa B_r^2(x_0)} $ is non-trivial.\label{pra}
\item If $ x\in K $ is an endpoint of exactly $ q $ segments $ L_{i_1},...,L_{i_q} $, then $ q\neq 1 $. If $ q=2 $, the angle of $ L_{i_1} $ and $ L_{i_2} $ is $ \pi $. If $ q=3 $, $ L_{i_1} $, $ L_{i_2} $, and $ L_{i_2} $ are in the same plane and the angles of them are all $ 2\pi/3 $.\label{prb}
\end{enumerate}
\label{pr6}
\end{enumerate}
\end{thm}

This result generalizes Theorem 1 and Proposition 2 from \cite{C17} to our specific model. The proof of this theorem relies on the monotonicity formula, which is derived from the Pohozaev identity. Notably, the proof of the sixth property is based on well-established findings regarding $ 1 $-dimensional stationary varifolds, as outlined in \cite{AA76}.

It is important to emphasize that in \cite{C17}, the author demonstrated that the endpoints of straight line segments in $ \cS_{\op{line}} $ must be even. However, in our model, obtaining such a result is not possible due to the differing topological structure of the vacuum manifold. Instead, for our model, we investigate the endpoints of segments with a maximum of three elements and provide their corresponding structure.

There are additional results available for the generalized Ginzburg-Landau model, and interested readers can refer to \cite{CS18,CO21,RJ21} for further developments.

\subsection{Main steps of the proof} 

The proofs of the main theorems in this paper are based on the methods and arguments presented in the works \cite{C17,MZ10,NZ13}. However, there are still some essential difficulties in dealing with our model.

\begin{enumerate}
\item The different tangent and normal space structures of $ \cN $ and $ \cN_u $ require the use of different quantities when considering higher-order convergence of the minimizers. This necessitates the application of modified iterative arguments from \cite{NZ13}.
\item The topological characterization of the vacuum manifold $ \cN $ differs significantly from the uniaxial case. In fact, we will prove that $ \cN=\M/(\Z_2\times \Z_2)=\Ss^3/Q_8 $, where $ Q_8 $ denotes the quaternion group (see Lemma \ref{FG}). The structure of $ \cN $ presents challenges in understanding the obstructions on $ \cN $ compared to $ \cN_u $.
\item When establishing Jerrard-Sandier type estimates (see, for example, \cite{S98}) for our model, the functions associated with the uniaxial case are not applicable. We need to develop more refined estimates.
\end{enumerate}

To overcome all these difficulties and establish our results, we conclude our framework of the proof into the following steps.

\begin{enumerate}
\item[Step 1.] We introduce quantities associated with the minimal polynomial of matrices in $ \cN $ and present new iterative arguments to obtain a priori estimates for the solutions of the Euler-Lagrange equation related to the minimizing problem. These calculations heavily rely on the characterization of the tangent and normal spaces of $ \cN $.

\item[Step 2.] We examine the free homotopy class on the manifold $ \cN $ and use a direct approach to characterize the topological properties of loops on $ \cN $.

\item[Step 3.] Building upon the definitions provided in \cite{C17}, we introduce new functions to evaluate the differences in eigenvalues for symmetric traceless matrices. Additionally, we emphasize that these functions require higher regularity. By utilizing these functions and the results of Step 2, we establish a new Jerrard-Sandier type estimate (see Proposition \ref{lowerboundprop}).

\item[Step 4.] By combining the results obtained from the first three steps, we use the well-known Luckhaus arguments (see \cite{L88}) to prove the main theorems outlined in this paper.
\end{enumerate}
 
\subsection{Notations}

\begin{itemize}
\item Throughout this paper, we will use $ C $ to denote positive constants, which may change from line to line.
\item $ \MM^{k\times k} $: the set of $ k\times k $ real. $ \Ss^{k\times k} $: the set of $ k\times k $ real symmetric matrices. $ \Ss_0 $: the set of traceless matrices in $ \Ss^{3\times 3} $.
\item We represent matrices using bold capital letters like $ \A,\B $. In this paper, without special clarification, matrices are all in $ \MM^{3\times 3} $.
\item $ \I_k $: identity matrix of order $ k $. $ \mathbf{O}_k $: zero matrix of order $ k $. If $ k=3 $, we omit the subscript and use $ \I,\mathbf{O} $.
\item $ \mathrm{O}(3):=\{\A\in\MM^{3\times 3}:\A^{\T}\A=\I\} $, and $ \mathrm{SO}(3):=\{\A\in\mathrm{O}(3),\,\,\det\A=1\} $.
\item $ \mathrm{SO}(3) $ is considered as a smooth Riemannian submanifold $ \MM^{3\times 3} $.
\item Let $ \A,\B\in\MM^{3\times 3} $. Their inner product is $ \A: \B:=\A_{ij}\B_{ij} $. In particular, $ |\A|^2=\A:\A $ if and only if $ \A\perp\B\Leftrightarrow\A:\B=0 $.
\item Assume that $ \A,\B:\om\subset\R^3\to\mathbb{M}^{3\times 3} $ are two differentiable matrix valued functions. The gradient or the first derivative of $ \A $ is $ \na\A=D\A:=(\pa_1\A,\pa_2\A,\pa_3\A) $. Furthermore, $
\na\A:\na \B=D\A:D\B:=\pa_{\ell}\A_{ij}\pa_{\ell}\B_{ij} $. In addition, $ |\na\A|^2=|D\A|^2=\na\A:\na \B $. For $ j\in\Z_+ $, $ |D^j\A|^2:=\sum_{|\al|=j}|\pa^{\al}\A|^2 $, where $ \al $ is the multi-index.
\item For $ g\in C^{\ift}(\MM^{3\times 3},\R) $, we denote $ \f{\pa g}{\pa\A_{ij}}=g_{ij} $, $ \f{\pa^2g}{\pa\A_{ij}\pa\A_{pq}}=g_{ij,pq} $, and $ \f{\pa^3g}{\pa\A_{ij}\pa\A_{pq}\pa\A_{mn}}=g_{ij,pq,mn} $.
\item For $ \m,\n\in\R^k $, $ k\in\Z_{\geq 2} $, $ \m\cdot\n=\n_i\m_i $ and $ \m\n=\m\otimes\n\in\mathbb{M}^{k\times k} $ with $ (\m\otimes\n)_{ij}=(\m_i\n_j) $.
\item For $ k\in\Z_{\geq 2} $, $ B_r^k(x):=\{x\in\R^k:|y-x|<r\} $. We will drop the superscript $ k $ if $ k=3 $ and drop $ x $ when $ x $ is the original point.
\item $ \HH^k $: $ k $-dimensional Hausdorff measure. If $ k=3 $, we denote $ \ud\HH^3:=\ud x $.
\item $ E\subset\R^3 $ is called to be countably $ \HH^1 $-rectifiable if there exist countably many Lipschitz functions $ f_i:\R^1\to\R^3 $ such that $ \HH^1(E\backslash\cup_{i=1}^{+\ift}f_i(\R^1))=0 $.
\item Let $ \ell,k\in\Z_+ $. Assume that $ \cX\subset\R^{\ell} $ is a Lipschitz Riemannian $ k $-submanifold of $ \R^{\ell} $, we define $ \na_{\cX} $ as the tangent gradient on $ \cX $. Such gradient exists a.e. by Rademacher theorem. If $ \cX $ is $ C^j $ with $ j\in\Z_{\geq 2} $, we denote $ D_{\cX}^j $ as the higher derivatives on $ \cX $. If $ \Q\in H^1(\cX,\Ss_0) $, we define
$$ 
e_{\va}(\Q,\cX):=\f{1}{2}|\na_{\cX}\Q|^2+\f{1}{\va^2}f_b(\Q),\text{ and }E_{\va}(\Q,\cX):=\int_\cX e_{\va}(\Q,\cX)\ud\HH^k, 
$$ 
for $ 0<\va<1 $. If $ \ell=k $, we write $ \na_{\cX}=\na $, $ D_{\cX}^j=D^j $ and $ e_{\va}(\Q,\cX)=e_{\va}(\Q) $.
\item $ \e^{(1)}:=(1,0,0)^{\T} $, $ \e^{(2)}:=(0,1,0)^{\T} $, and $ \e^{(3)}:=(0,0,1)^{\T} $.
\item For $ k\in\Z_{\geq 2} $, $ \R\PP^{k-1}:=\{\n\n\in\MM^{k\times k}:\n\in\Ss^{k-1}\} $ is the $ (k-1) $-projective space. $ \R\PP^{k-1} $ is a compact smooth Riemannian $ (k-1) $-submanifold of $ \MM^{k\times k} $.
\end{itemize}

Next, we recall the regularity assumption of bounded domain.

\begin{defn}\label{DefnLocal}
Let $ U\subset\R^3 $ be a bounded domain. We say $ U $ is $ C^{k,1} $ with $ k\in\Z_{\geq 0} $, (for $ k=0 $, $ U $ is called Lipschitz), if there exist
\be
r_{U,k}>0,\,\,M_{U,k}>0,\label{r0M0}
\ee
such that the following property holds. For any $ x_0\in\pa U $, there exist a $ C^{k,1} $ function $ \psi $ defined on $ \R^2 $, a coordinate system under a translation, and a rotation such that $ x_0=(0,0,0)\in\R^3 $, 
\be
U\cap B_r(x_0)=\{(y_1,y_2,y_3)\in\R^3:(y_1,y_2)\in\R^2,\,\,y_3>\psi(y_1,y_2)\}\cap B_r(x_0)\label{UcapB}
\ee
for any $ 0<r<20(M_{U,k}+1)r_{U,k} $, and
\be
\psi(0,0)=0,\quad\|D^i\psi\|_{L^{\ift}(\R^2)}\leq M_{U,k}\text{ for }i=1,2,...,k+1.\label{Dkest}
\ee
To this end, we say $ U $ is a bounded $ C^{k,1} $ domain with $ r_{U,k} $ and $ M_{U,k} $, if $ U $ is a bounded domain satisfying \eqref{UcapB} and \eqref{Dkest}, where $ r_{U,k} $ and $ M_{U,k} $ are given by \eqref{r0M0}. We call $ U $ is a smooth domain, if $ U $ is $ C^{k,1} $ domain for any $ k\in\Z_+ $.
\end{defn}

\section{Preliminaries}\label{Preliminaries}

\subsection{Some properties of \texorpdfstring{$ \Ss_0 $}{}, \texorpdfstring{$ \cN $}{}, \texorpdfstring{$ \M $}{} and \texorpdfstring{$ f_b $}{}}

Firstly, we will present some properties related to the set of symmetric traceless matrices $ \Ss_0 $, the set associated with the nematic phase $ \cN $ and $ \M $, as well as the bulk energy density $ f_b $. These results will serve as a foundation for the subsequent analysis.

\begin{lem}\label{reprelem}
For any fixed non-zero $ \Q\in\Ss_0 $, there exist two numbers $ s\in\R_+ $, $ r\in[0,1] $ and a pair $ (\n,\m)\in\M $ such that
\be
\Q=s\left\{\n\n-\f{1}{3}\I+r\(\m\m-\f{1}{3}\I\)\right\}.\label{representationQ}
\ee
The maps $ \Q\mapsto s(\Q) $ and $ \Q\mapsto r(\Q) $ defined on $ \Ss_0\backslash\{\mathbf{O}\} $ are continuous. Moreover, the parameters $ s(\Q),r(\Q) $ can be determined uniquely from the eigenvalues $ \lda_1(\Q)\geq \lda_2(\Q)\geq\lda_3(\Q) $ of $ \Q $ by
$$
s(\Q):=2\lda_1(\Q)+\lda_2(\Q)\text{ and }r(\Q):=\f{\lda_1(\Q)+2\lda_2(\Q)}{2\lda_1(\Q)+\lda_2(\Q)}.
$$
\end{lem}
\begin{proof}
Since $ \Q $ is a real symmetric matrix, it can be diagonalized as $ \Q=\lda_1\n\n+\lda_2\m\m+\lda_3\p\p $, where $ \n,\m $, and $ \p $ form an orthonormal basis of $ \R^3 $. This implies that $ \n\n+\m\m+\p\p=\I $. Additionally, from $ \tr\Q=0 $, we have $ \sum_{i=1}^3\lda_i=0 $ and 
$$
\Q=(2\lda_1+\lda_2)\n\n+(\lda_1+2\lda_2)\m\m-(\lda_1+\lda_2)\I,
$$
which directly implies \eqref{representationQ}. As $ \{\lda_i(\Q)\}_{i=1}^3 $ are continuous functions with respect to $ \Q $, $ s(\Q) $ and $ r(\Q) $ are continuous in $ \Ss_0\backslash\{\mathbf{O}\} $.
\end{proof}

\begin{lem}\label{mu1mu2mu3S}
Assuming $ \mu_1>\mu_2>\mu_3 $ and $ \sum_{i=1}^3\mu_i=0 $, we can define
$$
\Ss_0(\mu_1,\mu_2,\mu_3):=\{\Q\in\Ss_0:\lda_i(\Q)=\mu_i\text{ for any }i=1,2,3\}.
$$
This set is a compact, smooth $ 3 $-dimensional Riemannian manifold without boundary, which is isometrically embedded into $ \Ss_0 $.
\end{lem}

\begin{cor}\label{PropLdaM}
$ \cN $ is a compact and smooth $ 3 $-dimensional Riemannian manifold without boundary, which is isometrically embedded into $ \Ss_0 $.
\end{cor}
\begin{proof}
Choosing $ (\mu_1,\mu_2,\mu_3)=(r_*,0,-r_*) $, the result follows from Lemma \ref{mu1mu2mu3S}.
\end{proof}

\begin{proof}[Proof of Lemma \ref{mu1mu2mu3S}]
Define a map $ \Lda:\Ss_0\to\R^3 $ by
\be
\Lda(\Q):=(\lda_1(\Q),\lda_2(\Q),\lda_3(\Q)),\label{LdaFunctionDef}
\ee
where $ \lda_1(\Q)\geq\lda_2(\Q)\geq\lda_3(\Q) $ are eigenvalues of $ \Q $. Clearly, $ \Ss_0(\mu_1,\mu_2,\mu_3)=\Lda^{-1}((\mu_1,\mu_2,\mu_3)) $. Define $ \{F_i\}_{i=1}^3 $ as $ F_i:\Ss_0\times V_0\to\R $ such that $ F_i(\Q,(t_1,t_2,t_3)):=p_{\Q}(t_i) $,
where $ p_{\Q}(t):=\det(\Q-t\I) $ and
\be
V_0:=\{(t_1,t_2,t_3)\in\R^3:t_1>t_2>t_3\}.\label{V0SetDef}
\ee 
We also define $ F:\Ss_0\times V_0\to\R^3 $ by $
F:=(F_1,F_2,F_3) $. For $ \Q_0\in\Ss_0(\mu_1,\mu_2,\mu_3) $, we have $
p_{\Q_0}(t)=-(t-\mu_1)(t-\mu_2)(t-\mu_3) $ and $ F(\Q_0,\Lda(\Q_0))=0 $. For $ D_2F=(\pa_{t_i}F_j)_{i,j\in\{1,2,3\}} $, since $ \mu_1>\mu_2>\mu_3 $, we have
$$
D_2F|_{(\Q_0,\Lda(\Q_0))}=\diag\{p_{\Q_0}'(\mu_1),p_{\Q_0}'(\mu_2),p_{\Q_0}'(\mu_3)\}
$$
is non-degenerate. We note that $ F $ is a smooth map. By the implicit function theorem, there exist open sets $ U\subset\Ss_0 $ and $ V\subset V_0 $,together with a smooth map $ \Lambda_*:U\to V $ such that $ \Q_0\in U $, $ (\mu_1,\mu_2,\mu_3)\in V $, and
\be
F(\Q,\Lda_*(\Q))=0\text{ for any }\Q\in U.\label{FLda}
\ee
The implicit function theorem also implies that the choice of $ \Lda_* $ is unique and then $ \Lda_*=\Lda $ in $ U $. Taking the derivative with respect to the variable $ \Q $ for both sides of \eqref{FLda}, we have, for any $ \Q\in U $ and $ \PP\in\Ss_0 $, $ D_{\Q}\Lda:\Ss_0\to\R^3 $ satisfies
$$
D_{\Q}\Lda(\PP)=-(D_{\mathbf{t}}F|_{(\Q,\Lda(\Q))})^{-1}(D_{\Q}F|_{(\Q,\Lda(\Q))})(\PP).
$$
By the previous analysis, $ D_2F|_{(\Q,\Lda(\Q))} $ is non-degenerate in $ U $. We claim that $ D_{\Q}F|_{(\Q_0,\Lda(\Q_0))}:\Ss_0\to\R^3 $ is surjective. If the claim is true, we can choose an open set $ U_1\subset U $ such that $ D_{\Q}\Lda $ is surjective in $ U_1 $ and then $ (\mu_1,\mu_2,\mu_3) $ is a regular point of $ \Lda $, which directly implies that $ \Ss_0(\mu_1,\mu_2,\mu_3) $ is a $ 3 $-dimensional manifold. Obviously, by the natural including map $ \iota:\Ss_0(\mu_1,\mu_2,\mu_3)\hookrightarrow\Ss_0 $, $ \Ss_0(\mu_1,\mu_2,\mu_3) $ is isometrically embedded into $ \Ss_0 $. Next we will show the claim. By the Jacobi formula, for any $ \PP\in\Ss_0 $, we have
\begin{align*}
D_{\Q}F|_{(\Q_0,\Lda(\Q_0))}(\PP)=(\tr(\ad(\Q_0-\mu_i\I)\PP))_{i=1,2,3},
\end{align*}
where $ \ad(\Q) $ is the adjoint matrix for $ \Q $. Without loss of generality, we assume $ \Q_0=\diag\{\mu_1,\mu_2,\mu_3\} $. For otherwise, there is $ \U\in\mathrm{O}(3) $ such that $ \U^{\T}\Q_0\U=\diag\{\mu_1,\mu_2,\mu_3\} $ and the same arguments can be applied. By simple calculations,
\begin{align*}
\ad(\Q_0-\mu_1\I)&=\diag\{(\mu_2-\mu_1)(\mu_3-\mu_1),0,0\},\\
\ad(\Q_0-\mu_2\I)&=\diag\{0,(\mu_1-\mu_2)(\mu_3-\mu_2),0\},\\
\ad(\Q_0-\mu_3\I)&=\diag\{0,0,(\mu_1-\mu_3)(\mu_2-\mu_3)\}.
\end{align*}
This directly implies that $ D_{\Q}F|_{(\Q_0,\Lda(\Q_0))} $ is surjective since $ \mu_1>\mu_2>\mu_3 $. 
\end{proof}

In the proof of Lemma \ref{mu1mu2mu3S}, we obtain a byproduct, which can be stated as follows.

\begin{lem}\label{smoothLda}
Assume that the map $ \Lda $ and the set $ V_0 $ are given by \eqref{LdaFunctionDef} and \eqref{V0SetDef}, respectively. There holds that $ \Lda $ is continuous in $ \Ss_0 $ and smooth in $ \Lda^{-1}(V_0) $.
\end{lem}

Now we can characterize the tangent and normal spaces of the smooth manifold $ \cN $ as a submanifold of $ \Ss_0 $. In particular, we have the following lemmas.

\begin{lem}[\cite{HL22}, Proposition 3.1]\label{minimalp}
For $ \Q\in\cN $, the following properties hold.
\begin{enumerate}
\item $ \tr\Q^2=2r_*^2 $ and $ \tr\Q^3=0 $.
\item The minimal polynomial of $ \Q $ is $ \lda^3-r_*^2\lda $.
\end{enumerate}
\end{lem}

\begin{lem}[\cite{HL22}, Proposition 3.2]\label{tangentnormal}
For $ \Q=r_*(\n\n-\m\m)\in\cN $ and $ \p=\n\times\m $ with $ (\n,\m)\in\M $, there hold
\begin{align*}
T_{\Q}\cN&=\op{span}\left\{\f{1}{\sqrt{2}}(\n\m+\m\n), \f{1}{\sqrt{2}}(\n\p+\p\n),\f{1}{\sqrt{2}}(\m\p+\p\m)\right\},\\
(T_{\Q}\cN)_{\Ss_0}^{\perp}&=\op{span}\left\{\f{1}{\sqrt{2}}(\n\n-\m\m),\sqrt{6}\(\f{1}{2}\n\n+\f{1}{2}\m\m-\f{1}{3}\I\)\right\}.
\end{align*}
\end{lem}

\begin{rem}\label{normalremQ}
For any $ \Q=r_*(\n\n-\m\m)\in\cN $ with $ (\n,\m)\in\M $, by  Lemma \ref{tangentnormal}, we have
$$
(T_{\Q}\cN)_{\Ss_0}^{\perp}=\op{span}\left\{\f{\Q}{\sqrt{2}r_*},\f{\sqrt{6}}{2r_*^2}\(\Q^2-\f{2r_*^2}{3}\I\)\right\}.
$$
Consequently, the tangent space $ T_{\Q}\cN $ can be characterized as
$$
T_{\Q}\cN=\{\PP\in\Ss_0:\tr(\PP\Q)=0,\,\,\tr(\PP\Q^2)=0\}.
$$
\end{rem}

In view of the characterization of the tangent and normal space of the manifold $ \cN $, we can further calculate its second fundamental form, which is useful for studying harmonic maps with $ \cN $ as the target manifold.

\begin{lem}
The second fundamental form of $ \cN $ in $ \Ss_0 $ is given by
\be
\Pi_{\cN}(\mathbf{X},\Y)=-\f{1}{2r_*^2}\tr(\mathbf{X}\Y)\Q-\f{3}{r_*^4}\tr(\mathbf{X}\Y\Q)\(\Q^2-\f{2r_*^2}{3}\I\),\label{SecondFundamentalForm}
\ee
where $ \mathbf{X} $ and $ \Y $ are tangent vectors to $ \cN $ at the point $ \Q\in\cN $.
\end{lem}
\begin{proof}
The calculations presented here are analogous to the proof of Lemma 3.5 in \cite{NZ13}. Due to the transitive action of the special orthogonal group $ \mathrm{SO}(3) $ on the manifold $ \cN $, we can assume that $ \Q=r_*\diag\{1,-1,0\} $. We can isometrically embed $ \cN $ into the space of $ 3\times 3 $ matrices $ \MM^{3\times 3} $ through inclusion. For a matrix in $ \MM^{3\times 3} $, there exists a parametrization
$$
\A=\(\begin{matrix}
y_1&y_2&y_3\\
y_4&y_5&y_6\\
y_7&y_8&y_9
\end{matrix}\)=\sum_{i=1}^9y_i\pa_i.
$$
By this, we can choose tangential vectors at $ \Q $ in the form of differential operators, which are given by
\begin{align*}
\X_1&=-(y_2+y_4)\pa_1+(y_1-y_5)\pa_2-y_6\pa_3+(y_1-y_5)\pa_4+(y_2+y_4)\pa_5+y_3\pa_6-y_8\pa_7+y_7\pa_8,\\
\X_2&=-(y_3+y_7)\pa_1-y_8\pa_2+(y_1-y_9)\pa_3-y_6\pa_4+y_4\pa_6+(y_1-y_9)\pa_7+y_2\pa_8+(y_3+y_7)\pa_9,\\
\X_3&=-y_3\pa_2+y_2\pa_3-y_7\pa_4-(y_6+y_8)\pa_5+(y_5-y_9)\pa_6+y_4\pa_7+(y_5-y_9)\pa_8+(y_6+y_8)\pa_9.
\end{align*}
We have that $ \{\X_i\}_{i=1}^3 $ form a basis of $ T_{\Q}\cN $. Simple calculations yield that
\begin{align*}
\ol{\na}_{\X_1}\X_1&=2(y_1-y_5)(-\pa_1+\pa_5)+\cdots,\,\,\ol{\na}_{\X_1}\X_2=(y_1-y_5)(-\pa_6+\pa_8)+\cdots,\\
\ol{\na}_{\X_2}\X_2&=2(y_1-y_9)(-\pa_1+\pa_9)+\cdots,\,\,\ol{\na}_{\X_1}\X_3=(y_1-y_5)(-\pa_3+\pa_7)+\cdots,\\
\ol{\na}_{\X_2}\X_3&=(y_1-y_9)(-\pa_2-\pa_4)+\cdots,\,\,\ol{\na}_{\X_3}\X_3=2(y_5-y_9)(-\pa_5+\pa_9)+\cdots,
\end{align*}
where the dots comprise of terms that vanish at $ \Q $ and $ \ol{\na} $ denotes the connection of $ \MM^{3\times 3} $. Since $ \Pi_{\cN}(\X,\Y) $ is the normal component of $ \ol{\na}_{\X}\Y $, we can obtain
\begin{align*}
\Pi_{\cN}(\X_1,\X_1)(\Q)&=[4r_*(-\pa_1+\pa_5)]^{\perp}=4r_*\diag\{-1,1,0\},\\
\Pi_{\cN}(\X_1,\X_2)(\Q)&=[2r_*(\pa_6+\pa_8)]^{\perp}=0,\\
\Pi_{\cN}(\X_2,\X_2)(\Q)&=[2r_*(-\pa_1+\pa_5)]^{\perp}=2r_*\diag\{-1,0,1\},\\
\Pi_{\cN}(\X_1,\X_3)(\Q)&=[r_*(\pa_3+\pa_7)]^{\perp}=0,\\
\Pi_{\cN}(\X_2,\X_3)(\Q)&=[r_*(-\pa_2-\pa_4)]^{\perp}=0,\\
\Pi_{\cN}(\X_3,\X_3)(\Q)&=[-2r_*(-\pa_5+\pa_9)]^{\perp}=-2r_*\diag\{0,-1,1\}.
\end{align*}
Since it is easy to check that \eqref{SecondFundamentalForm} is true when $ \X,\Y\in\{\X_1,\X_2,\X_3\} $ and $ \Pi_{\cN}(\cdot,\cdot) $ is bilinear, we can verify the result of this lemma for any $ \X,\Y\in T_{\Q}\cN $.
\end{proof}

The next lemma shows that the set $ \M $ is a smooth manifold and give a Riemannian metric on it.

\begin{lem}
$ \M $ is a compact, smooth $ 3 $-dimensional Riemannian manifold without boundary, which is isometrically embedded into $ \R^6 $.    
\end{lem}
\begin{proof}
Considering the parametrization of $ \R^6 $ by $ (\n,\m) $ with $ \n,\m\in\R^3 $, we can define $ F:\R^6\to\R^3 $ by $
F(\n,\m):=(|\n|^2,|\m|^2,\n\cdot\m) $. Obviously, we have $ \M=F^{-1}((1,1,0)) $. Through simple calculations, we observe that on $ \M $, $ (1,1,0) $ is a regular point of $ F $, which means that $ DF|_{\M} $ has rank $ 3 $. Consequently, $ \M $ is a compact, smooth $ 3 $-dimensional Riemannian submanifold of $ \R^6 $ inheriting the Euclidean metric of $ \R^6 $.
\end{proof}

The sets $ \{\cC_i\}_{i=1}^2 $ are defined by
\begin{align*}
\cC_1&:=\{\Q\in\Ss_0:\lda_1(\Q)=\lda_2(\Q)\},\\
\cC_2&:=\{\Q\in\Ss_0:\lda_2(\Q)=\lda_3(\Q)\},
\end{align*}
where $ \lda_1(\Q)\geq\lda_2(\Q)\geq\lda_3(\Q) $ are eigenvalues of $ \Q $. For matrices $ \Q\in\Ss_0\backslash\cC_1 $ with distinct first two eigenvalues and $ \Q\in\Ss_0\backslash\cC_2 $ with distinct second and third eigenvalues, we can choose eigenvectors corresponding to $ \lda_1(\Q) $ and $ \lda_3(\Q) $, denoted as $ \n(\Q) $ and $ \m(\Q) $, respectively. Moreover, the following maps
\be
\begin{aligned}
\varrho&:\Ss_0\backslash(\cC_1\cup\cC_2)\to\cN,\quad\varrho(\Q):=r_*(\n(\Q)\n(\Q)-\m(\Q)\m(\Q),\\
\varrho_1&:\Ss_0\backslash\cC_1\to\cP,\quad\quad\quad\,\,\,\,\varrho_1(\Q):=r_*\n(\Q)\n(\Q),\\
\varrho_2&:\Ss_0\backslash\cC_2\to\cP,\quad\quad\quad\,\,\,\,\varrho_2(\Q):=r_*\m(\Q)\m(\Q)
\end{aligned}\label{rho1rho2rho}
\ee
are well defined, where $ \cP $ is given by
\be
\cP:=\{r_*\n\n\in\MM^{3\times 3}:\n\in\Ss^2\}.\label{N0}
\ee

The following lemma shows that $ \varrho $ is exactly the nearest point projection map to the manifold $ \cN $.

\begin{lem}\label{varrhop}
The maps $ \varrho_1,\varrho_2 $ are of class $ C^1 $ on $ \Ss_0\backslash\cC_1 $ and $ \Ss_0\backslash\cC_2 $, respectively and satisfy
\be
|D\varrho_i(\Q)|\leq C\text{ for any }\Q\in\Ss_0\backslash\cC_i\label{Dvarrhoestimate}
\ee
with $ i=1,2 $, where $ C>0 $ depends only on $ \cA $. The map $ \varrho $ is of class $ C^1 $ and $ \varrho=\varrho_1-\varrho_2 $ on $ \Ss_0\backslash(\cC_1\cup\cC_2) $. Furthermore, $ \varrho $ precisely agrees with the nearest point projection onto $ \cN $, in the sense that
\be
|\Q-\varrho(\Q)|\leq|\Q-\PP|\label{rhomin}
\ee
holds for all $ \Q\in\Ss_0\backslash(\cC_1\cup\cC_2) $ and any $ \PP\in\cN $, with strict inequality when $ \PP\neq\varrho(\Q) $. In other words, $ \varrho(\Q) $ uniquely achieves the minimum distance from $ \Q\in\Ss_0\backslash(\mathcal{C}_1\cup\mathcal{C}_2) $ to $ \cN $.
\end{lem}

\begin{proof}
The $ C^1 $ regularity of $ \varrho,\varrho_1,\varrho_2 $ and \eqref{Dvarrhoestimate} can be deduced from classical differentiability results for matrix eigenvectors (see Section 9.1 of \cite{A89} for example). Therefore, we focus on establishing property \eqref{rhomin}. For $ \Q\in\Ss_0 $, we write its representation as $
\Q=\lda_1\n\n+\lda_2\p\p+\lda_3\m\m $, with $ \lda_1\geq\lda_2\geq\lda_3 $, where $ (\n,\p,\m) $ forms an orthonormal basis of $ \R^3 $ and $ \sum_{i=1}^3\lda_i=0 $. Clearly, $ \lda_1\geq 0\geq\lda_3 $. Let $ \PP\in\cN $ be given by $ \PP=r_*(\uu\uu-\vv\vv) $ with $ (\uu,\vv)\in\M $. Simple calculations yield that
\begin{align*}
|\Q-\PP|^2&=\sum_{i=1}^3\lda_i^2+2r_*^2-2\lda_1r_*(\n\cdot\uu)^2-2\lda_2r_*(\p\cdot\uu)^2-2\lda_3r_*(\m\cdot\uu)^2\\
&\quad\quad+2\lda_1r_*(\n\cdot\vv)^2+2\lda_2r_*(\p\cdot\vv)^2+2\lda_3r_*(\m\cdot\vv)^2:=f(\uu,\vv).
\end{align*}
Since $ \Q\notin\cC_1\cup\cC_2 $, we have $ \lda_1\geq\lda_2\geq\lda_3 $ and $ \lda_1>0>\lda_3 $. Hence, $ f(\uu,\vv) $ achieves its minimum when $ \uu=\n $ and $ \vv=\m $, which implies \eqref{rhomin}.
\end{proof}

\begin{rem}\label{rhorem}
In fact, one can show that $ \varrho\in C^{\ift}(\Ss_0\backslash(\mathcal{C}_1\cup\mathcal{C}_2)) $ by using Lemma \ref{mu1mu2mu3S} and standard arguments. For any $ \Q\in\Ss_0\backslash(\mathcal{C}_1\cup\mathcal{C}_2) $, $ \varrho(\Q) $ is well-defined. Let $ \Q=\lda_1\n\n+\lda_2\p\p+\lda_3\m\m $ with $ \lda_1\geq\lda_2\geq\lda_3 $. Lemma \ref{varrhop} implies that $ \varrho(\Q)=r_*(\n\n-\m\m) $ and $
\dist^2(\Q,\cN)=(\lda_1-r_*)^2+\lda_2^2+(\lda_3+r_*)^2 $. Since $ \lda_2=-\lda_1-\lda_3 $, we can use Cauchy inequality to get
\be
(\lda_1-r_*)^2+(\lda_3+r_*)^2\leq\dist^2(\Q,\cN)\leq 2((\lda_1-r_*)^2+(\lda_3+r_*)^2).\label{eqdist}
\ee
Since $ \cN\cap(\mathcal{C}_1\cup\mathcal{C}_2)\neq\emptyset $ and $ \cN $ is a compact smooth manifold, there exists $ 0<\delta_0<1 $ depending only on $ \cA $ such that $ \varrho(\Q) $ is well defined if $ \dist(\Q,\cN)<\delta_0 $.
\end{rem}

Consider functions $ \phi_1,\phi_2:\Ss_0\to\R $ defined by
\begin{align*}
\phi_1(\Q)&:=r_*^{-1}(\lda_1(\Q)-\lda_2(\Q)),\\
\phi_2(\Q)&:=r_*^{-1}(\lda_2(\Q)-\lda_3(\Q)).
\end{align*}
Obviously, $ \phi_1,\phi_2\geq 0 $, $ \phi_1(\Q)=0 $ if and only if $ \Q\in\cC_1 $ and $ \phi_2(\Q)=0 $ if and only if $ \Q\in\cC_2 $. With the help of representation formula \eqref{representationQ}, it follows that
\begin{align*}
\phi_1(\Q)&=r_*^{-1}s(\Q)(1-r(\Q)),\\
\phi_2(\Q)&=r_*^{-1}s(\Q)r(\Q).
\end{align*}
with $ s(\Q)\in(0,+\ift) $ and $ r(\Q)\in[0,1] $ when $ \Q\in\Ss_0\backslash\{\mathbf{O}\} $. For functions $ \phi_1 $ and $ \phi_2 $, we have the following results on Lipschitz estimates of them.

\begin{lem}\label{Lipphi}
For each $ i=1,2 $, the following properties hold.
\begin{enumerate}
\item $ \phi_i $ is homogeneous of degree $ 1 $.
\item $ \phi_i\in W^{1,\ift}(\Ss_0)\cap C^{\ift}(\Ss_0\backslash(\cC_1\cup\cC_2)) $  and satisfies
\be
\sqrt{2}r_*^{-1}\leq |D\phi_i(\Q)|\leq 2r_*^{-1}\text{ for any }\Q\in\Ss_0\backslash(\cC_1\cup\cC_2).\label{Dphiestimate}
\ee
\end{enumerate}
\end{lem}
\begin{proof}
Firstly, for any $ i=1,2 $, $ t>0 $ and $ \Q\in\Ss_0 $, by the definition of $ \phi_i $, we have $ \phi_i(t\Q)=t\phi_i(\Q) $. This implies that $ \{\phi_i\}_{i=1}^2 $ are homogeneous of degree $ 1 $. By similar arguments in Lemma 13 of \cite{C17}, for each $ i=1,2 $, we deduce $ \phi_i\in W^{1,\ift}(\Ss_0)\cap C^1(\Ss_0\backslash(\cC_1\cup\cC_2)) $ and \eqref{Dphiestimate}. By using Lemma \ref{smoothLda}, we obtain $ \{\phi_i\}_{i=1}^2 $ are smooth in $ \Ss_0\backslash(\cC_1\cup\cC_2) $.
\end{proof}

We can further define the function $ \phi_0:\Ss_0\to\R $ by
$$
\phi_0(\Q):=\min\{\phi_1(\Q),\phi_2(\Q)\}.
$$
Now we have, $ \Q\in\cC_1\cup\cC_2 $ if and only if $ \phi_0(\Q)=0 $.

\begin{lem}\label{Lipphi0lem}
For $ \phi_0 $, the following properties hold.
\begin{enumerate}
\item $ \phi_0 $ is homogeneous of degree $ 1 $.
\item $ \phi_0\in W^{1,\ift}(\Ss_0) $ and satisfies 
\be
|D\phi_0(\Q)|\leq 2r_*^{-1}\text{ for any }\Q\in\Ss_0\backslash(\cC_1\cup\cC_2).\label{Lipestiphi0eq}
\ee
\end{enumerate}
\end{lem}
\begin{proof}
For $ \Q\in\Ss_0 $, there exists $ i\in\{1,2\} $ such that $ \phi_0(\Q)=\phi_i(\Q) $. Applying the first property of Lemma \ref{Lipphi}, for any $ t>0 $, we have $ \phi_0(t\Q)=\phi_i(t\Q)=t\phi_i(\Q)=t\phi_0(\Q) $. In view of the second property of Lemma \ref{Lipphi} and the fact that the $ \phi_0=(\phi_1+\phi_2-|\phi_1-\phi_2|)/2 $, we deduce $ \phi_0\in W^{1,\ift}(\Ss_0) $ and satisfies \eqref{Lipestiphi0eq}.
\end{proof}

Noticing that $ \phi_0 $ is not $ C^2 $ near the set $ \{\Q\in\Ss_0:\phi_1(\Q)=\phi_2(\Q)\} $, for later use, we define a comparable function of $ \phi_0 $, which has higher regularity. Precisely, for $ 0<\tau<1/2 $, we set $ \phi_{\tau}:\Ss_0\to\R $ by
\be
\begin{aligned}
\phi_{\tau}(\Q):=\left\{\begin{aligned}
&r_*^{-1}\vp_{\tau}(s(\Q),r(\Q))&\text{ if }&\Q\neq\mathbf{O},\\
&0&\text{ if }&\Q=\mathbf{O},
\end{aligned}\right.
\end{aligned}\label{phitaudefinition}
\ee
where
$$
\vp_{\tau}(s,r):=\left\{\begin{aligned}
&sr&\text{ if }&s\in\R_+,\,\,r\in[0,(1-2\tau)/2),\\
&s\beta_{\tau}(r)&\text{ if }&s\in\R_+,\,\,r\in[(1-2\tau)/2,(1+2\tau)/2],\\
&s(1-r)&\text{ if }&s\in\R_+,\,\,r\in((1+2\tau)/2,1],
\end{aligned}\right.
$$
with
$$
\beta_{\tau}(r):=\f{1}{24\tau^5}\(\f{2r-1}{2}\)^6-\f{5}{8\tau}\(\f{2r-1}{2}\)^2+\f{6-5\tau}{12}.
$$

\begin{lem}\label{Lipphi4}
For each $ 0<\tau<1/2 $ the following properties hold.
\begin{enumerate}
\item $ \phi_{\tau}^{-1}(0)=\cC_1\cup\cC_2 $ and for any $ \Q\in\Ss_0 $,
\be
\phi_{\tau}(\Q)\leq\phi_0(\Q)\leq\f{6}{6-5\tau}\phi_{\tau}(\Q).\label{phi42phi3}
\ee 
\item $ \phi_{\tau}\in W^{1,\ift}(\Ss_0)\cap C^2(\Ss_0\backslash(\cC_1\cup\cC_2)) $ and satisfies
\be
|D\phi_{\tau}(\Q)|\leq 5r_*^{-1}\text{ for any }\Q\in\Ss_0\backslash(\cC_1\cup\cC_2).\label{Dphi4es}
\ee
\end{enumerate}
\end{lem}
\begin{proof}
By the definition of $ \{\phi_{\tau}\}_{0<\tau<1/2} $, we have 
$$
\phi_{\tau}^{-1}(0)=\{\mathbf{O}\}\cup\{\Q\in\Ss_0\backslash\{\mathbf{O}\}:r(\Q)=0\text{ or }r(\Q)=1\}=\cC_1\cup\cC_2.
$$
To show the first property, we assume $ \Q\neq\mathbf{O} $, since for otherwise, there is nothing to prove. Indeed, if $ \Q\neq\mathbf{O} $, we have
\begin{align*}
\phi_0(\Q):=\left\{\begin{aligned}
&r_*^{-1}s(\Q)r(\Q)&\text{ if }&r(\Q)\in[0,1/2),\\
&r_*^{-1}s(\Q)(1-r(\Q))&\text{ if }&r(\Q)\in[1/2,1].
\end{aligned}\right.
\end{align*}
This directly implies that \eqref{phi42phi3} when $ r(\Q)\in[0,(1-2\tau)/2)\cup((1+2\tau)/2,1] $. On the other hand, since $ \phi_{\tau} $ is an even function with respect to $ r $, it suffices to show $
\beta_{\tau}(r)\leq r\leq 6\beta_{\tau}(r)/(6-5\tau) $ for any $ r\in[(1-2\tau)/2,1/2] $. Let $ (1-2\tau t)/2=r $, it is reduced to prove that
$$
\f{(6-5\tau)(1-2\tau t)}{12}\leq\beta_{\tau}\(\f{1-2\tau t}{2}\)\leq\f{1-2\tau t}{2}
$$
for any $ t\in[0,1] $. By simple calculations, we get
\begin{align*}
\beta_{\tau}\(\f{1-2\tau t}{2}\)-\f{1-2\tau t}{2}&=\f{\tau(t-1)^3(t^3+3t^2+6t+10)}{24}\leq 0,\\
\beta_{\tau}\(\f{1-2\tau t}{2}\)-\f{(6-5\tau)(1-2\tau t)}{12}&=\f{\tau t(t^5-15t-20\tau+24)}{24}\geq 0,
\end{align*}
where for the second inequality, we have used $ \tau\in(0,1/2) $ and the fact
\be
\min_{t\in[0,1]}\{t^5-15t\}=-14.\label{t515}
\ee
Therefore, we complete the proof of \eqref{phi42phi3}. Next, we show the second property. Firstly, $ \{\phi_{\tau}(s,r)\}_{0<\tau<1/2} $ is $ C^2 $ in $ (0,+\ift)\times(0,1) $. Combined with Lemma \ref{smoothLda}, we have, $ \{\phi_{\tau}\}_{0<\tau<1/2}\subset C^{\ift}(\Ss_0\backslash(\cC_1\cup\cC_2)) $. In view of the regularity results for eigenvalues given by equation (9.1.32) of \cite{A89}, we have, for $ i=1,2 $ and $ \Q\in\Ss_0\backslash(\cC_1\cup\cC_2) $, there holds
\be
|D\lda_i(\Q)|=\max_{\mathbf{B}\in\Ss_0,\,\,|\mathbf{B}|=1}|\B\n\cdot\n|\leq 1,\label{Dlda12}
\ee
where $ \B\n\cdot\n=\B_{ij}\n_i\n_j $. By simple calculations, for $ \Q\in\Ss_0\backslash(\cC_1\cup\cC_2) $, we deduce that
\be
D\phi_{\tau}(\Q)=\left\{\begin{aligned}
&r_*^{-1}(D\lda_1(\Q)+2D\lda_2(\Q)),\quad\quad r(\Q)\in(0,(1-2\tau)/2),\\
&r_*^{-1}(2D\lda_1(\Q)+D\lda_2(\Q))\beta_{\tau}(r(\Q))\\
&\quad\quad+r_*^{-1}(1-2r(\Q))D\lda_1(\Q)+(2-r(\Q))D\lda_2(\Q))\beta_{\tau}'(r(\Q)),\\
&\quad\quad\quad\quad\quad\quad\quad\quad\quad\quad\quad\quad\quad r(\Q)\in[(1-2\tau)/2,(1+2\tau)/2],\\
&r_*^{-1}(D\lda_1(\Q)-D\lda_2(\Q)),\quad\quad r(\Q)\in((1-2\tau)/2,1).\\
\end{aligned}\right.\label{Dphitauestimate2}
\ee
Direct computations yield that 
\be
\f{1-2\tau}{2}\leq\beta_{\tau}(r)\leq\f{6-5\tau}{12},\,\,0\leq|\beta_{\tau}'(r)|\leq 1,\label{estibetatau}
\ee
for $ r\in[(1-2\tau)/2,(1+2\tau)/2] $. This, together with \eqref{Dlda12} and \eqref{Dphitauestimate2}, implies \eqref{Dphi4es}. Next, we claim, for any $ \PP,\Q\in\Ss_0 $, and $ \tau\in(0,1/2) $,
\be
|\phi_{\tau}(\Q)-\phi_{\tau}(\PP)|\leq 5r_*^{-1}|\Q-\PP|.\label{Lipclaim}
\ee
This claim shows that $ \{\phi_{\tau}\}_{\tau\in(0,1/2)}\subset W^{1,\ift}(\Ss_0) $. We firstly note that $ \phi_{\tau} $ is continuous. Indeed, by Lemma \ref{reprelem}, we have, for any $ \tau\in(0,1/2) $, $ \phi_{\tau}\in C^2(\Ss_0\backslash\{\mathbf{O}\}) $. On the other hand, as $ |\Q|\to 0^+ $, we have $ s(\Q)\to 0^+ $ and $ r(\Q)\in[0,1] $, which implies that $ \phi_{\tau}(\Q)\to 0=\phi_{\tau}(\mathbf{O}) $. Next, we will divide the proof of the claim into several cases. We fix $ \tau\in(0,1/2) $.\smallskip

\noindent
\underline{\textbf{Case 1.}} $ \Q,\PP\in\cC_1\cup\cC_2 $. Now we have $ \phi_{\tau}(\Q)=\phi_{\tau}(\PP)=0 $ and there is nothing to prove.\smallskip

\noindent
\underline{\textbf{Case 2.}} $ \Q\in\cC_1\cup\cC_2 $ and $ \PP\notin\cC_1\cup\cC_2 $. Since $ \cC_1\cup\cC_2 $ is closed in $ \Ss_0 $, there exists $ \PP_0 $ such that $ |\PP-\PP_0|=\dist(\PP,\cC_1\cup\cC_2) $. Choosing $ \PP(t)=t\PP+(1-t)\PP_0 $ with $ t\in(0,1) $, we have $ \PP(t)\notin\cC_1\cup\cC_2 $ and $ \PP(t)\to\PP_0 $ as $ t\to 0^+ $. By using \eqref{Dphi4es}, we deduce
$$
|\phi_{\tau}(\PP(t))-\phi_{\tau}(\PP)|\leq 5r_*^{-1}|\PP(t)-\PP|\leq 5r_*^{-1}|\Q-\PP|.
$$
Since $ \phi_{\tau} $ is continuous, we let $ t\to 0^+ $ and obtain \eqref{Lipclaim}.\smallskip

\noindent
\underline{\textbf{Case 3.}} $ \Q\notin\cC_1\cup\cC_2 $ and $ \PP\in\cC_1\cup\cC_2 $. This case is almost the same as Case 2 and we omit the proof of it for simplicity.\smallskip

\noindent
\underline{\textbf{Case 4.}} $ \Q,\PP\notin\cC_1\cup\cC_2 $. Let $ \PP(t)=(1-t)\Q+t\PP $ with $ t\in[0,1] $. If $ \PP(t)\notin\cC_1\cup\cC_2 $ for any $ t\in[0,1] $, the claim follows directly from \eqref{Dphi4es}. If there is some $ t\in(0,1) $, such that $ \PP(t)\in\cC_1\cup\cC_2 $, we apply the result in Case 2 to get
\begin{align*}
|\phi_{\tau}(\Q)-\phi_{\tau}(\PP)|&\leq|\phi_{\tau}(\Q)-\phi_{\tau}(\PP(t))|+|\phi_{\tau}(\PP(t))-\phi_{\tau}(\PP)|\\
&\leq 5r_*^{-1}(|\Q-\PP(t)|+|\PP-\PP(t)|)\\
&=5r_*^{-1}|\Q-\PP|.
\end{align*}
Consequently, we complete the proof of this claim.
\end{proof}

\begin{lem}\label{fBl}
There exist constants $ c_0,c_1>0 $, and $ C_1>0 $ which only depend on $ \cA $ such that the following properties hold. For any $ \Q\in\Ss_0 $ and $ \tau\in(0,1/2) $,
\begin{align}
f_b(\Q)&\geq c_0\max_{i=0,1,2}\{(1-\phi_i(\Q))^2\},\label{1phi}\\
f_b(\Q)&\geq c_0\min\left\{(1-\phi_{\tau}(\Q))^2,\(1-\f{6}{6-5\tau}\phi_{\tau}(\Q)\)^2\right\}.\label{tauphi}
\end{align}
If $ \Q^*\in\cN $ and $ \PP\in(T_{\Q^*}\cN)_{\Ss_0}^{\perp} $, then
\be
c_1|\PP|^2\leq\left.\f{\ud^2}{\ud t^2}\right|_{t=0}f_b(\Q^*+t\PP)\leq C_1|\PP|^2.\label{d2es}
\ee
\end{lem}

\begin{cor}\label{fBc}
There exist constants $ c'>0 $ and $ \delta_0>0 $ depending only on $ \cA $ such that if $ \Q\in\Ss_0 $ satisfies $ \dist(\Q,\cN)<\delta_0 $, then
\begin{align}
f_b(\Q)&\geq c'\dist^2(\Q,\cN),\label{fBnond}\\
f_b(t\Q+(1-t)\varrho(\Q))&\leq c't^2f_b(\Q)\text{ for any }t\in[0,1].\label{fBconvex}
\end{align}
\end{cor}

\begin{proof}
This Corollary follows directly from the lemma \ref{fBl} and the Taylor expansion formula. By \eqref{eqdist}, there exists $ \delta_0>0 $ depending only only on $ \cA $ such that if $ \dist(\Q,\cN)<\delta_0 $, then $ \Q\in\Ss_0\backslash(\cC_1\cup\cC_2) $ and $ \varrho(\Q) $ is well defined. For such $ \Q $, set $ \Q^*=\varrho(\Q) $ and $ \PP=\Q-\Q^* $. We assume that $ |\PP|\neq 0 $ since for otherwise there is nothing to prove. Now define $ g(s):=f_b(\Q^*+s\PP/|\PP|) $. Applying Taylor expansion (taking a smaller $ \delta_0 $ if necessary), we have
$$
g(s)=g(0)+g'(0)t+\f{g''(0)}{2}s^2+R(s),\,\,R(s)\leq C_2|s|^3
$$
for some $ C_2>0 $ depending only on $ \cA $. Choosing $ \delta_0>0 $ such that $ C_2\delta_0<c_1/4 $, we can use \eqref{d2es} to obtain $ c_1\leq g''(0)\leq C_1 $ and $ c_1s^2/4\leq g(s)\leq Cs^2 $ for any $ s\in(0,\delta_0) $. Letting $ s=|\PP| $, we get $ f_b(\Q)=g(\PP)\geq c_1|\PP|^2/4 $. On the other hand, letting $ s=t|\PP| $ with $ t\in[0,1] $, we deduce that $
f_b(t\Q+(1-t)\Q^*)\leq Ct^2|\PP|^2\leq Ct^2f_b(\Q) $, which completes the proof.
\end{proof}

\begin{proof}[Proof of Lemma \ref{fBl}]
By the representation formula \eqref{representationQ}, we have
\begin{align*}
f_b(\Q)&=a_1-\f{a_2(s^2-st+t^2)}{3}+\f{a_4(s^2-st+t^2)^2}{9}+\f{4a_6(s^2-st+t^2)^3}{81}\\
&\quad\quad\,\,+\f{a_6'(2s^3-3s^2t-3st^2+2t^3)^2}{486}:=f(s,t),
\end{align*}
where $ s=s(\Q) $ and $ t=s(\Q)r(\Q) $. Here, $ f(s,t) $ is a function defined on $
D:=\{(s,t)\in\R^2:0\leq t\leq s\} $. Since $ f_b(\Q) $ attains its minimum on $ \cN $, $ f(s,t) $ attains its minimum at the point $ (2r_*,r_*) $. Moreover $ (2r_*,r_*) $ is the unique minimum  point of $ f(s,t) $ in $ D $. By simple calculations, we obtain
\begin{gather*}
f(2r_*,r_*)=\pa_sf(2r_*,r_*)=\pa_tf(2r_*,r_*)=0,\\
\pa_s^2f(2r_*,r_*)=\f{6a_4r_*^2+(24a_6+a_6')r_*^4}{3},\,\,\pa_t^2f(2r_*,r_*)=\f{4a_6'r_*^4}{3},\\
\det(D^2f(2r_*,r_*))=\f{8a_6'r_*^6(a_4+4a_6r_*^2)}{3},
\end{gather*}
which gives 
\be
C_0^{-1}\I\leq D^2f(2r_*,r_*)\leq C_0\I,\label{strictlyPositive}
\ee
for some $ C_0>0 $ depending only on $ \cA $. On the other hand, we note that by the definition of $ f(s,t) $, we deduce that $ |f(s,t)|\geq c_0(s^6+t^6) $ for sufficiently large $ |t|,|s|>0,(s,t)\in D $. This, together with \eqref{strictlyPositive} and the fact that $ (2r_*,r_*) $ is the unique minimum point of $ f(s,t) $ in $ D $ such that $ f(2r_*,r_*)=0 $, implies that
\be
f_b(\Q)\geq C((s-2r_*)^2+(sr-r_*)^2)\geq C\max_{i=1,2}\{(1-\phi_i(\Q))^2\}.\label{fb2es}
\ee
This also implies that $ f_b(\Q)\geq C(1-\phi_0(\Q))^2 $ by the definition of $ \phi_0 $. Next, we will show \eqref{tauphi}. By \eqref{fb2es} and the definition of $ \phi_{\tau} $, it suffices to prove that
\be
\begin{aligned}
&2((s-2r_*)^2+(sr-r_*)^2)\\
&\quad\quad\geq\left\{\begin{aligned}
&(r_*-sr)^2&\text{ if }&s\in\R_+,\,\,r\in[0,(1-2\tau)/2),\\
&\(r_*-\f{6s\beta_{\tau}(r)}{6-5\tau}\)^2&\text{ if }&s\in\R_+,\,\,r\in[(1-2\tau)/2,(1+2\tau)/2],\\
&(r_*-s(1-r))^2&\text{ if }&s\in\R_+,\,\,r\in((1+2\tau)/2,1].
\end{aligned}\right.\label{phi4red}
\end{aligned}
\ee
Moreover, we only need to consider the case that $ s\in\R_+ $ and $ r\in[(1-2\tau)/2,(1+2\tau)/2] $. Simple calculations yield that
$$
(s-2r_*)^2+(sr-r_*)^2-\f{1}{2}\(r_*-\f{6s\beta_{\tau}(r)}{6-5\tau}\)^2=a_{\tau}(r)s^2+b_{\tau}(r)s+\f{9r_*^2}{2},
$$
where
\begin{align*}
a_{\tau}(r)&:=1+r^2-\f{18\beta_{\tau}(r)^2}{(6-5\tau)^2},\\
b_{\tau}(r)&:=-4r_*-2r_*r+\f{6r_*\beta_{\tau}(r)}{6-5\tau}.
\end{align*}
We claim, for $ \tau\in(0,1/2) $ and $ \lda\in[-1,1] $,
\be
a_{\tau}\(\f{1-2\tau\lda}{2}\)\geq 0\text{ and }\left\{b_{\tau}\(\f{1-2\tau\lda}{2}\)\right\}^2-18r_*^2a_{\tau}\(\f{1-2\tau\lda}{2}\)\leq 0,\label{abclaim}
\ee
which completes the proof. By direct computations, we obtain
\begin{align*}
a_{\tau}\(\f{1-2\tau\lda}{2}\)&=1+\(\f{1-2\tau\lda}{2}\)^2-\f{(12-10\tau-15\tau\lda^2+\tau\lda^6)^2}{32(6-5\tau)^2}\\
&\geq 1+\(\f{1-2\tau\lda}{2}\)^2-\f{(12-10\tau)^2}{32(6-5\tau)^2}\\
&\geq 1-\f{1}{8}>0,
\end{align*}
where for the first inequality, we have used the fact that $ -14\leq\lda^6-15\lda^2\leq 0 $ for any $ \lda\in[-1,1] $. On the other hand, we note
\begin{align*}
&\left\{b_{\tau}\(\f{1-2\tau\lda}{2}\)\right\}^2-18r_*^2a_{\tau}\(\f{1-2\tau\lda}{2}\)\\
&\quad\quad=\f{r_*^2\tau^2\lda^2(\lda^5-15\lda+20\tau-24)(5\lda^5-75\lda-140\tau+168)}{8(6-5\tau)^2}.
\end{align*}
For $ \lda\in[-1,1] $, by \eqref{t515}, we have $ \lda^5-15\lda\in[-14,14] $ and then
\begin{align*}
\lda^5-15\lda+20\tau-24&\leq 14+\f{20}{2}-24=0,\\
5\lda^5-75\lda-140\tau+168&\geq-70-\f{140}{2}+168=28>0.
\end{align*}
Therefore, the claim \eqref{abclaim} is true. 

Finally, we show \eqref{d2es}. Since $ \mathrm{SO}(3) $ acts transitively on $ \cN $, we can assume $ \Q^*=r_*\diag\{1,-1,0\} $. Indeed, for $ \Q^*=r_*(\n\n-\m\m)\in\cN $ with $ (\n,\m)\in\M $, there exists a matrix $ \mathbf{R}\in\mathrm{SO}(3) $, such that $ \mathbf{R}\n=\e^{(1)} $ and $ \mathbf{R}\n=\e^{(2)} $. Consequently, $ \mathbf{R}\Q^*\mathbf{R}^{\T}=r_*\diag\{1,-1,0\} $ and $ f_b(\mathbf{R}\Q^*\mathbf{R}^{\T})=f_b(\Q^*) $. In view of the second property of Lemma \ref{tangentnormal}, we have $ \PP=\diag\{x,y,-x-y\} $. As a result, simple computations yield that
$$
\left.\f{\ud^2}{\ud t^2}\right|_{t=0}f_b(\Q_0+t\PP)=3a_6'r_*^4(x+y)^2+2r_*^2(a_4+4a_6r_*^2)(x-y)^2.
$$
Note that $ |\PP|^2=x^2+y^2+(x+y)^2 $, we can use Cauchy inequality to obtain
$$
c_1|\PP|^2\leq\left.\f{\ud^2}{\ud t^2}\right|_{t=0}f_b(\Q_0+t\PP)\leq C_1|\PP|^2,
$$
which completes the proof.
\end{proof}

\begin{lem}\label{TechLemma1}
Let $ U\subset\R^k $ with $ k\in\Z_+ $ be a domain and $ \Q\in C^1(U,\Ss_0) $. For any $ \tau\in(0,1/2) $,
\be
\begin{aligned}
|\na\Q|^2&\geq\f{r_*^2}{18}|\na(\phi_{\tau}\circ\Q)|^2+(\phi_{\tau}\circ\Q)^2|\na(\varrho\circ\Q)|^2,\quad\HH^k\text{-a.e. in }U,
\end{aligned}\label{naQdetau}
\ee
where we set 
\be
(\phi_{\tau}\circ\Q)|\na(\varrho\circ\Q)|(x)=0\text{ if }\Q(x)\in\cC_1\cup\cC_2.\label{conven0phi}
\ee
\end{lem}

\begin{rem}
By the definition of $ \varrho $, it is well defined in the set $ \Ss_0\backslash(\cC_1\cup\cC_2) $. This implies that the right hand side of \eqref{naQdetau} is always well defined.
\end{rem}

\begin{proof}[Proof of Lemma \ref{TechLemma1}]
Since there is the formula $
|\na\Q|^2=\sum_{i=1}^k|\pa_i\Q|^2 $, we only need to show \eqref{naQdetau} for $ k=1 $. If $ k\in\Z_{\geq 2} $, we can simply sum $ |\pa_i\Q|^2 $ up and obtain the result. Firstly, we fix $ \tau\in(0,1/2) $. By Lemma \ref{Lipphi4}, $ \phi_{\tau} $ is Lipschitz continuous, which implies that $ \phi_{\tau}\circ\Q\in W_{\loc}^{1,\ift}(U) $. Combining the fact that $ \phi_{\tau}\circ\Q=0 $ in $ \Q^{-1}(\cC_1\cup\cC_2) $, we have 
$$
(\phi_{\tau}\circ\Q)'=0,\quad\HH^1\text{-a.e. in }\Q^{-1}(\cC_1\cup\cC_2).
$$
This, together with \eqref{conven0phi}, implies that \eqref{naQdetau} is true in the set $ \Q^{-1}(\cC_1\cup\cC_2) $. Now, we fix $ x\in U\backslash\Q^{-1}(\cC_1\cup\cC_2) $. By Lemma \ref{reprelem}, we obtain 
\be
\Q=s\(\n\n-\f{1}{3}\I\)+sr\(\m\m-\f{1}{3}\I\),\label{Repre2}
\ee
where $ s>0 $, $ 0<r<1 $ and $ (\n,\m)\in\M $. In view of Lemma \ref{smoothLda}, $ s,r,\n,\m\in C^1 $ in a neighborhood of $ x $. It follows from the definition $ \varrho $ that $
\varrho\circ\Q=r_*(\n\n-\p\p) $, where $ \p=\n\times\m $. Now, $ \{|(\phi_i\circ\Q)'|^2\}_{i=1}^2 $, $ \{|(\varrho_i\circ\Q)'|^2\}_{i=1}^2 $, $ |(\phi_{\tau}\circ\Q)'|^2 $, $ |(\varrho\circ\Q)'|^2 $ and $ |\Q'|^2 $ can be computed in terms of $ s,r,\n,\m $ and their derivatives. Setting $ t=sr $, we get
\be
\begin{gathered}
|(\phi_1\circ\Q)'|^2=r_*^{-2}((s')^2-2s't'+(t')^2),\,\,|(\phi_2\circ\Q)'|^2=r_*^{-2}(t')^2,\\
|(\varrho_1\circ\Q)'|^2=2r_*^2|\n'|^2,\,\,|(\varrho_2\circ\Q)'|^2=2r_*^2|\p'|^2,\\
(\varrho_1\circ\Q)':(\varrho_2\circ\Q)'=-2r_*^2(\n'\cdot\p)^2,
\end{gathered}\label{rho1rho2sr}
\ee
and
\be
\begin{gathered}
|(\phi_{\tau}\circ\Q)'|^2=\left\{\begin{aligned}
&r_*^{-2}(t')^2&\text{ if }&s\in\R_+,\,\,r\in[0,(1-2\tau)/2),\\
&r_*^{-2}(s'\beta_{\tau}(r)+sr'\beta_{\tau}'(r))^2&\text{ if }&s\in\R_+,\,\,r\in[(1-2\tau)/2,(1+2\tau)/2],\\
&r_*^{-2}((s')^2-2s't'+(t')^2)&\text{ if }&s\in\R_+,\,\,r\in((1+2\tau)/2,1],
\end{aligned}\right.\\
|(\varrho\circ\Q)'|^2=2r_*^2(|\n'|^2+|\p'|^2+2(\n'\cdot\p)^2).
\end{gathered}\label{phitaurepr}
\ee
Moreover, we have
\be
|\Q'|^2=\f{2((s')^2-s't'+(t')^2)}{3}+2s^2(|\n'|^2+r^2|\m'|^2+2r(\n'\cdot\m)(\n\cdot\m')).\label{Qlowerprime}
\ee
By differentiating the orthogonality conditions for $ (\n,\m,\p) $, we obtain
$$
\n'=\alpha\m+\beta\p,\,\,\m'=-\alpha\n+\ga\p,\text{ and }\p'=-\beta\n-\ga\m
$$
for some continuous, real-valued functions $ \alpha,\beta,\gamma $. Consequently, we deduce
\begin{gather*}
|(\varrho_1\circ\Q)'|^2=2r_*^2(\al^2+\beta^2),\,\,
|(\varrho_2\circ\Q)'|^2=2r_*^2(\beta^2+\ga^2),\\
(\varrho_1\circ\Q)':(\varrho_2\circ\Q)'=-2r_*^2\beta^2\leq 0,
\end{gather*}
and then
\begin{align*}
&|\Q'|^2-\f{2((s')^2-s't'+(t')^2)}{3}\\
&\quad\quad=2s^2(1-r)^2(\al^2+\beta^2)+4s^2r(1-r)\beta^2+2s^2r^2(\ga^2+\beta^2)\\
&\quad\quad=r_*^{-2}s^2(1-r)^2|(\varrho_1\circ\Q)'|^2+r_*^{-2}s^2r^2|(\varrho_2\circ\Q)'|^2\\
&\quad\quad\quad\quad-2(r_*^{-1}sr)\cdot(r_*^{-1}s(1-r))((\varrho_1\circ\Q)':(\varrho_2\circ\Q)')\\
&\quad\quad=\sum_{i=1,2}(\phi_i\circ\Q)^2|(\varrho_i\circ\Q)'|^2-2(\phi_1\circ\Q)(\phi_2\circ\Q)((\varrho_1\circ\Q)':(\varrho_2\circ\Q)').
\end{align*}
This, together with \eqref{phi42phi3}, implies that
\be
\begin{aligned}
|\Q'|^2&\geq\f{2((s')^2-s't'+(t')^2)}{3}+(\phi_0\circ\Q)^2|(\varrho\circ\Q)'|^2\\
&\geq\f{2((s')^2-s't'+(t')^2)}{3}+(\phi_{\tau}\circ\Q)^2|(\varrho\circ\Q)'|^2.
\end{aligned}\label{23stgeq2}
\ee
In view of \eqref{phitaurepr}, there holds $ |(\phi_{\tau}\circ\Q)'|=|(\phi_0\circ\Q)'| $ if $ r\in[0,(1-2\tau)/2)\cup((1+2\tau)/2,1] $. Combining \eqref{rho1rho2sr}, we have
$$
\f{2((s')^2-s't'+(t')^2)}{3}\geq\f{r_*^2}{3}(|(\phi_1\circ\Q)'|^2+|(\phi_2\circ\Q)'|^2)\geq\f{r_*^2}{3}|(\phi_{\tau}\circ\Q)'|^2.
$$
This, together with \eqref{23stgeq2}, implies \eqref{naQdetau} for $ r\in[0,(1-2\tau)/2)\cup((1+2\tau)/2,1] $. On the other hand, for $ r\in[(1-2\tau)/2,(1+2\tau)/2] $, we have
\begin{align*}
|(\phi_{\tau}\circ\Q)'|^2&=r_*^{-2}(s'\beta_{\tau}(r)+sr'\beta_{\tau}'(r))^2\\
&\leq 2r_*^{-2}(s')^2(\beta_{\tau}(r))^2+2(sr')^2(\beta_{\tau}'(r))^2\\
&=2r_*^{-2}(s')^2(\beta_{\tau}(r))^2+2(t'-s'r)^2(\beta_{\tau}'(r))^2\\
&\leq r_*^{-2}(s')^2(2(\beta_{\tau}(r))^2+4r_*^{-2}r^2(\beta_{\tau}'(r))^2)+4(t')^2(\beta_{\tau}'(r))^2\\
&\leq 6r_*^{-2}((s')^2+(t')^2)\leq \f{1}{18r_*^2}\(\f{2((s')^2-s't'+(t')^2)}{3}\),
\end{align*}
where for the first, second and forth inequalities, we have used Cauchy inequality and for the third inequality, we have used \eqref{estibetatau}. Applying \eqref{23stgeq2} again, we can complete the proof of \eqref{naQdetau}.
\end{proof}

\begin{cor}\label{lowerQcorollary}
Let $ U\subset\R^k $ be a domain with $ k\in\Z_{\geq 2} $. For $ \tau\in(0,1/2) $, the map $ \sg_{\tau}:\Ss_0\to\Ss_0 $ given by
$$
\sg_{\tau}:\Q\mapsto\left\{\begin{aligned}
&\phi_{\tau}(\Q)\varrho(\Q)&\text{ if }&\Q\in\Ss_0\backslash(\cC_1\cup\cC_2)\\
&\mathbf{O}&\text{ if }&\Q\in \cC_1\cup\cC_2
\end{aligned}\right.
$$
is Lipschitz continuous. Moreover, for $ \Q\in H^1(U,\Ss_0) $, there holds $ \sg_{\tau}\circ\Q\in H^1(U,\Ss_0) $ and
\be
\f{1}{4}|\na(\sg_{\tau}\circ\Q)|^2\leq\f{r_*^2}{18}|\na(\phi_{\tau}\circ\Q)|^2+(\phi_{\tau}\circ\Q)^2|\na(\varrho\circ\Q)|^2\leq|\na\Q|^2\label{H1Qlowerupper}
\ee
holds $ \HH^k $-a.e. in $ U $.
\end{cor}
\begin{proof}
By using \eqref{Dvarrhoestimate}, \eqref{Dphi4es}, and \eqref{naQdetau}, we have, for any $ \Q\in\Ss_0\backslash(\cC_1\cup\cC_2) $,
\begin{align*}
|D\sg_{\tau}(\Q)|^2&=|D\phi_{\tau}(\Q)\varrho(\Q)+(\phi_{\tau}\circ\Q)D\varrho(\Q)|^2\\
&\leq 4r_*^2|D\phi_{\tau}(\Q)|^2+2(\phi_{\tau}\circ\Q)^2|D\varrho(\Q)|^2\leq C,
\end{align*}
where $ C>0 $ depends only on $ \cA $. Also, we can show that $ \sg_{\tau}\in C(\Ss_0,\Ss_0) $. By applying arguments in the proof of the property that $ \phi_{\tau}\in W^{1,\ift}(\Ss_0) $ in Lemma \ref{Lipphi4}, since $ \sg_{\tau}|_{\cC_1\cup\cC_2}=\mathbf{O} $, we deduce that $ \sg_{\tau}\in W^{1,\ift}(\Ss_0,\Ss_0) $ and the lower bound of \eqref{H1Qlowerupper} holds. The upper bound of \eqref{H1Qlowerupper} follows directly from smoothness approximation of $ H^1 $ maps by smooth maps and \eqref{naQdetau}.
\end{proof}
\subsection{Existence and smoothness of minimizers}

In this subsection, we first give some basic results for minimizers of \eqref{energy}. Next we will give some examples on the boundary conditions such that \eqref{assumptionbound1} and \eqref{assumptionbound} hold true. 

\begin{prop}[Existence of the minimizer]\label{ExistenceThm}
Let $ \om\subset\R^3 $ be a bounded Lipschitz domain and $ 0<\va<1 $. There exists $ \Q_{\va} $, being a global minimizer of the energy \eqref{energy} in the space $ H^1(\om,\Ss_0;\Q_{b,\va}) $ with $ \Q_{b,\va}\in H^{1/2}(\pa\om,\Ss_0) $.
\end{prop}

\begin{proof}
Fix $ 0<\va<1 $. Obviously, since $ \Q_{b,\va}\in H^{1/2}(\pa\om,\Ss_0) $, there exists $ \PP\in H^1(\om,\Ss_0) $ such that $ \PP|_{\pa\om}=\Q_{b,\va} $ and $ \|\PP\|_{H^1(\om)}\leq\|\Q_{b,\va}\|_{H^{1/2}(\pa\om)} $. 
By using Sobolev embedding theorem, when $ \om\subset\R^3 $ is bounded with a Lipschitz boundary, we have $ H^1(\om)\hookrightarrow L^6(\om) $. Consequently, $ 0\leq E_{\va}(\PP,\om)\leq C $, which implies that 
$$
0\leq E_{\min}=\min\{E_{\va}(\Q,\om):\Q\in H^1(\om,\Ss_0;\Q_{b,\va})\}<+\ift.
$$
Choosing $ \{\Q_{\va}^{(n)}\}_{n\in\Z_+}\subset H^1(\om,\Ss_0;\Q_{b,\va}) $ such that $ \lim_{n\to+\ift}E_{\va}(\Q_{\va}^{(n)},\om)=E_{\min} $, it follows from the definition of $ a_1 $ and $ f_b $ ($ f_b\geq 0 $) that $ \|\na\Q_{\va}\|_{L^2(\om)} $ is uniformly bounded with respect to $ n $. Furthermore, by Poincar\'{e} inequality, we have $
\|\Q_{\va}^{(n)}\|_{L^2(\om)}\leq C(\|\na\Q_{\va}^{(n)}\|_{L^2(\om)}+\|\Q_{b,\va}\|_{H^{1/2}(\pa\om)}) $, and then $ \{\Q_{\va}^{(n)}\}_{n\in\Z_+} $ is uniformly bounded in $ H^1(\om,\Ss_0) $. By using Sobolev embedding theorem once again, we can extract a subsequence from $ \{\Q_{\va}^{(n)}\}_{n\in\Z_+} $, without changing the notation, and obtain $ \Q_{\va}\in H^1(\om,\Ss_0) $ such that as $ n\to+\ift $,
\begin{align*}
\Q_{\va}^{(n)}&\wc\Q_{\va}\text{ weakly in }H^1(\om,\Ss_0),\\
\Q_{\va}^{(n)}&\to\Q_{\va}\text{ strongly in }L^p(\om,\Ss_0)\text{ for any }1\leq p<6,\text{ a.e. in }\om.
\end{align*}
and $ \sup_{n\geq 1}\|\Q_{\va}^{(n)}\|_{L^6(\om)}<+\ift $. By applying Fatou lemma and the property of weak convergence, we deduce $ E_{\va}(\Q_{\va},\om)\leq\liminf_{n\to+\ift}E_{\va}(\Q_{\va}^{(n)},\om)=E_{\min} $. Finally, by the compactness of the trace operator, i.e., $ H^1(\om,\Ss_0)\hookrightarrow\hookrightarrow L^2(\pa\om,\Ss_0) $,we can conclude that $ \Q_{\va}|_{\pa\om}=\Q_{b,\va} $.
\end{proof}

\begin{prop}\label{boundednessofminimizer}
Let $ \om\subset\R^3 $ be a bounded Lipschitz domain and $ 0<\va<1 $. Assume that $ \Q_{\va} $ is a global minimizer of \eqref{energy} in the space $ H^1(\om,\Ss_0;\Q_{b,\va}) $ with $ \Q_{b,\va}\in(H^{1/2}\cap L^{\ift})(\pa\om,\Ss_0) $. There holds
$$
\|\Q_{\va}\|_{L^{\ift}(\om)}\leq\max\{\sqrt{2}r_*,\|\Q_{b,\va}\|_{L^{\ift}(\pa\om)}\}.
$$
\end{prop}
\begin{proof}
We define a map $ \pi:\Ss_0\to\Ss_0 $ by
$$
\pi(\Q):=\left\{\begin{aligned}
&C_0|\Q|^{-1}\Q&\text{ if }&|\Q|\geq C_0,\\  
&\Q&\text{ if }&|\Q|<C_0,
\end{aligned}\right.
$$
where $ C_0:=\max\{\sqrt{2}r_*,\|\Q_{b,\va}\|_{L^{\ift}(\pa\om)}\} $. Simple calculations yield that
\begin{align*}
Df_b(\Q):\Q&=-a_2|\Q|^2+a_4|\Q|^4+a_6|\Q|^6+a_6'(\tr \Q^3)^2\\
&\geq-a_2|\Q|^2+a_4|\Q|^4+a_6|\Q|^6,
\end{align*}
where for the last inequality we have used the fact $ \sqrt{6}|\tr\Q^3|\leq|\Q|^3 $ (see Lemma 1 of \cite{MZ10} for details). If $ |\Q|>\sqrt{2}r_* $, we can obtain from the definition of $ r_* $ that $ Df_b(\Q):\Q>0 $. By this, we deduce that $ f_b(\pi(\Q))\leq f_b(\Q) $ for any $ \Q\in\Ss_0 $. On the other hand, $ \pi $ is $ 1 $-Lipschitz continuous. Hence, $ \pi\circ\Q_{\va}\in H^1(\om,\Ss_0;\Q_{b,\va}) $ and satisfies $ |\na(\pi\circ\Q_{\va})|\leq|\na\Q_{\va}| $ a.e. in $ \om $. Choosing $ \pi\circ\Q_{\va} $ as a competitor, we have
$$
E_{\va}(\pi\circ\Q_{\va},\om)\leq \int_{\om}\(\f{1}{2}|\na\Q_{\va}|^2+\f{1}{\va^2}f_b(\Q_{\va})\)\ud x=E_{\va}(\Q_{\va},\om),
$$
with strict inequality if the set where $ |\Q_{\va}|>C_0 $ has positive measure. Since $ \Q_{\va} $ is a global minimizer in $ H^1(\om,\Ss_0;\Q_{b,\va}) $, we conclude that $ |\Q_{\va}|\leq C_0 $ a.e. in $ \om $. 
\end{proof}

\begin{prop}\label{logboundenergy}
Let $ \om\subset\R^3 $ be a bounded smooth domain. Assume that $ \{\Q_{\va}\}_{0<\va<1} $ are global minimizers of \eqref{energy} in the space $ H^1(\om,\Ss_0;\Q_{b,\va}) $ with $ \Q_{b,\va}\in H^{1/2}(\pa\om,\cN) $. The following properties hold.
\begin{enumerate}
\item If 
\be
\sup_{0<\va<1}\|\Q_{b,\va}\|_{H^{1/2}(\pa\om,\cN)}\leq C_0\label{QbvaH12esbou}
\ee
for some $ C_0>0 $, then there exists $ C>0 $ depending only on $ \cA,C_0 $, and $ \om $ such that
\be
E_{\va}(\Q_{\va},\om)\leq C\(\log\f{1}{\va}+1\).\label{logH1}
\ee
\item If $ \Q_{b,\va}=\Q_b\in H^{1/2}(\pa\om,\cN) $ for any $ 0<\va<1 $, then
\be
E_{\va}(\Q_{\va},\om)\leq C,\label{boundedH1com}
\ee
where $ C>0 $ depends only on $ \cA,\om $, and $ \Q_b $. Moreover, there exists a subsequence $ \va_n\to 0^+ $ and $ \Q_0\in H^1(\om,\cN) $ such that $ \Q_0 $ is a minimizer of the problem
\be
\min\left\{\int_{\om}|\na\Q|^2\ud x:\Q\in H^1(\om,\cN),\,\,\Q|_{\pa\om}=\Q_b\right\},\label{limitlift}
\ee
$ \Q_{\va_n}\to\Q_0 $ in $ H^1(\om,\Ss_0) $ as $ n\to+\ift $, and
\be
\lim_{n\to+\ift}\f{1}{\va^2}\int_{\om}f_b(\Q_{\va_n})\ud x=0.\label{fvanlimit}
\ee
\end{enumerate}
\end{prop}

\begin{lem}[\cite{C17}, Lemma 69]\label{MNprime}
Let $ \cY $ be a smooth Riemannian $ m $-manifold and $ \cX $ be a compact smooth Riemannian $ n $-submanifold of $ \cY $. There exist $ \delta,C>0 $, depending only on $ \cX,\cY $ such that for any non-increasing function $ h:\R_+\to\R_+ $, there holds
\be
\int_{U_{\delta}}h(\dist_{\cY}(x,\cX))\ud\HH^m(x)\leq C\int_0^{C\delta}s^{m-n-1}h(s)\ud s,\label{hsmnminus}
\ee
where $
U_{\delta}:=\{x\in\cY:\dist_{\cY}(x,\cX)<\delta\} $ is the $ \delta $-neighborhood of $ \cX $ in $ \cY $. 
\end{lem}

\begin{proof}[Proof of Proposition \ref{logboundenergy}]
For fixed $ 0<\va<1 $, choose $ \U_{\va}\in H^1(\om,\Ss_0) $ such that
$$
\left\{\begin{aligned}
-\Delta\U_{\va}&=0&\text{ in }&\om,\\
\U_{\va}&=\Q_{b,\va}&\text{ on }&\pa\om.
\end{aligned}\right.
$$
By \eqref{QbvaH12esbou}, $ \{\U_{\va}\}_{0<\va<1} $ is a bounded sequence in $ (H^1\cap L^{\ift})(\om,\Ss_0) $, which satisfies
\be
\|\U_{\va}\|_{H^1(\om)}+\|\U_{\va}\|_{L^{\ift}(\om)}\leq C,\label{UvanablaestimateOm}
\ee
where $ C>0 $ depends only on $ \cA,C_0 $ and $ \om $. Choose $ 0<\delta_1<1 $ sufficiently small, depending only on $ \cA $, such that the maps $ \{\varrho_{\A}\}_{|\A|<\delta_1}\subset C^1(\cN,\cN) $ defined by $ \varrho_{\A}(\Q)=\varrho(\Q-\A) $ are diffeomorphism. We define the function $ \eta_{\va}(r):=\va^{-1}r $ if $ 0\leq r<\va $, and $ \eta_{\va}(r)=1 $ if $ r\geq\va $. For $ \A\in\Ss_0 $, $ |\A|<\delta_1 $, let $
\U_{\va,\A}:=(\eta_{\va}\circ\phi_{1/4})(\U_{\va}-\A)\varrho(\U_{\va}-\A) $, where $ \varrho:\Ss_0\backslash(\cC_1\cup\cC_2)\to\cN $ and $ \phi_{1/4}:\Ss_0\to\R $ are given by \eqref{rho1rho2rho} and \eqref{phitaudefinition}. Since $ \phi_{1/4}(\U_{\va}(x)-\A)=0 $ when $ \U_{\va}(x)-\A\in\cC_1\cup\cC_2 $, $ \U_{\va,\A} $ is always well defined. It follows from Lemma \ref{varrhop} and Corollary \ref{lowerQcorollary} that $ \U_{\va,\A}\in(H^1\cap L^{\ift})(\om,\Ss_0) $. Using \eqref{Dphi4es}, we have
$$
|\na\U_{\va,\A}|^2\leq C((\eta_{\va}'\circ\phi_{1/4})^2(\U_{\va}-\A)|\na\U_{\va}|^2+(\eta_{\va}\circ\phi_{1/4})^2(\U_{\va}(\cdot)-\A)|\na(\varrho(\U_{\va}-\A))|^2).
$$
By Corollary \ref{lowerQcorollary} again, we deduce
$$
|\na\U_{\va,\A}|^2\leq C\((\eta_{\va}'\circ\phi_{1/4})(\U_{\va}-\A)+\f{(\eta_{\va}\circ\phi_{1/4})^2(\U_{\va}-\A)}{\phi_{1/4}^2(\U_{\va}-\A)}\)|\na\U_{\va}|^2.
$$
Combining the fact that $ f_b(\U_{\va,\A})\leq C\chi_{\{\phi(\U_{\va}-\A)\leq\va\}} $, there holds
\be
\begin{aligned}
&E_{\va}(\U_{\va,\A},\om)\\
&\quad\leq C\int_{\om}\left\{\(\f{\chi_{\{\phi_{1/4}(\U_{\va}-\A)\geq\va\}}}{\phi_{1/4}^2(\U_{\va}-\A)}+\f{\chi_{\{\phi_{1/4}(\U_{\va}-\A)\leq\va\}}}{\va^2}\)|\na\U_{\va}|^2+\f{\chi_{\{\phi_{1/4}(\U_{\va}-\A)\leq\va\}}}{\va^2}\right\}\ud x.
\end{aligned}\label{Evaphi14es}
\ee
For simplicity, for $ \Q\in\Ss_0 $, we denote $ B_r^{\Ss_0}(\Q):=\{\PP\in\Ss_0:|\PP-\Q|<r\} $, which is the ball in $ \Ss_0 $ with radius $ r $ and center $ \Q $. If the center $ \Q=\mathbf{O} $, we denote it by $ B_r^{\Ss_0} $. In view of \eqref{UvanablaestimateOm}, we can choose $ R>0 $, depending only on $ \cA,C_0 $, and $ \om $, such that $
B_{\delta_1}^{\Ss_0}(\U_{\va}(x))\in B_R^{\Ss_0} $ for a.e. $ x\in\om $.
Define $ K:=B_R^{\Ss_0} $ and $ K_{\va}:=K\cap\{\Q:\phi_{1/4}(\Q)\leq\va\} $. Integrating \eqref{Evaphi14es} with respect to $ \A\in B_{\delta_1}^{\Ss_0} $ and applying Fubini theorem, we have
\begin{align*}
\int_{B_{\delta_1}^{\Ss_0}}E_{\va}(\U_{\va,\A},\om)\ud\HH^5(\A)\leq C\int_{\om}\left\{\(\int_{K\backslash K_{\va}}\f{\ud\HH^5(\B)}{\phi_{1/4}^2(\B)}+\f{\HH^5(K_{\va})}{\va^2}\)|\na\U_{\va}|^2+\f{\HH^5(K_{\va})}{\va^2}\right\}\ud x.
\end{align*}
We claim
\be
\HH^5(K_{\va})\leq C\va^2\text{ and }\int_{K\backslash K_{\va}}\f{\ud\HH^5(\B)}{\phi_{1/4}^2(\B)}\leq C\(\log\f{1}{\va}+1\),\label{KvaH5}
\ee
where $ C>0 $ depends only on $ \cA,C_0 $ and $ \om $. If this claim is true, by \eqref{UvanablaestimateOm}, we have
\begin{align*}
\int_{B_{\delta}^{\Ss_0}}E_{\va}(\U_{\va,\A},\om)\ud\HH^5(\A)&\leq C\(\log\f{1}{\va}+1\)\|\na\U_{\va}\|_{L^2(\om)}^2+C\leq C\(\log\f{1}{\va}+1\).
\end{align*}
By using average arguments, we can choose some $ \A_0\in\Ss_0 $ such that $ |\A_0|<\delta_1 $,
\be
E_{\va}(\U_{\va},\om)\leq C\delta_1^{-5}\(\log\f{1}{\va}+1\)\leq C\(\log\f{1}{\va}+1\).\label{deltamius5}
\ee
and
\be
\U_{\va}=\varrho(\Q_{b,\va}-\A_0)\text{ on }\pa\om.\label{UvaAboundary}
\ee
Define $ \mathcal{U}:=\{\lda\Q:\lda\in\R_{+}:\Q\in\cN\} $. The map $ F:\mathcal{U}\to\mathcal{U} $ given by $ F(\lda\Q):=\lda^{-1}\varrho_{\A_0}^{-1}(\Q) $ for $ \lda\in\R_+,\Q\in\cN $ is Lipschitz and well defined. Also we note that $ f_b(F(\PP))=f_b(\PP) $ for any $ \PP\in\mathcal{U} $ since if $ \PP=\lda r_*(\n\n-\m\m) $ with $ (\n,\m)\in\M $, the value $ f_b(\PP) $ is independent of $ \n $ and $ \m $. In view of \eqref{UvaAboundary} and the fact that $ \U_{\va,\A_0}\in\cN $, we have that $ \PP_{\va}:=F(\U_{\va,\A_0}) $ is well defined and $ \PP_{\va}\in H^1(\om,\Ss_0;\Q_{b,\va}) $. Moreover, it follows from \eqref{deltamius5} that $
E_{\va}(\PP_{\va},\om)\leq C(\log(1/\va)+1) $. By using the assumption that $ \Q_{\va} $ is a global minimizers, we deduce \eqref{logH1}. 

Next, let us prove the claim \eqref{KvaH5}. Actually, we prove the following property. For any non-increasing, non-negative function $ g:\R_+\to\R_+ $, there holds
\be
\int_{B_R^{\Ss_0}}(g\circ\phi_{1/4})(\Q)\ud\HH^5(\Q)\leq C_1\int_0^{C_2}(s+s^4)g(s)\ud s,\label{ss4gs}
\ee
where $ C_1,C_2>0 $ depends only on $ \cA,C_0 $, and $ \om $. This directly implies \eqref{KvaH5} by choosing $ g(s)=\chi_{(0,\va)}(s) $ and $ g(s)=\va^{-2}\chi_{(0,\va)}(s)+s^{-2}\chi_{[0,+\ift)}(s) $ for $ s\in\R_+ $. Before we prove \eqref{ss4gs}, we firstly introduce some notations. For $ r>0 $, define
$$
\dist_r(\Q,\PP):=\inf\left\{\int_0^1|\ga'(t)|\ud t:\ga\in C^1([0,1],\pa B_r^{\Ss_0}),\,\,\ga(0)=\Q,\,\,\ga(1)=\PP\right\}
$$
for any $ \Q,\PP\in\pa B_r^{\Ss_0} $. Let $ \cN_r':=(\cC_1\cup\cC_2)\cap\pa B_r^{\Ss_0} $. We note that $ \cN_r' $ is a smooth Riemannian $ 2 $-submanifold of $ \pa B_r^{\Ss_0} $ since in $ \Ss_0 $, the only singularity of $ \cC_1\cup\cC_2 $ is $ \mathbf{O} $. There holds, for some constant $ C_3 $ depending only on $ \cA $, such that
\be
\phi_0(\Q)\geq C_3\dist_{|\Q|}(\Q,\cN_{|\Q|}')\text{ for any }\Q\in\Ss_0.\label{phiQaldist}
\ee
To show this, we fix $ \Q\in\Ss_0 $ and $ \PP\in\cN_{|\Q|}' $, and we can use the the formula \eqref{representationQ} to obtain
$$
\Q=s\(\n\n-\f{1}{3}\I\)+sr\(\m\m-\f{1}{3}\I\),\,\,\PP=t\(\p\p-\f{1}{3}\I\)
$$
for some $ (\n,\m)\in\M $, $ \p\in\Ss^2 $, $ s\in\R_+ $, $ t\in\R $ and $ r\in[0,1] $. Simple calculations yield that 
$$
|\Q-\PP|^2=\f{2s^2(r^2-r+1)-2st(1+r)+2t^2}{3}+r(\m\cdot\p)^2
$$
and $ t^2=s^2(r^2-r+1) $. Consequently, we have
\begin{align*}
\dist^2(\Q,\cN_{|\Q|}')&=\left\{\begin{aligned}
&\f{2s^2\sqrt{r^2-r+1}((1-r)^2-(\sqrt{r^2-r+1}-1)^2)}{3}&\text{ if }&r\geq\f{1}{2}\\
&\f{2s^2\sqrt{r^2-r+1}(r^2-(\sqrt{r^2-r+1}-1)^2)}{3}&\text{ if }&r<\f{1}{2}
\end{aligned}\right.\\
&\leq\f{2\min\left\{s^2r^2,s^2(1-r)^2\right\}}{3}=\f{2r_*^2}{3}\phi_0^2(\Q).
\end{align*}
In view of the fact that $
\dist_{|\Q|}(\Q,\PP)\leq C|\Q-\PP| $ for any $ \Q,\PP\in\Ss_0 $, $ |\Q|=|\PP| $, we deduce \eqref{phiQaldist}. By using the first property of Lemma \ref{Lipphi0lem}, \eqref{phi42phi3}, \eqref{phiQaldist} and the assumption that $ g $ is non-increasing, we can obtain
\begin{align*}
\int_{B_R^{\Ss_0}}(g\circ\phi_{1/4})(\Q)\ud\HH^5(\Q)&\leq\int_0^R\rho^4\int_{\pa B_1^{\Ss_0}}g(\rho\phi_0(\Q))\ud\HH^4(\Q)\ud\rho\\
&\leq\int_0^R\rho^4\int_{\pa B_1^{\Ss_0}}g(C_3\rho\dist_1(\Q,\cN_1'))\ud\HH^4\ud\rho.
\end{align*}
Applying Lemma \ref{MNprime} to $ \cX=\cN_1' $ and $ \cY=\pa B_1^{\Ss_0} $, and $ h(s)=g(C_3\rho s) $, we can find $ \delta,C>0 $ such that
\begin{align*}
\int_{B_R^{\Ss_0}}(g\circ\phi_{1/4})(\Q)\ud\HH^5(\Q)&\leq\int_0^R\rho^4\left\{\(\int_{U_{\delta}}+\int_{V_{\delta}}\)g(C_3\rho\dist_1(\Q,\cN_1'))\ud\HH^4(\Q)\right\}\ud\rho\\
&\leq C\int_0^R\rho^4\(\int_0^{C\delta}sg(C_3\rho s)\ud s+g(C_3\rho\delta)\HH^4(V_{\delta})\)\ud\rho,
\end{align*}
where $ U_{\delta} $ is the $ \delta $-neighborhood of $ \cN_1' $ in $ \pa B_1^{\Ss_0} $ and $ V_{\delta}:=\pa B_1^{\Ss_0}\backslash U_{\delta} $. Finally, it follows from Fubini theorem and change of variables that
\begin{align*}
\int_{B_R^{\Ss_0}}(g\circ\phi_{1/4})(\Q)\ud\HH^5(\Q)&\leq \f{C}{C_3^2}\int_0^R\rho^2\(\int_0^{CC_3\rho\delta}tg(t)\ud t\)\ud\rho+\f{C}{C_3^5\delta^5}\HH^4(V_{\delta})\int_0^{C_3\delta R}t^4g(t)\ud t\\
&\leq C_1\int_0^{C_2}(t+t^4)g(t)\ud t,
\end{align*}
which completes the proof of \eqref{ss4gs}. 

If $ \Q_{b,\va}=\Q_b $ for any $ 0<\va<1 $, we choose $ \PP\in H^1(\om,\cN) $ as a minimizer of the problem \eqref{limitlift}. Such minimizer exists through standard arguments like Proposition \ref{ExistenceThm}. Obviously, we have $ \PP|_{\pa\om}=\Q_b $ on $ \pa\om $ and the minimizing property of $ \Q_{\va} $ implies that
\be
E_{\va}(\Q_{\va},\om)\leq E_{\va}(\PP,\om)\leq\f{1}{2}\int_{\om}|\na\PP|^2\ud x,\label{EvaQvawhQ}
\ee
which completes the proof of \eqref{boundedH1com}. Now we can choose a subsequence $ \va_n\to 0^+ $ as $ n\to+\ift $ and $ \Q_0\in H^1(\om,\Ss_0) $ such that
\begin{align*}
&\Q_{\va_n}\wc\Q_0\text{ weakly in }H^1(\om,\Ss_0),\\
&\Q_{\va_n}\to\Q_0\text{ strongly in }L^2(\om,\Ss_0)\text{ and a.e. in }\om.
\end{align*}
In view of compactness of the trace operator ($ H^1(\om,\Ss_0)\hookrightarrow\hookrightarrow L^2(\pa\om,\Ss_0) $) and the property that $ \Q_{\va}|_{\pa\om}=\Q_b $, we have $ \Q_0|_{\pa\om}=\Q_b $. Using \eqref{EvaQvawhQ} and Fatou lemma, there holds
$$
\int_{\om}f_b(\Q_0)\ud x\leq\liminf_{n\to+\ift}\int_{\om}f_b(\Q_{\va_n})\ud x\leq\liminf_{n\to+\ift}\f{\va_n^2}{2}\int_{\om}|\na\PP|^2\ud x=0,
$$
which implies that $ \Q_0\in H^1(\om,\cN) $. By \eqref{EvaQvawhQ} and the property of the weak convergence, we deduce
\begin{align*}
\int_{\om}|\na\PP|^2\ud x\leq\int_{\om}|\na\Q_0|^2\ud x\leq\liminf_{n\to+\ift}\int_{\om}|\na\Q_{\va_n}|^2\ud x\leq \limsup_{n\to+\ift}\int_{\om}|\na\Q_{\va_n}|^2\ud x\leq\int_{\om}|\na\PP|^2\ud x.
\end{align*}
This implies that all the inequalities in the above formula are actually equalities, and then $ \Q_0 $ is a minimizer of the problem \eqref{limitlift}. Also, we have
\be
\lim_{n\to+\ift}\int_{\om}|\na\Q_{\va_n}|^2\ud x=\int_{\om}|\na\Q_0|^2\ud x\label{nablaconvergence}
\ee
and then $ \Q_{\va_n}\to\Q_0 $ strongly in $ H^1(\om,\Ss_0) $ as $ n\to+\ift $. Taking $ \Q_0 $ as a competitor, there holds
$$
E_{\va}(\Q_{\va},\om)\leq\f{1}{2}\int_{\om}|\na\Q_0|^2\ud x.
$$
This, together with \eqref{nablaconvergence}, implies \eqref{fvanlimit}.
\end{proof}

For $ 0<\va<1 $, if $ \Q_{\va} $ is a local minimizer of the minimizing problem \eqref{energy}, then it satisfies the Euler-Lagrange equation
\be
-\va^2\Delta \Q_{\va}+\Psi(\Q_{\va})=0\label{E-L}
\ee
in weak sense, where the map $ \Psi=(\Psi_{ij})_{i,j=1}^3:\Ss_0\to\Ss_0 $ is defined as
\be
\begin{aligned}
\Psi_{ij}(\Q):=(f_b)_{ij}(\Q)-\f{a_6'}{3}(\tr\Q^3)(\tr\Q^2)\delta_{ij}
\end{aligned}\label{PsiQ}
\ee
for any $ 1\leq i,j\leq 3 $, or, more precisely,
$$
\Psi(\Q)=(-a_2+a_4(\tr\Q^2)+a_6(\tr \Q^2)^2)\Q+a_6'(\tr\Q^3)\(\Q^2-\f{1}{3}(\tr\Q^2)\I\).
$$
That is to say, for any $ \PP\in C_c^{\ift}(\om,\mathbb{M}^{3\times 3}) $,
$$
\int_{\om}(\na \Q_{\va}:\na\PP+\Psi(\Q_{\va}):\PP)\ud x=0.
$$

\begin{cor}\label{smooth}
Let $ \om\subset\R^3 $ be a bounded domain. Assume that $ \Q_{\va} $ satisfies \eqref{E-L} in weak sense and $ \|\Q_{\va}\|_{L^{\ift}(\om)}\leq C_0 $ for some constant $ C_0>0 $. The following properties hold.
\begin{enumerate}
\item $ \Q_{\va}\in C^{\ift}(\om,\Ss_0) $. For any $ U\subset\subset\om $, $ \va\|\na\Q_{\va}\|_{L^{\ift}(U)}\leq C $, where $ C>0 $ depends only on $ \cA,C_0 $, and $ U $.
\item If $ \om $ is a smooth domain and $ \Q_{\va}|_{\pa\om}=\Q_b $ for some $ \Q_b\in C^{\ift}(\pa\om,\Ss_0) $, then $ \Q_{\va}\in C^{\ift}(\ol{\om},\Ss_0) $ such that $ \va\|\na\Q_{\va}\|_{L^{\ift}(\om)}\leq C $, where where $ C>0 $ depends only on $ \cA,C_0,\Q_b $, and $ \om $.
\end{enumerate}
\end{cor}

\begin{proof}
The proof directly follows from the regularity theory of elliptic equations. One can see \cite{HL97} for references.
\end{proof}

\subsection{Some useful identities}

In this subsection, we will give some identities of ourLandau-de Gennes model, which are important in the analysis of the fine structure of line defect.

\begin{lem}[Stress-energy identity]\label{StressThm}
Let $ \om\subset\R^3 $ be a bounded domain. Assume that $ \Q_{\va}\in C^{\ift}(\om,\Ss_0) $ and satisfies \eqref{E-L}. There holds
\be
\pa_j(e_{\va}(\Q_{\va})\delta_{ij}-\pa_i\Q_{\va}:\pa_j\Q_{\va})=0\label{StressEnergy}
\ee
in $ \om $, for any $ i=1,2,3 $.
\end{lem}
\begin{proof}
For $ i=1,2,3 $, simple calculations yield that
\begin{align*}
&\pa_j(e_{\va}(\Q_{\va})\delta_{ij}-\pa_i\Q_{\va}:\pa_j\Q_{\va})\\
&\quad\quad=\pa_i\pa_k\Q_{\va}:\pa_k\Q_{\va}+\f{1}{\va^2}(f_b)_{pq}(\Q_{\va})\pa_i(\Q_{\va})_{pq}-\pa_i\pa_j\Q_{\va}:\pa_j\Q_{\va}-\pa_i\Q_{\va}:\pa_j\pa_j\Q_{\va}\\
&\quad\quad=\pa_i\pa_k\Q_{\va}:\pa_k\Q_{\va}+\Delta \Q_{\va}:\pa_i \Q_{\va}-\pa_i\pa_j\Q_{\va}:\pa_j\Q_{\va}-\pa_i\Q_{\va}:\Delta \Q_{\va}=0,
\end{align*}
where for the second inequality, we have used \eqref{E-L} and the fact that $ \tr(\pa_i \Q_{\va})=0 $ for $ i=1,2,3 $.
\end{proof}

By using Stress-energy identity, we can derive the following Pohozaev identity.

\begin{lem}[Pohozaev identity]\label{lemPohozaev}
Let $ \om\subset\R^3 $ be a bounded domain and $ U\subset\subset\om $ be a Lipschitz subdomain with $ x_0
\in U $. Assume that $ \Q_{\va}\in C^{\ift}(\om,\Ss_0) $ and satisfies \eqref{E-L}. There holds
\be
\begin{aligned}
&\f{1}{2}\int_{U}|\na \Q_{\va}|^2\ud x+\int_{U}\f{3}{\va^2}f_b(\Q_{\va})\ud x+\int_{\pa U}\nu(x)\cdot(x-x_0)|\pa_{\nu}\Q_{\va}|^2\ud\HH^2\\
&\quad\quad=\int_{\pa U}\nu(x)\cdot(x-x_0)e_{\va}(\Q_{\va})\ud\HH^2-\int_{\pa U}\pa_{\nu}\Q_{\va}:(\na \Q_{\va}\cdot\mathbb{P}_{\pa U}(x-x_0))\ud\HH^2,
\end{aligned}\label{Pohozaev}
\ee
where $ \nu(x) $ is the unit outward normal vector to $ \pa U $ at $ x $, $ \pa_{\nu}\Q_{\va}=\nu\cdot\na\Q_{\va} $ and $ \mathbb{P}_{\pa U}(x-x_0) $ is the component of $ x-x_0 $ that is tangent to $ \pa U $.
\end{lem}

\begin{proof}
Up to a translation, we assume that $ x_0=0 $. Firstly, by simple calculations, we note that
\be
\pa_{\nu}\Q_{\va}:(x\cdot\na\Q_{\va})=\pa_{\nu}\Q_{\va}:(\mathbb{P}_{\pa U}(x)\cdot\na\Q_{\va})+(x\cdot\nu)|\pa_{\nu}\Q_{\va}|^2.\label{nux}
\ee
In view of Stress-energy identity \eqref{StressEnergy} and integration by parts, we have
\begin{align*}
0&=\int_Ux_i\pa_j(e_{\va}(\Q_{\va})\delta_{ij}-\pa_i\Q:\pa_j\Q)\ud x\\
&=\int_U(x_j\pa_je_{\va}(\Q_{\va})-x_i\pa_j(\pa_i\Q_{\va}:\pa_j\Q_{\va}))\ud x\\
&=\int_{U}(|\na\Q_{\va}|^2-3e_{\va}(\Q_{\va}))\ud x+\int_{\pa U}\(e_{\va}(\Q_{\va})(x\cdot\nu)-(x\cdot\na\Q_{\va}):(\pa_{\nu}\Q_{\va})\)\ud\HH^2.
\end{align*}
This, together with \eqref{nux}, implies the identity \eqref{Pohozaev}.
\end{proof}

Applying Pohozaev identity, we can obtain the following monotonicity formula.

\begin{cor}[Monotonicity formula]\label{Monotonicity}
Let $ \om\subset\R^3 $ be a bounded domain. Assume that $ \Q_{\va}\in C^{\ift}(\om,\Ss_0) $ and satisfies \eqref{E-L}. If $ x_0\in\om $ and $ 0<r_1<r_2<\dist(x_0,\pa\om) $, then
\be
\f{1}{r_1}E_{\va}(\Q_{\va},B_{r_1}(x_0))\leq\f{1}{r_2}E_{\va}(\Q_{\va},B_{r_2}(x_0)).\label{Mo}
\ee
\end{cor}

\begin{proof}
Up to a translation, we assume that $ x_0=0 $. Choosing $ U=B_{\rho} $ with $ 0<\rho<\dist(x_0,\pa\om) $ in \eqref{Pohozaev}, using $ \nu(x)=x/\rho $ on $ \pa B_{\rho} $, we obtain
\begin{align*}
-\int_{B_{\rho}}e_{\va}(\Q_{\va})\ud x+\rho\int_{\pa B_{\rho}}e_{\va}(\Q_{\va})\ud\HH^2=\f{1}{\rho}\int_{\pa B_{\rho}}|x\cdot\na\Q_{\va}|^2\ud x+\f{2}{\va^2}\int_{B_{\rho}}f_b(\Q_{\va})\ud x.
\end{align*}
By this, we can calculate that
\begin{align*}
\f{\pa}{\pa\rho}\(\f{1}{\rho}E_{\va}(\Q_{\va},B_{\rho})\)&=-\f{1}{\rho^2}\(\int_{B_{\rho}}e_{\va}(\Q_{\va})\ud x-\rho\int_{\pa B_{\rho}}e_{\va}(\Q_{\va})\ud\HH^2\)\\
&=\f{1}{\rho}\int_{\pa B_{\rho}}|\pa_{\nu}\Q_{\va}|^2\ud x+\f{2}{\va^2\rho^2}\int_{B_{\rho}}f_b(\Q_{\va})\ud x\geq 0,
\end{align*}
which completes the proof.
\end{proof}

\section{A priori estimates for the Euler-Lagrange equation}\label{AprioriSec}

In this section, for $ 0<\va<1 $, we will provide some a priori estimates for $ \Q_{\va} $, the classical solution of \eqref{E-L}, namely $ -\va^2\Delta\Q_{\va}+\Psi(\Q_{\va})=0 $. Using Lemma \ref{minimalp}, we can observe that $ \Q\in\cN $ if and only if $ \Q^3-r_*^2\Q=\mathbf{O} $. Additionally, when $ \Q\in\cN $, we have $ \tr\Q^2=2r_*^2 $. Therefore, the functions $ \Q^3-r_*^2\Q $ and $ \tr\Q^2-2r_*^2 $ are essential in characterizing the ``distance" between $ \Q $ and $ \cN $. In light of this, we introduce
\begin{align}
\zeta(\Q)&:=|\Q^3-r_*^2\Q|^2\label{zetaQ}\\
\xi(\Q)&:=(\tr\Q^2-2r_*^2)^2.\label{xiQ}
\end{align}

\subsection{Lower order estimates}
We start with some elementary estimates for functions of $ \Q\in\Ss_0 $ in a small tubular neighborhood of the manifold $ \cN $. These estimates will be useful in the study for the solutions of \eqref{E-L}. Specifically, Lemmas \ref{smallreg1} and \ref{smallreg2} fall under  more general results of \cite{CLR18}. However, we still present the direct proof here for the sake of completeness. In the proofs, we introduce some quantities that will be used in studying higher regularity results.

\begin{lem}\label{fBcomparison}
There exist two constants $ \delta_0>0 $ and $ C>0 $ depending only $ \cA $ such that  if $ \dist(\Q,\cN)<\delta_0 $, then
\begin{align}
C^{-1}\dist^2(\Q,\cN)&\leq \zeta(\Q),|\Psi(\Q)|^2,f_b(\Q)\leq C\dist^2(\Q,\cN),\label{uPsifQnon}\\
0&\leq \xi(\Q)\leq C\dist^2(\Q,\cN),\label{vBound}
\end{align}
and
\be
C^{-1}\dist^2(\Q,\cN)\leq\zeta_{ij}(\Q)\Psi_{ij}(\Q)\leq C\dist^2(\Q,\cN),\label{uPsiQnon}
\ee
where $ \Psi(\Q),\zeta(\Q) $, and $ \xi(\Q) $ are defined by \eqref{PsiQ}, \eqref{zetaQ}, and \eqref{xiQ} respectively.
\end{lem}

\begin{proof}
Assume that the set of eigenvalues of the matrix $ \Q $ is $ \{x,y,-x-y\} $. Define
\begin{align*}
F_1(x,y)&:=\zeta(\Q),\,\,F_2(x,y):=|\Psi(\Q)|^2,\,\,
F_3(x,y):=f_b(\Q),\\
F_4(x,y)&:=\pa\zeta_{ij}(\Q)\Psi_{ij}(\Q),\,\,F_5(x,y):=\xi(\Q).
\end{align*}
Now, we only need to show that there are $ \delta_i>0 $ ($ i=1,2,...,6 $) such that 
\begin{align}
C^{-1}\dist^2(\Q,\cN)\leq F_j(x,y)&\leq C\dist^2(\Q,\cN),\quad j\in\{1,2,3,4\},\label{F14non}\\
0\leq F_5(x,y)&\leq C\dist^2(\Q,\cN),\label{F5non}
\end{align}
for $ (x,y) $ in any of the following balls:
\be
\begin{aligned}
&B_{\delta_1}^2(r_*,-r_*),\,\,B_{\delta_2}^2(r_*,0),\,\,B_{\delta_3}^2(-r_*,0),\\
&B_{\delta_4}^2(-r_*,r_*),\,\,B_{\delta_5}^2(0,r_*),\,\,B_{\delta_6}^2(0,-r_*).
\end{aligned}\label{ball16}
\ee
Indeed, choosing $ \delta_0>0 $ sufficiently small, by the continuity of eigenvalues with respect to $ \Q $, we have, if $ \dist(\Q,\cN)<\delta_0 $, then $ (x,y) $ must in one of the balls given by \eqref{ball16}. Consequently, \eqref{F14non} and \eqref{F5non} directly imply \eqref{uPsifQnon}, \eqref{vBound}, and \eqref{uPsiQnon}. 

For simplicity, we only consider the case that $ (x,y)\in B_{\delta_1}^2(r_*,-r_*) $ and other cases can be dealt with by almost the same argument. By direct calculations and \eqref{rstar}, we have
\begin{gather*}
F_j(r_*,-r_*)=\pa_xF_j(r_*,-r_*)=\pa_yF_j(r_*,-r_*)=0,\quad j\in\{1,2,3,4,5\},\\
\pa_x^2F_1(r_*,-r_*)=\pa_y^2F_1(r_*,-r_*)=10r_*^2,\\
\det(D^2F_1(r_*,-r_*))=96r_*^4,\\
\pa_x^2F_2(r_*,-r_*)=\pa_y^2F_2(r_*,-r_*)=4r_*^4(4a_4^2+32a_4a_6r_*^2+(64a_6^2+3(a_6')^2)r_*^4),\\
\det(D^2F_2(r_*,-r_*))=768r_*^{12}(a_4+4a_6r_*^2)^2(a_6')^2,\\
\pa_x^2F_3(r_*,-r_*)=\pa_y^2F_3(r_*,-r_*)=2a_4r_*^2+(8a_6+3a_6')r_*^4,\\
\det(D^2F_3(r_*,-r_*))=24r_*^6(a_4+ 4a_6'r_*^2)a_6',\\
\pa_x^2F_4(r_*,-r_*)=\pa_y^2F_4(r_*,-r_*)=8r_*^6(4a_4+(16a_6+3a_6')r_*^2),\\
\det(D^2F_4(r_*,-r_*))=3072r_*^{14}(a_4+4a_6r_*^2)a_6',
\end{gather*}
and
$$
\pa_x^2F_5(r_*,-r_*)=\pa_x^2F_5(r_*,-r_*)=8r_*^2,\,\,\det(D^2F_5(r_*,-r_*))=0.
$$
The above results implies that there is a constant $ C_0>0 $, depending only on $ \cA $ such that  
\be
\begin{aligned}
C_0^{-1}\I_2\leq D^2F_{j}(r_*,-r_*)&\leq C_0\I_2,\quad j\in\{1,2,3,4\},\\
0\leq D^2F_5(r_*,-r_*)&\leq C_0\I_2. 
\end{aligned}\label{nonmatrix}
\ee
Now by Taylor expansion, for $ j=1,2,3,4,5 $, if $ \delta_1=\delta_1(\cA)>0 $ is sufficiently small, and $ (x,y)\in B_{\delta_1}^2(r_*,-r_*) $, we have
\be
\begin{aligned}
F_j(x,y)&=\pa_x^2F_j(r_*,-r_*)(x-r_*)^2+2\pa_{xy}^2F_j(r_*,-r_*)(x-r_*)(x+r_*)\\
&\quad\quad+\pa_y^2F_j(r_*,-r_*)(x+r_*)^2+R_{j}(x,y),
\end{aligned}\label{Taylor}
\ee
where $ \{R_j(x,y)\}_{j=1}^5 $ are remainders, satisfying 
$$ 
|R_j(x,y)|\leq\f{1}{2C_0}((x-r_*)^2+(y+r_*)^2).
$$
This, together with \eqref{eqdist} and \eqref{Taylor}, implies \eqref{F14non} and \eqref{F5non} as desired.
\end{proof}

\begin{lem}\label{minimalHessian}
If $ \Q^*\in\cN $ and $ \PP\in\Ss_0 $, then
$$
(f_b)_{ij,pq}(\Q^*)\PP_{ij}\PP_{pq}\geq 0,\,\,\zeta_{ij,pq}(\Q^*)\PP_{ij}\PP_{pq}\geq 0.
$$
\end{lem}
\begin{proof}
Given that both $ f_b $ and $ \zeta $ achieve their minimums (both $ 0 $) on the manifold $ \cN $, it follows that the Hessian matrices of these functions, when restricted to $ \cN $, are non-negative. This observation directly leads to the desired result.
\end{proof}

Next, we will consider solutions of $ -\va^2\Delta\Q_{\va}+\Psi(\Q_{\va})=0 $ in $ U\cap B_r(x_0) $ with boundary data given on $ \pa U\cap B_r(x_0) $ for appropriate domain $ U $, $ x_0\in\pa U $, and $ r>0 $. 

For $ \Q_{\va}\in C^{\ift} $ satisfying $ -\va^2\Delta\Q_{\va}+\Psi(\Q_{\va})=0 $, we define two functions
\begin{align}
\Y_{\va}&:=\va^{-2}(\Q_{\va}^3-r_*^2\Q_{\va}),\label{Xva}\\
h_{\va}&:=\va^{-2}(\tr(\Q_{\va}^2)-2r_*^2).\label{hva}
\end{align}

\begin{lem}\label{LiftQNll1}
Let $ 0<\va<1 $. There exists $ \delta_0>0 $ depending only on $ \cA $ such that the following properties hold. 
\begin{enumerate}
\item If $ \Q_{\va}\in C^{\ift}(B_r,\Ss_0) $ is a solution of $
-\va^2\Delta\Q_{\va}+\Psi(\Q_{\va})=0 $ in $ B_r $, and satisfies the estimate $
\|\dist(\Q_{\va},\cN)\|_{L^{\ift}(B_r)}<\delta_0 $, then
$$
\|(|\Delta\Q_{\va}|+\va^{-2}f_b^{1/2}(\Q_{\va})+|h_{\va}|+|\Y_{\va}|)\|_{L^{\ift}(B_{r/2})}\leq C(\|\na\Q_{\va}\|_{L^{\ift}(B_r)}+r^{-2}).
$$
where $ C>0 $ depends only on $ \cA $.
\item If $ U $ is a bounded Lipschitz domain with $ r_{U,0},M_{U,0} $, and $ \Q_{\va}\in C^{\ift}(U\cap B_r(x_0),\Ss_0)\cap C^0(\ol{U\cap B_r(x_0)},\Ss_0) $ is a solution of 
$$
\left\{\begin{aligned}
-\va^2\Delta\Q_{\va}+\Psi(\Q_{\va})&=0&\text{ in }&U\cap B_r(x_0),\\
\Q_{\va}&\in\cN&\text{ on }&\pa U\cap B_r(x_0)
\end{aligned}\right.
$$
for some $ x_0\in\pa U $ and $ 0<r<r_{U,0} $, satisfying $ \|\dist(\Q_{\va},\cN)\|_{L^{\ift}(U\cap B_r(x_0))}<\delta_0 $,
then
$$
\|(|\Delta\Q_{\va}|+\va^{-2}f_b^{1/2}(\Q_{\va})+|h_{\va}|+|\Y_{\va}|)\|_{L^{\ift}(U\cap B_{r/2}(x_0))}\leq C(\|\na\Q_{\va}\|_{L^{\ift}(U\cap B_r(x_0))}+r^{-2}),
$$
where $ C>0 $ depends only on $ \cA $.
\end{enumerate}
\end{lem}

\begin{proof}
For simplicity, we only prove the second property and give some remarks for the other one at the end. Up to a translation, we assume that $ x_0=0 $. By Remark \ref{rhorem}, if $ 0<\delta_0<1 $ is sufficiently small, depending only on $\cA$, then the projection $ \varrho $ makes sense and $ |\Q_{\va}-\Q_{\va}^{*}|=\dist(\Q_{\va},\cN) $ with $ \Q_{\va}^*=\varrho(\Q_{\va}) $. In particular, in $ U\cap B_r $,
\be
|\Q_{\va}|\leq\dist(\Q_{\va},\cN)+|\Q_{\va}^*|\leq 1+\sqrt{2}r_*.\label{deltabound}
\ee
Moreover, by using \eqref{uPsifQnon}, we have
\be
C^{-1}\zeta_{\va}\leq|\Q_{\va}-\Q_{\va}^*|^2\leq C\zeta_{\va}\text{ in } U\cap B_r,\label{uQUB}
\ee
where $ \zeta_{\va}:=\zeta(\Q_{\va}) $. Now, direct computations yield that 
\be
\Delta \zeta_{\va}=\zeta_{ij}(\Q_{\va})\Delta(\Q_{\va})_{ij}+\zeta_{ij,pq}(\Q_{\va})\pa_{k}(\Q_{\va})_{ij}\pa_{k}(\Q_{\va})_{pq}.\label{DeltauQ}
\ee
By Taylor expansion, we get
\be
\begin{aligned}
\zeta_{ij,pq}(\Q_{\va})&=\zeta_{ij,pq}(\Q_{\va}^*)+\zeta_{ij,pq,mn}(\Q_{\va}^*)(\Q_{\va}-\Q_{\va}^*)_{mn}+\mathcal{R}^{ijpq}(\Q_{\va},\Q_{\va}^*)
\end{aligned}\label{pa2uQ}
\ee
in $  U\cap B_r $, where $ \mathcal{R}^{ijpq} $ is a remainder satisfying
\be
|\mathcal{R}^{ijpq}(\Q_{\va},\Q_{\va}^*)|\leq C|\Q_{\va}-\Q_{\va}^*|^2.\label{EsR}
\ee
Now, using \eqref{E-L}, \eqref{uPsiQnon}, \eqref{uQUB}, \eqref{pa2uQ}, \eqref{EsR} and Lemma \ref{minimalHessian}, we have
\begin{align*}
\va^2\Delta\zeta_{\va}&=\va^2\zeta_{ij}(\Q_{\va})\Delta(\Q_{\va})_{ij}+\va^2\zeta_{ij,pq}(\Q_{\va})\pa_{k}(\Q_{\va})_{ij}\pa_{k}(\Q_{\va})_{pq}\\
&\geq C\zeta_{\va}+\va^2\zeta_{ij,pq}(\Q_{\va}^*)\pa_{k}(\Q_{\va})_{ij}\pa_{k}(\Q_{\va})_{pq}-C\va^2\|\na\Q_{\va}\|_{L^{\ift}(U\cap B_r)}^2|\Q_{\va}-\Q_{\va}^*|\\
&\geq C\zeta_{\va}-C\va^2\zeta_{\va}^{1/2}\|\na\Q_{\va}\|_{L^{\ift}(U\cap B_r)}^2\\
&\geq C\zeta_{\va}-C\va^4\|\na\Q_{\va}\|_{L^{\ift}(U\cap B_r)}^4,
\end{align*}
Since $ \Q_{\va}\in\cN $ on $ \pa  U\cap B_r $, we have $ \zeta_{\va}=0 $ on $ \pa U\cap B_r $. Applying \eqref{deltabound} and Corollary \ref{aub} to $ \zeta_{\va} $ with $ L=\va^2 $, $ \ga=2 $, we can obtain
\be
\|\zeta_{\va}\|_{L^{\ift}(U\cap B_{r/2})}\leq C\va^4(\|\na\Q_{\va}\|_{L^{\ift}(U\cap B_r)}^4+r^{-4}).\label{uQes}
\ee
In view of Lemma \ref{fBcomparison}, we can choose smaller $ 0<\delta_0<1 $ such that when $ \dist(\Q_{\va},\cN)<\delta_0 $, there holds
$$
\xi(\Q_{\va})=|\va^2h_{\va}|^2,f_b(\Q_{\va}),|\va^2\Delta\Q_{\va}|^2=|\Psi(\Q_{\va})|^2\leq C\zeta_{\va}=|\va^2\Y_{\va}|^2,
$$
and the result follows directly from this and \eqref{uQes}. For the interior case, one only need to use almost the same argument as above and replace Corollary \ref{aub} by Lemma \ref{au}.
\end{proof}

\begin{prop}\label{DeltaQstarformula}
Let $ U\subset\R^3 $ be a bounded domain and $ 0<\va<1 $. If $ \Q_{\va}\in C^{\ift}(U,\Ss_0) $ satisfies $ -\va^2\Delta\Q_{\va}+\Psi(\Q_{\va})=0 $ and $ \phi_0(\Q_{\va})>0 $ in $ U $, then $ \Q_{\va}^*=\varrho(\Q_{\va})\in C^{\ift}(U,\cN) $ is well defined and there holds
\be
\begin{aligned}
&\Delta\Q_{\va}^*+(\mu_{\va}^{(1)})^{-1}\mu_{\va}^{(2)}((\Delta\Q_{\va}^*)\Q_{\va}^*+\Q_{\va}^*\Delta\Q_{\va}^*)\\
&\quad\quad=-\f{1}{2r_*^2}\tr(\na\Q_{\va}^*\na\Q_{\va}^*)\Q_{\va}^*-\f{3}{r_*^4}\tr(\na\Q_{\va}^*\na\Q_{\va}^*\Q_{\va}^*)\wh{\Q}_{\va}\\
&\quad\quad\quad\quad+\f{(\mu_{\va}^{(1)})^{-1}\mu_{\va}^{(2)}}{r_*^4}(r_*^2\tr(\na\Q_{\va}^*\na\Q_{\va}^*\Q_{\va}^*)\Q_{\va}^*-3\tr(\na\Q_{\va}^*\na\Q_{\va}^*(\Q_{\va}^*)^2)\wh{\Q}_{\va}\\
&\quad\quad\quad\quad\quad\quad\quad\quad\quad\quad+2r_*^2\tr(\na\Q_{\va}^*\na\Q_{\va}^*)\wh{\Q}_{\va}+2r_*^4\na\Q_{\va}^*\na\Q_{\va}^*),
\end{aligned}\label{longfor}
\ee
where
$$
\mu_{\va}^{(1)}:=\f{1}{2r_*^2}\tr(\Q_{\va}\Q_{\va}^*),\,\,\mu_{\va}^{(2)}:=\f{3}{2r_*^4}\tr(\Q_{\va}(\Q_{\va}^*)^2),\text{ and }\wh{\Q}_{\va}:=(\Q_{\va}^*)^2-\f{2r_*^2}{3}\I.
$$
Moreover, there exist $ 0<\delta_0<1 $ and $ C>0 $, depending only on $ \cA $, such that if $ \dist(\Q_{\va}^*,\cN)<\delta_0 $, then
\be
|\Delta\Q_{\va}^*|\leq C|\na\Q_{\va}|^2.\label{Deltanablacontrol}
\ee
\end{prop}

\begin{proof}
In view of Remark \ref{rhorem}, $ \Q_{\va}^* $ is well defined in $ U $ and smooth. In view of Lemma \ref{tangentnormal} and Remark \ref{normalremQ}, we have
\be
\begin{aligned}
\Q_{\va}&=\Q_{\va}^*+\left\{(\Q_{\va}-\Q_{\va}^*):\f{\Q_{\va}^*}{\sqrt{2}r_*}\right\}\f{\Q_{\va}^*}{\sqrt{2}r_*}+\left\{(\Q_{\va}-\Q_{\va}^*):\f{\sqrt{6}\wh{\Q}_{\va}}{2r_*^2}\right\}\f{\sqrt{6}\wh{\Q}_{\va}}{2r_*^2}\\
&=\mu_{\va}^{(1)}\Q_{\va}^*+\mu_{\va}^{(2)}\wh{\Q}_{\va},
\end{aligned}\label{ProjectionPva}
\ee
where we have used the property that
\be
\tr(\Q_{\va}^*)=\tr((\Q_{\va}^*)^3)=\tr(\Q_{\va})=0\text{ and }\tr((\Q_{\va}^*)^2)=2r_*^2.\label{Pvatr}
\ee
Applying Lemma \ref{tangentnormal}, we deduce
\be
\Q_{\va}^*,\wh{\Q}_{\va}\in (T_{\Q_{\va}^*}\cN)_{\Ss_0}^{\perp}.\label{TanPstar}
\ee
Since $ \Q_{\va}^*\in\cN $ and $ \phi_0(\Q_{\va})>0 $ in $ U $, we have $ \mu_{\va}^{(1)}>0 $. To compute $ \Delta\Q_{\va}^* $, we note that $
\Delta\Q_{\va}^*=(\Delta\Q_{\va}^*)^{\perp}+(\Delta\Q_{\va}^*)^{\parallel} $, where $ (\Delta\Q_{\va}^*)^{\perp}\in (T_{\Q_{\va}^*}\cN)_{\Ss_0}^{\perp} $ and $ (\Delta\Q_{\va}^*)^{\parallel}\in T_{\Q_{\va}^*}\cN $. In view of \eqref{SecondFundamentalForm}, it is widely known (see Proposition 1.3.1 of \cite{LW08}) that
\be
\begin{aligned}
(\Delta\Q_{\va}^*)^{\perp}&=\Pi_{\cN}(\na\Q_{\va}^*,\na\Q_{\va}^*)(\Q_{\va}^*)\\
&=-\f{1}{2r_*^2}\tr(\na\Q_{\va}^*\na\Q_{\va}^*)\Q_{\va}^*-\f{3}{r_*^4}\tr(\na\Q_{\va}^*\na\Q_{\va}^*\Q_{\va}^*)\wh{\Q}_{\va}.
\end{aligned}\label{perpQvastar}
\ee
Moreover, it follows from \eqref{Pvatr} and Remark \ref{normalremQ} that
\begin{align}
\tr(\Delta\Q_{\va}^*\Q_{\va}^*)&=\tr((\Delta\Q_{\va}^*)^{\perp}\Q_{\va}^*)=-\tr(\na\Q_{\va}^*\na\Q_{\va}^*),\label{Qva1star}\\
\tr(\Delta\Q_{\va}^*(\Q_{\va}^*)^2)&=\tr(\Delta\Q_{\va}^*\wh{\Q}_{\va})=\tr((\Delta\Q_{\va}^*)^{\perp}\wh{\Q}_{\va})=-2\tr(\na\Q_{\va}^*\na\Q_{\va}^*\Q_{\va}^*).\label{Qva1star2}
\end{align}
Taking $ \Delta $ for both sides of \eqref{ProjectionPva} to $ T_{\Q_{\va}^*}\cN $, we can obtain from \eqref{TanPstar} that
\be
(\Delta\Q_{\va})^{\parallel}=\mu_{\va}^{(1)}(\Delta\Q_{\va}^*)^{\parallel}+\mu_{\va}^{(2)}(\Delta\wh{\Q}_{\va})^{\parallel},\label{Deltaparallel}
\ee
By simple calculations, \eqref{TanPstar} and \eqref{ProjectionPva}, we deduce
\begin{align*}
\Q_{\va},\Q_{\va}^2-\f{1}{3}(\tr\Q_{\va}^2)\I\in\op{span}\{\Q_{\va}^*,\wh{\Q}_{\va}\}.
\end{align*}
This, together with \eqref{E-L}, implies that $
(\Delta\Q_{\va})^{\parallel}=0 $. Now we can obtain from \eqref{Deltaparallel} that
\be
(\Delta\Q_{\va}^*)^{\parallel}=-(\mu_{\va}^{(1)})^{-1}\mu_{\va}^{(2)}(\Delta\wh{\Q}_{\va})^{\parallel}.\label{QvaDeltamu1}
\ee
By Remark \ref{normalremQ}, \eqref{Qva1star}, and \eqref{Qva1star2}, we have
\begin{align*}
(\Delta\wh{\Q}_{\va})^{\parallel}&=(2\na\Q_{\va}^*\na\Q_{\va}^*+(\Delta\Q_{\va}^*)\Q_{\va}^*+\Q_{\va}^*\Delta\Q_{\va}^*)^{\parallel}\\
&=2\na\Q_{\va}^*\na\Q_{\va}^*+(\Delta\Q_{\va}^*)\Q_{\va}^*+\Q_{\va}^*\Delta\Q_{\va}^*-\f{1}{r_*^2}\tr(\na\Q_{\va}^*\na\Q_{\va}^*\Q_{\va}^*)\Q_{\va}^*\\
&\quad\quad-\f{3}{r_*^4}\tr(\na\Q_{\va}^*\na\Q_{\va}^*\wh{\Q}_{\va})\wh{\Q}_{\va}-\f{1}{r_*^2}\tr(\Delta\Q_{\va}^*(\Q_{\va}^*)^2)\Q_{\va}^*-\f{3}{r_*^4}\tr(\Delta\Q_{\va}^*\Q_{\va}^*\wh{\Q}_{\va})\wh{\Q}_{\va}\\
&=2\na\Q_{\va}^*\na\Q_{\va}^*+(\Delta\Q_{\va}^*)\Q_{\va}^*+\Q_{\va}^*\Delta\Q_{\va}^*+\f{1}{r_*^2}\tr(\na\Q_{\va}^*\na\Q_{\va}^*\Q_{\va}^*)\Q_{\va}^*\\
&\quad\quad-\f{3}{r_*^4}\tr(\na\Q_{\va}^*\na\Q_{\va}^*(\Q_{\va}^*)^2)\wh{\Q}_{\va}+\f{2}{r_*^2}\tr(\na\Q_{\va}^*\na\Q_{\va}^*)\wh{\Q}_{\va}.
\end{align*}
This, together with \eqref{Deltaparallel}, implies that
\begin{align*}
&(\Delta\Q_{\va}^*)^{\parallel}+(\mu_{\va}^{(1)})^{-1}\mu_{\va}^{(2)}((\Delta\Q_{\va}^*)\Q_{\va}^*+\Q_{\va}^*\Delta\Q_{\va}^*)\\    
&\quad\quad=\f{(\mu_{\va}^{(1)})^{-1}\mu_{\va}^{(2)}}{r_*^4}(r_*^2\tr(\na\Q_{\va}^*\na\Q_{\va}^*\Q_{\va}^*)\Q_{\va}^*-3\tr(\na\Q_{\va}^*\na\Q_{\va}^*(\Q_{\va}^*)^2)\wh{\Q}_{\va}\\
&\quad\quad\quad\quad\quad\quad\quad\quad+2r_*^2\tr(\na\Q_{\va}^*\na\Q_{\va}^*)\wh{\Q}_{\va}+2r_*^4\na\Q_{\va}^*\na\Q_{\va}^*)
\end{align*}
Combined with \eqref{perpQvastar}, we can complete the proof of \eqref{longfor}. To prove \eqref{Deltanablacontrol}, we note that $ \mu_{\va}^{(1)}\to 1 $ and $ \mu_{\va}^{(2)}\to 0 $ as $ \dist(\Q_{\va},\cN)\to 0^+ $. Thus \eqref{Deltanablacontrol} directly follows from \eqref{longfor} by using the property that $ |\na\Q_{\va}^*|\leq C|\na\Q_{\va}| $.
\end{proof}

\begin{lem}\label{BoundaryUse}
Let $ 0<\va<1 $. Assume that $ U $ is a bounded $ C^{2,1} $ domain with $ r_{U,2} $ and $ M_{U,2} $. There exists $ \delta_0>0 $ depending only on $ \cA $ such that the following property holds. If $ \Q_{\va}\in C^{\ift}(U\cap B_r(x_0),\Ss_0)\cap C^2(\ol{U\cap B_r(x_0)},\Ss_0) $ is a solution of
$$
\left\{\begin{aligned}
-\va^2\Delta\Q_{\va}+\Psi(\Q_{\va})&=0&\text{ in }&U\cap B_r(x_0),\\
\Q_{\va}&\in\cN&\text{ on }&\pa  U\cap B_r(x_0),
\end{aligned}\right.
$$
for some $ x_0\in\pa U $ and $ 0<r<r_{U,2} $, satisfying $
\|\dist(\Q_{\va},\cN)\|_{L^{\ift}(U\cap B_r)}<\delta_0 $, then $ \Q_{\va}^*=\varrho(\Q_{\va}) $ is well defined and 
\begin{align*}
\|\na(\Q_{\va}-\Q_{\va}^*)\|_{L^{\ift}(\pa U\cap B_{r/2}(x_0))}&\leq Cr^{-1}\|\Q_{\va}-\Q_{\va}^*\|_{L^{\ift}(U\cap B_r(x_0))}\\
&\quad\quad+Cr^{-1/2}\|\na\Q_{\va}\|_{L^2(U\cap B_r(x_0))}\|\na\Q_{\va}\|_{L^{\ift}(U\cap B_r(x_0))}\\
&\quad\quad+Cr^{1/4}\|\na\Q_{\va}\|_{L^2(U\cap B_r(x_0))}^{1/2}\|\na\Q_{\va}\|_{L^{\ift}(U\cap B_r(x_0))}^{3/2},
\end{align*}
where $ C>0 $ depends only on $ \cA,r_{U,2} $, and $ M_{U,2} $.
\end{lem}
\begin{proof}
Up to a translation, we assume that $ x_0=0 $. Firstly, define $ \mathbf{Z}_{\va}:=\va^2\Y_{\va} $. By \eqref{uQUB}, if $ 0<\delta_0<1 $ is sufficiently small, then for $ \Q_{\va}^*=\varrho(\Q_{\va}) $, there holds
\be
C^{-1}|\mathbf{Z}_{\va}|\leq|\Q_{\va}-\Q_{\va}^*|=\dist(\Q_{\va},\cN)\leq C|\mathbf{Z}_{\va}|\text{ in } U\cap B_r.\label{wtYva}
\ee
We note that $ |\mathbf{Z}_{\va}|\in W^{1,\ift}(U\cap B_r)\cap C^{\ift}(U\cap B_r\cap\{|\mathbf{Z}_{\va}|>0\}) $ and 
\be
\na|\mathbf{Z}_{\va}|=0\text{ a.e. in } U\cap B_r\cap\{|\mathbf{Z}_{\va}|=0\},\label{paX0}
\ee
Denoting $ \zeta_{\va}=|\mathbf{Z}_{\va}|^2 $ and using \eqref{DeltauQ}, we have, in $ U\cap B_r\cap\{|\mathbf{Z}_{\va}|>0\} $,
\be
\begin{aligned}
&\Delta(|\mathbf{Z}_{\va}|)=\f{\zeta_{ij}(\Q_{\va})\Delta(\Q_{\va})_{ij}}{2|\mathbf{Z}_{\va}|}\\
&\quad\quad+\f{2(\zeta_{ij,pq}\zeta)(\Q_{\va})\pa_{k}(\Q_{\va})_{ij}\pa_{k}(\Q_{\va})_{pq}-(\zeta_{ij}\zeta_{pq})(\Q_{\va})\pa_{k}(\Q_{\va})_{ij}\pa_{k}(\Q_{\va})_{pq}}{4|\mathbf{Z}_{\va}|^3}:=I_1+I_2.
\end{aligned}\label{DeltaYvawt}
\ee
By Taylor expansion, when $ |\Q_{\va}-\Q_{\va}^*|<\delta_0 $ is sufficiently small, we have
\begin{align*}
\zeta_{\va}(\Q_{\va})&=\f{1}{2}\zeta_{ij,pq}(\Q_{\va}^*)(\Q_{\va}-\Q_{\va}^*)_{ij}(\Q_{\va}-\Q_{\va}^*)_{pq}+O(|\Q_{\va}-\Q_{\va}^*|^3),\\
\zeta_{ij,pq}(\Q_{\va})&=\zeta_{ij,pq}(\Q_{\va}^*)+O(|\Q_{\va}-\Q_{\va}^*|),\\
(\zeta_{ij}\zeta_{pq})(\Q_{\va})&=(\zeta_{ij,mn}\zeta_{pq,k\ell})(\Q_{\va}^*)(\Q_{\va}-\Q_{\va}^*)_{mn}(\Q_{\va}-\Q_{\va}^*)_{k\ell}+O(|\Q_{\va}-\Q_{\va}^*|^3).
\end{align*}
In view of Lemma \ref{minimalHessian}, we obtain
$$
(\zeta_{ij,pq}(\Q_{\va}^*)\PP_{ij}\PP_{pq})(\zeta_{ij,pq}(\Q_{\va}^*)\wh{\PP}_{ij}\wh{\PP}_{pq})-(\zeta_{ij,pq}(\Q_{\va}^*)\PP_{ij}\wh{\PP}_{pq})^2\geq 0,
$$
for any $ \PP,\wh{\PP}\in\Ss_0 $.
This, together with \eqref{wtYva}, implies 
\be
I_2\geq-\f{C|\Q_{\va}-\Q_{\va}^*|^3|\na\Q_{\va}|^2}{|\mathbf{Z}_{\va}|^3}\geq-C|\na\Q_{\va}|^2.\label{I2es2}
\ee
In view \eqref{E-L} and \eqref{uPsiQnon}, when $ \dist(\Q_{\va},\cN)<\delta_0 $ is sufficiently small, $
I_1\geq 0 $. By \eqref{DeltaYvawt} and \eqref{I2es2}, we have $
-\Delta(|\mathbf{Z}_{\va}|)\leq C|\na\Q_{\va}|^2 $ in $ U\cap B_r\cap\{|\mathbf{Z}_{\va}|>0\} $. Combined with \eqref{paX0}, there holds
\be
-\Delta(|\mathbf{Z}_{\va}|)\leq C_1|\na\Q_{\va}|^2\text{ in } U\cap B_r\text{ in the sense of distribution,}\label{distriY}
\ee
for some constant $ C_1>0 $ depending only on $ \cA $. Indeed, for $ 0\leq\vp\in C_c^{\ift}(U\cap B_r) $, 
$$
\int_{ U\cap B_r}\na(|\mathbf{Z}_{\va}|)\cdot\na\vp\ud x\leq C_1\int_{ U\cap B_r}|\na\Q_{\va}|^2\vp\ud x.
$$
Now, we define $ w_{\va} $ as the solution of
$$
\left\{\begin{aligned}
-\Delta w_{\va}&=C_1|\na\Q_{\va}|^2&\text{ in }& U\cap B_r,\\
w_{\va}&=|\mathbf{Z}_{\va}|&\text{ on }&\pa(U\cap B_r).
\end{aligned}\right.
$$
Define $ \{w_{\va,i}\}_{i=1}^2 $ such that
\be
\left\{\begin{aligned}
-\Delta w_{\va,1}&=C_1|\na\Q_{\va}|^2&\text{ in }& U\cap B_r,\\
w_{\va,1}&=0&\text{ on }&\pa(U\cap B_r),
\end{aligned}\right.\label{wva1}
\ee
and $ w_{\va,2}:=w_{\va}-w_{\va,1} $. Applying Lemma \ref{LempRe} with $ p=4 $ and noting that $ w_{\va}=|\mathbf{Z}_{\va}|=0 $ on $ \pa  U\cap B_r $, we have, for any $ y\in U\cap B_{r/2} $,
\be
\begin{aligned}
w_{\va}(y)&\leq C\|\na w_{\va}\|_{L^{\ift}(U\cap B_{2r/3})}\dist(y,\pa U)\\
&\leq Cr^{-3/2}\|\na w_{\va}\|_{L^2(U\cap B_{3r/4})}\dist(y,\pa U)\\
&\quad\quad+Cr^{1/4}\|\na\Q_{\va}\|_{L^2(U\cap B_{3r/4})}^{1/2}\|\na\Q_{\va}\|_{L^{\ift}(U\cap B_{3r/4})}^{3/2}\dist(y,\pa U).
\end{aligned}\label{wvay}
\ee
For $ w_{\va,1} $, testing \eqref{wva1} by $ w_{\va,1} $ itself and using Poincar\'{e} inequality, we have
\begin{align*}
\|\na w_{\va,1}\|_{L^2(U\cap B_r)}^2&\leq C_1\int_{U\cap B_r}|\na\Q_{\va}|^2|w_{\va,1}|\ud x\\
&\leq Cr\|\na\Q_{\va}\|_{L^{\ift}(U\cap B_r)}\|\na\Q_{\va}\|_{L^2(U\cap B_r)}\|\na w_{\va,1}\|_{L^2(U\cap B_r)},
\end{align*}
which implies that
\be
\|\na w_{\va,1}\|_{L^2(U\cap B_r)}\leq Cr\|\na\Q_{\va}\|_{L^{\ift}(U\cap B_r)}\|\na\Q_{\va}\|_{L^2(U\cap B_r)}.\label{wva1est}
\ee
Consequently, $ w_{\va,2} $ satisfies 
$$
\left\{\begin{aligned}
-\Delta w_{\va,2}&=C_1|\na\Q_{\va}|^2&\text{ in }& U\cap B_r,\\
w_{\va,2}&=|\mathbf{Z}_{\va}|=0&\text{ on }&\pa U\cap B_r,\\
w_{\va,2}&=|\mathbf{Z}_{\va}|&\text{ on }&U\cap\pa B_r.
\end{aligned}\right.
$$
By the boundary Caccioppoli inequality and the maximum principle, we get
\be
\|\na w_{\va,2}\|_{L^2(U\cap B_{3r/4})}\leq Cr^{-1}\|w_{\va,2}\|_{L^2(U\cap B_r)}\leq Cr^{1/2}\|\mathbf{Z}_{\va}\|_{L^{\ift}(U\cap B_r)}.\label{wva2est}
\ee
By \eqref{distriY} and the maximum principle, we have $
|\mathbf{Z}_{\va}|\leq w_{\va}\text{ in }U\cap B_{r/2} $. This, together with \eqref{wtYva}, \eqref{wvay}, \eqref{wva1est} and \eqref{wva2est}, gives that if $ y\in U\cap B_{r/2} $, then
\be
\begin{aligned}
|(\Q_{\va}-\Q_{\va}^*)(y)|&\leq Cr^{-1}\|\Q_{\va}-\Q_{\va}^*\|_{L^{\ift}(U\cap B_r)}\dist(y,\pa U)\\
&\quad\quad+Cr^{-1/2}\|\na\Q_{\va}\|_{L^2(U\cap B_r)}\|\na\Q_{\va}\|_{L^{\ift}(U\cap B_r)}\dist(y,\pa U)\\
&\quad\quad+Cr^{1/4}\|\na\Q_{\va}\|_{L^2(U\cap B_r)}^{1/2}\|\na\Q_{\va}\|_{L^{\ift}(U\cap B_r)}^{3/2}\dist(y,\pa U),
\end{aligned}\label{normalde}
\ee
Since $ \Q_{\va}\in\cN $ on $ \pa U\cap B_r $, we have $ |\Q_{\va}-\Q_{\va}^*|=0 $ on $ \pa  U\cap B_r $ and hence the result direct follows from \eqref{normalde}.
\end{proof}

\begin{lem}\label{EnergyEstimate}
Let $ U\subset\R^3 $ be a bounded domain. There exist $ \delta_0>0 $ and $ C>0 $ depending only on $ \cA $ such that if $ \Q_{\va}\in C^{\ift}(U,\Ss_0) $ is a solution of $ -\va^2\Delta\Q_{\va}+\Psi(\Q_{\va})=0 $ in $ U $ and $
\|\dist(\Q_{\va},\cN)\|_{L^{\ift}(U)}<\delta_0 $, then
$$
-\Delta(e_{\va}(\Q_{\va}))(x)\leq Ce_{\va}^2(\Q_{\va})(x)\text{ for any }x\in U.
$$
\end{lem}

\begin{proof}
Firstly, by \eqref{E-L} and the property that $ \tr(\pa_k\Q_{\va})=0 $, we have
\be
\begin{aligned}
-\f{1}{2}\Delta(|\na\Q_{\va}|^2)&=-\Delta(\pa_k(\Q_{\va})_{ij})\pa_k(\Q_{\va})_{ij}-\pa_{k}\pa_{\ell}(\Q_{\va})_{ij}\pa_{k}\pa_{\ell}(\Q_{\va})_{ij}\\
&\leq-\f{1}{\va^2}\pa_k((f_b)_{ij}(\Q_{\va}))\pa_{k}(\Q_{\va})_{ij},
\end{aligned}\label{naQde22}
\ee
Moreover, we can obtain
\be
\begin{aligned}
-\f{1}{\va^2}\Delta(f_b(\Q_{\va}))&=-\f{1}{\va^2}\pa_k((f_b)_{ij}(\Q_{\va})\pa_k(\Q_{\va})_{ij})\\
&=-\f{1}{\va^4}|\Psi(\Q_{\va})|^2-\f{1}{\va^2}\pa_k((f_b)_{ij}(\Q_{\va}))\pa_{k}(\Q_{\va})_{ij}.
\end{aligned}\label{fBde}
\ee
For $ 0<\delta_0<1 $ sufficiently small, if $ \dist(\Q_{\va},\cN)<\delta_0 $ in $ U $, then $ \Q_{\va}^*=\varrho(\Q_{\va}) $ is well defined. We can apply Taylor expansion of the function $ (f_b)_{ij,mn} $ and obtain
\begin{align*}
(f_b)_{ij,mn}(\Q_{\va})=(f_b)_{ij,mn}(\Q_{\va}^*)+(f_b)_{ij,mn,pq}(\Q_{\va}^*)(\Q_{\va}-\Q_{\va}^*)_{pq}+\mathcal{R}^{ijmn}(\Q_{\va},\Q_{\va}^*),
\end{align*}
where $ \mathcal{R}^{ijmn}(\Q_{\va},\Q_{\va}^*) $ is a remainder satisfying
\be
|\mathcal{R}^{ijmn}(\Q_{\va},\Q_{\va}^*)|\leq C|\Q_{\va}-\Q_{\va}^*|^2.\label{EstimatesR}
\ee
Direct calculations yield that
\begin{align*}
-\f{1}{\va^2}\pa_k((f_b)_{ij}(\Q_{\va}))\pa_k(\Q_{\va})_{ij}&=-\f{1}{\va^2}\((f_b)_{ij,mn}(\Q_{\va}^*)\)\pa_k(\Q_{\va})_{mn}\pa_k(\Q_{\va})_{ij}\\
&\quad-\f{1}{\va^2}((f_b)_{ij,mn,pq}(\Q_{\va}^*))(\Q_{\va}-\Q_{\va}^*)_{pq}\pa_k(\Q_{\va})_{mn}\pa_k(\Q_{\va})_{ij}\\
&\quad-\f{1}{\va^2}\mathcal{R}^{ijmn}(\Q_{\va},\Q_{\va}^*)\pa_k(\Q_{\va})_{mn}\pa_k(\Q_{\va})_{ij}\\
&=I_1+I_2+I_3.
\end{align*}
By using Lemma \ref{minimalHessian}, we have $ I_1\leq 0 $. In view of \eqref{EstimatesR} and Cauchy inequality \be
\delta a^2+\f{b^2}{4\delta}\geq ab,\label{Cauchy}
\ee
with $ a,b,\delta>0 $, we can obtain
$$
|I_2+I_3|\leq\f{C\delta}{\va^4}\dist^2(\Q_{\va},\cN)+\f{1}{\delta}|\na\Q_{\va}|^4.
$$
This, together with \eqref{uPsifQnon}, \eqref{naQde22} and \eqref{fBde}, implies that in $ U $,
\begin{align*}
-\Delta(e_{\va}(\Q_{\va}))\leq-\f{1}{\va^4}|\Psi(\Q_{\va})|^2+\f{C\delta}{\va^4}\dist^2(\Q_{\va},\cN)+\f{1}{\delta}|\na\Q_{\va}|^4\leq Ce_{\va}^2(\Q_{\va}),
\end{align*}
if we choose $ 0<\delta<1 $ sufficiently small.
\end{proof}

\begin{lem}\label{smallreg1}
Let $ 0<\va<1 $. There exist $ \delta_0,E_0>0 $ depending only on $ \cA $ such that if $ \Q_{\va}\in C^{\ift}(B_r,\Ss_0) $ is a solution of $ -\va^2\Delta\Q_{\va}+\Psi(\Q_{\va})=0 $ in $ B_r $, satisfying
\begin{align*}
\|\dist(\Q_{\va},\cN)\|_{L^{\ift}(B_r)}&<\delta_0\text{ and }\f{1}{r}\int_{B_r}e_{\va}(\Q_{\va})\ud x<E_0,
\end{align*}
then $ r^2\|e_{\va}(\Q_{\va})\|_{L^{\ift}(B_{r/2})}\leq C $, where $ C>0 $ depends only on $ \cA $.
\end{lem}

\begin{proof}
Define $ w_{\va}:=e_{\va}(\Q_{\va}) $. Choose $ 0<r_1<r $ such that
$$
\sup_{0\leq\rho\leq r}\left\{(r-\rho)^2\sup_{B_{\rho}}w_{\va}\right\}=(r-r_1)^2\sup_{B_{r_1}}w_{\va}.
$$
Set $ \al_{\va}:=\sup_{B_{r_1}}w_{\va}=w_{\va}(x_1) $ with $ x_1\in\ol{B_{r_1}} $. If $ \al_{\va}(r-r_1)^2\leq 4 $, by the definition of $ r_1 $, we can choose $ \rho=r/2 $ and get $ (r/2)^2\sup_{B_{r/2}}w_{\va}\leq\al_{\va}(r-r_1)^2\leq 4 $, from which the result follows directly. For otherwise, we define $ \PP_{\va}(x):=\Q_{\va}(x_1+\al_{\va}^{-1/2}x) $ and $ w_{\va,1}:=e_{\ol{\va}}(\PP_{\va}) $, with $ \ol{\va}=\al_{\va}\va $. By simple calculations, we have $
w_{\va,1}(x)=\al_{\va}^{-1}w_{\va}(x_1+\al_{\va}^{-1/2}x) $ and $ w_{\va,1}(0)=1 $. In view of Lemma \ref{EnergyEstimate}, it follows that
\be
-\Delta w_{\va,1}(x)=-\al_{\va}^{-2}\Delta w_{\va}(x_1+\al_{\va}^{-1/2}x)\leq Cw_{\va,1}^2(x),\label{Bo2}
\ee
for any $ x $ such that $ |x_1+\al_{\va}^{-1/2}x|\leq r $. As a result, we obtain
\be
\sup_{B_1}w_{\va,1}\leq\f{1}{\al_{\va}}\sup_{B_{(r+r_1)/2}}w_{\va}\leq\f{\al_{\va}(r-r_1)^2}{\al_{\va}(r-(r_1+r)/2)}=4,\label{bound4}
\ee
where for the first inequality, we have used $ \al_{\va}(r-r_1)^2>4 $ and the property that
$$ 
|x_1+\al_{\va}^{-1/2}x|\leq r_1+\f{r-r_1}{2}=\f{r+r_1}{2}, 
$$
when $ |x|<\al_{\va}^{1/2}/2 $. For the second inequality of \eqref{bound4}, we have used the definition of $ r_1 $. Combined with \eqref{Bo2}, we have $
-\Delta w_{\va,1}\leq Cw_{\va,1} $ in $ B_1 $. Applying Lemma \ref{Hanarck}, we have
\be
1=w_{\va,1}(0)\leq C\int_{B_1}w_{\va,1}\ud x.\label{Low}
\ee
On the other hand, we note that $ \PP_{\ol{\va}} $ satisfies $ -\ol{\va}^2\Delta\PP_{\ol{\va}}+\Psi(\PP_{\ol{\va}})=0 $. By the change of variables and \eqref{Mo}, we deduce
\begin{align*}
\int_{B_1}w_{\va,1}\ud x&\leq\f{1}{(r-r_1)\al_{\va}^{1/2}/2}\int_{B_{(r-r_1)\al_{\va}^{1/2}/2}}w_{\va,1}\ud x\leq\f{2}{r}\int_{B_{r/2}(x_1)}w_{\va}\ud x\leq 2E_0,
\end{align*}
This, together with \eqref{Low}, implies $ 1\leq CE_0 $. If $ E_0 $ is sufficiently small, there is a contradiction and then we can complete the proof.
\end{proof}

\begin{lem}\label{smallreg2}
Let $ 0<\va<1 $. Assume that $ U $ is a bounded $ C^{2,1} $ domain with $ r_{U,2} $ and $ M_{U,2} $. There exist $ \delta_0>0 $ and $ E_0>0 $, depending only $ \cA,r_{U,2} $, and $ M_{U,2} $ such that the following property holds. If $ \Q_{\va}\in C^{\ift}(U\cap B_r(x_0),\Ss_0)\cap C^2(\ol{ U\cap B_r(x_0)},\Ss_0) $ is a solution of
$$
\left\{\begin{aligned}
-\va^2\Delta\Q_{\va}+\Psi(\Q_{\va})&=0&\text{ in }& U\cap B_r(x_0),\\
\Q_{\va}&=\Q_{b,\va}&\text{ on }&\pa U\cap B_r(x_0),
\end{aligned}\right.
$$
for some $ 0<r<r_{U,2} $ and $ \Q_{b,\va}\in C^2(\pa U\cap B_{2r}(x_0),\cN) $, satisfying
$$
\|\dist(\Q_{\va},\cN)\|_{L^{\ift}(U\cap B_{2r}(x_0))}<\delta_0,\text{ and }E_{\va}:=\sup_{x\in U\cap B_r(x_0)}\sup_{0<\rho<r}\f{1}{r}\int_{U\cap B_{\rho}(x)}e_{\va}(\Q_{\va})\ud y<E_0,
$$
then
$$
\|e_{\va}(\Q_{\va})\|_{L^{\ift}(U\cap B_{r/2}(x_0))}\leq C(b_{\va}^2+r^{-2}(\delta_{\va}^2+E_{\va})),
$$
where
\begin{align*}
b_{\va}&:=\||\na_{\pa U}\Q_{b,\va}|+r|D_{\pa U}^2\Q_{b,\va}|\|_{L^{\ift}(\pa U\cap B_{2r})},\\
\delta_{\va}&:=\|\dist(\Q_{\va},\cN)\|_{L^{\ift}(U\cap B_{2r})}
\end{align*}
and $ C>0 $ depends only on $ \cA,r_{U,2},M_{U,2} $.
\end{lem}

\begin{proof}
For simplicity, we assume $ x_0=0 $ and set $ w_{\va}:=e_{\va}(\Q_{\va}) $. Choosing $ 0<\delta_0<1 $ sufficiently small, by Lemma \ref{EnergyEstimate}, we have
\be
-\Delta w_{\va}\leq Cw_{\va}^2\text{ in } U\cap B_{2r},\label{BoB}
\ee
as $ \delta_{\va}<\delta_0 $. To complete the proof, we intend to show that
$$
\sup_{ U\cap B_{r/2}}w_{\va}\leq C(b_{\va}^2+r^{-2}(\delta_{\va}^2+E_{\va})).
$$
This directly follows from
\be
M_{\va}:=\sup_{0\leq\rho\leq r}\left\{(r-\rho)^2\sup_{ U\cap B_{r}}(w_{\va}-C_0b_{\va}^2)\right\}\leq C(E_{\va}+\delta_{\va}^2),\label{Mvaboun}
\ee
for some $ \delta_0,E_0>0 $, where $ C_0 $ is to be determined later. Set $ 0<r_1<r $ and $ x_1\in\ol{ U\cap B_{r_1}} $ such that $ M_{\va}=(r-r_1)^2\sup_{ U\cap B_{r_1}}(w_{\va}-C_0b_{\va}^2) $ and $ \al_{\va}:=w_{\va}(x_1)=\sup_{ U\cap B_{r_1}}w_{\va} $. As a result, we have
\begin{align}
\sup_{U\cap B_{(r-r_1)/2}(x_1)}w_{\va}&\leq\sup_{ U\cap B_{(r+r_1)/2}}w_{\va}\leq 4\al_{\va},\label{inftbounwva}\\
M_{\va}&=(r-r_1)^2(\al_{\va}-C_0b_{\va}^2).\label{Mvavalu}
\end{align}
Define $ r_2:=(r-r_1)/2 $, $ U_1:=\{y\in\R^3:x_1+\al_{\va}^{-1/2}y\in U\} $, and $
\rho_1:=\al_{\va}^{1/2}r_2 $. Let $
\PP_{\va}(y):=\Q_{\va}(x_1+\al_{\va}^{-1/2}y) $, and $ w_{\va,1}(y):=e_{\ol{\va}}(\PP_{\va})(y) $ for $ y\in U_1\cap B_{\rho_1} $ with $ \ol{\va}=\al_{\va}\va $. In view of the arguments in \eqref{bound4} and \eqref{inftbounwva}, we have
\be
\sup_{U_1\cap B_{\rho_1}}w_{\va,1}\leq 4,\label{wva4boun}
\ee
This, together with \eqref{BoB}, implies that
\be
-\Delta w_{\va,1}\leq Cw_{\va,1}\text{ in }U_1\cap B_{\rho_1}.\label{HarBa}
\ee
Moreover, by the change of variables, we also get
\be
\f{1}{\rho}\int_{U_1\cap B_{\rho}}w_{\va,1}\ud y\leq E_{\va}<E_0\text{ for any }0<\rho\leq\rho_1.\label{Boundenwva}
\ee
Set $ \rho_2=\min\{\dist(0,\pa U_1),\rho_1/10\} $. We claim that if $ E_0 $ is sufficiently small, then $ \rho_2\leq 1 $. Indeed, if $ \rho_2>1 $, we have $ \rho_1>10 $ and $ \dist(0,\pa U_1)>1 $. By using \eqref{Boundenwva}, Lemma \ref{Hanarck} and the fact that $ w_{\va,1}(0)=1 $, we obtain x
\begin{align*}
1\leq C\int_{B_1}w_{\va,1}\ud y\leq CE_0. 
\end{align*}
Choosing $ E_0 $ sufficiently small, it is a contradiction. If $ \rho_2=\rho_1/10 $, then $ B_{\rho_2}\subset U_1 $ and we can apply \eqref{Boundenwva} and Lemma \ref{Hanarck} again to get
$$
1=w_{\va,1}(0)\leq\f{C}{\rho_2^3}\int_{B_{\rho_2}}w_{\va,1}\ud y\leq\f{CE_{\va}}{\rho_2^2}.
$$
Consequently, we have $
M_{\va}\leq 4\rho_1^2\leq 400\rho_2^2\leq CE_{\va} $, which directly implies \eqref{Mvaboun}. For this reason, we assume 
\be
\rho_2=\dist(0,\pa U_1)<\f{\rho_1}{10}.\label{rho210rho1}
\ee
Also, we can assume that $ \al_{\va}\geq C_0b_{\va}^2 $, since for otherwise \eqref{Mvaboun} follows directly and we are done. Next, we define $ \PP_{\va}^*=\varrho(\PP_{\va}) $ and $ \Q_{\va}^*=\varrho(\Q_{\va}) $. Since $ 0<\delta_0<1 $ is sufficiently small, these two are well defined. We claim, for $ 0<\rho<\rho_1/10 $,
\be
\begin{aligned}
\|\na(\PP_{\va}-\PP_{\va}^*)\|_{L^{\ift}(\pa 
U_1\cap B_{\rho})}&\leq C\rho^{-1}\|\PP_{\va}-\PP_{\va}^*\|_{L^{\ift}(U_1\cap B_{5\rho})}\\
&\quad\quad+C\rho^{-1/2}\|\na\PP_{\va}\|_{L^2(U_1\cap B_{5\rho})}\|\na\PP_{\va}\|_{L^{\ift}(U_1\cap B_{5\rho})}\\
&\quad\quad+C\rho^{1/4}\|\na\PP_{\va}\|_{L^2(U_1\cap B_{5\rho})}^{1/2}\|\na\PP_{\va}\|_{L^{\ift}(U_1\cap B_{5\rho})}^{3/2}.
\end{aligned}\label{Boundaryuse2}
\ee
Set $ 0<s=\al_{\va}^{-1/2}\rho<r_2/10 $. By \eqref{rho210rho1}, we have
$$
\dist(x_1,\pa U)=\al_{\va}^{-1/2}\rho_2<\f{r_2}{10}\leq\f{r}{10}<\f{r_{U,2}}{10}.
$$
If $ 0<s<\dist(x_1,\pa U) $, then $ \pa U\cap B_s=\pa U_1\cap B_{\rho}=\emptyset $. For this case, there is nothing to prove. If $ s>\dist(x_1,\pa U) $, we can choose $ x_2\in\pa U $ such that $ |x_1-x_2|=\dist(x_1,\pa U) $. As a result, $
B_s(x_1)\subset B_{2s}(x_2)$ and $ B_{4s}(x_2)\subset B_{5s}(x_1) $. By applying Lemma \ref{BoundaryUse}, we have
\begin{align*}
\|\na(\Q_{\va}-\Q_{\va}^*)\|_{L^{\ift}(\pa U\cap B_s(x_1))}&\leq \|\na(\Q_{\va}-\Q_{\va}^*)\|_{L^{\ift}(\pa U\cap B_{2s}(x_2))}\\
&\leq Cs^{-1}\|\Q_{\va}-\Q_{\va}^*\|_{L^{\ift}(U\cap B_{5s}(x_1))}\\
&\quad\quad+Cs^{-1/2}\|\na\Q_{\va}\|_{L^2(U\cap B_{5s}(x_1))}\|\na\Q_{\va}\|_{L^{\ift}(U\cap B_{5s}(x_1))}\\
&\quad\quad+Cs^{1/4}\|\na\Q_{\va}\|_{L^2(U\cap B_{5s}(x_1))}^{1/2}\|\na\Q_{\va}\|_{L^{\ift}(U\cap B_{5s}(x_1))}^{3/2},
\end{align*}
which directly implies \eqref{Boundaryuse2} by the change of variables. For $ \PP_{\va}^* $, it satisfies $ \ol{\va}^2\Delta\PP_{\va}^*+\Psi(\PP_{\va}^*)=0 $. For $ 0<\delta_0<1 $ sufficiently small, we can use \eqref{Deltanablacontrol} to obtain
\be
|\Delta\PP_{\va}^*|\leq C|\na\PP_{\va}|^2.\label{DeltaPPva}
\ee
By using this, Lemma \ref{LempRe}, and arguments in the proof of \eqref{Boundaryuse2}, we have, if $ 0<\rho<\rho_1/10 $, then
\be
\begin{aligned}
\|\na\PP_{\va}^*\|_{L^{\ift}(\pa U_1\cap B_{\rho})}&\leq C\|\na_{\pa U_1}\PP_{\va}\|_{L^{\ift}(\pa U_1\cap B_{5\rho})}+C\rho\|D_{\pa U_1}^2\PP_{\va}\|_{L^{\ift}(\pa U_1\cap B_{5\rho})}\\
&\quad\quad+C\rho^{-3/2}\|\na\PP_{\va}\|_{L^2(U_1\cap B_{5\rho})}+C\rho^{1/4}\|\Delta\PP_{\va}\|_{L^4(U_1\cap B_{5\rho})}\\
&\leq C\al_{\va}^{-1/2}b_{\va}+C\rho^{-3/2}\|\na\PP_{\va}\|_{L^2(U_1\cap B_{5\rho})}\\
&\quad\quad+C\rho^{1/4}\|\na\PP_{\va}\|_{L^2(U_1\cap B_{5\rho})}^{1/2}\|\na\PP_{\va}\|_{L^{\ift}(U_1\cap B_{5\rho})}^{3/2}.
\end{aligned}
\ee
This, together with \eqref{wva4boun}, \eqref{Boundenwva}, \eqref{Boundaryuse2}, and the fact that $ \PP_{\va}\in\cN $ on $ \pa U_1 $, implies that
$$
\sup_{\pa U\cap B_{\rho}}w_{\va,1}\leq C_1\(\f{1}{C_0}+\f{(E_{\va}+\delta_{\va}^2)}{\rho^2}+\rho E_{\va}^{1/2}+E_{\va}^{1/2}\):=C_1C_{\va}(\rho)
$$
for any $ 0<\rho<\rho_1/10 $. Set
$$
w_{\va,2}=\left\{\begin{aligned}
&\max\left\{C_1C_{\va}(\rho),w_{\va,1}\right\}&\text{ in }&U_1\cap B_{\rho},\\
&C_1C_{\va}(\rho)&\text{ in }&B_{\rho}\backslash U_1.
\end{aligned}\right.
$$
Using \eqref{HarBa}, we have $ w_{\va,2}\in W^{1,\ift}(B_{\rho}) $ and $
-\Delta w_{\va,2}\leq Cw_{\va,2} $ in $ B_{\rho} $. By applying \eqref{Boundenwva}, Lemma \ref{Hanarck}, and the fact that $ w_{\va,1}(0)=1 $, we get
$$
1\leq\f{C}{\rho^3}\int_{B_{\rho}}w_{\va,2}\ud y\leq C_2C_{\va}(\rho).
$$
This implies that for any $ 0<\rho<\rho_1/10 $, $
1\leq C_2C_{\va}(\rho) $. Now, we choose $ E_0>0 $ and $ \delta_0>0 $ such that
$$
C_0>8C_2,\,\,E_0^{1/2}<\f{1}{8C_2},\,\,(E_0+\delta_0^2)^{4/3}<\f{1}{16C_2},\text{ and }\f{E_0^{1/2}(E_0+\delta_0^2)^{1/3}}{2}+E_0^{1/2}<\f{1}{4C_2},
$$
For $ \rho_3=(E_{\va}+\delta_{\va}^2)^{1/3} $, we have $
C_2C_{\va}(\rho_3)<3/8 $, which implies that $ \rho_1/10<\rho_3 $ and then
$$
1\leq C_2C_{\va}\(\f{\rho_1}{20}\)\leq \f{1}{8}+\f{C(E_{\va}+\delta_{\va}^2)}{\rho_1^2}+\f{1}{4}.
$$
As a result, $ M_{\va}\leq 4\rho_1^2\leq C(E_{\va}+\delta_{\va}^2)^{1/2} $, which completes the proof.
\end{proof}

\subsection{Higher order estimates}

In this subsection, we deal with the higher regularity estimates for solutions of \eqref{E-L}. The method in the proof is based on bootstrap arguments.

\begin{lem}\label{Ckestimate}
Let $ 0<\va<1 $. Assume that $ \Q_{\va}\in C^{\ift}(B_r,\Ss_0) $ is a solution of $ -\va^2\Delta\Q_{\va}+\Psi(\Q_{\va})=0 $ in $ B_r $ satisfying
\be
\|(|\tr \Y_{\va}|+|h_{\va}|+|\na\Q_{\va}|)\|_{L^{\ift}(B_r)}\leq C_0.\label{0deriva}
\ee
For any $ j\in\Z_{\geq 0} $ and $ 0<\rho<r $, there holds
$$
\|(|D^j(\tr\Y_{\va})|+|D^jh_{\va}|+|D^{j+1}\Q_{\va}|)\|_{L^{\ift}(B_{\rho})}\leq C(1+(r-\rho)^{-j}),
$$
where $ \Y_{\va},h_{\va} $ are given by \eqref{Xva}, \eqref{hva}, and $ C>0 $ is a constant depending only on $ \cA,C_0,j $. 
\end{lem}

\begin{proof}
To begin with, for $ 0<s\leq r $ and $ j\in\Z_{\geq 0} $, we define 
$$
\begin{aligned}
d_{\va}(s,j)&:=\left\|\(\sum_{\ell=0}^j|D^{\ell}(\tr\Y_{\va})|+\sum_{\ell=0}^j|D^{\ell}h_{\va}|+\sum_{\ell=0}^{j+1}|D^{\ell}\Q_{\va}|\)\right\|_{L^{\ift}(B_s)},\\
D_{\va}(s,j)&:=d_{\va}(s,j)+\sum_{\substack{\ell_1,\ell_2,...,\ell_p,\ell_1+\ell_2+\cdots+\ell_p=j+1,\\ p\geq 2,\,\,\forall q\geq 1,\,\,\ell_q\geq 1}}\prod_{k=1}^pd_{\va}(s,\ell_k),\\
\Phi_{\va}(s,j)&:=\|(|D^{j+1}(\tr\Y_{\va})|+|D^{j+1}h_{\va}|+|D^{j+2}\Q_{\va}|)\|_{L^{\ift}(B_{s})}.
\end{aligned}
$$
For $ \Q_{\va} $, using $ -\va^2\Delta\Q_{\va}+\Psi(\Q_{\va})=0 $, we have
\be
\Delta\Q_{\va}
=2a_6(\tr\Y_{\va})\Q_{\va}+a_6'(\tr\Y_{\va})\(\Q_{\va}^2-\f{1}{3}(\tr\Q_{\va}^2)\I\)+\f{a_2}{2r_*^2}h_{\va}\Q_{\va}.\label{Qeq}
\ee
For $ \Y_{\va} $, we can obtain 
$$
\va^2\Delta\Y_{\va}=3\Delta\Q_{\va}\Q_{\va}^2+2\pa_i\Q_{\va}\Q_{\va}\pa_i\Q_{\va}+2\Q_{\va}\pa_i\Q_{\va}\pa_i\Q_{\va}+2\pa_i\Q_{\va}\pa_i\Q_{\va}\Q_{\va}-r_*^2\Delta\Q_{\va}.
$$
Taking the trace on both sides and using \eqref{Qeq}, it follows that
\be
\begin{aligned}
\va^2\Delta(\tr\Y_{\va})&=6a_6\va^2(\tr\Y_{\va})^2+\f{a_6'}{2}(\va^2h_{\va}+2r_*^2)^2(\tr\Y_{\va})\\
&\quad\quad+\f{3a_2}{2r_*^2}\va^2h_{\va}(\tr\Y_{\va})+6\tr(\na\Q_{\va}\na\Q_{\va}\Q_{\va}).
\end{aligned}\label{Xeq}
\ee
Taking derivatives $ D^{j+1} $ on both sides of \eqref{Qeq}, we have, for $ 0<t\leq r $,
$$
\|\Delta(D^{j+1}\Q_{\va})\|_{L^{\ift}(B_{t})}\leq C(\|D^{j+1}(\tr\Y_{\va})\|_{L^{\ift}(B_{t})}+\|D^{j+1}h_{\va}\|_{L^{\ift}(B_{t})}+D_{\va}(t,j)).
$$
Consequently, for $ 0<s<t\leq r $ and $ \delta>0 $, we have
\be
\begin{aligned}
\|D^{j+2}\Q_{\va}\|_{L^{\ift}(B_s)}&\leq C(s-t)^{-1}\|D^{j+1}\Q_{\va}\|_{L^{\ift}(B_t)}+C\|\Delta(D^{j+1}\Q_{\va})\|_{L^{\ift}(B_t)}^{1/2}\|D^{j+1}\Q_{\va}\|_{L^{\ift}(B_t)}^{1/2}\\
&\leq C(1+(t-s)^{-1})D_{\va}(t,j)+\delta(\|D^{j+1}(\tr\Y_{\va})\|_{L^{\ift}(B_{t})}+\|D^{j+1}h_{\va}\|_{L^{\ift}(B_{t})})\\
&\leq C(1+(t-s)^{-1})D_{\va}(t,j)+\delta\Phi_{\va}(t,j),
\end{aligned}\label{QE1}
\ee
where for the first inequality, we have used Lemma \ref{Intp} and for the second inequality, we have used \eqref{Cauchy} and the fact $
\|D^{j+1}\Q_{\va}\|_{L^{\ift}(B_{t})}\leq D_{\va}(t,j) $. If we take derivatives $ D^j $ on both sides of \eqref{Xeq}, there holds
\begin{align*}
\|\Delta(D^j(\tr\Y_{\va}))\|_{L^{\ift}(B_{t})}\leq C\va^{-2}(\|D^j(\tr\Y_{\va})\|_{L^{\ift}(B_{t})}+\|D^jh_{\va}\|_{L^{\ift}(B_{t})}+D_{\va}(t,j))
\end{align*}
for $ 0<t<r $. Again by using Lemma \ref{Intp}, we have
\be
\|D^{j+1}(\tr\Y_{\va})\|_{L^{\ift}(B_{t})}\leq C(\va^{-1}+(t-s)^{-1})D_{\va}(t,j),\label{trXE}
\ee
for any $ 0<s<t\leq r $. Now we write the equation \eqref{Xeq} as the form
\be
-\va^2\Delta(\tr \Y_{\va})+2a_6'r_*^4(\tr\Y_{\va})=h_{\va,1}+\va^2h_{\va,2},\label{Xeq2}
\ee
where $ h_{\va,1} $ and $ h_{\va,2} $ are given by
\begin{align*}
h_{\va,1}&=-6\tr(\na\Q_{\va}\na\Q_{\va}\Q_{\va}),\\
h_{\va,2}&=-6a_6(\tr\Y_{\va})^2-\f{3a_2+4a_6'r_*^4}{2r_*^2}h_{\va}(\tr\Y_{\va})-\f{a_6'}{2}\va^2h_{\va}^2(\tr\Y_{\va})^2.
\end{align*}
By \eqref{QE1}, \eqref{trXE}, and direct computations, we have, for any $ \delta>0 $,
\begin{align*}
\|D^{j+1}h_{\va,1}\|_{L^{\ift}(B_{(s+t)/2})}&\leq C(1+(t-s)^{-1})D_{\va}(t,j)+\delta\Phi_s(t,j),\\
\|D^{j+1}h_{\va,2}\|_{L^{\ift}(B_{(s+t)/2})}&\leq C(\va^{-1}+(t-s)^{-1})D_{\va}(t,j)+C\|D^{j+1}h_{\va}\|_{L^{\ift}(B_{t})}.
\end{align*}
Combined with \eqref{trXE}, \eqref{Xeq2}, and Lemma \ref{corau}, we deduce
\be
\begin{aligned}
\|D^{j+1}(\tr\Y_{\va})\|_{L^{\ift}(B_{s})}&\leq C\|D^{j+1}(h_{\va,1}+\va^2h_{\va,2})\|_{L^{\ift}(B_{(s+t)/2})}\\
&\quad\quad+C\min\{1,\va(t-s)^{-1}\}\|D^{j+1}(\tr\Y_{\va})\|_{L^{\ift}(B_{(s+t)/2})}\\
&\leq\delta\Phi_{\va}(t,j)+C(1+(t-s)^{-1})D_{\va}(t,j)+C\va^2\|D^{j+1}h_{\va}\|_{L^{\ift}(B_{t})}.
\end{aligned}\label{trXz}
\ee
For the function $ h_{\va} $, by applying \eqref{Qeq}, we have
\be
\va^2\Delta h_{\va}=4a_6(\tr\Y_{\va})(\tr\Q_{\va}^2)+2a_6'\va^2(\tr\Y_{\va})^2+a_2r_*^{-2}h_{\va}(\va^2h_{\va}+2r_*^2)+2|\na\Q_{\va}|^2.\label{heq1}
\ee
Taking derivatives $ D^j $ on both sides of \eqref{heq1} and using Lemma \ref{Intp}, we get $ \|\Delta(D^jh_{\va})\|_{L^{\ift}(B_{t})}\leq C\va^{-2}D_{\va}(t,j) $, and then
\be
\|D^{j+1}h_{\va}\|_{L^{\ift}(B_{s})}\leq C(\va^{-1}+(t-s)^{-1})D_{\va}(t,j).\label{hvaes1}
\ee
This, together with \eqref{trXz}, imply that
\be
\|D^{j+1}(\tr\Y_{\va})\|_{L^{\ift}(B_{s})}\leq C(1+(t-s)^{-1})D_{\va}(t,j)+\delta\Phi_{\va}(t,j).\label{trX3}
\ee
Write the equation \eqref{heq1} in the form as
\be
-\va^2\Delta h_{\va}+2a_2h_{\va}=h_{\va,3}+\va^2h_{\va,4},\label{heq2}
\ee
where 
\begin{align*}
h_{\va,3}&=-4a_6(\tr\Y_{\va})(\tr\Q_{\va}^2)-2|\na\Q_{\va}|^2,\\
h_{\va,4}&=-2a_6'(\tr\Y_{\va})^2-\f{a_2}{r_*^2}h_{\va}^2.
\end{align*}
Using \eqref{QE1}, \eqref{hvaes1}, and \eqref{trX3}, for any $ 0<s<t\leq r $ and $ \delta>0 $, there hold
\begin{align*}
\|D^{j+1}h_{\va,3}\|_{L^{\ift}(B_{(s+t)/2})}&\leq C(1+(t-s)^{-1})D_{\va}(t,j)+\delta\Phi_{\va}(t,j),\\
\|D^{j+1}h_{\va,4}\|_{L^{\ift}(B_{(s+t)/2})}&\leq C(\va^{-1}+(t-s)^{-1})D_{\va}(t,j).
\end{align*}
Taking derivatives $ D^{j+1} $ on both sides of \eqref{heq2} and using Lemma \ref{au}, it follows that for any $ 0<s<t\leq r $ and $ \delta>0 $, we have
\begin{align*}
\|D^{j+1}h_{\va}\|_{L^{\ift}(B_{s})}&\leq C\|D^{j+1}(h_{\va,3}+\va^2h_{\va,4})\|_{L^{\ift}(B_{(s+t)/2})}\\
&\quad\quad+C\min\{1,\va(t-s)^{-1}\}\|D^{j+1}h_{\va}\|_{L^{\ift}(B_{(s+t)/2})}\\
&\leq C(1+(t-s)^{-1})D_{\va}(t,j)+\delta\Phi_{\va}(t,j).
\end{align*}
This, combined with \eqref{QE1} and \eqref{trX3}, implies that
$$
\Phi_{\va}(s)\leq\delta\Phi_{\va}(t)+C(1+(t-s)^{-1})D_{\va}(t,j)\text{ for any }\delta>0,
$$
where $ C>0 $ depends only on $ \cA $ and $ \delta $. For $ 0<s<t\leq r $, choosing $ \delta=1/2 $ and
\begin{align*}
\rho_0=s,\,\,\rho_{i+1}-\rho_i=\f{2^i}{3^{i+1}}(t-s).
\end{align*}
Now we have
$$
\Phi_{\va}(s)\leq\f{1}{2^i}\Phi_{\va}(\rho_i)+C(1+(t-s)^{-1})D_{\va}(s,j)\cdot\sum_{\ell=0}^{i-1}\(\f{3}{4}\)^{\ell}.
$$
Letting $ i\to+\ift $, we deduce
$$
\Phi_{\va}(s)\leq C(1+(t-s)^{-1})D_{\va}(t,j),
$$
which, together with \eqref{0deriva} and some interpolation inequalities, directly implies that
$$
d_{\va}(s,j+1)\leq C(1+(t-s)^{-1})D_{\va}(t,j).
$$
Consequently, the result follows by inductions on $ j $.
\end{proof}

\begin{cor}\label{BounCjcor}
Let $ 0<\va<1 $. Assume that $ U $ is a bounded Lipschitz domain with $ r_{U,0} $ and $ M_{U,0} $. If $ \Q_{\va}\in C^{\ift}(U\cap B_r(x_0),\Ss_0) $ is a solution of $ -\va^2\Delta\Q_{\va}+\Psi(\Q_{\va})=0 $ satisfying 
$$
\|(|\tr\Y_{\va}|+|h_{\va}|+|\na\Q_{\va}|)\|_{L^{\ift}(U\cap B_r(x_0))}\leq C_0
$$
for some $ 0<r<r_{U,0} $, $ x_0\in\pa U $, and $ C_0>0 $, then for any $ y\in U\cap B_{r/2}(x_0) $, there holds
\begin{align*}
(|D^j(\tr\Y_{\va})|+|D^jh_{\va}|+|D^{j+1}\Q_{\va}|)(y)\leq C\dist^{-j}(y,\pa U\cap B_r(x_0)),
\end{align*}
where $ C>0 $ depends only on $ \cA,r_0, $ and $ C_0 $.
\end{cor}
\begin{proof}
For any $ y\in U\cap B_{r/2}(x_0) $, we choose $ y_0\in\pa U\cap B_r(x_0) $ such that $ |y-y_0|=\dist(y,\pa U\cap B_r(x_0)) $. Set $ \rho:=|y-y_0| $. We have $ B_{\rho}(y)\subset U\cap B_r(x_0) $. By using Lemma \ref{Ckestimate}, we deduce
\begin{align*}
(|D^j(\tr\Y_{\va})|+|D^jh_{\va}|+|D^{j+1}\Q_{\va}|)(y)&\leq \|(|D^j(\tr\Y_{\va})|+|D^jh_{\va}|+|D^{j+1}\Q_{\va}|)\|_{L^{\ift}(B_{\rho/2}(y))}\\
&\leq C(1+(\rho-\rho/2)^{-j})\leq C\rho^{-j},
\end{align*}
which directly implies the result.
\end{proof}

\begin{lem}\label{Cknearlabel}
Let $ 0<\va<1 $ and $ U $ be a bounded $ C^{3,1} $ domain with $ r_{U,3} $ and $ M_{U,3} $. Assume that $ \Q_{\va}\in C^{\ift}(U\cap B_r(x_0),\Ss_0)\cap C^3(\ol{U\cap B_r(x_0)},\Ss_0) $ is a solution of
$$
\left\{\begin{aligned}
-\va^2\Delta\Q_{\va}+\Psi(\Q_{\va})&=0&\text{ in }& U\cap B_r,\\
\Q_{\va}&=\Q_{b,\va}&\text{ on }&\pa  U\cap B_r,
\end{aligned}\right.
$$
for some $ 0<r<r_{U,3},x_0\in\pa U,C_0>0 $, and $ \Q_{b,\va}\in C^3(\pa U\cap B_r(x_0)) $, satisfying
$$
\|(|\na_{\pa U}\Q_{b,\va}|+|D_{\pa U}^2\Q_{b,\va}+|D_{\pa U}^3\Q_{b,\va}|)\|_{L^{\ift}(U\cap B_r(x_0))}\leq C_0.
$$
If there is $ C_1>0 $, such that
\be
\|(|\tr\Y_{\va}|+|h_{\va}|+|\na\Q_{\va}|)\|_{L^{\ift}(U\cap B_r(x_0))}\leq C_1,\label{assumptionC1hva}
\ee
then for any $ \mu\in(0,1) $ and $ j\in\Z_{\geq 0} $, there holds 
$$
|D^{j+2}\Q_{\va}(y)|\leq Cr^{-1+\mu}\dist^{-(j+\mu)}(y,\pa U)\text{ for any }y\in \pa U\cap B_{r/2}(x_0),
$$
where $ C>0 $ depends only on $ \cA,C_0,C_1,r_{U,3},M_{U,3},j $, and $ \mu $.
\end{lem}
\begin{proof}
We will give the proof by an induction on $ j $. It is divided into several steps. \smallskip

\noindent
\underline{\textbf{Step 1.}} Proof of the case $ j=0 $. Since $ U $ is a bounded $ C^{3,1} $ domain, then under a translation, we can assume that $ x_0=0 $ and $ U\cap B_{r_0}=\bH_{\psi}\cap B_{r_0} $,
where
$$
\bH_{\psi}=\{(y_1,y_2,y_3)\in\R^3:(y_1,y_2)\in\R^2,\,\,y_3>\psi(y_1,y_2)\}
$$
and $ \psi $ satisfies \eqref{Dkest} for $ k=3 $. For simplicity, we define $
T_{\rho}:=T(\rho,\psi)=\bH_{\psi}\cap B_{\rho} $. For $ y\in T_{7r/8} $, we choose $ y_0\in\pa\bH_{\psi} $ such that $
|y-y_0|=\dist(y,\pa\bH_{\psi}) $. Defining $ \rho:=|y-y_0|/10 $, we have $ B_{\rho}(y)\subset B_r $. In view of \eqref{Qeq} and Lemma \ref{Ckestimate}, for any $ y\in T_{7r/8} $, we have
\be
\begin{aligned}
|\na(\Delta\Q_{\va})(y)|&\leq C\|(|\na\Q_{\va}|+|\na(\tr\Y_{\va})|+|\na h_{\va}|)\|_{L^{\ift}(B_{\rho/10}(y))}\\
&\leq C(1+(\rho-\rho/2)^{-1})\\
&\leq C\dist^{-1}(y,\pa\bH_{\psi}).
\end{aligned}\label{DeltaQvaes}
\ee
We can define a smooth orthonormal frame $ \{\vv^{(i)}\}_{i=1}^3 $ by
$$
\vv^{(1)}(y),\vv^{(2)}(y)\in T_{(y_1,y_2,\psi(y_1,y_2))}\pa\bH_{\psi}
$$
and $ \vv^{(3)}=\vv^{(1)}\times\vv^{(2)} $ for any $ y\in T_r $. By the definition of $ \{\vv^{(i)}\}_{i=1}^3 $, we get, for any $ u\in C^{\ift}(\bH_{\psi}\cap B_r) $,
\be
|\na u|^2=\sum_{i=1}^3|\pa_iu|^2=\sum_{i=1}^3|\pa_{\vv^{(i)}}u|^2,\label{naRenew}
\ee
and
\be
\Delta u=\sum_{i=1}^3\pa_i^2u=\sum_{i=1}^3\pa_{\vv^{(i)}}^2u.\label{DeltaRepr}
\ee
Now, setting $ i=1,2 $ and $ \U_{\va}=\pa_{\vv^{(i)}}\Q_{\va} $, we define $ \W_{\va} $ as a solution of
$$
\left\{\begin{aligned}
\Delta\W_{\va}&=0&\text{ in }&T_{7r/8},\\
\W_{\va}&=\U_{\va}&\text{ on }&\pa T_{7r/8}.
\end{aligned}\right.
$$
We also define $
\W_{\va,1}(y):=\U_{\va}(y_1,y_2,\psi(y_1,y_2)) $, $
\W_{\va,2}(y):=\W_{\va}(y)-\W_{\va,1}(y) $, and $
\V_{\va}(y):=\U_{\va}(y)-\W_{\va}(y) $ for $ y\in T_{7r/8} $. As a result,
\be
\left\{\begin{aligned}
\Delta\W_{\va,2}&=-\Delta\W_{\va,1}&\text{ in }&T_{7r/8},\\
\W_{\va,2}&=0&\text{ on }&\pa T_{7r/8}.
\end{aligned}\right.\label{Wva2equ}
\ee
By using Lemma \ref{LempRe} with $ p=4 $, we have
\be
\begin{aligned}
\|\na\W_{\va}\|_{L^{\ift}(T_{2r/3})}&\leq\|\na\W_{\va,1}\|_{L^{\ift}(T_{2r/3})}+\|\na\W_{\va,2}\|_{L^{\ift}(T_{2r/3})}\\
&\leq C(\|\na\W_{\va,1}\|_{L^{\ift}(T_{3r/4})}+r^{-3/2}\|\na\W_{\va,2}\|_{L^2(T_{3r/4})}+r^{1/4}\|\Delta\W_{\va,1}\|_{L^4(T_{3r/4})})\\
&\leq C(\|\na\W_{\va,1}\|_{L^{\ift}(T_{3r/4})}+r^{-3/2}\|\na\W_{\va,2}\|_{L^2(T_{3r/4})}+r\|\Delta\W_{\va,1}\|_{L^{\ift}(T_{3r/4})})\\
&\leq C(\||\na_{\pa\bH_{\psi}}\U_{\va}|+r|D_{\pa\bH_{\psi}}^2\U_{\va}|\|_{L^{\ift}(T_{7r/8})}+r^{-3/2}\|\na\W_{\va,2}\|_{L^2(T_{3r/4})}).
\end{aligned}\label{Wvaestimate}
\ee
In view of \eqref{Wva2equ}, we get
\begin{align*}
\|\na\W_{\va,2}\|_{L^2(T_{3r/4})}&\leq Cr^{3/2}(\|\Delta\W_{\va,2}\|_{L^2(T_{7r/8})}+r^{-1}\|\W_{\va,2}\|_{L^{\ift}(T_{7r/8})})\\
&\leq Cr^{3/2}(\|\Delta\W_{\va,1}\|_{L^2(T_{7r/8})}+r^{-1}\|\W_{\va}\|_{L^{\ift}(T_{7r/8})}+r^{-1}\|\W_{\va,1}\|_{L^{\ift}(T_{7r/8})})\\
&\leq Cr^{3/2}(\|\Delta\W_{\va,1}\|_{L^2(T_{7r/8})}+Cr^{-1}\|\U_{\va}\|_{L^{\ift}(T_{7r/8})}),
\end{align*}
where for the first inequality, we have used boundary Caccioppoli inequality, for the second inequality, we have used the definition of $ \W_{\va,2} $, and for the last inequality, we have used the maximum principle. Combined with \eqref{Wvaestimate}, we can obtain
\be
\begin{aligned}
\|\na\W_{\va}\|_{L^{\ift}(U\cap B_{2r/3})}&\leq Cr^{-1}\|\U_{\va}\|_{L^{\ift}(U\cap B_{7r/8})}+C\||\na_{\pa\bH_{\psi}}\U_{\va}|+r|D_{\pa\bH_{\psi}}^2\U_{\va}|\|_{L^{\ift}(U\cap B_{7r/8})}\\
&\leq Cr^{-1}.
\end{aligned}\label{Cr}
\ee
Again, by the maximum principle, we have
$$
\|\V_{\va}\|_{L^{\ift}(T_{7r/8})}\leq C\|\U_{\va}\|_{L^{\ift}(T_{7r/8})}\leq C.
$$
In view of \eqref{DeltaQvaes}, we obtain
\be
|\Delta\V_{\va}(y)|\leq|\Delta\U_{\va}(y)|\leq C\dist^{-1}(y,\pa\bH_{\psi})\leq Cr^{\mu}\dist^{-\mu-1}(y,\pa\bH_{\psi}),\label{DeltaVva}
\ee
for any $ y\in T_{7r/8} $ and $ \V_{\va}=0 $ on $ \pa T_{7r/8} $. Consequently, Lemma \ref{Boundarymu} implies
\be
\begin{aligned}
|\V_{\va}(y)|&\leq C(r^{-1}\|\V_{\va}\|_{L^{\ift}(\bH_{\psi}\cap B_r)}\dist(y,\pa \bH_{\psi})+r^{\mu}\dist^{1-\mu}(y,\pa\bH_{\psi}))\\
&\leq Cr^{\mu-1}\dist^{1-\mu}(y,\pa\bH_{\psi})
\end{aligned}\label{Vvaift}
\ee
for any $ y\in T_{2r/3} $. By using \eqref{DeltaVva}, \eqref{Vvaift} and applying Lemma \ref{Intp}, we have
\begin{align*}
|\na\V_{\va}(y)|^2&\leq C\|\V_{\va}\|_{L^{\ift}(B_{R_1/2}(y))}(\|\Delta\V_{\va}\|_{L^{\ift}(B_{R_1/2}(y))}+R_1^{-2}\|\V_{\va}\|_{L^{\ift}(B_{R_1/2}(y))})\\
&\leq Cr^{2\mu-2}\dist^{-2\mu}(y,\pa\bH_{\psi})
\end{align*}
for any $ y\in T_{r/2} $, where $ R_1=\dist(y,\pa\bH_{\psi})/2 $. Combining \eqref{Cr}, we deduce
$$
|\na\U_{\va}(y)|\leq Cr^{\mu-1}\dist^{-\mu}(y,\pa\bH_{\psi})\text{ for any }y\in T_{r/2}.
$$
By \eqref{naRenew}, for $ (i,k)\neq(3,3) $, there holds
$$
|\pa_{\vv^{(i)}}\pa_{\vv^{(k)}}\Q_{\va}(y)|\leq Cr^{\mu-1}\dist^{-\mu}(y,\pa\bH_{\psi})
$$
when $ y\in T_{r/2} $. For $ i=k=3 $, we note that by \eqref{Qeq} and \eqref{assumptionC1hva}, there holds $ |\Delta\Q_{\va}|\leq C $. This, together with \eqref{DeltaRepr}, implies that $ |\pa_{\vv^{(3)}}\pa_{\vv^{(3)}}\Q_{\va}| $.\smallskip

\noindent
\underline{\textbf{Step 2}.} Assume that the result is true for $ j $, we prove for the case $ j+1 $. By \eqref{Qeq} and \eqref{BounCjcor}, we have $
|\Delta(D^{j+2}\Q_{\va})(y)|\leq C\dist^{-(j+2)}(y,\pa\bH_{\psi}) $ for any $ y\in T_{2r/3} $. Now the result follows from Lemma \ref{Intp}.
\end{proof}

\begin{lem}\label{uvavvvainterior}
Let $ 0<\va<1 $ and $ \Q_{\va}\in C^{\ift}(B_r,\Ss_0) $ be a solution of $ -\va^2\Delta\Q_{\va}+\Psi(\Q_{\va})=0 $ in $ B_r $ with $ r>0 $. Assume that
$$
\|(|\tr\Y_{\va}|+|h_{\va}|+|\na\Q_{\va}|)\|_{L^{\ift}(B_r)}\leq C_0
$$
for some $ C_0>0 $. For any $ j\in\Z_{\geq 0} $ and $ 0<s<r $, there holds,
$$
\|(|D^ju_{\va}|+|D^jv_{\va}|)\|_{L^{\ift}(B_s)}\leq C(1+(r-s)^{-(j+2)}),
$$
where 
\be
\begin{aligned}
u_{\va}&:=h_{\va}-\f{12a_6}{a_2a_6'r_*^2}\tr(\na\Q_{\va}\na\Q_{\va}\Q_{\va})+\f{1}{a_2}|\na\Q_{\va}|^2,\\
v_{\va}&:=\tr\Y_{\va}+\f{3}{a_6'r_*^4}\tr(\na\Q_{\va}\na\Q_{\va}\Q_{\va}),
\end{aligned}\label{uvavvadef}
\ee
and $ C>0 $ depends only on $ \cA,C_0 $, and $ j $.
\end{lem}
\begin{proof}
Define $ v_{\va,1}:=\va^{-2}v_{\va} $, we can rewrite \eqref{Xeq} and \eqref{heq1} as the following form
\begin{align}
-\va^2\Delta u_{\va}+2a_2u_{\va}&=\va^2f_{\va},\label{uvaeq}\\
-\va^2\Delta v_{\va}+2a_6'r_*^4v_{\va}&=\va^2g_{\va},\label{vvaeq} 
\end{align}
where
\begin{align*}
f_{\va}&:=\f{12a_6}{a_2a_6'r_*^2}\Delta(\tr(\na\Q_{\va}\na\Q_{\va}\Q_{\va}))-\f{1}{a_2}\Delta(|\na\Q_{\va}|^2)-4a_6\va^2v_{\va,1}h_{\va}-8a_6r_*^2v_{\va,1}-2a_6'(\tr\Y_{\va})^2-\f{a_2}{r_*^2}h_{\va}^2,\\
g_{\va}&:=-\f{3}{a_6'r_*^4}\Delta(\tr(\na\Q_{\va}\na\Q_{\va}\Q_{\va}))-\f{a_6'}{2}\va^2h_{\va}^2(\tr\Y_{\va})^2-6a_6(\tr\Y_{\va})^2-\f{3a_2+4a_6'r_*^4}{2r_*^2}h_{\va}(\tr\Y_{\va}).
\end{align*}
In view of \eqref{Qeq}, we have
\begin{align*}
f_{\va}&=c_1\tr(\pa_{ij}^2\Q_{\va}\pa_i\Q_{\va}\pa_j\Q_{\va})+\tr(\pa_{ij}^2\Q_{\va}\pa_{ij}^2\Q_{\va}(c_2\I+c_3\Q_{\va}))\\
&\quad\quad+\pa_i(\tr\Y_{\va})\tr(\pa_i\Q_{\va}(c_4\I+c_5\Q_{\va}^2+c_6\Q_{\va}^3))\\
&\quad\quad+(\tr\Y_{\va})\tr(\pa_i\Q_{\va}\pa_i\Q_{\va}(c_7\I+c_8\Q_{\va}+c_9\Q_{\va}^2))\\
&\quad\quad+\pa_ih_{\va}\tr(\pa_i\Q_{\va}(c_{10}\Q_{\va}+c_{11}\Q_{\va}^2))+h_{\va}\tr(\pa_i\Q_{\va}\pa_i\Q_{\va}(c_{12}\I+c_{13}\Q_{\va}))\\
&\quad\quad+c_{14}h_{\va}^2+c_{15}(\tr\Y_{\va})^2+c_{16}v_{\va,1}+c_{17}\va^2v_{\va,1}h_{\va},
\end{align*}
and
\begin{align*}
g_{\va}&=c_1'\tr(\pa_{ij}^2\Q_{\va}\pa_i\Q_{\va}\pa_j\Q_{\va})+c_2'\tr(\pa_{ij}^2\Q_{\va}\pa_{ij}^2\Q_{\va}\Q_{\va})\\
&\quad\quad+\pa_i(\tr\Y_{\va})(\tr(\pa_i\Q_{\va}(c_3'\Q_{\va}^2+c_4'\Q_{\va}^3)+c_5'\tr(\pa_i\Q_{\va}\Q_{\va})\tr(\Q_{\va}^2))\\
&\quad\quad+(\tr\Y_{\va})(\tr(\pa_i\Q_{\va}\pa_i\Q_{\va}(c_6'\Q_{\va}+c_7'\Q_{\va}^2))+c_8'\tr(\pa_i\Q_{\va}\Q_{\va})\tr(\pa_i\Q_{\va}\Q_{\va}))\\
&\quad\quad+c_9'\pa_ih_{\va}\tr(\pa_i\Q_{\va}\Q_{\va}^2)+c_{10}'h_{\va}\tr(\pa_i\Q_{\va}\pa_i\Q_{\va}\Q_{\va})\\
&\quad\quad+c_{11}'(\tr\Y_{\va})^2+c_{12}'h_{\va}(\tr\Y_{\va})+c_{13}'\va^2h_{\va}^2(\tr\Y_{\va})^2,
\end{align*}
where $ \{c_i\}_{i=1}^{17} $, $ \{c_i'\}_{i=1}^{13} $ are constants depending only on $ \cA $. In view of Lemma \ref{Ckestimate}, for any $ j\in\Z_{\geq 0} $, we have
\begin{align*}
\|D^jg_{\va}\|_{L^{\ift}(B_{(s+r)/2})}&\leq C(1+(r-s)^{-(j+2)}),\\
\|D^jv_{\va}\|_{L^{\ift}(B_{(s+r)/2})}&\leq C(1+(r-s)^{-j}).
\end{align*}
Taking $ D^j $ on both sides of \eqref{vvaeq} and applying Corollary \ref{corau}, we get
\begin{align*}
\|D^jv_{\va}\|_{L^{\ift}(B_s)}&\leq C\va^2\|D^jg_{\va}\|_{L^{\ift}(B_{(s+r)/2})}+C\min\{1,\va^2(r-s)^{-2}\}\|D^jv_{\va}\|_{L^{\ift}(B_{(s+r)/2})}\\
&\leq C\va^2(1+(r-s)^{-(j+2)}).
\end{align*}
This implies that $
\|D^jv_{\va,1}\|_{L^{\ift}(B_{(s+r)/2})}\leq C(1+(r-s)^{-(j+2)}) $. Again, by using Lemma \ref{Ckestimate}, there hold
\begin{align*}
\|D^jf_{\va}\|_{L^{\ift}(B_{(s+r)/2})}&\leq C(1+(r-s)^{-(j+2)}),\\
\|D^ju_{\va}\|_{L^{\ift}(B_{(s+r)/2})}&\leq C(1+(r-s)^{-j}).
\end{align*}
It follows from Corollary \ref{corau} that
\begin{align*}
\|D^ju_{\va}\|_{L^{\ift}(B_s)}&\leq C\va^2\|D^jf_{\va}\|_{L^{\ift}(B_{(s+r)/2})}+C\min\{1,\va^2(r-s)^{-2}\}\|D^ju_{\va}\|_{L^{\ift}(B_{(s+r)/2})}\\
&\leq C\va^2(1+(r-s)^{-(j+2)}),
\end{align*}
which completes the proof.
\end{proof}

\begin{lem}\label{lemexpdestim}
Let $ 0<\va<1 $ and $ U $ be a bounded $ C^{3,1} $ domain with $ r_{U,3} $ and $ M_{U,3} $. Assume that $ \Q_{\va}\in C^{\ift}(U,\Ss_0)\cap C^3(\ol{U},\Ss_0) $ is a solution of
$$
\left\{\begin{aligned}
-\va^2\Delta\Q_{\va}+\Psi(\Q_{\va})&=0&\text{ in }&U,\\
\Q_{\va}&=\Q_{b,\va}&\text{ on }&\pa U,
\end{aligned}\right.
$$
such that 
\begin{align*}
\|(|\na_{\pa U}\Q_{b,\va}|+|D_{\pa U}^2\Q_{b,\va}+|D_{\pa U}^3\Q_{b,\va}|)\|_{L^{\ift}(\pa U)}&\leq C_0,\\
\|(|\tr\Y_{\va}|+|h_{\va}|+|\na\Q_{\va}|)\|_{L^{\ift}(U)}&\leq C_1,
\end{align*}
for some $ C_0,C_1>0 $. For any $ B_r(x_0)\subset\subset U $, $ j\in\Z_{\geq 0} $, and $ 0<s<r $, there holds
$$
\|(|D^ju_{\va}|+|D^jv_{\va}|)\|_{L^{\ift}(B_s(x_0))}\leq\f{C}{(r-s)^j}\left\{\f{\va^2}{(r-s)}+\exp\(-\f{C'(r-s)}{\va}\)\right\},
$$
where $ C,C'>0 $ depend only on $ \cA,r_{U,3},M_{U,3},C_0,C_1 $, and $ j $.
\end{lem}
\begin{proof}
For simplicity, under a translation, we assume that $ x_0=0 $. By some basic covering arguments and Lemma \ref{Cknearlabel}, we have that for any $ \mu\in(0,1) $ and $ j\in\Z_{\geq 0} $,
$$
\sum_{\ell=0}^{j+2}|D^{\ell}\Q_{\va}(y)|\leq C\dist^{-j-\mu}(y,\pa U),\quad y\in U.
$$
Since $ U $ is bounded and $ B_r\subset\subset U $, we have, $
\sum_{\ell=0}^{j+2}|D^{\ell}\Q_{\va}(y)|\leq C(r-s)^{-j-\mu} $ for any $ y\in B_s $ with $ 0<s<r $. This implies that
\be
\sum_{\ell=0}^{j+2}\|D^{\ell}\Q_{\va}\|_{L^{\ift}(B_s)}\leq C(r-s)^{-j-\mu}.\label{Dj2Qmu}
\ee
In view of Lemma \ref{Ckestimate}, we have
\be
\sum_{\ell=0}^{j}\|(|D^{\ell}(\tr\Y_{\va})|+|D^{\ell}h_{\va}|+|D^{\ell}u_{\va}|+|D^{\ell}v_{\va}|)\|_{L^{\ift}(B_s)}\leq C(r-s)^{-j}.\label{uvavvaCj}
\ee
Using \eqref{Dj2Qmu} and \eqref{uvavvaCj} with sufficiently small $ \mu\in(0,1) $, we get
\be
\|D^jv_{\va}\|_{L^{\ift}(B_{(s+r)/2})}\leq C(r-s)^{-j-1}.\label{Cjvvaest}
\ee
This, together with \eqref{vvaeq}, \eqref{uvavvaCj}, and Corollary \ref{corau}, implies that
$$
\sum_{\ell=0}^j\|D^{\ell}v_{\va}\|_{L^{\ift}(B_s)}\leq A_{\va}(r,s):=\f{C}{(r-s)^j}\left\{\f{\va^2}{(r-s)}+\exp\(-\f{r-s}{C\va}\)\right\}
$$
for any $ 0<s<r $. As a result, for any $ j\in\Z_{\geq 0} $, $
\sum_{\ell=0}^j\|D^{\ell}v_{\va,1}\|_{L^{\ift}(B_s)}\leq CA_{\va}(r,s)/\va^2 $. Imitating the proof of \eqref{Cjvvaest}, we obtain
$ \|D^ju_{\va,1}\|_{L^{\ift}(B_{(s+r)/2})}\leq CA_{\va}(r,s)/\va^2 $. This, together with \eqref{uvaeq}, \eqref{uvavvaCj}, and Corollary \ref{corau}, implies that
$ \|D^ju_{\va}\|_{L^{\ift}(B_s)}\leq CA_{\va}(r,s) $, which completes the proof.
\end{proof}

\section{Lower bound of energy in two dimensional balls}\label{LoweTwodi}

In this section, we aim to establish a lower bound on the sextic Landau-de Gennes energy within the unit ball $ B_1^2 $. 
Deriving such an estimate is a necessary step for the subsequent analysis. To obtain this, a more refined examination of the manifold $ \cN $ is required. In \cite{C15}, the author used the minimum energies of loops with some restrictions to characterize the topological obstacle of manifolds and gives some rough results for more general case. Motivated by this approach, we will consider more dedicated results for our special model \eqref{energy}.

\subsection{Free homotopy class}

Firstly, we give the definition of the free homotopy classes on topological space. Unlike the homotopy classes considered in the fundamental group, this definition does not impose any conditions on base points. 

\begin{defn}[Free homotopy class]\label{freehomodef}
Let $ \cX $ be a topological space. The free homotopy class of $ \cX $, defined by $ [\Ss^1,\cX] $, is the set of homotopy classes of continuous maps (loops) $ \Ss^1\to\cX $, with no conditions on base points. In particular, if $ \al_0,\al_1:\Ss^1\to\cX $ are two loops on $ \cX $, then $ \al_0 $ and $ \al_1 $ are in the same homotopy class of $ [\Ss^1,\cX] $, denoted by  $ \al_0\sim_{\cX}\al_1 $, if and only if there exists a continuous map $ H:[0,1]\times\Ss^1\to\cX $ such that $ H(0,\cdot)=\al_0(\cdot) $ and $ H(1,\cdot)=\al_1(\cdot) $. In $ [\Ss^1,\cX] $, we use $ [\al_0]_{\cX} $ to denote the class whose representative element is $ \al_0 $, i.e.,
$$
[\al_0]_{\cX}=\{\al\in C^0(\Ss^1,\cX):\al\sim_{\cX}\al_0\}.
$$
By this, $ [\al_1]_{\cX}=[\al_2]_X $ if and only if $ \al_1\sim_{\cX}\al_2 $. If $ [\al]_{\cX}=[x_0]_X $, for some point $ x_0\in\cX $, then we call $ [\al]_{\cX} $ is the trivial class in $ [\Ss^1,\cX] $ and $ \al $ is a trivial loop on $ \cX $. For otherwise, we replace the word trivial by non-trivial.
\end{defn}

\begin{rem}
For a loop $ \al\in C^0(\Ss^1,\cX) $, by identifying the point on $ \Ss^1 $ by $ \exp(i\theta) $ for $ \theta\in[0,2\pi] $ we have a parametrization of $ \al:[0,2\pi]\to\cX $ with $ \al(0)=\al(2\pi) $. For curves $ \ga,\ga_1,\ga_2\in C^0([0,2\pi],\cX) $ such that $ \ga_2(2\pi)=\ga_1(0) $, we use $ \wt{\ga} $ to denote the reverse of $ \ga $ and $ \ga_1*\ga_2 $ to represent the composition of $ \ga_1 $ and $ \ga_2 $. Precisely speaking, $ \wt{\ga}(\theta)=:\ga(2\pi-\theta) $ and
$$
(\ga_1*\ga_2)(\theta):=\left\{\begin{aligned}
&\ga_1(2\theta)&\text{ if }&\theta\in[0,\pi],\\
&\ga_2(2\theta-2\pi)&\text{ if }&\theta\in(\pi,2\pi].
\end{aligned}\right.
$$
For a point $ x_0\in\cX $, we denote $ x_0\in\ga $ if there is some $ \theta_0\in[0,2\pi] $ such that $ \ga(\theta_0)=x_0 $.
\end{rem}

\begin{rem}
Let $ \cX $ be a topological space and $ x_0\in\cX $. To distinguish between the (free) homotopy classes in $ [\Ss^1,\cX] $ and the fundamental group $ \pi_1(\cX,x_0) $, we use $ [\al]_{\cX,x_0} $ to denote an element in $ \pi_1(\cX,x_0) $ (here $ \al:\Ss^1\to\cX $ is a loop with $ x_0\in x_0 $) and $ \al_1\sim_{\cX,x_0}\al_2 $ to represent the relation $ [\al_1]_{\cX,x_0}=[\al_2]_{\cX,x_0} $ in $ \pi_1(\cX,x_0) $ for $ \al_1,\al_2\in C^0(\Ss^1,\cX) $ with $ x_0\in \al_1,\al_2 $. Moreover, we have $ [\al_1*\al_2]_{\cX,x_0}=[\al_1]_{\cX,x_0}[\al_2]_{\cX,x_0} $.
\end{rem}

\begin{defn}
Let $ G $ be a group and $ g_1,g_2\in G $. We call $ g_1 $ and $ g_2 $ are conjugate if there exists $ g\in G $ such that $ g_1=gg_2g^{-1} $. We denote this by $ g_1\stackrel{c}{\sim}g_2 $ in $ G $.
\end{defn}

\begin{lem}\label{LemmaHom}
Let $ \cX $ be a topological space and $ x_0\in\cX $. Assume that the natural map $ \Phi_{x_0}:\pi_1(\cX,x_0)\to[\Ss^1,\cX] $ is obtained by ignoring the base point, i.e., $ \Phi_{x_0}([\al]_{\cX,x_0})=[\al]_{\cX} $ for $ \al\in C^0(\Ss^1,\cX) $ with $ x_0\in\al $. $ \Phi_{x_0} $ satisfies the following properties.
\begin{enumerate}
\item $ \Phi_{x_0} $ is surjective if $ \cX $ is path-connected.
\item $ \Phi_{x_0}([\al_1]_{\cX,x_0})=\Phi_{x_0}([\al_2]_{\cX,x_0}) $ if and only if $
[\al_1]_{\cX,x_0}\stackrel{c}{\sim}[\al_2]_{\cX,x_0} $ in $ \pi_1(\cX,x_0) $.
\end{enumerate} 
Hence $ \Phi_{x_0} $ induces a bijective correspondence between $ [\Ss^1,\cX] $ and the set of conjugacy classes in $ \pi_1(\cX,x_0) $, if $ \cX $ is path-connected.
\end{lem}
\begin{proof}
We can regard $ \pi_1(\cX,x_0) $ as the set of base point-preserving homotopy classes of continuous $ (\Ss^1,s_0)\to(\cX,x_0) $. To show that $ \Phi_{x_0} $ is surjective, let $ \al_1\in C^0(\Ss^1,\cX) $ and $ [\al_1]_{\cX}\in[\Ss^1,\cX] $. Choosing a point $ x_1\in\al_1 $, we can use the path-connectedness to get a curve $ \ga $ connecting $ x_0 $ and $ x_1 $. Consequently, the curve $ \ga*\al_1*\wt{\ga} $ is based at $ x_0 $ and $ [\ga*\al_1*\wt{\ga}]_{\cX}=[\al_1]_{\cX} $, since we can continuously moving the base point from $ x_0 $ to $ x_1 $ through the curve $ \ga $. For the second property, we assume that $ \al_1(s_0)=\al_2(s_0)=x_0 $. Since $ [\al_1]_{\cX}=[\al_2]_{\cX} $, there is a homotopy function $ H(\cdot,\cdot):[0,1]\times\Ss^1\to\cX $ from $ \al_1 $ to $ \al_2 $. Let $ \al_0(\theta)=H(\theta/(2\pi),s_0) $ be a loop. By Theorem 1.19 of \cite{H02}, we have, 
\be
[\al_2]_{\cX,x_0}=[\al_0]_{\cX,x_0}[\al_1]_{\cX,x_0}[\al_0]_{\cX,x_0}^{-1}.\label{c1c2}
\ee
On the other hand, for $ i=0,1,2 $, we set $ \al_i:\Ss^1\to\cX $ such that $
\al_i(0)=\al_i(2\pi)=x_0 $, satisfying \eqref{c1c2}. For $ \theta\in[0,2\pi] $, let $ H(\theta,\cdot):[0,2\pi]\to\cX $ be the curve $ H(\theta,s)=\al_0((1-s/(2\pi))\theta+s) $, i.e., a curve from $ \al_0(\theta) $ to $ \al_0(2\pi)=x_0 $. Observe that $ H(0,\cdot)=\al_0(\cdot) $ and $ H(2\pi,\cdot)=x_0 $. Now, $ G(t,\cdot)=H(2\pi t,\cdot)*\al_1(\cdot)*\wt{H}(2\pi t,\cdot) $ is a free homotopy with $
G(0,\cdot)=\al_0(\cdot)*\al_1(\cdot)*\wt{\al}_0(\cdot) $ and $ G(1,\cdot)=\al_1(\cdot) $. Therefore, $ \Phi_{x_0}([\al_1]_{\cX,x_0})=\Phi_{x_0}([\al_2]_{\cX,x_0}) $.
\end{proof}

For the rest of this subsection, we assume that $ \cX $ is a path-connected, compact, and smooth Riemannian manifold without boundary, which can be isometrically embedded into Euclidean space $ \R^k $ for some $ k\in\Z_+ $. For each $ \al\in C^0(\Ss^1,\cX) $, we define the energy of the homotopy class $ [\al]_{\cX} $ as
\be
\mathcal{E}([\al]_{\cX}):=\inf\left\{\f{1}{2}\int_{\Ss^1}|\beta'(\theta)|^2\ud\theta:\beta\in H^1(\Ss^1,\cX),\,\,[\beta]_{\cX}=[\al]_{\cX}\right\}.\label{ldaalbeta}
\ee
Since the embedding $ H^1(\Ss^1,\cX)\hookrightarrow C^0(\Ss^1,\cX) $ is compact and dense, the infimum in \eqref{ldaalbeta} can be achieved by similar arguments in Proposition \ref{ExistenceThm}. Each minimizer $ \beta $ is a geodesic and then $ |\beta'| $ is constant. 

For each $ \al\in C^0(\Ss^1,\cN) $ and $ x_0\in\cX $, we define
\begin{align*}
\cE_{x_0}^*([\al]_{\cX}):=\inf\left\{\sum_{i=1}^n\cE([\al_i]_{\cX}):\{[\al_i]_{\cX,x_0}\}_{i=1}^n\subset\pi_1(\cX,x_0),\,\,[\al]_{\cX}=[\al_1*\cdots*\al_n]_{\cX}\right\}.
\end{align*}
We see that the definition of $ \cE_{x_0}^*([\al]_{\cX}) $ does not depend on the choice of $ x_0 $. Indeed, for $ x_1\in\cX $, with $ x_1\neq x_0 $, since $ \cX $ is path-connected, there exists $ \ga\in C^0([0,2\pi],\cX) $ such that $ \ga(0)=x_1 $ and $ \ga(2\pi)=x_0 $. If $ \{[\al_i]_{\cX,x_0}\}_{i=1}^n\subset\pi_1(\cX,x_0) $ satisfies $ [\al_1*\cdots*\al_n]_{\cX}=[\al]_{\cX} $, we set $ \beta_i=\ga*\al_i*\wt{\ga} $ for $ i=1,2,...,n $ and get $ [\beta_i]_{\cX}=[\al_i]_{\cX} $. Moreover,
$$
[\beta_1*\cdots*\beta_n]_{\cX}=[\ga*\al_1*\cdots*\al_n*\wt{\ga}]_{\cX}=[\al]_{\cX},
$$
and we have $ \sum_{i=1}^n\cE([\beta_i]_{\cX})=\sum_{i=1}^n\cE([\al_i]_{\cX}) $, which implies $ \cE_{x_1}^*([\al]_{\cX})\leq\cE_{x_0}^*([\al]_{\cX}) $. Similarly we can obtain the reverse inequality and then $ \cE_{x_1}^*([\al]_{\cX})=\cE_{x_0}^*([\al]_{\cX}) $. For this reason, we will drop the subscription $ x_0 $ and use the notation $ \cE^*([\al]_{\cX}) $. Now we have the lemma as follows. 

\begin{lem}\label{singuenergessum}
Let $ \al\in C^0(\Ss^1,\cX) $ and $ x_0\in\cX $. The following properties hold.
\begin{enumerate}
\item $ 0\leq\cE^*([\al]_{\cX})\leq\cE([\al]_{\cX}) $. $ \al $ is a trivial loop, if and only if $ \cE^*([\al]_{\cX})=\cE([\al]_{\cX})=0 $.
\item If $ x_0\in\al $, then 
$$
\cE^*([\al]_{\cX})=\inf\left\{\sum_{i=1}^n\cE([\al]_{\cX}):\{[\al_i]_{\cX,x_0}\}_{i=1}^n\subset\pi_1(\cX,x_0),\,\,[\al]_{\cX,x_0}=\prod_{i=1}^n[\al_i]_{\cX,x_0}\right\}.
$$
\item If $ \{[\al_i]_{\cX,x_0}\}_{i=1}^n\subset\pi_1(\cX,x_0) $ with $ n\in\Z_+ $ and $ [\al]_{\cX}=[\al_1*\cdots*\al_n]_{\cX} $, then $ \cE^*([\al]_{\cX})\leq\sum_{i=1}^n\cE^*([\al_i]_{\cX}) $.
\end{enumerate}
\end{lem}
\begin{proof}
The first property follows directly from the definition. Next, we show the second one. If $ \{[\al_i]_{\cX,x_0}\}_{i=1}^n\subset\pi_1(\cX,x_0) $ satisfies $ [\al_1*\cdots*\al_n]_{\cX}=[\al]_{\cX} $, then it follows from Lemma \ref{LemmaHom} that $ [\al_1*\cdots*\al_n]_{\cX,x_0}=[\ga*\al*\wt{\ga}]_{\cX,x_0} $ for some $ [\ga]_{\cX,x_0}\in\pi_1(\cX,x_0) $. Choosing $ \beta_i=\ga*\al_i*\wt{\ga} $ for $ i=1,2,...,n $, we can obtain $ [\al_i]_{\cX}=[\beta_i]_{\cX} $, $ \prod_{i=1}^n[\beta_i]_{\cX,x_0}=[\al]_{\cX,x_0} $, and $ \sum_{i=1}^n\cE([\al_i]_{\cX})=\sum_{i=1}^n\cE([\beta_i]_{\cX}) $, this directly implies the result. For the third property, fix $ \va>0 $. For any $ \al_i $, by the definition of $ \cE^*(\cdot) $ and the second property, we can construct $ \{[\al_{ij}]_{\cX,x_0}\}_{j=1}^{m_i} $ such that $ [\al_i]_{\cX,x_0}=\prod_{j=1}^{m_i}[\al_{ij}]_{\cX,x_0} $ and $ \sum_{j=1}^{m_i}\cE([\al_{ij}]_{\cX})\leq\cE^*([\al_i]_{\cX})+\va/n $. Now we deduce that
$$
[(\al_{11}*\cdots*\al_{1m_1})*\cdots*(\al_{n1}*\cdots*\al_{nm_n})]_{\cX}=[\al]_{\cX}
$$
and then
$$
\cE^*([\al]_{\cX})\leq\sum_{i=1}^n\sum_{j=1}^{m_i}\cE([\al_{ij}]_{\cX})\leq\sum_{i=1}^n\cE^*([\al_i]_{\cX})+\va.
$$
Letting $ \va\to 0^+ $, the result holds.
\end{proof}

For $ 0<s<r $, we define the annulus by 
\be
A_{s,r}^2(x_0):=\{x\in\R^2:s<|x-x_0|<r\}.\label{Annulus}
\ee
If $ x_0=0 $, we will use the notation $ A_{s,r}^2 $ for simplicity. 

Assume that $ u\in H^1(A_{s,r}^2(x_0),\cX) $. We give the definition for the free homotopy class of $ u|_{\pa B_r^2(x_0)} $ in the following lemma. 

\begin{lem}\label{propHomt}
If $ 0<s<r $ and $ u\in H^1(A_{s,r}^2(x_0),\cX) $, then for a.e. $ t\in(s,r) $, $ u|_{\pa B_t^2(x_0)}\in C^0(\Ss^1,\cX) $ is a loop on $ \cX $ and are in the same free homotopy class. We call this the free homotopy class of $ u|_{\pa B_r^2(x_0)} $, denoted by $ [u|_{\pa B_r^2(x_0)}]_{\cX} $.
\end{lem}

\begin{proof}
Under a translation, we assume that $ x_0=0 $. Since $ u\in H^1(A_{s,r}^2,\cX) $, by Fubini theorem, we obtain that for a.e. $ t\in(s,r) $, $ u|_{\pa B_t^2}\in H^1(\pa B_t^2,\cX) $. By Sobolev embedding theorem, for such $ t $, $ u|_{\pa B_t^2} $ is continuous and the free homotopy class of it is well defined. Using Lemma \ref{SmoothApproximation}, we can construct a smooth approximation $ u_{\delta}:A_{s,r}^2\to\cX $ such that $ u_{\delta}\to u $ in $ H^1(A_{s,r}^2,\R^k) $ as $ \delta\to 0^+ $, where $ \cX $ is isometrically embedded into $ \R^k $. By using Fubini and Sobolev embedding theorems again, we get that $ u_{\delta}\to u $ uniformly on $ \pa B_t^2 $ for a.e. $ t\in(s,r) $. As a result, for a.e. $ t\in(s,r) $, $ u|_{\pa B_t^2} $ belong to the same homotopy class.
\end{proof}

In view of the above lemma, for $ u\in H^1(A_{s,r}^2(x_0),\cX) $, $ \cE^*([u|_{\pa B_r^2(x_0)}]_{\cX}) $ is well defined. Furthermore, we have the following result.

\begin{lem}\label{Numbersum}
Assume that $ n\in\Z_+ $, $ \{s_i\}_{i=0}^n\cup\{r_i\}_{i=0}^n\subset\R_+ $, and $ \{x_i\}_{i=0}^n\subset\R^2 $ satisfy the following properties:
\begin{itemize}
\item $ 0<s_i<r_i $ for any $ i=0,1,2,...,n $ and $ \cup_{i=1}^nB_{r_i}^2(x_i)\subset B_{s_0}^2(x_0) $.
\item $ B_{r_i}^2(x_i)\cap B_{r_j}^2(x_j)=\emptyset $ for any $ i,j\in\{1,2,...,n\} $ such that $ i\neq j $.
\end{itemize} 
If $ u\in H^1(B_{r_0}^2(x_0)\backslash\cup_{i=1}^nB_{s_i}^2(x_i),\cX) $,
then $ \{\cE^*([u|_{\pa B_{r_i}^2(x_i)}]_{\cX})\}_{i=0}^n $ are well defined and the following inequality holds
\be
\cE^*([u|_{\pa B_{r_0}^2(x_0)}]_{\cX})\leq \sum_{i=1}^n\cE^*([u|_{\pa B_{r_i}^2(x_i)}]_{\cX}).\label{deg1}
\ee
In particular, if $ n=1 $, then the inequality in \eqref{deg1} will change to equality.
\end{lem}

\begin{proof}
We denote $ U:=B_{r_0}^2(x_0)\backslash\cup_{i=1}^nB_{s_i}^2(x_i) $. In view of Lemma \ref{propHomt} and \ref{SmoothApproximation}, we can assume that $ u\in (H^1\cap C^{\ift}(U))\cap C^0(\ol{U}) $. If $ u|_{\pa B_r^2(x_0)} $ is trivial, we have $ \cE^*([u|_{\pa B_r^2(x_0)}]_{\cX})=0 $, and then there is nothing to prove. Now we assume that $ u|_{\pa B_r^2(x_0)} $ is non-trivial. Choose $ \wh{x}_0 $, $ \{y_i\}_{i=1}^n $, and curves $ \{c_i\}_{i=1}^n\subset C^0([0,2\pi],U) $ such that $ c_i(0)=\wh{x}_0 $ and $ c_i(2\pi)=y_i $ for any $ i=1,2,...,n $. Set $ \{\al_i\}_{i=1}^n:=\{u|_{c_i}\}_{i=1}^n $. We have,
$$
[(\al_1*u|_{\pa B_{r_1}^2(x_1)}*\wt{\al}_1)*\cdots*(\al_n*u|_{\pa B_{r_n}^2(x_n)}*\wt{\al}_n)]_{\cX}=[u|_{\pa B_{r_0}^2(x_0)}]_{\cX}
$$
since $ u $ itself is the free homotopy function. By Lemma \ref{singuenergessum}, we have
$$
\cE^*([u|_{\pa B_{r_0}^2(x_0)}]_{\cX})\leq\sum_{i=1}^n\cE^*([\al_i*u|_{\pa B_{r_i}^2(x_i)}*\wt{\al}_i]_{\cX})=\sum_{i=1}^n\cE^*([u|_{\pa B_{r_i}^2(x_i)}]_{\cX}).
$$
Finally, when $ n=1 $, we have that $ u $ is a free homotopy function of $ u|_{\pa B_{r_1}^2(x_1)} $ and $ u|_{\pa B_{r_0}^2(x_0)} $, which implies that $
\cE^*([u|_{\pa B_{r_0}^2(x_0)}]_{\cX})=\cE^*([u|_{\pa B_{r_1}^2(x_1)}]_{\cX}) $.
\end{proof}

\begin{defn}
For any $ V\subset\subset\R^2 $, the radius of $ V $ is defined by
$$
\op{rad}(V):=\inf\left\{\sum_{i=1}^nr_i:\ol{V}\subset\bigcup_{i=1}^nB_{r_i}^2(x_i)\right\}.
$$
\end{defn}

\begin{lem}\label{radlem}
For the radius defined as above, the following properties hold.
\begin{enumerate}
\item $ \op{rad}(V)\leq\op{rad}(\ol{V}) $.
\item For a bounded set $ V $, $ \op{rad}(V)\leq\diam(\pa V) $.
\item For bounded sets $ \{V_i\}_{i=1}^n $, $ \op{rad}\(\cup_{i=1}^nV_i\)\leq\sum_{i=1}^n\op{rad}(V_i) $. 
\end{enumerate}
\end{lem}
\begin{proof}
The first property follows directly from the definition and the second is the consequence of $ \diam(\ol{V})=\diam(\pa V) $. For the third one, by inductive arguments, it suffices to show the case of $ n=2 $. For $ \va>0 $ and $ i=1,2 $, we choose $ \{B_{r_{ij}}^2(x_{ij})\}_{j=1}^{n_i} $, such that $ \ol{V_i}\subset\cup_{j=1}^{n_i}B_{r_{ij}}^2(x_{ij}) $, and $
\sum_{j=1}^{n_i}r_{ij}\leq\op{rad}(V_i)+\va/2 $. Now we have $ \ol{V_1\cup V_2}\subset\cup_{i=1}^2\cup_{j=1}^{n_i}B_{r_{ij}}^2(x_{ij}) $ and
$$
\op{rad}(V_1\cup V_2)\leq\sum_{i=1}^2\sum_{j=1}^{n_i}r_{ij}\leq\op{rad}(V_1)+\op{rad}(V_2)+\va.
$$
Letting $ \va\to 0^+ $, we complete the proof.
\end{proof}

\begin{prop}\label{Plower}
Let $ \w\subset B_r^2 $ with $ r>0 $ such that $ \dist(\w,\pa B_r^2)\geq 2\lda r $ for some $ \lda\geq 19/40 $. For any $ u\in H^1(B_r^2\backslash\w,\cX) $, there holds
$$
\f{1}{2}\int_{B_r^2\backslash\w}|\na u|^2\ud\HH^2\geq\cE^*([u|_{\pa B_r^2}]_{\cX})\log\f{\lda r}{\op{rad}(\w)}.
$$
\end{prop}

\begin{lem}\label{lemmalowers}
Assume that $ 0<s<r $, and $ u\in H^1(A_{s,r}^2,\cX) $, then
$$
\f{1}{2}\int_{A_{s,r}^2}|\na u|^2\ud\HH^2\geq\cE^*([u|_{\pa B_r^2}]_{\cX})\log\f{r}{s}.
$$
\end{lem}
\begin{proof}
By Lemma \ref{SmoothApproximation}, without loss of generality, we can assume that $ u $ is smooth. Using polar coordinates and the first property of Lemma \ref{singuenergessum}, we have
\begin{align*}
\f{1}{2}\int_{A_{s,r}^2}|\na u|^2\ud\HH^2&=\int_s^r\int_{\Ss^1}\(\rho|\pa_{\rho}u|^2+\f{1}{\rho}|\pa_{\theta}u|^2\)\ud\theta\ud\rho\geq\cE^*([u|_{\pa B_r^2}]_{\cX})\int_s^r\f{\ud\rho}{\rho},
\end{align*}
which directly implies the result.
\end{proof}

\begin{proof}[Proof of Proposition \ref{Plower}]
By a scaling argument, we can assume that $ r=1 $. For any $ 0<\eta<1/40 $, by the definition of radius for an arbitrary set, we obtain $ \{B_{r_i}^2(x_i)\}_{i=1}^n $ such that
\be
\ol{\w}\subset\bigcup_{i=1}^nB_{r_i}^2(x_i)\text{ and }\sum_{i=1}^nr_i\leq\op{rad}(\w)+\eta.\label{etasmall}
\ee
Moreover, by $ \lda\geq 19/40 $, we have $ \ol{\w}\subset\ol{B_{1/20}^2} $. For this reason, we assume $ x_i\in\ol{B_{1/20}^2} $ and $ 0<r_i\leq 1/10 $ for any $ i=1,2,...,n $. Without loss of generality, we can assume that $ \{\ol{B_{r_i}^2(x_i)}\}_{i=1}^n $ are mutually disjoint. If not, for example, $ \ol{B_{r_i}^2(x_i)}\cap \ol{B_{r_j}^2(x_j)}\neq\emptyset $, we construct a ball $ B_{r_{ij}}^2(x_{ij}) $ such that $
x_{ij}=(x_i+x_j)/2\in\ol{B_{1/20}^2} $, $ r_{ij}\leq r_i+r_j $, and $ B_{r_i}^2(x_i)\cup B_{r_j}^2(x_j)\subset B_{r_{ij}}^2(x_{ij}) $. If $ \ol{B_{r_{ij}}^2(x_{ij})} $ intersects $ \ol{B_{r_k}^2(x_k)} $ for another $ k\neq i,j $, we construct $ B_{r_{ijk}}^2(x_{ijk}) $ such that $ x_{ijk}=(x_{ij}+x_k)/2\in\ol{B_{1/20}^2} $, $ r_{ijk}\leq r_i+r_j+r_k $ and $ B_{r_{ij}}^2(x_{ij})\cup B_{r_k}^2(x_k)\subset B_{r_{ijk}}^2(x_{ijk}) $. Repeating this procedure, we finally get a ball $ B_r^2(x) $ that contains balls $ \{B_{r_i}^2(x_i)\}_{i=1}^{\ell} $ (relabeled) such that $ x\in\ol{B_{1/20}^2} $, $ r\leq\sum_{i=1}^{\ell}r_i $, and $ \ol{B_r^2(x)}\cap(\cup_{i=\ell+1}^n\ol{B_{r_i}^2(x_i)})=\emptyset $. In view of \eqref{etasmall} and the fact that $ \ol{\w}\subset\ol{B_{1/20}^2} $, we have that $
r\leq\op{rad}(\w)+\eta<1/10 $. Replacing $ \{B_{r_i}^2(x_i)\}_{i=1}^n $ by $ \{B_r^2(x)\}\cup\{B_{r_i}^2(x_i)\}_{i=\ell+1}^n $, we have $ \ol{\w}\subset B_r^2(x)\cup(\cup_{i=\ell+1}^nB_{r_i}^2(x_i)) $ and $ r+\sum_{i=\ell+1}^nr_i\leq\op{rad}(\w)+\eta $. Now we can continue this procedure for finite times such that all the closures of balls in the construction are mutually disjoint. Moreover, through some mild modifications, we can also assume that $
\w\cap(\cup_{i=1}^n\pa B_{r_i}^2(x_i))=\emptyset $.

Since $ \{\ol{B_{r_i}^2(x_i)}\}_{i=1}^n $ are mutually disjoint, for any $ i\in\{1,2,...,n\} $, there is a ball $ B_r^2(x_i) $ such that $ r_i<r $ and $ B_r^2(x_i)\cap(\cup_{1\leq j\leq n,j\neq i}B_{r_j}^2(x_j))=\emptyset $. By Lemma \ref{Numbersum}, $ \{\cE(u|_{\pa B_{r_i}^2(x_i)})\}_{i=1}^n $ are well defined. We call that $ \{B_{r_i}^2(x_i)\}_{i=1}^n $ is our family of balls at time $ t=0 $. Next, we proceed to define the family $ \{B^2(t;i)\}_{i=1}^{n(t)} $ for $ t>0 $. We will denote by $ r_i(t) $ the radius of $ B^2(t;i) $, by $ x_i(t) $ the center of $ B^2(t;i) $, and define in
addition the seed size $ \mu_i(t) $ of $ B^2(t;i) $. Initially, $ \mu_i(0)=r_i(0)=r_i $. Suppose that the family is defined at some time $ t>0 $, and that the
quantity $ \log(r_i(t)/\mu_i(0)) $
are the same for all the balls $ B^2(t;i) $, call it $ \al(t) $. Three
cases will occur.\smallskip

\noindent
\underline{\textbf{Case 1.}} The family $ \{B^2(t;i)\}_{i=1}^{n(t)} $ satisfies ``one ball condition", i.e., $ n(t)\geq 2 $ and there exists a $ B_r^2(x) $ such that $ x\in\ol{B_{1/20}^2} $, $ 0<r\leq 3/20 $, satisfying $ \cup_{i=1}^{n(t)}\ol{B^2(t;i)}\subset B_r^2(x)$ and $ 0<r\leq\sum_{i=1}^{n(t)}r_i(t) $. For this case, we refine the family $ \{B^2(t;i)\}_{i=1}^{n(t)} $ to the one ball $ B_r^2(x) $. The seed size $ \mu $ of $ B_r^2(x) $ is defined to be such that
\be
\log\f{r}{\mu}=\log\f{r_i(t)}{\mu_i(t)}=\al(t)\text{ for any }i=1,2,...,n(t).\label{seedsizegiven}
\ee
Here we have to show that such $ \mu $ exists. Firstly we fix $ 1\leq i_0\leq n(t) $. This implies $ r\geq r_{i_0}(t) $ and thus, as $ s $ increases from $ \mu_{i_0}(t) $ to $ r $, $ \log(r/s) $ continuously decreases from some number greater than $ \log(r_{i_0}(t)/\mu_{i_0}(t)) $ to zero. There must be a suitable $ s=\mu $ in the interval $ [\mu_{i_0}(t),r] $ satisfying \eqref{seedsizegiven}.\smallskip

\noindent
\underline{\textbf{Case 2.}} $ \{B^2(t;i)\}_{i=1}^{n(t)} $ does not satisfy the ``one ball condition" such that $ n(t)\geq 2 $, 
$$
\ol{B^2(t;i)}\cap\ol{B^2(t;j)}=\emptyset\text{ for any }1\leq i,j\leq n(t),\,\,i\neq j,
$$ 
or $ n(t)=1 $. In this case, as time increases, we leave $ \{\mu_i(t)\}_{i=1}^{n(t)} $ remaining constant and expand each ball without changing their centers in a way such that the quantity $ \al(t)=\log(r_i(t)/\mu_i(t)) $ remains the same for all the balls, and is a continuous increasing function of $ t $. We stop expanding either when there are two balls $ \ol{B^2(t;i)}\cap\ol{B^2(t;j)}\neq\emptyset $ with $ i\neq j $ or one of closures of balls in the family intersects $ \pa B_1^2 $. Note that when we are expanding the ball, the first case will not occur by the definition. \smallskip

\noindent
\underline{\textbf{Case 3.}} $ \{B^2(t;i)\}_{i=1}^{n(t)} $ does not satisfy the ``one ball condition" and there are some $ i\neq j $ such that $ \ol{B^2(t;i)}\cap\ol{B^2(t;j)}\neq\emptyset $. For this case, we construct $ B_{r_{ij}}^2(x_{ij}) $ such that $
x_{ij}=(x_i(t)+x_j(t))/2\in\ol{B_{1/20}^2} $, $ r_{ij}\leq r_i(t)+r_j(t) $, and $ B^2(t;i)\cup B^2(t;j)\subset B_{r_{ij}}^2(x_{ij}) $. If $ \ol{B_{r_{ij}}^2(x_{ij})} $ intersects $ \ol{B^2(t;k)} $ for some $ k\neq i,j $, we enlarge it again to the ball $ B_{r_{ijk}}^2(x_{ijk}) $ such that $ x_{ijk}=(x_{ij}+x_k(t))/2\in \ol{B_{1/20}^2} $, $ r_{ijk}\leq r_i(t)+r_j(t)+r_k(t) $, and $ B_{r_{ij}}^2(x_{ij})\cup B^2(t;k)\subset B_{r_{ijk}}^2(x_{ijk}) $. Repeating such procedure, we can finally get a ball $ B_r^2(x) $ that contains balls $ \{B^2(t;i)\}_{i=1}^{\ell} $, (relabeled), $ x\in\ol{B_{1/20}^2} $ and $ r\leq\sum_{i=1}^{\ell}r_i(t) $. We claim, when we are constructing larger balls, the procedure will never produce a ball with $ r>3/20 $. Indeed, for this time $ t $ before the enlarging step, we have $ \ol{B^2(t;i)}\cap \ol{B^2(t;j)}=\emptyset $ for any $ i\neq j $ and $ \{x_i(t)\}_{i=1}^{n(t)}\subset \ol{B_{1/20}^2} $. Since $ n(t)\geq 2 $, this implies that $
\cup_{i=1}^{n(t)}\ol{B^2(t;i)}\subset B_{1/10}^2 $. If there produces a ball $ B_r^2(x) $ such that $ r>3/20 $, then by the fact $ x\in\ol{B_{1/20}^2} $, we have $ \cup_{i=1}^{n(t)}B^2(t;i)\subset B_{3/20}^2(x) $, which is a contradiction to the assumption that $ \{B^2(t;i)\}_{i=1}^{n(t)} $ does not satisfy ``one ball condition". After the enlarging procedure, we refine the family at the time $ t $ to be the original family less $ \{B^2(t;i)\}_{i=1}^{\ell} $,
plus the ball $ B_r^2(x) $ we just constructed. The seed size $ \mu $ of $ B_r^2(x) $ is given by \eqref{seedsizegiven} for $ i=1,2,...,\ell $ and the existence of such $ \mu $ is analogous to the first case.

For the constructions above, we have the following properties.
\begin{enumerate}
\item $ \ol{\w}\subset\cup_{i=1}^{n(t)}B^2(t;i) $.
\item The seed sizes satisfy
\be
\mu_i(t)\leq\sum_{j:B^2(0;j)\subset B^2(t;i)}r_j(0)\text{ for any }i=1,2,...,n(t).\label{muit}
\ee
\item For any $ B^2(t;i) $ in the family,
$$
\f{1}{2}\int_{B^2(t;i)\backslash\w}|\na u|^2\ud\HH^2\geq\cE^*([u|_{\pa B_{r_i(t)}^2(x_i(t))}]_{\cX})\al(t),
$$
where $ \al(t)=\log(r_i(t)/\mu_i(t)) $. 
\end{enumerate}

Here, we note $ \{\cE^*([u|_{\pa B_{r_i(t)}^2(x_i(t))}]_{\cX})\}_{i=1}^{n(t)} $ are well defined by the first property. Moreover, $ \al(t) $ does not depend on $ i $, which is ensured by the constructions. The first property is trivial. For the second property, it is easy to see that this is true for $ t=0 $ and remains true under the procedure of Case 2. So we only need to show that \eqref{muit} remains true under the procedure of Case 1 and Case 3 above. Since the proof is almost the same, we focus on the constructions in Case 3 and the other follows analogously. We assume that for the time $ t $, $ \{B^2(t;i)\}_{i=1}^{n(t)} $ satisfies \eqref{muit} and $ \{B^2(t;i)\}_{i=1}^{\ell} $ are merged to $ B_r^2(x) $ with seed size $ \mu $, we have that \eqref{seedsizegiven} is true. By using the property of the function $ \log(\cdot) $, we obtain 
$$
\al(t)=\log\f{r}{\mu}\geq\log\f{\sum_{i=1}^{\ell}r_i(t)}{\sum_{i=1}^{\ell}\mu_i(t)}.
$$
From the construction $ B_r^2(x) $, we have that $ r\leq\sum_{i=1}^{\ell}r_i(t) $, and then $ \mu\leq\sum_{i=1}^{\ell}\mu_i(t) $. By the assumption on $ \{B^2(t;i)\}_{i=1}^{n(t)} $, we deduce
$$
\mu\leq\sum_{i=1}^{\ell}\mu_i(t)\leq\sum_{i=1}^{\ell}\sum_{j:B^2(0;j)\subset B^2(t;i)}r_j(0)=\sum_{j:B^2(0;j)\subset B_r^2(x)}r_j(0),
$$
which completes the proof. Now let us prove the third property. Like the proof of the second one, we still only consider Case 3. Similarly, we assume that $ \{B^2(t;i)\}_{i=1}^{\ell} $ are merged to $ B_r^2(x) $ with seed size $ \mu $. We use the inductive arguments, i.e., we suppose that the third property holds for $ B^2(t;i) $ with $ i=1,2,...,\ell $. By using Lemma \ref{lemmalowers}, we have
\be
\begin{aligned}
\f{1}{2}\int_{B_r^2(x)\backslash\w}|\na u|^2\ud\HH^2&\geq\sum_{i=1}^{\ell}\f{1}{2}\int_{B^2(t;i)\backslash\w}|\na u|^2\ud\HH^2\\
&\geq\sum_{i=1}^{\ell}\cE^*([u|_{\pa B_{r_i(t)}^2(x_i(t))}]_{\cX})\log\f{r_i(t)}{\mu_i(t)}\\
&\geq\cE^*([u|_{\pa B_r^2(x)}]_{\cX})\al(t),
\end{aligned}\label{largein}
\ee
where for the second inequality, we have used the inductive assumption, and for the last inequality, we have used Lemma \ref{Numbersum}. When we are expanding the radius of $ B_r^2(x) $ to some $ R>r $ in the procedure of Case 2, there holds
\begin{align*}
\f{1}{2}\int_{B_R^2(x)\backslash\w}|\na u|^2\ud\HH^2&=\f{1}{2}\int_{A_{r,R}^2(x)}|\na u|^2\ud\HH^2+\f{1}{2}\int_{B_r^2(x)}|\na u|^2\ud\HH^2\\
&\geq\cE^*([u|_{\pa B_r^2(x)}]_{\cX})\log\f{R}{r}+\cE^*([u|_{\pa B_r^2(x)}]_{\cX})\log\f{r}{\mu}\\
&=\cE^*([u|_{\pa B_R^2(x)}]_{\cX})\log\f{R}{\mu},
\end{align*}
where for the inequality and last equality, we have used Lemma \ref{lemmalowers} again. Now we have completed the proof of the third property.

In the constructions for the family of balls, for finite time, there is only one ball in the collection and the procedure ends when the closure of this ball is expanded to touch the boundary of $ B_1^2 $. At this time, we denote this ball $ B_{r(T)}^2(x(T)) $ with seed size $ \mu(T) $. By the second property, we have 
\be
\mu(T)\leq\sum_{i=1}^{n(0)}r_i(0)=\sum_{i=1}^nr_i\leq\op{rad}(\w)+\eta.\label{muTeta}
\ee
Obviously, here $ r(T)>\lda $. Also, we observe that for sufficiently small $ \delta>0 $, $ x(t)=x(T) $ when $ t\in[T-\delta,T] $, $ r(t)<r(T) $ and $
\w\subset B_{r(t)}^2(x(t))\subset B_{r(T)}^2(x(T))\subset B_1^2 $. It follows from the third property that
\begin{align*}
\f{1}{2}\int_{B_1^2\backslash\w}|\na u|^2\ud\HH^2&\geq\f{1}{2}\int_{B_{r(T-\delta/2)}^2(x(T))\backslash\w}|\na u|^2\ud\HH^2\\
&\geq\cE^*([u|_{\pa B_{r(T-\delta/2)}^2(x(T))}]_{\cX})\log\f{r(T-\delta/2)}{\mu(T)}\\
&\geq\cE^*([u|_{\pa B_1^2}]_{\cX})\log\f{r(T-\delta/2)}{\op{rad}(\w)+\eta},
\end{align*}
where for the last inequality, we have used the Lemma \ref{Numbersum} for $ n=1 $ and \eqref{muTeta}. Since $ r(T)>\lda $ and $ r(t) $ is continuous, we can let $ \eta\to 0^+ $ and $ \delta\to 0^+ $ and complete the proof.
\end{proof}

\subsection{Some geometric results of the manifold \texorpdfstring{$ \cP $}{} and \texorpdfstring{$ \cN $}{}}

In this subsection, we present some geometric properties of $ \cP $ and $ \cN $. For a path-connected topological space $ \cX $, sometimes we use $ \pi_1(\cX) $ to denote the fundamental group of it and drop the base point for simplicity.

\begin{lem}\label{N0fgroup}
$ \pi_1(\cP)\cong\Z_2=\{\ol{0},\ol{1}\} $ and $ \Ss^2 $ is the universal covering space of $ \cP $ with the covering map
\be
p_0:\Ss^2\to\cN,\quad p_0(\n)=r_*\n\n,\,\,\n\in\Ss^2.\label{nnp0}
\ee
In particular, for a fixed point $ \PP_0\in\cP $ and a loop $ \A\in C^0(\Ss^1,\cP) $ with $ \PP_0\in\A $, $ [\A]_{\cP,\PP_0}=[\PP_0]_{\cP,\PP_0} $ if and only if $ \A=p_0(\n) $, where $ \n\in C^0(\Ss^1,\Ss^2) $.
\end{lem}

\begin{proof}
By the definition of \eqref{N0}, we have that $ \pi_1(\cP)\cong\pi_1(\R\PP^2)\cong\Z_2 $. Moreover, by the construction of $ \R\PP^2 $, $ \Ss^2 $ is the universal covering space of $ \cP $ and one can easily check that \eqref{nnp0} gives the covering map.
\end{proof}

Since $ \#\pi_1(\cP)=2 $, then for any point $ \PP_0\in\cP $, there is only one isomorphism from $ \pi_1(\cP,\PP_0) $ to $ \Z_2 $. It is of no ambiguity to identify $ \pi_1(\cP) $ with $ \Z_2 $, i.e., for a loop $ \mathbf{A}\in C^0(\Ss^1,\cP) $ with $ \PP_0\in\A $, we will use the convention throughout this paper that $
[\A]_{\cP,\PP_0}:=\ol{0} $ when $ \A\sim_{\cP,\PP_0}\PP_0 $ and $ [\A]_{\cP,\PP_0}:=\ol{1} $ for otherwise. 

\begin{prop}\label{freeN0}
Let $ \A,\B\in C^0(\Ss^1,\cP) $ be two loops on $ \cP $. Assume that $ \PP_1,\PP_2 $ such that $ \PP_1\in\A $ and $ \PP_2\in\B $. There holds, $ [\A]_{\cP}=[\B]_{\cP} $ if and only if $ [\A]_{\cP,\PP_1}=[\B]_{\cP,\PP_2} $ in $ \Z_2 $. In particular, $ \A $ is a trivial loop if and only if $ [\A]_{\cP,\PP_1}=\ol{0} $. Moreover, there are two classes in $ [\Ss^1,\cP] $, denoted by $ \rh_0 $ and $ \rh_1 $ corresponding the sets of to trivial loops and non-trivial loops respectively.
\end{prop}

\begin{rem}
Firstly, we note that $ [\A]_{\cP,\PP_1},[\B]_{\cP,\PP_2} $ do not depend on the choices of $ \PP_1,\PP_2 $. Take the loop $ \A $ for example. For $ \wh{\PP}_1\in\A $ and $ \wh{\PP}_1\neq\PP_1 $, by using the fact that $ \cP $ is path-connected, there is a curve $ \CC:[0,2\pi]\to\cP $ such that $ \CC(0)=\wh{\PP}_1 $ and $ \CC(2\pi)=\PP_1 $. It induces an isomorphism $ \tau_{\cP,\CC}:\pi_1(\cP,\PP_1)\to \pi_1(\cP,\wh{\PP}_1) $, such that $ \tau_{\cP,\CC}([\mathbf{L}]_{\cP,\PP_1})=[\CC*\mathbf{L}*\wt{\CC}]_{\cN,\wh{\PP}_1} $ for any loop $ \mathbf{L}\in C^0(\Ss^1,\cP) $ based on $ \PP_1 $. Since $ \tau_{\cP,\CC} $ is the automorphism of $ \Z_2 $, we have $ [\A]_{\cP,\wh{\PP}_1}=[\A]_{\cP,\PP_1} $.
\end{rem}

\begin{proof}[Proof of Proposition \ref{freeN0}]
Choose another point $ \PP_0\in\cP $ and construct two curves $ \CC_1,\CC_2:[0,2\pi]\to\cP $ such that $ \CC_1(0)=\PP_0 $, $ \CC_1(2\pi)=\PP_1 $, $ \CC_2(0)=\PP_0 $ and $ \CC_2(2\pi)=\PP_2 $. Define $ \{\tau_{\cP,\CC_{\ell}}\}_{\ell=1}^2 $ by $ \tau_{\cP,\CC_{\ell}}:\pi_1(\cP,\PP_{\ell})\to \pi_1(\cP,\PP_0) $ for $ \ell=1,2 $, where $ \tau_{\cP,\CC_{\ell}}([\mathbf{L}_{\ell}]_{\cP,\PP_{\ell}})=[\CC_{\ell}*\mathbf{L}_{\ell}*\wt{\CC}_{\ell}]_{\cP,\PP_0} $ for $ \mathbf{L}_{\ell}\in C^0(\Ss^1,\cP) $ based on $ \PP_{\ell} $. In view of Lemma \ref{LemmaHom}, we have
$$
[\A]_{\cP}=[\B]_{\cP}\Leftrightarrow[\CC_1*\A*\wt{\CC}_1]_{\cP}=[\CC_2*\B*\wt{\CC}_2]_{\cP}\Leftrightarrow\tau_{\cP,\CC_1}([\A]_{\cP,\PP_1})\stackrel{c}{\sim}\tau_{\cP,\CC_2}([\B]_{\cP,\PP_2})\text{ in }\Z_2.
$$
By the property that $ \{\tau_{\cP,\CC_{\ell}}\}_{\ell=1}^2 $ are both automorphisms of $ \Z_2 $, we have $ [\A]_{\cP}=[\B]_{\cP} $ if and only if $ [\A]_{\cP,\PP_1}=[\B]_{\cP,\PP_2} $.
\end{proof}

Now, we consider the manifold $ \M $, which is a covering space of $ \cN $.

\begin{lem}\label{MSO(3)}
$ \M $ is diffeomorphic to $ \mathrm{SO}(3) $ and $ \R\PP^3 $.
\end{lem}

\begin{proof}
We see that $ \M $ and $ \mathrm{SO}(3) $ are diffeomorphic, since the map $
\iota:\M\to\mathrm{SO}(3) $ defined by $ \iota((\n,\m))=(\n,\m,\n\times\m) $ is a smooth diffeomorphism from $ \M $ to $ \mathrm{SO}(3) $. Here, $ (\n,\m,\n\times\m) $ denotes the orthogonal matrix with three columns formed by $ \n,\m $ and $ \n\times\m $ respectively. To complete the proof, we will establish a diffeomorphism from $ \R\PP^3 $ to $ \mathrm{SO}(3) $. Firstly, we have
$ \R\PP^3=\{\p\p\in\MM^{4\times 4}:\p\in\Ss^3\} $. For any $ \p=(a_1,a_2,a_3,a_4)^{\T}\in\Ss^3 $, we define a map $ \sg:\R\PP^3\to\mathrm{SO}(3) $ by
\be
\sg (\p\p)=\(\begin{matrix}
a_1^2+a_2^2-a_3^2-a_4^2&2(a_2a_3-a_1a_4)&2(a_1a_3+a_2a_4)\\
2(a_2a_3+a_1a_4)&a_1^2-a_2^2+a_3^2-a_4^2&2(a_3a_4-a_1a_2)\\
2(a_2a_4-a_1a_3)&2(a_1a_2+a_3a_4)&a_1^2-a_2^2-a_3^2+a_4^2\\
\end{matrix}\).\label{DifRP3SO3}
\ee
Simple calculations imply that $ \sg  $ is a smooth homeomorphism. Meanwhile, the reverse $ \sg ^{-1} $ is also smooth. Indeed, in constructing $ \sg  $, we first represent $ \p=a_1+a_2\ii+a_3\jj+a_4\kk $ with $ \sum_{i=1}^4a_i^2=1 $ as an element in the quaternion field
$$
\mathbb{H}=\{a+b\ii+c\jj+d\kk:a,b,c,d\in\R,\,\,\ii^2=\jj^2=\kk^2=\ii\jj\kk=-1\}.
$$
The matrix in \eqref{DifRP3SO3} corresponds to the linear map
$$
(a_1+a_2\ii+a_3\jj+a_4\kk)(\cdot)(a_1+a_2\ii+a_3\jj+a_4\kk)^{-1}:\op{span}\{\ii,\jj,\kk\}\to\op{span}\{\ii,\jj,\kk\}.
$$
To be precise, we have $ \p(\ii,\jj,\kk)\p^{-1}=(\ii,\jj,\kk)(\sg(\p\p)) $.
\end{proof}

\begin{rem}\label{SO3rem}
In view of Lemma \ref{MSO(3)} and formula \eqref{DifRP3SO3}, it can be seen that $ \Ss^3 $ is a covering space of $ \mathrm{SO}(3) $ with order two and the covering map $ \psi:\Ss^3\to \mathrm{SO}(3) $ is given by $ \psi(\p)=(\uu^{(1)},\uu^{(2)},\uu^{(3)}) $,
where $ \p=(a_1,a_2,a_3,a_4)^{\T}\in\Ss^3 $ and
\begin{align*}
\uu^{(1)}&=(a_1^2+a_2^2-a_3^2-a_4^2,2(a_2a_3+a_1a_4),2(a_2a_4-a_1a_3))^{\T},\\
\uu^{(2)}&=(2(a_2a_3-a_1a_4),a_1^2-a_2^2+a_3^2-a_4^2,2(a_1a_2+a_3a_4))^{\T},\\
\uu^{(3)}&=(2(a_1a_3+a_2a_4),2(a_3a_4-a_1a_2),a_1^2-a_2^2-a_3^2+a_4^2)^{\T}.
\end{align*}
Choose $ \va>0 $ and $ \ga:(-\va,\va)\to\Ss^3 $ such that $ \ga(s)=(\ga_1(s),\ga_2(s),\ga_3(s),\ga_4(s))^{\T} $ with $ \sum_{i=1}^4\ga_i^2(s)=1 $ for $ s\in(-\va,\va) $. Simple calculations yield that
\be
\left|\(\left.\f{\ud}{\ud t}\right|_{t=0}\psi(\ga)\)\right|^2=\sum_{i=1}^3|(\uu_i(\ga))'(0)|^2=2|\ga'(0)|^2.\label{TwoDeri1}
\ee
Due to the arbitrariness of the choice of $ \ga $, we deduce that
\be
|\ud\psi(\p)\vv|^2=2|\vv|^2\text{ for any }\p\in\Ss^3\text{ and }\vv\in T_{\p}\Ss^3.\label{TwoDeri}
\ee
This implies that the pull-back metric induced by $ \psi $ coincides with the fundamental form of $ \mathrm{SO}(3) $ up to a constant. Consequently, the Levi-Civita connections for the two metrics are the same since the Christoffel symbols remain the same under a multiplication of a constant number. Therefore, we conclude that $ \PP:\Ss^1\to\mathrm{SO}(3) $ is a geodesic if and only if it can be expressed as $ \PP=\psi(\p) $, where $ \p:[0,2\pi]\to\Ss^3 $ is a geodesic curve in $ \Ss^3 $ such that $ \p(0)=\pm\p(2\pi) $.
\end{rem}

Based on the results above, we can discuss some geometric properties of $ \cN $. To begin with, we give the fundamental group of it.

\begin{lem}\label{FG}
$ \pi_1(\cN)\cong Q_8 $ and $ \Ss^3 $ is the universal covering space of $ \cN $. Here $ Q_8 $ is the quaternion group, i.e., $ Q_8=\{\pm 1,\pm\ii,\pm\jj,\pm\kk\} $, where $ \ii,\jj,\kk $ satisfy the relations $ \ii^2=\jj^2=\kk^2=\ii\jj\kk=-1 $.
\end{lem}

\begin{proof}[Proof of Lemma \ref{FG}]
Firstly, we have that $ \M $ is a covering space of $ \cN $ with order four. Indeed, the covering map $ p_1 $ is defined by 
$$
p_1:\M\to\cN,\quad (\n,\m)\mapsto r_*(\n\n-\m\m).
$$
Moreover, the set of preimages for $ r_*(\n\n-\m\m)\in\cN $ is given by
$$
p_1^{-1}(r_*(\n\n-\m\m))=\{(\n,\m),(\n,-\m),(-\n,\m),(-\n,-\m)\}.
$$
Applying Lemma \ref{MSO(3)}, $ \M $ is diffeomorphic to $ \mathrm{SO}(3) $ and $ \R\PP^3 $. This implies that $ \Ss^3 $ is a covering space of $ \M $ with order two. In view of Remark \ref{SO3rem}, the covering map is given by
$$
p_2:\Ss^3\to\M,\quad\vv=(a_1,a_2,a_3,a_4)^{\T}\to(\n,\m),
$$
where
\be
\begin{aligned}
\n&=(a_1^2+a_2^2-a_3^2-a_4^2,2(a_2a_3+a_1a_4),2(a_2a_4-a_1a_3))^{\T},\\
\m&=(2(a_2a_3-a_1a_4),a_1^2-a_2^2+a_3^2-a_4^2,2(a_1a_2+a_3a_4))^{\T}.
\end{aligned}\label{nmdef}
\ee
By using the fact that the coverings above are all with finite order, we have that the composition $ p=p_1\circ p_2 $ is a covering map from $ \Ss^3 $ to $ \cN $. This implies that $ \Ss^3 $ is the universal covering space of $ \cN $ with order eight. Precisely speaking, we obtain 
\be
p:\Ss^3\to\cN,\quad\vv=(a_1,a_2,a_3,a_4)^{\T}\to r_*(\n\n-\m\m)\label{universal}
\ee
is a covering map, where $ \n,\m $ are given by \eqref{nmdef}. Since $ \pi_1(\Ss^3)=\{0\} $, in view of Proposition 1.32, 1.39 in \cite{H02}, $ \pi_1(\cN)\cong G(\Ss^3,\cN) $ and $ \#\pi_1(\cN)=8 $, where $ G(\Ss^3,\cN) $ is the deck transformation group of the covering space $ \Ss^3\to\cN $. Consider the action of $ Q_8 $ on $ \Ss^3 $, which is characterized as the set of unit quaternions by right multiplication $ R_{g} $ with $ g\in Q_8 $. Taking $ \ii $ as an example, for any $ \vv\in\Ss^3 $ represented by $ \vv=a_1+a_2\ii+a_3\jj+a_4\kk $, with $ \sum_{i=1}^4a_i^2=1 $, 
\be
R_{\ii}(\vv)=(a_1+a_2\ii+a_3\jj+a_4\kk)\ii=-a_2+a_1\ii+a_4\jj-a_3\kk.\label{right}
\ee
We deduce that $ R_g(g\in Q_8) $ is compatible with $ p:\Ss^3\to\cN $, i.e., $ p\circ R_g=p $. Meanwhile, $ R_g $ is a diffeomorphism from $ \Ss^3 $ to itself. This implies that $ Q_8 $ can be regard as a subgroup of $ G(\Ss^3,\cN) $. By the previous analysis, $ \#G(\Ss^3,\cN)=8 $ and then $ 
\pi_1(\cN)\cong G(\Ss^3,\cN)\cong Q_8 $. Indeed, by the results given as above, we have $ \Ss^3/Q_8\cong\cN $, where the notation ``$ \cong $'' here means the homeomorphism and the quotient relations are defined under the group actions by right multiplication like \eqref{right}.
\end{proof}

Recall the projections $ \{\varrho_i\}_{i=1}^2 $ given in \eqref{rho1rho2rho}, we have that they are both well defined on $ \cN $, that is,
\begin{align*}
\varrho_1&:\cN\to\cP,\quad r_*(\n\n-\m\m)\mapsto r_*\n\n,\\
\varrho_2&:\cN\to\cP,\quad r_*(\n\n-\m\m)\mapsto r_*\m\m.
\end{align*}
Obviously, $ \varrho_1 $ and $ \varrho_2 $ are smooth maps restricting on $ \cN $, so they induce homomorphisms from $ \pi_1(\cN) $ to $ \pi_1(\cP) $. If we fix a point $ \X\in\cN $, then for $ \ell=1,2 $,
$$
(\varrho_{\ell})_{\X,*}:\pi_1(\cN,\X)\to\pi_1(\cP,\varrho_{\ell}(\X)),
$$
where $ (\varrho_{\ell})_{\X,*}([\mathbf{L}]_{\cN,\X})=[\varrho_{\ell}(\mathbf{L})]_{\cP,\varrho_{\ell}(\X)} $ with $ \mathbf{L}\in C^0(\Ss^1,\cN) $ based on $ \X $. 

Firstly, we choose the base point $
\X_0=r_*(\e^{(1)}\e^{(1)}-\e^{(2)}\e^{(2)})\in\cN $, then $ \{(\varrho_{\ell})_{\X_0,*}\}_{\ell=1}^2 $ are homomorphisms as follows. For any $ \ell=1,2 $,
\be
(\varrho_{\ell})_{\X_0,*}:\pi_1(\cN,\X_0)\to\pi_1(\cP,r_*\e^{(\ell)}\e^{(\ell)}),\label{p1p2hom}
\ee
where $
(\varrho_{\ell})_{\X_0,*}([\mathbf{L}]_{\cN,\X_0})=[\varrho_{\ell}(\mathbf{L})]_{\cP,r_*\e^{(\ell)}\e^{(\ell)}} $ with $ \mathbf{L}\in C^0(\Ss^1,\cN) $ based on $ \X_0 $. We claim that for $ \ell=1,2 $, $ (\varrho_{\ell})_{\X_0,*} $ defined above are surjective. In fact, taking $ \varrho_1 $ for example, we choose $
\LL(\theta)=r_*(\n(\theta)\n(\theta)-\e^{(2)}\e^{(2)}) $, where $ \n(\theta)=(\cos(\theta/2),0,\sin(\theta/2))^{\T} $. Obviously $ \LL\in C^0(\Ss^1,\cN) $ and $ \LL(0)=\LL(2\pi)=\X_0 $. In view of Lemma \ref{N0fgroup}, we have 
$$ 
(\varrho_1)_{\X_0,*}([\LL]_{\cN,\X_0})=[r_*\n\n]_{\cP,r_*\e^{(1)}\e^{(1)}}=\ol{1},
$$ 
and then $ (\varrho_1)_{\X_0,*} $ is surjective. Thus the claim is true. Therefore, by theorem of isomorphism, we obtain 
\be
\pi_1(\cN,\X_0)/\ker((\varrho_{\ell})_{\X_0,*})\cong\pi_1(\cP,r_*\e^{(\ell)}\e^{(\ell)}),\quad \ell=1,2.\label{Isothm}
\ee
By this, we have the following result.

\begin{lem}\label{ij}
Assume that $ \{(\varrho_1)_{\X_0,*}\}_{\ell=1}^2 $ are homomorphisms given by \eqref{p1p2hom}. There hold $ \#\ker((\varrho_1)_{\X_0,*})=\#\ker((\varrho_2)_{\X_0,*})=4 $, and $ \ker((\varrho_1)_{\X_0,*})\neq\ker((\varrho_2)_{\X_0,*}) $. Moreover,
\begin{align*}
\ker((\varrho_1)_{\X_0,*})&=\{[\X_0]_{\cN,\X_0},[\A^{(0)}]_{\cN,\X_0},[\A^{(0)}]_{\cN,\X_0}^2,[\A^{(0)}]_{\cN,\X_0}^3\},\\
\ker((\varrho_2)_{\X_0,*})&=\{[\X_0]_{\cN,\X_0},[\B^{(0)}]_{\cN,\X_0},[\B^{(0)}]_{\cN,\X_0}^2,[\B^{(0)}]_{\cN,\X_0}^3\},
\end{align*}
where for $ \theta\in[0,2\pi] $,
\be
\begin{aligned}
\A^{(0)}(\theta)&=r_*(\e^{(1)}\e^{(1)}-\n(\theta)\n(\theta)),\\
\B^{(0)}(\theta)&=r_*(\m(\theta)\m(\theta)-\e^{(2)}\e^{(2)})\label{ABtheta}
\end{aligned}
\ee
with $ \n(\theta)=(0,\cos(\theta/2),\sin(\theta/2))^{\T} $ and $ \m(\theta)=(\cos(\theta/2),0,\sin(\theta/2))^{\T} $. In particular, $ [\A^{(0)}]_{\cN,\X_0}^2=[\B^{(0)}]_{\cN,\X_0}^2 $.
\end{lem}

\begin{proof}
By Lemma \ref{FG}, we have an isomorphism $ \vp_{\X_0}:Q_8\to\pi_1(\cN,\X_0) $. This, together with Lemma \ref{N0fgroup} and \eqref{Isothm}, implies that $ \# Q_8=\#\Z_2\cdot\#\ker((\varrho_{\ell})_{\X_0,*}) $ for $ \ell=1,2 $, and then $ \#\ker((\varrho_1)_{\X_0,*})=\#\ker((\varrho_2)_{\X_0,*})=4 $. Obviously, by using Lemma \ref{N0fgroup}, we have, for $ \A^{(0)} $ and $ \B^{(0)} $ given in \eqref{ABtheta},
\be
\begin{aligned}
[\varrho_1(\A^{(0)})]_{\cP,r_*\e^{(1)}\e^{(1)}}&=\ol{0},\,\,[\varrho_2(\A^{(0)})]_{\cP,r_*\e^{(2)}\e^{(2)}}=\ol{1},\\
[\varrho_1(\B^{(0)})]_{\cP,r_*\e^{(1)}\e^{(1)}}&=\ol{1},\,\,[\varrho_2(\B^{(0)})]_{\cP,r_*\e^{(2)}\e^{(2)}}=\ol{0}.
\end{aligned}\label{AB10}
\ee
On the other hand, since we identify $ \pi_1(\cP) $ with $ \Z_2 $, we deduce that
$$
(\varrho_{\ell})_{\X_0,*}(\vp_{\X_0}(-1))=(\varrho_{\ell})_{\X_0,*}(\vp_{\X_0}(\ii^2))=(\varrho_{\ell})_{\X_0,*}(\vp_{\X_0}(\ii)^2)=\ol{0}
$$
for any $ \ell=1,2 $. Combined with \eqref{AB10}, we have $ [\A^{(0)}]_{\cN,\X_0} $, $ [\B^{(0)}]_{\cN,\X_0}\neq\vp_{\X_0}(-1) $, and then $ [\A^{(0)}]_{\cN,\X_0} $, $ [\B^{(0)}]_{\cN,\X_0} $ are all elements with order four. Since $ -1 $ is the only one element in $ Q_8 $ of order two, we get that $ [\A^{(0)}]_{\cN,\X_0}^2=[\B^{(0)}]_{\cN,\X_0}^2=\vp_{\X_0}(-1) $ and can complete the proof. 
\end{proof}

Without loss of generality, for the isomorphism $ \vp_{\X_0} $ defined in the proof of Lemma \ref{ij}, we can further set that
\be
\vp_{\X_0}^{-1}([\A^{(0)}]_{\cN,\X_0})=\ii\text{ and }\vp_{\X_0}^{-1}([\B^{(0)}]_{\cN,\X_0})=\jj,\label{vpP0}
\ee
where $ \A^{(0)} $ and $ \B^{(0)} $ are given by \eqref{ABtheta}. Since $ Q_8 $ is generated by $ \ii $ and $ \jj $, the isomorphism $ \vp_{\X_0} $ is uniquely determined by \eqref{vpP0}. For another point $ \X_1\in\cN $, denoted by $ \X_1=r_*(\n^{(1)}\n^{(1)}-\n^{(2)}\n^{(2)}) $, where $ (\n^{(1)},\n^{(2)})\in\M $, we obtain $ \U\in\mathrm{SO}(3) $ such that $
\U\e^{(1)}=\n^{(1)} $ and $ \U\e^{(2)}=\n^{(2)} $. Here $ \U $ is uniquely determined by $ (\n^{(1)},\n^{(2)}) $. Choose
\begin{align*}
\A^{(1)}(\theta)&=\U\A(\theta)\U^{\T}=r_*(\n^{(1)}\n^{(1)}-(\U\n(\theta))(\U\n(\theta))),\\
\B^{(1)}(\theta)&=\U\B(\theta)\U^{\T}=r_*((\U\m(\theta))(\U\m(\theta))-\n^{(2)}\n^{(2)}),
\end{align*}
being two loops based on $ \X_1 $. Now we can define an isomorphism $ \vp_{\X_1}:Q_8\to\pi_1(\cN,\X_1) $ generated by the relations
\be
\vp_{\X_1}^{-1}([\A^{(1)}]_{\cN,\X_1})=\ii\text{ and }\vp_{\X_1}^{-1}([\B^{(1)}]_{\cN,\X_1})=\jj.\label{vpP1}
\ee
Moreover, in view of the proof of Lemma \ref{ij}, we have
\be
\begin{aligned}
\ker((\varrho_1)_{\X_1,*})&=\{[\X_1]_{\cN,\X_1},[\A^{(1)}]_{\cN,\X_1},[\A^{(1)}]_{\cN,\X_1}^2,[\A^{(1)}]_{\cN,\X_1}^3\},\\
\ker((\varrho_2)_{\X_1,*})&=\{[\X_1]_{\cN,\X_1},[\B^{(1)}]_{\cN,\X_1},[\B^{(1)}]_{\cN,\X_1}^2,[\B^{(1)}]_{\cN,\X_1}^3\}.
\end{aligned}\label{kerP1}
\ee
and
\be
\begin{aligned}
\pi_1&(\cN,\X_1)=\{[\X_1]_{\cN,\X_1},[\A^{(1)}]_{\cN,\X_1},[\A^{(1)}]_{\cN,\X_1}^2,[\A^{(1)}]_{\cN,\X_1}^3,\\
&\quad\quad[\B^{(1)}]_{\cN,\X_1},[\B^{(1)}]_{\cN,\X_1}^3,[\A^{(1)}]_{\cN,\X_1}[\B^{(1)}]_{\cN,\X_1},[\B^{(1)}]_{\cN,\X_1}^{-1}[\A^{(1)}]_{\cN,\X_1}^{-1}\}   
\end{aligned}.\label{GroupEl}
\ee
Since $ \cN $ is path-connected, there exists a curve $ \CC\in C^0([0,2\pi],\cN) $ connecting $ \X_0 $ and $ \X_1 $, i.e., $ \CC(0)=\X_1 $ and $ \CC(2\pi)=\X_0 $. It induces an isomorphism $
\tau_{\cN,\CC}:\pi_1(\cN,\X_0)\to \pi_1(\cN,\X_1) $, such that $ \tau_{\cN,\CC}([\mathbf{L}]_{\cN,\X_0}):=[\CC*\mathbf{L}*\wt{\CC}]_{\cN,\X_1} $ for $ \mathbf{L}\in C^0(\Ss^1,\cN) $ based on $ \X_0 $. Using $ \vp_{\X_0},\vp_{\X_1} $ and $ \tau_{\cN,\CC} $, we can construct an automorphism of $ Q_8 $ by $
\vp_{\X_1}^{-1}\circ\tau_{\cN,\CC}\circ\vp_{\X_0}:Q_8\to Q_8 $. For this automorphism, we have the following result.

\begin{lem}\label{conjulem}
$ \vp_{\X_1}^{-1}\circ\tau_{\cN,\CC}\circ\vp_{\X_0} $ preserves the conjugacy classes of $ Q_8 $.
\end{lem}
\begin{proof}
There are five conjugacy classes of $ Q_8 $ given by $ \{1\} $, $ \{-1\} $, $ \{\pm\ii\} $, $ \{\pm\jj\} $, and $ \{\pm\kk\} $. Since $ -1 $ is the unique element of order two in $ Q_8 $ and $ 1 $ is the identity, we have, $ \vp_{\X_1}^{-1}\circ\tau_{\cN,\CC}\circ\vp_{\X_0} $ preserves $ \{1\} $ and $ \{-1\} $. Moreover, by $ \ii\jj=\kk $, we only need to show that for any $ g\in\{\ii,\jj\} $,
\be
(\vp_{\X_1}^{-1}\circ\tau_{\cN,\CC}\circ\vp_{\X_0})(\{\pm g\})=\{\pm g\}.\label{iijj}
\ee
For any $ \ell=1,2 $, $ \varrho_{\ell}(\CC) $ is a curve on $ \cP $ with $ \varrho_{\ell}(\CC)(0)=\varrho_{\ell}(\X_1) $, $ \varrho_{\ell}(\CC)(2\pi)=\varrho_{\ell}(\X_0) $. We claim that
\be
\begin{aligned}
&(\varrho_{\ell})_{\X_0,*}([\A]_{\cN,\X_0})=\ol{1}\text{ in }\pi_1(\cP,\varrho_{\ell}(\X_0))\\
&\quad\quad\text{ if and only if }(\varrho_{\ell})_{\X_1,*}([\CC*\A*\wt{\CC}]_{\cN,\X_1})=\ol{1}\text{ in }\pi_1(\cP,\varrho_{\ell}(\X_1)),
\end{aligned}\label{iffloop}
\ee
with $ \ell=1,2 $. Indeed, since $ \pi_1(\cP)=\Z_2 $, the map $
\tau_{\cP,\rho_{\ell}(\CC)}:\pi_1(\cP,\varrho_{\ell}(\X_0))\to\pi_1(\cP,\varrho_{\ell}(\X_1)) $ defined by $ \tau_{\cP,\rho_{\ell}(\CC)}([\mathbf{L}_{\ell}]_{\cP,\varrho_{\ell}(\X_0)}):=[\varrho_{\ell}(\CC)*\mathbf{L}_{\ell}*\varrho_{\ell}(\wt{\CC})]_{\cP,\varrho_{\ell}(\X_1)} $ for $ \mathbf{L}_{\ell}\in C^0(\Ss^1,\cP) $ based on $ \varrho_{\ell}(\X_0) $ with $ \ell=1,2 $ are both automorphisms of $ \Z_2 $. This implies that
$$
[\mathbf{L}_{\ell}]_{\cP,\varrho_{\ell}(\X_0)}=\ol{1}\text{ if and only if }[\varrho_{\ell}(\CC)*\mathbf{L}_{\ell}*\varrho_{\ell}(\wt{\CC})]_{\cP,\varrho_{\ell}(\X_1)}=\ol{1}.
$$
This, together with the fact that
$$
(\varrho_{\ell})_{\X_0,*}([\CC*\mathbf{L}*\wt{\CC}]_{\cN,\X_1})=[\varrho_{\ell}(\CC)*\varrho_{\ell}(\mathbf{L})*\varrho_{\ell}(\wt{\CC})]_{\cP,\varrho_{\ell}(\X_1)},
$$
for any $ \mathbf{L}\in C^0(\Ss^1,\cN) $, implies \eqref{iffloop}. Combining \eqref{AB10}, \eqref{vpP0} and \eqref{iffloop}, we have 
\begin{align*}
((\varrho_1)_{\X_1,*}\circ\tau_{\cN,\CC}\circ\vp_{\X_0})(\ii)&=\ol{0},\,\,((\varrho_2)_{\X_1,*}\circ\tau_{\cN,\CC}\circ\vp_{\X_0})(\ii)=\ol{1}\\
((\varrho_1)_{\X_1,*}\circ\tau_{\cN,\CC}\circ\vp_{\X_0})(\jj)&=\ol{1},\,\,((\varrho_2)_{\X_1,*}\circ\tau_{\cN,\CC}\circ\vp_{\X_0})(\jj)=\ol{0}.
\end{align*}
This implies that $ (\tau_{\cN,\CC}\circ\vp_{\X_0})(\ii)\in\ker((\varrho_1)_{\X_1,*}) $ and $ (\tau_{\cN,\CC}\circ\vp_{\X_0})(\jj)\in\ker((\varrho_2)_{\X_1,*}) $. In view of \eqref{kerP1} and the fact that $ \{(\tau_{\cN,\CC}\circ\vp_{\X_0})(g)\}_{g\in\{\ii,\jj\}} $ are two elements with order four, we have that $
(\tau_{\cN,\CC}\circ\vp_{\X_0})(\ii)\in\{[\A^{(1)}]_{\cN,\X_1},[\A^{(1)}]_{\cN,\X_1}^3\} $ and $
(\tau_{\cN,\CC}\circ\vp_{\X_0})(\jj)\in\{[\B^{(1)}]_{\cN,\X_1},[\B^{(1)}]_{\cN,\X_1}^3\} $. Using \eqref{vpP1}, we deduce $ (\vp_{\X_1}^{-1}\circ\tau_{\cN,\CC}\circ\vp_{\X_0})(g)\in\{\pm g\} $, for $ g\in\{\ii,\jj\} $, which directly implies \eqref{iijj} and completes the proof.
\end{proof}

\begin{cor}\label{C1C2P1P2}
Let $ \X_1,\X_2\in\cN $ be two points. Assume that $ \A\in C^0(\Ss^1,\cN) $ is a loop and $ \X_1,\X_2\in\A $. There hold, $ \vp_{\X_1}^{-1}([\A]_{\cN,\X_1}) $ and $ \vp_{\X_2}^{-1}([\A]_{\cN,\X_2}) $ are in the same conjugacy class in $ Q_8 $.
\end{cor}
\begin{proof}
By the path-connectness of $ \cN $, we can construct two curves $ \CC_1 $ and $ \CC_2 $ such that $ \CC_{\ell}(0)=\X_{\ell} $ and $ \CC_{\ell}(2\pi)=\X_0 $ for any $ \ell=1,2 $. There are two isomorphisms $
\tau_{\cN,\CC_{\ell}}:\pi_1(\cN,\X_0)\to\pi_1(\cN,\X_{\ell}) $ given by $ \tau_{\cN,\CC_{\ell}}([\mathbf{L}]_{\cN,\X_0})=[\CC_{\ell}*\mathbf{L}*\wt{\CC}_{\ell}]_{\cN,\X_{\ell}} $ for $ \mathbf{L}\in C^0(\Ss^1,\cN) $ based on $ \X_0 $ and $ \ell=1,2 $. By Lemma \ref{LemmaHom}, we have
\begin{align*}
&[\wt{\CC}_1*\A*\CC_1]_{\cN}=[\wt{\CC}_2*\A*\CC_2]_{\cN}=[\A]_{\cN},\\
\Rightarrow&[\wt{\CC}_1*\A*\CC_1]_{\cN,\X_0}\stackrel{c}{\sim}[\wt{\CC}_2*\A*\CC_2]_{\cN,\X_0}\text{ in }\pi_1(\cN,\X_0),\\
\Rightarrow&(\vp_{\X_0}^{-1}\circ\tau_{\cN,\CC_1}^{-1}\circ\vp_{\X_1})(\vp_{\X_1}^{-1}([\A]_{\cN,\X_1}))\stackrel{c}{\sim}(\vp_{\X_0}^{-1}\circ\tau_{\cN,\CC_2}^{-1}\circ\vp_{\X_2})(\vp_{\X_2}^{-1}([\A]_{\cN,\X_2}))\text{ in }Q_8.
\end{align*}
Using Lemma \ref{conjulem}, we have $ \{\vp_{\X_0}^{-1}\circ\tau_{\cN,\CC_{\ell}}^{-1}\circ\vp_{\X_{\ell}}\}_{\ell=1}^2 $ preserve the conjugacy classes in $ Q_8 $. This implies that $ \vp_{\X_1}^{-1}([\A]_{\cN,\X_1})\stackrel{c}{\sim}\vp_{\X_2}^{-1}([\A]_{\cN,\X_2}) $ in $ Q_8 $, which completes the proof.
\end{proof}

Now we can give a characterization of $ [\Ss^1,\cN] $.

\begin{prop}\label{freeNprop}
Let $ \A,\B\in C^0(\Ss^1,\cN) $ be two loops. There holds, $ [\A]_{\cN}=[\B]_{\cN} $ if and only if for two points $ \X_1\in\A $ and $ \X_2\in\B $, the two elements $ \vp_{\X_1}^{-1}([\A]_{\cN,\X_1}) $ and $ \vp_{\X_2}^{-1}([\B]_{\cN,\X_2}) $ are in the same conjugacy classes of $ Q_8 $. In particular, If $ \A,\B $ are both non-trivial loops, then $ [\A]_{\cN}=[\B]_{\cN} $ if and only if $ [\varrho_{\ell}(\A)]_{\cP}=[\varrho_{\ell}(\B)]_{\cP} $ for any $ \ell=1,2 $.
\end{prop}

\begin{proof}
In view of Corollary \ref{C1C2P1P2}, the conjugacy classes of $ Q_8 $ containing $ \vp_{\X_1}^{-1}([\A]_{\cN,\X_1}) $ and $ \vp_{\X_2}^{-1}([\B]_{\cN,\X_2}) $ are independent of the choices of $ \X_1 $ and $ \X_2 $, which implies that the result given in the proposition is well defined. Since $ \cN $ is path-connected, there are two curves $ \CC_1 $ and $ \CC_2 $ such that $ \CC_{\ell}(0)=\X_{\ell} $ and $ \CC_{\ell}(2\pi)=\X_0 $ for $ \ell=1,2 $. Consequently, we have
\begin{align*}
&[\A]_{\cN}=[\B]_{\cN}\Leftrightarrow[\wt{\CC}_1*\A*\CC_1]_{\cN}=[\wt{\CC}_2*\B*\CC_2]_{\cN},\\
\Leftrightarrow& [\wt{\CC}_1*\A*\CC_1]_{\cN,\X_0}\stackrel{c}{\sim}[\wt{\CC}_2*\B*\CC_2]_{\cN,\X_0}\text{ in }\pi_1(\cN,\X_0),\\
\Leftrightarrow&(\vp_{\X_0}^{-1}\circ\tau_{\cN,\CC_1}^{-1}\circ\vp_{\X_1})(\vp_{\X_1}^{-1}([\A]_{\cN,\X_1}))\stackrel{c}{\sim}(\vp_{\X_0}^{-1}\circ\tau_{\cN,\CC_2}^{-1}\circ\vp_{\X_2})(\vp_{\X_2}^{-1}([\B]_{\cN,\X_2}))\text{ in }Q_8,
\end{align*}
where for the second ``$ \Leftrightarrow $", we have used Lemma \ref{LemmaHom}. Now, by Lemma \ref{conjulem}, $ [\A]_{\cN}=[\B]_{\cN} $ if and only if $ \vp_{\X_1}^{-1}([\A]_{\cN,\X_1})\stackrel{c}{\sim}\vp_{\X_2}^{-1}([\B]_{\cN,\X_2}) $ in $ Q_8 $. Moreover, by Proposition \ref{freeN0}, the construction of $ \vp_{\X_{\ell}} $ with $ \ell=1,2 $, \eqref{kerP1} and \eqref{GroupEl}, we have that $ [\A]_{\cN}=[\B]_{\cN} $ if and only if $ [\varrho_{\ell}(\A)]_{\cP}=[\varrho_{\ell}(\B)]_{\cP} $ for any $ \ell=1,2 $.
\end{proof}

\begin{cor}\label{homcor}
There are five elements in the set $ [\Ss^1,\cN] $, namely,
$$
[\Ss^1,\cN]=\{\rH_0,\rH_1,\rH_3,\rH_4,\rH_2\},
$$
where $ \rH_0 $ is the class containing trivial loops. For a loop $ \LL\in C^0(\Ss^1,\cN) $ and a point $ \X_1\in\LL $, 
\begin{align*}
[\LL]_{\cN}&=\rH_0\text{ if and only if }\vp_{\X_1}^{-1}([\LL]_{\cN,\X_1})=1,\\
[\LL]_{\cN}&=\rH_1\text{ if and only if }\vp_{\X_1}^{-1}([\LL]_{\cN,\X_1})\in\{\pm\ii\},\\
[\LL]_{\cN}&=\rH_2\text{ if and only if }\vp_{\X_1}^{-1}([\LL]_{\cN,\X_1})\in\{\pm\jj\},\\
[\LL]_{\cN}&=\rH_3\text{ if and only if }\vp_{\X_1}^{-1}([\LL]_{\cN,\X_1})\in\{\pm\kk\},\\
[\LL]_{\cN}&=\rH_4\text{ if and only if }\vp_{\X_1}^{-1}([\LL]_{\cN,\X_1})=-1.
\end{align*}
Moreover, if $ [\LL]_{\cN}\neq\rH_0 $, there is the characterization as follows.
\begin{align*}
[\LL]_{\cN}&=\rH_1\text{ if and only if }[\varrho_1(\LL)]_{\cP}=\rh_0\text{ and }[\varrho_2(\LL)]_{\cP}=\rh_1,\\
[\LL]_{\cN}&=\rH_2\text{ if and only if }[\varrho_1(\LL)]_{\cP}=\rh_1\text{ and }[\varrho_2(\LL)]_{\cP}=\rh_0,\\
[\LL]_{\cN}&=\rH_3\text{ if and only if }[\varrho_1(\LL)]_{\cP}=\rh_1\text{ and }[\varrho_2(\LL)]_{\cP}=\rh_1,\\
[\LL]_{\cN}&=\rH_4\text{ if and only if }[\varrho_1(\LL)]_{\cP}=\rh_0\text{ and }[\varrho_2(\LL)]_{\cP}=\rh_0,
\end{align*}
where $ [\Ss^1,\cP]=\{\rh_0,\rh_1\} $.
\end{cor}
\begin{proof}
This is directly consequence from Lemma \ref{LemmaHom}, Proposition \ref{freeN0} and \ref{freeNprop}.
\end{proof}

Now, let us consider the map $ T:\cN\to\cN $ defined as follows.
\be
T(r_*(\n\n-\m\m))=r_*(\m\m-\n\n).\label{Tdef}
\ee
Obviously $ T $ is an diffeomorphism, preserving the metric on $ \cN $ and induces a map $ T_* $ from $ [\Ss^1,\cN] $ to itself by $ T_*([\mathbf{L}]_{\cN})=[T(\mathbf{L})]_{\cN} $ for $ \mathbf{L}\in C^0(\Ss^1,\cN) $. Now we have the following result on $ T_* $.

\begin{lem}\label{fij}
$ T_* $ is a permutation on $ [\Ss^1,\cN] $ such that $ T_*(\rH_0)=\rH_0 $, $ T_*(\rH_1)=\rH_2 $, $ T_*(\rH_2)=\rH_1 $, $ T_*(\rH_3)=\rH_3 $, and $ T_*(\rH_4)=\rH_4 $.
\end{lem}

\begin{proof}
We firstly claim that $ \LL\in C^0(\Ss^1,\cN) $, $ T_*([\LL]_{\cN})=\rH_0 $ if and only if $ [\LL]_{\cN}=\rH_0 $. Indeed, if $ [\LL]_{\cN}=\rH_0 $, then there exists $ H:[0,1]\times\Ss^1\to\cN $ such that $ H(0,\cdot)=\LL(\cdot) $ and $ H(1,\cdot)=\X_0 $, where $ \X_0\in\cN $ is a point. $ T\circ H $ induces a free homotopy from $ T(\LL) $ to $ T(\X_0) $. For the other side, we can replace $ T $ by $ T^{-1} $. Moreover, we have $ \varrho_2=\varrho_1\circ T $. The other results in the lemma directly follows from Corollary \ref{homcor}.
\end{proof}

We see that $ \cP $ and $ \cN $ are path-connected, compact and smooth Riemannian manifolds without boundary, isometrically embedded into $ \MM^{3\times 3} $. We can define
\be
\kappa_*:=\cE^*(\rh_1).\label{kappastardefn}
\ee
In view of Lemma \ref{singuenergessum}, we have $ \kappa_*>0 $. Since there are only two different free homotopy classes in $ [\Ss^1,\cP] $, we deduce from the definition of $ \cE^*(\cdot) $ that $ \kappa_*=\cE^*(\rh_1)=\cE(\rh_1) $. Furthermore, the following lemma give the explicit value of it.

\begin{lem}[\cite{C15}, Lemma 3.6]\label{kappa}
$ \PP_*\in H^1(\Ss^1,\cN) $ is a minimizer for $ \cE(\rh_1) $ if and only if it can be represented by $ \PP_*=r_*\n\n $, where $ \n $ is a half of big circle on $ \Ss^2 $, for example, $ \n(\theta)=(\cos(\theta/2),\sin(\theta/2),0)^{\T} $ for $ 0\leq\theta\leq 2\pi $. In particular, 
$$
\kappa_*=\f{1}{2}\int_{\Ss^1}|\PP_*'(\theta)|\ud\theta=\f{\pi}{2}r_*^2.
$$
\end{lem}

On $ \cN $, we have that the map $ T $ defined in \eqref{Tdef} preserves the metric of $ \cN $. By Lemma \ref{fij}, we deduce $ \cE(\rH_1)=\cE(\rH_2) $. Precisely, we have the following result on the value of $ \{\cE(\rH_i)\}_{i=0}^4 $ and $ \{\cE^*(\rH_i)\}_{i=0}^4 $.

\begin{lem}\label{kappa123}
For $ \cE,\cE^*:[\Ss^1,\cN]\to\R $, the following properties hold.
\begin{enumerate}
\item $ \cE(\rH_0)=0 $, $ \cE(\rH_1)=\cE(\rH_2)=\kappa_* $, $ 2\kappa_*<\cE(\rH_3)<4\kappa_* $ and $ \cE(\rH_4)=4\kappa_* $. 
\item $ \cE^*(\rH_0)=0 $, $ \cE^*(\rH_1)=\cE^*(\rH_2)=\kappa_* $ and $ \cE^*(\rH_3)=\cE^*(\rH_4)=2\kappa_* $.
\end{enumerate}
\end{lem}
\begin{proof}
Firstly, by Lemma \ref{singuenergessum}, we have $ \cE(\rH_0)=\cE^*(\rH_0)=0 $. For a loop on $ \cN $, we represent $ \LL(\theta)=r_*(\n(\theta)\n(\theta)-\m(\theta)\m(\theta)) $. Here $ \theta\in[0,2\pi] $ with $ \n,\m\in C^0(\Ss^1,\Ss^2) $, $ \n\cdot\m=0 $ and $ \LL(0)=\LL(2\pi) $, that is,
$$
r_*(\n(0)\n(0)-\n(2\pi)\n(2\pi))=r_*(\m(0)\m(0)-\m(2\pi)\m(2\pi)).
$$
It implies that
\begin{align*}
0&\leq|\n(0)\n(0)-\n(2\pi)\n(2\pi)|^2\\
&=(\n(0)\n(0)-\n(2\pi)\n(2\pi)):(\m(0)\m(0)-\m(2\pi)\m(2\pi))\\
&=-(\n(0)\cdot\m(2\pi))^2-(\m(0)\cdot\n(2\pi))^2\leq 0.
\end{align*}
By this, we have $ \n(0)\n(0)=\n(2\pi)\n(2\pi) $ and $ \m(0)\m(0)=\m(2\pi)\m(2\pi) $. As a result, $ r_*\n(\cdot)\n(\cdot) $ and $ r_*\m(\cdot)\m(\cdot) $ are two loops on $ \cP $. By simple calculations,
\be
\f{1}{2}\int_0^{2\pi}|\LL'(\theta)|^2\ud\theta=\int_0^{2\pi}(|\n'(\theta)|^2+|\m'(\theta)|^2+2(\n'(\theta)\cdot\m(\theta))^2)\ud\theta.\label{LengthofB}
\ee

\underline{$ \cE(\rH_1)=\kappa_* $.} In view of Lemma \ref{N0fgroup}, Proposition \ref{freeN0} and Corollary \ref{homcor}, we construct a loop $ [\LL_1]_{\cN}=\rH_1 $, where $
\LL_1=r_*(\e^{(1)}\e^{(1)}-\n^{(1)}\n^{(1)}) $, with $ \n^{(1)}(\theta)=(0,\cos(\theta/2),\sin(\theta/2))^{\T} $. Obviously,
\be
\f{1}{2}\int_{0}^{2\pi}|\LL_1'(\theta)|^2\ud\theta=\kappa_*,\label{Aenergy}
\ee
and then $ \cE(\rH_1)\leq\kappa_* $. On the other hand, if $ [\LL]_{\cN}=\rH_1 $, then it follows from Corollary \ref{homcor} that $ [r_*\n\n]_{\cP}=\rh_0 $ and $ [r_*\m\m]_{\cP}=\rh_1 $. By Lemma \ref{kappa} and \eqref{LengthofB},
$$
\f{1}{2}\int_0^{2\pi}|\LL'(\theta)|^2\ud\theta\geq\int_0^{2\pi}|\m'(\theta)|^2\ud\theta\geq\kappa_*.
$$
By the arbitrariness of $ \LL $, we get $ \cE(\rH_1)\geq\kappa_* $. This, together with \eqref{Aenergy}, implies that $ \cE(\rH_1)=\kappa_* $. In particular, $ \LL_1 $ is a minimizer of $ \cE(\rH_1) $.\smallskip

\underline{$ 2\kappa_*<\cE(\rH_3)<4\kappa_* $.} If $ [\LL]_{\cN}=\rH_3 $, then by Corollary \ref{homcor} again, we have $ [r_*\n\n]_{\cP}=[r_*\m\m]_{\cP}=\rh_1 $. Combined with Lemma \ref{kappa} and \eqref{LengthofB}, we get
$$
\f{1}{2}\int_0^{2\pi}|\LL'(\theta)|^2\ud\theta\geq\int_0^{2\pi}(|\n'(\theta)|^2+|\m'(\theta)|^2)\ud\theta\geq 2\kappa_*.
$$
If $ \cE(\rH_3)=2\kappa_* $, Lemma \ref{kappa} implies that $ \n(\theta) $ and $ \m(\theta) $ are both half of big circles in $ \Ss^2 $ and $ \m'(\theta)\cdot\n(\theta)=0 $ for any $ \theta\in[0,2\pi] $. By simple calculations, we have $ \m'(\theta)=0 $, which is a contradiction, so $ \cE(\rH_3)>2\kappa_* $. Now we have to show that $ \cE(\rH_3)<4\kappa_* $. Choose $ \n^{(2)}(\theta)=(\cos(\theta/2),\sin(\theta/2),0)^{\T} $, $ \m^{(2)}(\theta)=(-\sin(\theta/2),\cos(\theta/2),0)^{\T} $ with $ \theta\in[0,2\pi] $ and
\be
\LL_2=r_*(\n^{(2)}\n^{(2)}-\m^{(2)}\m^{(2)}).\label{Bnm}
\ee
Using Proposition \ref{freeN0} and Corollary \ref{homcor}, we obtain that $ [\LL]_{\cN}=\rH_3 $. Also,
$$
\f{1}{2}\int_{0}^{2\pi}|\LL_2'(\theta)|^2\ud\theta=4\kappa_*.
$$
By this, we have that $ \cE^*(\rH_3)\leq 4\kappa_* $ and then we only need to prove that $ \LL_2 $ given in \eqref{Bnm} is not a minimizer of $ \cE^*(\rH_3) $. If $ \LL_2 $ is a minimizer, we claim that
\be
\LL_2''(\theta)=\Pi_{\cN}(\LL_2'(\theta),\LL_2'(\theta))(\LL_2(\theta)),\label{Harmonicmapequation}
\ee
where $ \Pi_{\cN} $ is the second fundamental form of $ \cN $ given by \eqref{SecondFundamentalForm}. Indeed, for any $ \PP\in C_c^{\ift}(\Ss^1,\cN) $, we have $ \varrho(\LL_2(\theta)+s\PP(\theta)) $ makes sense for any $ \theta\in[0,2\pi] $, if $ s>0 $ is sufficiently small. Moreover, $ [\LL_2+s\PP]_{\cN}=\rH_3 $, since $ \varrho(\LL_2(\theta)+s\PP(\theta)) $ itself is a free homotopy function when $ s\to 0 $. By the minimizing property, we have
$$
\left.\f{\ud}{\ud s}\right|_{s=0}\int_{0}^{2\pi}|(\varrho(\LL_2(\theta)+s\PP(\theta)))'|^2\ud\theta=0
$$
and the claim follows directly from the arguments in Proposition 1.3.1 of \cite{LW08}. By simple computations, we get that $ \LL_2 $ does not satisfy the equation \eqref{Harmonicmapequation}, which means that $ \cE(\rH_3)\neq 4\kappa_* $ and then $ \cE(\rH_3)<4\kappa_* $. \smallskip

\underline{$ \cE(\rH_4)=4\kappa_* $.} Set $ \n^{(3)}(\theta)=(0,\cos\theta,\sin\theta)^{\T} $ for $ \theta\in[0,2\pi] $, we claim that $
\LL_3=r_*(\e^{(1)}\e^{(1)}-\n^{(3)}\n^{(3)}) $ is a minimizer of $ \cE(\rH_4) $. Firstly, we note that $ \LL_3=\LL_1*\LL_1 $. By Proposition \ref{freeNprop}, $ \LL_3 $ is a non-trivial loop. Using Corollary \ref{homcor}, we have that $ [\LL_3]_{\cN}=\rH_4 $. By simple calculations,
$$
\f{1}{2}\int_0^{2\pi}|\LL_3'(\theta)|^2\ud\theta=4\kappa_*
$$
and then $ \cE(\rH_4)\leq 4\kappa_* $. Next, we will show that $ \cE(\rH_4)\geq 4\kappa_* $. Assume that $ \cE(\rH_4) $ is achieved for the loop $ \LL(\theta)=r_*(\n(\theta)\n(\theta)-\m(\theta)\m(\theta)) $ with $ \theta\in[0,2\pi] $. Since $ [\LL]_{\cN}=\rH_4 $, it can be deduced from Corollary \ref{homcor} that $ [\varrho_1(\LL)]_{\cP}=[\varrho_2(\LL)]_{\cP}=\rh_0 $. This implies that $ \n $ and $ \m $ are two loops on $ \Ss^2 $. Suppose that $ \n'=\al\m+\beta\p $ and $ \m'=-\al\n+\ga\p $ with $ \p=\n\times\m $ and $ \al,\beta,\ga:[0,2\pi]\to\R $. In view of \eqref{LengthofB}, we get
\begin{align*}
\f{1}{2r_*^2}\int_{0}^{2\pi}|\LL'(\theta)|^2\ud\theta&=\int_{0}^{2\pi}(4\al^2(\theta)+\beta^2(\theta)+\ga^2(\theta))\ud\theta\\
&\geq\f{1}{2}\int_{0}^{2\pi}2(\al^2(\theta)+\beta^2(\theta)+\ga^2(\theta))\ud\theta\\
&=\f{1}{2}\int_{0}^{2\pi}(|\n'(\theta)|^2+|\m'(\theta)|^2+|\p'(\theta)|^2)\ud\theta.
\end{align*}
By the proof of Lemma \ref{FG}, $ \mathrm{SO}(3) $ is a covering space of $ \cN $. Since $ [\LL]_{\cN}=\rH_4 $, $ (\n,\m,\p) $ is a non-trivial loop on $ \mathrm{SO}(3) $. In view of Lemma \ref{MSO(3)}, $ \pi_1(\mathrm{SO}(3))\cong\pi_1(\R\PP^3)\cong\Z_2 $. Like $ \cP $, there are also only two free homotopy classes of $ \mathrm{SO}(3) $. As a result, we only need to show that
\be
\inf\left\{\f{1}{2}\int_{\Ss^1}|\PP'(\theta)|^2\ud\theta:\PP\in H^1(\Ss^1,\mathrm{SO}(3))\text{ is non-trivial}\right\}=2\pi.\label{2pi}
\ee
In view of Remark \ref{SO3rem}, the above minimum is achieved by a geodesics $ (\n,\m,\p)=\psi(\vv)\in H^1(\Ss^1,\mathrm{SO}(3)) $ and $ \vv:[0,2\pi]\to\Ss^3 $ is a half of big circle of $ \Ss^3 $. For simplicity, we can choose $ \vv(\theta)=(\cos(\theta/2),\sin(\theta/2),0,0)^{\T} $ with $ \theta\in[0,2\pi] $. This gives 
$$
(\n,\m,\p)=\(\begin{matrix}
1&0&0\\
0&\cos\theta&-\sin\theta\\
0&\sin\theta&\cos\theta\\
\end{matrix}\),
$$
which directly implies \eqref{2pi} and completes the proof of the first property.

Next, we give the explicit value of $ \cE^*(\cdot) $ on $ [\Ss^1,\cN] $. Fix $ \X_0\in\cN $. Let $ \LL\in C^0(\Ss^1,\cN) $ such that $ \X_0\in\LL $ and $ [\LL]_{\cN}\neq\rH_0 $. For any $ \{[\LL_{\ell}]_{\cN,\X_0}\}_{\ell=1}^n\subset\pi_1(\cN,\X_0) $ such that $ [\LL]_{\cN,\X_0}=\prod_{\ell=1}^n[\LL_{\ell}]_{\cN,\X_0} $, we have 
$$
\vp_{\X_0}^{-1}([\LL]_{\cN,\X_0})=\prod_{\ell=1}^n\vp_{\X_0}^{-1}([\LL_\ell]_{\cN,\X_0})\text{ in }Q_8.
$$

\underline{$ \cE(\rH_1)=\cE(\rH_2)=\kappa_* $.} Assume that $ [\LL]_{\cN}=\rH_1,\rH_2 $. As a result, $ \vp_{\X_0}^{-1}([\LL]_{\cN,\X_0})\in\{\pm\ii,\pm\jj\} $ by Corollary \ref{homcor}. This implies that there is $ \ell_0\in\{1,2,...,n\} $ such that $ \vp_{\X_0}^{-1}([\LL_{\ell_0}]_{\cN,\X_0})\neq 1 $. By the first property, we get $ \cE([\LL_{\ell_0}]_{\cN})\geq\kappa_* $ and then $ \sum_{\ell=1}^n\cE([\LL_{\ell}]_{\cN})\geq\kappa_* $. In view of the second property of Lemma \ref{singuenergessum}, we obtain that $ \cE^*([\LL]_{\cN})\geq\kappa_* $. On the other hand, by the first property of Lemma \ref{singuenergessum}, we have that $ \cE^*([\LL]_{\cN})\leq\cE([\LL]_{\cN})=\kappa_* $. Therefore, $ \cE^*([\LL]_{\cN})=\kappa_* $.

\underline{$ \cE(\rH_3)=2\kappa_* $.} We set $ [\LL]_{\cN}=\rH_3 $. Using Corollary \ref{homcor} again, we obtain that $ \vp_{\X_0}^{-1}([\LL]_{\cN,\X_0})\in\{\pm\kk\} $. Consequently, either there exists some $ \ell_0\in\{1,2,...,n\} $ such that $ \vp_{\X_0}^{-1}([\LL_{\ell_0}]_{\cN,\X_0})\in\{\pm\kk\} $ or there are $ \ell_1,\ell_2\in\{1,2,...,n\} $ such that $ \vp_{\X_0}^{-1}([\LL_{\ell_p}]_{\cN,\X_0})\neq 1 $ for $ p=1,2 $. By the second property, $ \cE(\rH_3)>2\kappa_* $. Applying Lemma \ref{singuenergessum} again, we have that for each of the cases above, there holds $ \sum_{\ell=1}^n\cE([\LL_{\ell}]_{\cN})\geq2\kappa_* $ and consequently $ \cE^*([\LL]_{\cN})\geq 2\kappa_* $. On the other hand, without loss of generality, we assume that $ \vp_{\X_0}^{-1}([\LL]_{\cN,\X_0})=\kk $. Choose $ \{[\LL_{\ell}]_{\cN,\X_0}\}_{\ell}^2 $ such that $ \vp_{\X_0}^{-1}([\LL_1]_{\cN,\X_0})=\ii $ and $ \vp_{\X_0}^{-1}([\LL_2]_{\cN,\X_0})=\jj $. From the construction, we deduce that $ [\LL]_{\cN,\X_0}=[\LL_1]_{\cN,\X_0}[\LL_2]_{\cN,\X_0} $ and then $ [\LL]_{\cN}=[\LL_1*\LL_2]_{\cN} $. In view of Corollary \ref{homcor}, $ \cE([\LL_{\ell}]_{\cN})=\kappa_* $ for $ \ell=1,2 $. This implies that $ \cE^*([\LL]_{\cN})\leq \sum_{\ell=1}^2\cE([\LL_{\ell}]_{\cN})=2\kappa_* $ and completes the proof of this case.\smallskip

\underline{$ \cE(\rH_4)=2\kappa_* $.} The arguments is almost the same as the proof $ \cE(\rH_3)=2\kappa_* $ and we omit it for simplicity.
\end{proof}

\subsection{Lower energy bound for Landau-de Gennes model}

Let $ \Q\in H^1(B_r^2,\Ss_0) $ with $ r>0 $. For a subset $ U\subset B_r^2 $, we define
$$
\phi_0(\Q,U):=\essinf_{U}(\phi_0\circ\Q)=\essinf_{U}(\min\{(\phi_1\circ\Q),(\phi_2\circ\Q)\}).
$$
If $ \Q\in C^0(\ol{B_1^2},\Ss_0) $ and $ \phi_0(\Q,A_{s,r}^2)>0 $, then $ \Q(x)\in\Ss_0\backslash(\cC_1\cup\cC_2) $ for any $ x\in\ol{A_{s,r}^2} $.

\begin{prop}\label{lowerboundprop}
Let $ 0<\va<1 $ and $ A:=A_{1/80,1}^2 $. Assume that $ C_0>0 $ is a positive constant. There exists $ C>0 $ depending only on $ \cA $ and $ C_0 $ such that if $ \Q\in H^1(B_1^2,\Ss_0) $ satisfies 
\be
\phi_0(\Q,A)>0\text{ and }\|\Q\|_{L^{\ift}(B_1^2)}\leq C_0,\label{phigeq}
\ee
then
$$
E_{\va}(\Q,B_1^2)\geq\cE^*([\varrho\circ\Q|_{\pa B_1^2}]_{\cN})\phi_0^2(\Q,A)\log\f{1}{\va}-C.
$$
\end{prop}

Before the proof of this proposition, we firstly give a corollary as follows.

\begin{cor}\label{lowerboundcor}
Let $ \va,r>0 $ be such that $ 0<\va<r/80 $. Given a map $ \Q\in H^1(B_r^2,\Ss_0) $ such that $ \Q|_{\pa B_r^2}\in H^1(\pa B_r^2,\Ss_0) $, $ \|\Q\|_{L^{\ift}(B_r^2)}\leq C_0 $ for some positive constant $ C_0>0 $ and
$$
\phi_0(\Q,\pa B_r^2):=\essinf_{\pa B_r^2}(\phi_0\circ\Q)>0.
$$
There holds
$$
E_{\va}(\Q,B_r^2)+(4\log 5)E_{\va}(\Q,\pa B_r^2)\geq\cE^*([\varrho\circ\Q|_{\pa B_r^2}]_{\cN})\phi_0^2(\Q,\pa B_r^2)\log\f{r}{\va}-C,
$$
where $ C>0 $ depends only on $ \cA,C_0 $.
\end{cor}

\begin{proof}
Set $ \ol{\va}=80\va/r\in(0,1) $ and the map $ \PP\in H^1(B_1^2,\Ss_0) $ defined by
$$
\PP(x)=\left\{\begin{aligned}
&\Q(rx/|x|)&\text{ if }&x\in A,\\
&\Q(80rx)&\text{ if }&x\in\ol{B_{1/80}^2}.
\end{aligned}\right.
$$
Since $ \PP $ is homogeneous in $ A $, it follows that $ \phi_0(\PP,A)=\phi_0(\Q,\pa B_r^2) $. By simple calculations,
\begin{align*}
E_{\ol{\va}}(\PP,B_1^2)&= E_{\ol{\va}}(\PP,B_{1/80}^2)+\int_{1/80}^1 E_{\ol{\va}}(\PP,\pa B_s^2)\ud s\\
&\leq E_{\va}(\Q,B_r^2)+\int_{1/80}^1\int_{0}^{2\pi}\(\f{1}{s^2}|\pa_{\theta}\Q(r,\theta)|^2+\f{1}{\ol{\va}^2}f_b(\Q(r,\theta))\)\ud\theta\ud s\\
&\leq E_{\va}(\Q,B_r^2)+\(\int_{1/80}^1\f{r}{s}\ud s\)\int_{\pa B_r^2}\(\f{1}{2}|\na_{\pa B_r^2}\Q|^2+\f{1}{\va^2}f_b(\Q)\)\ud\HH^1\\
&=E_{\va}(\Q,B_r^2)+(4\log 5)E_{\va}(\Q,\pa B_r^2),
\end{align*}
where for the first inequality, we have used the change of variables. Using Proposition \ref{lowerboundprop}, we can complete the proof.
\end{proof}

\begin{lem}\label{smoothlower}
Let $ 0<\va<1 $ and $ A_*:=A_{\mu,1-\mu}^2 $ with $ \mu=1/40 $. Assume that $ C_0>0 $ is a positive constant. There exists $ C>0 $ depending only on $ \cA $ and $ C_0 $ such that if $ \Q\in C^{\ift}(B_{1-\mu}^2,\Ss_0)\cap C^0(\ol{B_{1-\mu}^2},\Ss_0) $ satisfies
\be
\phi_0(\Q,A_*)>0\text{ and }\|\Q\|_{L^{\ift}(B_{1-\mu}^2)}\leq C_0,\label{Awtgeq}
\ee
then 
\be
E_{\va}(\Q,B_{1-\mu}^2)\geq\cE^*([\varrho\circ\Q|_{B_{1-\mu}^2}]_{\cN})\phi_0^2(\Q,A_*)\log\f{1}{\va}-C.\label{Eva1mueslow}
\ee
\end{lem}
\begin{proof}
Due to \eqref{Awtgeq}, $ [\varrho\circ\Q|_{\pa B_{1-\mu}^2}]_{\cN} $ is well defined. We assume that 
\be
[\varrho\circ\Q|_{\pa B_{1-\mu}^2}]_{\cN}\neq\rH_0,\label{nonH0ass}
\ee
since for otherwise, there is nothing to prove. We claim that $ \phi_0(\Q,\ol{B_{1-\mu}^2})=0 $. Indeed, if not, $ \varrho\circ\Q:\ol{B_{1-\mu}^2}\to\cN $ is continuous in $ \ol{B_{1-\mu}^2} $ and is a free homotopy function from $ \varrho\circ\Q|_{\pa B_{1-\mu}^2} $ to a constant map, which is a contradiction to the assumption \eqref{nonH0ass}. By the first property of Lemma \ref{Lipphi4}, we have
\be
\phi_{\tau}^{(0)}:=\min_{\ol{A_*}}(\phi_{\tau}\circ\Q)\geq \f{6-5\tau}{6}\phi_0(\Q,A_*)>0\text{ and }\min_{\ol{B_{1-\mu}^2}}(\phi_{\tau}\circ\Q)\leq\min_{\ol{B_{1-\mu}^2}}(\phi_0\circ\Q)=0,\label{phitau0}
\ee
for any $ \tau\in(0,1/2) $. Next, we show that
\be
E_{\va}(\Q,B_{1-\mu}^2)\geq\(\f{6-5\tau}{6}\)^2\cE^*([\varrho\circ\Q|_{B_{1-\mu}^2}]_{\cN})\phi_0^2(\Q,A_*)\log\f{1}{\va}-C,\label{tauEva}
\ee
where $ C>0 $ is a constant depending only on $ \cA $ and $ C_0 $. \eqref{Eva1mueslow} directly follows by letting $ \tau\to 0^+ $. In the following arguments, we fix some $ \tau\in(0,1/2) $. For each $ t>0 $, we define
$$
\om_{t,\gtrless}^{(\tau)}:=\{x\in B_{1-\mu}^2:\phi_{\tau}\circ\Q(x)\gtrless t\}\text{ and }
\Ga_t^{(\tau)}:=\pa\om_t^{(i)}\cap B_{2\mu}^2.
$$
In view of \eqref{phitau0}, if $ t\in(0,\phi_{\tau}^{(0)}) $, then sets $ \om_t^{(\tau),\gtrless} $, and $ \Ga_t^{(\tau)} $ are all non-empty. Using Lemma \ref{smoothLda} and \ref{Lipphi4}, $ \phi_{\tau}\circ\Q $ is $ C^2 $ in $ \om_0^{(\tau)} $. By Morse-Sard lemma, for a.e. $ t\in(0,\phi_{\tau}^{(0)}) $, every connected component of $ \Ga_t^{(\tau)} $ is a $ C^2 $ regular curve. Meanwhile, for such $ t $, we note that the number of the connected components must be finite. If not, we assume that the result fails for some $ t_0 $ and $ \Ga_{t_0}^{(\tau)} $. By the definition of $ \phi_{\tau}^{(0)} $, we have $ \Ga_{t_0}^{(\tau)}\subset\ol{B_{\mu}^2} $. Assume that $
\Ga_{t_0}^{(\tau)}=\cup_{i\in\Z_+}\Ga_{t_0}^{(\tau)}(i) $, where $ \{\Ga_{t_0}^{(\tau)}(i)\}_{i\in\Z_+} $ is the collection of connected components of $ \Ga_{t_0}^{(\tau)} $. Now we can choose a sequence $ \{x_i\}_{i\in\Z_+} $, such that $ x_i\in\Ga_{t_0}^{(\tau)}(i) $ and $ x_i\to x_0\in\ol{B_{\mu}^2} $. Since $ \Ga_{t_0}^{(\tau)} $ is closed, $ x_0\in\Ga_{t_0}^{(\tau)} $, so $ x_0\in\Ga_{t_0}^{(\tau)}(j) $ for some $ j\in\Z_+ $. On the other hand, the implicit function theorem implies that there exists a neighborhood of $ x_0 $, denoted by $ U\subset B_{2\mu}^2 $, such that $ U\cap\Ga_{t_0}^{(\tau)}=U\cap\Ga_{t_0}^{(\tau)}(j) $, which is a contradiction to the choice of $ x_0 $. We also note that for a.e. $ t\in(0,\phi_{\tau}^{(0)}) $, $
\pa\om_{t,<}^{\tau}=\Ga_t^{(\tau)}\subset\{x\in B_{1-\mu}^2:\phi_{\tau}\circ\Q(x)=t\} $. Now for such $ t $, assume that $ \{\Ga_{t}^{(\tau)}(i)\}_{i=1}^{n_t} $ is the collection of connected components of $ \Ga_{t}^{(\tau)} $, we can use Lemma \ref{radlem} to obtain
\be
\HH^1(\Ga_t^{(\tau)})\geq\sum_{i=1}^{n_t}\diam(\Ga_{t}^{(\tau)}(i))\geq\sum_{i=1}^{n_t}\op{rad}(\Ga_{t}^{(\tau)}(i))\geq\op{rad}(\om_{t,<}^{\tau})\label{Gatradt}
\ee
Moreover, for any $ t\in(0,\phi_{\tau}^{(0)}) $, $ \varrho\circ\Q $ is well defined in $ \om_{t,>}^{\tau} $, so we define
$$
\Theta(t):=\int_{\om_{t,>}^{\tau}}|\na(\varrho\circ\Q)|^2\ud\HH^2.
$$
Now we deduce
\be
\begin{aligned}
E_{\va}(\Q,B_{1-\mu}^2)&\geq\int_{B_{1-\mu}^2}\(\f{r_*^2}{36}|\na(\phi_{\tau}\circ\Q)|^2+\f{1}{\va^2}f_b(\Q)\)\ud\HH^2+\int_{B_{1-\mu}^2}\f{1}{2}(\phi_{\tau}\circ\Q)^2|\na(\varrho\circ\Q)|^2\ud\HH^2\\
&\geq\int_{B_{1-\mu}^2}\f{r_*}{3\va}|\na(\phi_{\tau}\circ\Q)|(f_b(\Q))^{1/2}\ud\HH^2+\int_0^{\ift}t\Theta(t)\ud t\\
&\geq\int_0^{\phi_{\tau}^{(0)}}\int_{\Ga_t^{(\tau)}}\f{C}{\va}(f_b(\Q))^{1/2}\ud\HH^1\ud t+\int_0^{\ift}t\Theta(t)\ud t,
\end{aligned}\label{EvaEs}
\ee
where for the first inequality, we have used \eqref{naQdetau}. For the second inequality of \eqref{EvaEs}, we have used the Cauchy inequality and the formula
\begin{align*}
\int_{B_{1-\mu}^2}(\phi_{\tau}\circ\Q)^2&|\na(\varrho\circ\Q)|^2\ud\HH^2=\int_{B_{1-\mu}^2}\int_{0}^{(\phi_{\tau}\circ\Q)(x)}2t|\na(\varrho\circ\Q)|^2\ud t\ud\HH^2\\
&=\int_0^{\ift}2t\(\int_{\{x\in B_{1-\mu}^2:(\phi_{\tau}\circ\Q)(x)>t>0\}}|\na(\varrho\circ\Q)|^2\ud\HH^2\)\ud t\\
&=\int_0^{\ift}2t\(\int_{\om_{t,>}^{\tau}}|\na(\varrho\circ\Q)|^2\ud\HH^2\)\ud t=2\int_0^{\ift}t\Theta(t)\ud t.
\end{align*}
For the last inequality of \eqref{EvaEs}, we have used the property that $ \phi_{\tau}\circ\Q $ is Lipschitz (see Lemma \ref{Lipphi4}) and the coarea formula. Equation \eqref{1phi} implies that on $ \Ga_t^{(0)} $, $
f_b(\Q)\geq y_{\tau}(t) $, where
\begin{align*}
y_{\tau}(t):=\left\{\begin{aligned}
&\(1-\f{6t}{6-5\tau}\)^2&\text{ if }&t\in\[0,\f{12-10\tau}{12-5\tau}\],\\
&(1-t)^2&\text{ if }&t\in\(\f{12-10\tau}{12-5\tau},+\ift\).
\end{aligned}\right.
\end{align*}
This, together with \eqref{Gatradt} and \eqref{EvaEs}, implies that 
\be
\begin{aligned}
E_{\va}(\Q,B_{1-\mu}^2)&\geq\int_0^{\phi_{\tau}^{(0)}}\int_{\Ga_t^{(\tau)}}\f{C}{\va}(y_{\tau}(t))^{1/2}\ud\HH^1\ud t+\int_0^{\phi_{\tau}^{(0)}}t\Theta(t)\ud t\\
&\geq\int_0^{\phi_{\tau}^{(0)}}\(\f{C}{\va}(y_{\tau}(t))^{1/2}\op{rad}(\om_{t,<}^{\tau})+t\Theta(t)\)\ud t.
\end{aligned}\label{Thetapsi}
\ee
For a.e. $ t\in(0,\phi_{\tau}^{(0)}) $, by \eqref{Awtgeq}, we have $ \om_{t,<}^{\tau}\subset\ol{B_{\mu}^2} $, so $ \dist(\om_{t,<}^{\tau},\pa B_{1-\mu}^2)\geq 2\lda(1-\mu) $, with $ \lda=(1-2\mu)/(2(1-\mu))=19/39>19/40 $. Therefore, by applying Lemma \ref{Plower} to the manifold $ \cN $, we have
\be
\begin{aligned}
\Theta(t)&\geq\int_{B_{1-\mu}^2\backslash\om_{t,<}^{\tau}}|\na(\varrho\circ\Q)|^2\ud\HH^2\\
&\geq 4\kappa_*([\varrho\circ\Q|_{B_{1-\mu}^2}]_{\cN})(\log(\beta(1-\mu))-\log(\op{rad}(\om_{t,<}^{\tau})))\\
&\geq 2e_*\log\f{19}{20}-2e_*\log(\op{rad}(\om_{t,<}^{\tau})),
\end{aligned}\label{Thetalow}
\ee
where we denote $ e_*=\cE^*([\varrho\circ\Q|_{B_{1-\mu}^2}]_{\cN}) $ and use the fact that $ e_*\leq 2\kappa_* $ (by Lemma \ref{kappa123}). This, together with \eqref{Thetapsi}, implies that
$$
E_{\va}(\Q,B_{1-\mu}^2)\geq\int_0^{\phi_{\tau}^{(0)}}\(\f{C_1}{\va}(y_{\tau}(t))^{1/2}\op{rad}(\om_{t,<}^{\tau})-2e_*t\log(\op{rad}(\om_{t,<}^{\tau}))\)\ud t-C_2,
$$
where $ C_2>0 $ depends only on $ \cA $ and $ C_0 $. An easy analysis shows that if $ t\neq 1,1-5\tau/6 $, the function $
g(s)=C_1\va^{-1}(y_{\tau}(t))^{1/2}s-2e_*t\log s $ with $ s\in\R_+ $ has a unique minimizer $ s_* $, which is given by $ g'(s_*)=0 $, i.e., $ s_*=2e_*\va t/(C_1(y_{\tau}(t))^{1/2}) $. As a consequence, we obtain that
\be
\begin{aligned}
E_{\va}(\Q,B_{1-\mu}^2)&\geq\int_0^{\phi_{\tau}^{(0)}}\(\f{C_1}{\va}(y_{\tau}(t))^{1/2}s_*-2e_*t\log s_*\)\ud t-C_2\\
&=-2e_*\int_0^{\phi_{\tau}^{(0)}}\(t\log\va-t+t\log\f{e_*t}{C_1(y_{\tau}(t))^{1/2}}\)\ud t-C_2\\
&\geq-e_*(\phi_{\tau}^{(0)})^2\log\va-e_*(\phi_{\tau}^{(0)})^2-2e_*\int_0^{\phi_{\tau}^{(0)}}t|\log((y_{\tau}(t))^{1/2})|\ud t\\
&\quad-2e_*\int_{0}^{\phi_{\tau}^{(0)}}t(\log(e_*t)-\log C_1)\ud t-C_2.
\end{aligned}\label{EvaC2}
\ee
We claim that $ \int_0^{\phi_{\tau}^{(0)}}t|\log((y_{\tau}(t))^{1/2})|\ud t\leq C_3 $, where $ C_3 $ is independent of $ \tau $. This, together with \eqref{phitau0} and \eqref{EvaC2}, directly implies \eqref{tauEva}. Indeed, choosing $ M_{\tau}=\max\{\phi_{\tau}^{(0)},1\} $, we have 
$ 0\leq M_{\tau}\leq\max\{\sqrt{C_0/2}r_*^{-1},1\} $ and the claim follows from
\begin{align*}
\int_0^{\phi_{\tau}^{(0)}}t|\log((y_{\tau}(t))^{1/2})|\ud t&\leq\int_0^{M_{\tau}}t|\log((y_{\tau}(t))^{1/2})|\ud t\\
&\leq\int_0^{M_{\tau}}t\(\log\left|1-\f{6t}{6-5\tau}\right|+\log|1-t|\)\ud t\leq C_3.
\end{align*}
Now, we can complete the proof.
\end{proof}

\begin{proof}[Proof of Proposition \ref{lowerboundprop}] 
For $ \Q\in H^1(B_1^2,\Ss_0) $, we still assume that $ [\varrho\circ\Q|_{\pa B_1}^2]_{\cN}\neq\rH_0 $ and for otherwise, there is nothing to prove. By standard extension theorem, we can construct $ \PP\in H^1(\R^2,\Ss_0) $ such that $ \supp\PP\subset B_2^2 $, $ \PP|_{B_1^2}=\Q|_{B_1^2} $, $
\|\PP\|_{H^1(B_2^2)}\leq C\|\Q\|_{H^2(B_1^2)} $ and $ \|\PP\|_{L^{\ift}(B_2^2)}\leq C $. Fix $ \chi\in C_c^{\ift}(B_{1/2}^2) $ such that $ \chi\geq 0 $ and $ \int_{\R^2}\chi\ud x=1 $. For $ \delta>0 $, define
$$
\Q^{\delta}(x)=\PP*\chi^{\delta}(x)=\int_{\R^2}\PP(x-y)\chi^{\delta}(y)\ud y,
$$
where $ \chi^{\delta}(y)=\delta^{-2}\chi(y/\delta) $. For such $ \Q^{\delta} $, there hold the following properties.
\begin{enumerate}
\item $ \lim_{\delta\to 0^+}E_{\va}(\Q^{\delta},B_1^2)=E_{\va}(\Q,B_1^2) $.
\item Recall that $ A_*=A_{1/40,39/40}^2 $, if $ \delta>0 $ is sufficiently small, then 
\be
\phi_0(\Q^{\delta},A_*)\geq\f{1}{2}\phi_0(\Q,A),\label{Qdeltaphii}
\ee
and $ [\varrho\circ\Q|_{\pa B_1^2}]_{\cN}=[\varrho\circ\Q^{\delta}|_{\pa B_{39/40}^2}]_{\cN} $.
\item There exists $ \{\delta_n\}_{n=1}^{+\ift} $, such that $ \delta_n\to 0^+ $ as $ n\to+\ift $ and
\be
\lim_{n\to+\ift}\phi_0(\Q^{\delta},A_*)\geq \phi_0(\Q,A).\label{bigerlim}
\ee
\end{enumerate}

Firstly, by simple calculations, we have
\be
\lim_{\delta\to 0^+}\|\Q^{\delta}-\Q\|_{H^1(B_1^2)}=0.\label{H1QdeltaA}
\ee
Using Sobolev embedding theorem, we get $ \Q^{\delta}\to\Q $ in $ L^6(B_1^2,\Ss_0) $ as $ \delta\to 0^+ $ and the first property holds. For the second one, if \eqref{Qdeltaphii} is true, then $ \varrho\circ\Q^{\delta} $ is well define in $ A_* $. By Lemma \ref{varrhop} and \eqref{H1QdeltaA}, we have $
\lim_{\delta\to 0^+}\|\varrho\circ\Q^{\delta}-\varrho\circ\Q\|_{H^1(A_*)}=0 $. $ [\varrho\circ\Q|_{\pa B_1^2}]_{\cN}=[\varrho\circ\Q^{\delta}|_{\pa B_{39/40}^2}]_{\cN} $ follows from Lemma \ref{propHomt} and its proof. Next, we show \eqref{Qdeltaphii}. For any $ x\in A_* $ and $ 0<\delta<1/100 $, we have $ B_{\delta}^2(x)\subset A $, which implies that $
\phi_0(\Q,A)\leq\essinf_{B_{\delta}^2(x)}(\phi_0\circ\Q) $. Combined with Lemma \ref{Lipphi0lem}, for any $ x\in A_* $, there holds,
\begin{align*}
\dist((\phi_0\circ\Q^{\delta})(x)&,[\phi_0(\Q,A),+\ift))\leq\dist\((\phi_0\circ\Q^{\delta})(x),\essinf_{B_{\delta}^2(x)}(\phi_0\circ\Q)\)\\
&\leq\dashint_{B_{\delta}^2(x)}|(\phi_0\circ\Q^{\delta})(x)-(\phi_0\circ\Q)(y)|\ud\HH^2(y)\\
&\leq C\dashint_{B_{\delta}^2(x)}|\Q^{\delta}(x)-\Q(y)|\ud\HH^2(y).
\end{align*}
Note for such $ B_{\delta}^2(x) $, we have
$$
\Q^{\delta}(x)=\int_{\R^2}\PP(y)\chi^{\delta}(x-y)\ud y=\int_{B_{\delta}^2(x)}\Q(y)\chi^{\delta}(x-y)\ud y.
$$
Now, using generalized Poincar\'{e} inequality and H\"{o}lder inequality, we obtain
\begin{align*}
\int_{B_{\delta}^2(x)}|\Q^{\delta}(x)-\Q(y)|\,d\HH^2(y)\leq C\delta\|\na\Q\|_{L^1(B_{\delta}^2(x))}\leq C\delta^2\|\na\Q\|_{L^2(B_{\delta}^2(x))}.
\end{align*}
Combining the calculations above, we have
\be
\lim_{\delta\to 0}\sup_{x\in A_*}\dist((\phi_0\circ\Q^{\delta})(x),[\phi_0(\Q,A),+\ift))\to 0.\label{phi0Qdeltaphi012}
\ee
which directly implies \eqref{Qdeltaphii}. For the third property, we can firstly choose a subsequence of $ \Q^{\delta} $ (still denoted by $ \Q^{\delta} $), such that $ \Q^{\delta}\to\Q $ a.e. in $ B_1^2 $ and
$$
\lim_{\delta\to 0^+}\phi_0(\Q^{\delta},A_*)\leq\phi_0(\Q,A_*)\leq\phi_0(\Q,A).
$$
Combined with \eqref{phi0Qdeltaphi012}, we obtain \eqref{bigerlim}. In view of Lemma \ref{smoothlower} and \eqref{Qdeltaphii}, we have
\begin{align*}
E_{\va}(\Q^{\delta},B_{39/40}^2)&\geq\cE^*([\varrho\circ\Q|_{\pa B_{1-\mu}^2}]_{\cN})\phi_0^2(\Q^{\delta},A_*)\log\f{1}{\va}-C\\
&\geq\cE^*([\varrho\circ\Q|_{\pa B_1^2}]_{\cN})\phi_0^2(\Q^{\delta},A_*)\log\f{1}{\va}-C.
\end{align*}
Choose $ \delta=\delta_n $ in the third property above, and let $ n\to+\ift $, we can use the first property and \eqref{bigerlim} to get
\begin{align*}
E_{\va}(\Q^{\delta},B_1^2)\geq E_{\va}(\Q^{\delta},B_{39/40}^2)\geq\cE^*([\varrho\circ\Q|_{\pa B_1^2}]_{\cN})\phi_0^2(\Q^{\delta},A_*)\log\f{1}{\va}-C,
\end{align*}
as desired.
\end{proof}

\section{Some technical lemmas}\label{Technical lemmas}

\subsection{Extension theory}

In this subsection, we focus on extension results within the unit ball, specifically aiming to find a function $ \PP $ that satisfies certain conditions. Precisely, for $ \PP_b\in H^1(\pa B_r^k,\cN) $, we intend to construct $ \PP\in H^1(B_r^k,\cN) $ such that $ \PP|_{\pa B_r^k}=\PP_b $. We also want to control the energy of $ \PP $ to some extent based on the norm of $ \PP_b $.

The approach used here is similar to those used in considering the unaxial limit, but with some differences. Firstly, due to the change in the manifold $ \cN $, we require a detailed characterization and understanding of the covering space. In contrast to the approach presented in \cite{HL22}, where only $ \M $ is considered as the covering space for $ \cN $, we need to consider its universal cover, $ \Ss^3 $ (as demonstrated in Lemma \ref{FG}, for example).

Secondly, when dealing with two-dimensional case, careful consideration must be given to the free homotopy classes of the boundary values. This is necessary because the fundamental group of $ \cN $ is more complicated compared to $ \R\PP^2 $.

\begin{lem}\label{k3}
Let $ k\in\Z_{\geq 3} $. There exists $ C>0 $, depending only on $ \cA $ and $ k $ such that for any $ r>0 $ and $ \PP_b\in H^1(\pa B_r^k,\cN) $, there is $ \PP\in H^1(B_r^k,\cN) $ satisfying $ \PP=\PP_b $ on $ \pa B_r^k $ and
$$
\|\na\PP\|_{L^2(B_r^k)}^2\leq Cr^{k/2-1/2}\|\na_{\pa B_r^k}\PP_b\|_{L^2(\pa B_r^k)}.
$$    
\end{lem} 

\begin{lem}\label{k2trivial}
There exists $ C>0 $, depending only on $ \cA $ such that for any $ r>0 $ and $ \PP_b\in H^1(\pa B_r^2,\cN) $ with $ [\PP_b]_{\cN}=\rH_0 $, there is $ \PP\in H^1(B_r^2,\cN) $ satisfying $ \PP=\PP_b $ on $ \pa B_r^2 $ and
$$
\|\na\PP\|_{L^2(B_r^2)}^2\leq Cr\|\na_{\pa B_r^2}\PP_b\|_{L^2(\pa B_r^2)}^2.
$$
\end{lem}

Roughly speaking, the proof of Lemma \ref{k3} and \ref{k2trivial} follows from two steps. The firs step is the lift to the universal covering space $ \Ss^3 $ and the second step is to extend the lifting maps. To conduct the second procedure, we need the following result, which is firstly given by \cite{HKL86} for $ \Ss^2 $-valued maps.

\begin{lem}\label{Extlemball}
Given $ k\in\Z_{\geq 3} $, $ \ell\in\Z_{\geq 2} $ and $ \n\in H^1(\pa B_r^k,\Ss^{\ell}) $, then there exists $ \uu\in H^1(B_r^k,\Ss^{\ell}) $ such that $ \uu=\n $ on $ \pa B_r^k $ and
\begin{align}
\|\na\uu\|_{L^2(B_r^k)}^2&\leq Cr^{k/2-1/2}\|\na_{\pa B_r^k}\n\|_{L^2(\pa B_r^k)},\label{un1}\\
\|\na\uu\|_{L^2(B_r^k)}^2&\leq Cr\|\na_{\pa B_r^k}\n\|_{L^2(\pa B_r^k)}^2,\label{un2}
\end{align}
where $ C>0 $ depends only on $ k $ and $ \ell $.
\end{lem}
\begin{proof}
In Lemma 2.3 of \cite{HKL86} such result is given for $ k=3 $ and $ \ell=2 $. In \cite{HL87}, a more general case was given in Theorem 6.1. Here we combine both of them and present a proof here for completeness. For given $ \n\in H^1(\pa B_r^k,\Ss^{\ell}) $, we consider the Dirichlet problem
\be
\left\{\begin{aligned}
-\Delta\vv&=0&\text{ in }&B_r^k\\
\vv&=\n&\text{ on }&\pa B_r^k,
\end{aligned}\right.\label{existencevnfunction}
\ee
where $ \vv:B_r^k\to\R^{\ell+1} $.
The existence of such $ \vv $ is ensured by Poisson integral formula. Testing the equation \eqref{existencevnfunction} by $ x\cdot\na\vv $ and applying the integration by parts, we can obtain
$$
0=\int_{B_r^k}-\Delta\vv(x\cdot\na\vv)\ud x=\int_{B_r^k}|\na\vv|^2\ud\HH^k-r\int_{\pa B_r^k}\left|\f{x_i}{r}\pa_i\vv\right|^2\ud\HH^{k-1}+\int_{B_r^k}\pa_i\vv\cdot(x_j\pa_{i}\pa_{j}\vv)\ud\HH^k,
$$
and
$$
2\int_{B_r^k}\pa_i\vv\cdot(x_j\pa_{i}\pa_{j}\vv)\ud\HH^k=-k\int_{B_r^k}|\na\vv|^2\ud\HH^k+r\int_{\pa B_r^k}|\na\vv|^2\ud\HH^{k-1}.
$$
As a result, we have
\be
r\|\na_{\pa B_r^k}\n\|_{L^2(\pa B_r^k)}^2=(k-2)\|\na\vv\|_{L^2(B_r^k)}^2+r^{-1}\|x_i\pa_i\vv\|_{L^2(\pa B_r^k)}^2.\label{videntity}
\ee
Integrating the equality $ |\na\vv|^2=\op{div}((\vv-\n_*)\na\vv) $, where $ \n_*=\dashint_{\pa B_r^k}\n $, there holds
\be
\|\na\vv\|_{L^2(\pa B_r^k)}^2=\int_{\pa B_r^k}(\n-\n_*)\(\na\vv\cdot\f{x}{r}\)\ud\HH^{k-1}\leq\|\n-\n_*\|_{L^2(\pa B_r^k)}\|\na_{\pa B_r^k}\n\|_{L^2(\pa B_r^k)},\label{navnnstar}
\ee
where for the last inequality we have used Cauchy inequality and \eqref{videntity}. Poincar\'{e} inequality and the fact that $ |\n_*|\leq 1 $ imply that
$$
\|\n-\n_*\|_{L^2(\pa B_r^k)}\leq C\min\{r\|\na_{\pa B_r^k}\n\|_{L^2(\pa B_k)},r^{k/2-1/2}\}.
$$
Combining \eqref{navnnstar}, we can obtain
$$
\|\na\vv\|_{L^2(B_r^k)}^2\leq C\min\{r\|\na_{\pa B_r^k}\n\|_{L^2(\pa B_r^k)}^2,r^{k/2-1/2}\|\na_{\pa B_r^k}\n\|_{L^2(\pa B_r^k)}\},
$$
Define a projection $ p_y:\R^{\ell}\to\R^{\ell} $ such that 
$ p_y(w)=(w-y)/|w-y| $ with $ w\in\R^{\ell} $ and $ y\in B_{1/4}^{\ell} $. Since $ \vv $ is harmonic, it is smooth in $ B_r^k $. By using Sard theorem, we have, $ p_y\circ\vv\in H^1(B_r^k,\Ss^{\ell}) $ for a.e. $ y\in B_{1/4}^{\ell} $. In view of the observation $ -\Delta(|\vv|^2)=-|\na\vv|^2\leq 0 $ in $ B_r $ and $ |\vv|=|\n|=1 $ on $ \pa B_r $, we can use maximum principle to get that $ \|\vv\|_{L^{\ift}(B_r)}\leq 1 $. Applying Fubini theorem, we obtain 
\begin{align*}
\int_{B_{1/4}^{\ell}}\int_{B_r^k}|\na (p_y\circ\vv)(x)|^2&\ud\HH^k(x)\ud\HH^{\ell}(y)\leq C\int_{B_r^k}|\na \vv(x)|^2\int_{B_{1/4}^{\ell}}\f{1}{|\vv(x)-y|^2}\ud\HH^{\ell}(y)\ud\HH^k(x)\\
&\leq C\int_{B_r^k}|\na\vv(x)|^2\(\int_{B_1^{\ell}}\f{1}{|y|^2}\ud\HH^{\ell}(y)\)\ud\HH^k(x)\\
&=C_0\int_{B_r^k}|\na\vv(x)|^2\ud\HH^k(x).
\end{align*}
Thus we may choose $ y_0\in B_{1/4}^{\ell} $ so that 
$$ 
\int_{B_r^k}|\na(p_{y_0}\circ\vv)|^2\ud\HH^k\leq 4C_0\int_{B_r^k}|\na\vv|^2\ud\HH^k.
$$
Letting $ \uu=(p_{y_0}|_{\Ss^{\ell}})^{-1}\circ p_{y_0}\circ\vv $, we conclude that $ \uu|_{\pa B_r^k}=\vv|_{\pa B_r^k} $ and 
$$
\int_{B_r^k}|\na\uu|^2\ud\HH^k\leq\|D((p_{y_0}|_{\Ss^{\ell}})^{-1})\|_{\ift}^2 \int_{B_r^k}|\na(p_{y_0}\circ\vv)|^2\ud\HH^k\leq C\int_{B_r^k}|\na\vv|^2\ud\HH^k,
$$
where we have used the fact that when $ y\in B_{1/4}^{\ell} $, $ p_{y} $ is invertible on $ \Ss^{\ell} $ and the inverse of it is Lipschitz continuous.
\end{proof}

Now we can give the proof of Lemma \ref{k3} and \ref{k2trivial}.

\begin{proof}[Proof of Lemma \ref{k3}]
For $ k\in\Z_{\geq 3} $, it can be seen that $ \pa B_r^k $ is simply connected and then we can use Lemma \ref{weakApproximation} to find a sequence $ \{\PP_b^{(j)}\}_{j\in\Z_+}\subset C^{\ift}(\pa B_r^k,\cN) $ such that $ \PP_b^{(j)}\wc\PP_b $ weakly in $ H^1(\pa B_r^k,\Ss_0) $ as $ j\to+\ift $. By standard lifting theorem (see Proposition 1.33 in \cite{H02}) and the fact that $ \Ss^3 $ is the universal covering space of $ \cN $ (Lemma \ref{FG}), we have $ \{\vv^{(j)}\}_{j\in\Z_+}\subset C^{\ift}(\pa B_r^k,\Ss^3) $ such that $
p(\vv^{(j)})=\PP_b^{(j)}=r_*(\n^{(j)}\n^{(j)}-\m^{(j)}\m^{(j)}) $, where $ p:\Ss^3\to\cN $, $ \{\n^{(j)}\}_{j\in\Z_+} $, and $ \{\m^{(j)}\}_{j\in\Z_+} $ are given by \eqref{universal}. Assuming that $ \vv^{(j)}=(a_1^{(j)},a_2^{(j)},a_3^{(j)},a_4^{(j)}) $, we have
\be
\begin{aligned}
|\pa_i\PP_b^{(j)}|^2&=r_*^2(|\pa_i\n^{(j)}|^2+|\pa_i\m^{(j)}|^2+|\pa_i\n^{(j)}\cdot\m^{(j)}|)\\
&\geq \f{r_*^2}{4}(4|\pa_i\n^{(j)}|^2+4|\pa_i\n^{(j)}|^2)\\
&\geq \f{r_*^2}{4}(2|\pa_i\n^{(j)}|^2+2|\pa_i\n^{(j)}|^2+2|\pa_i(\n^{(j)}\times\m^{(j)})|^2)\\
&\geq r_*^2(|\pa_ia_1^{(j)}|^2+|\pa_ia_2^{(j)}|^2+|\pa_ia_3^{(j)}|^2+|\pa_ia_4^{(j)}|^2)\\
&=r_*|\pa_i\vv^{(j)}|^2,
\end{aligned}\label{paigpaiv}
\ee
for any $ i=1,2,...,k $, where for the second inequality we have used Cauchy inequality and for the last inequality we have used \eqref{TwoDeri1} and \eqref{TwoDeri}. Since $ \{\PP_b^{(j)}\}_{j\in\Z_+} $ is convergent weakly in $ H^1(\pa B_r^k,\Ss_0) $, we have $
\sup_{j\in\Z_+}\|\PP_b^{(j)}\|_{H^1(\pa B_r^k,\cN)}<+\ift $. Consequently, $ \|\vv^{(j)}\|_{H^1(\pa B_r^k,\Ss^3)} $ is also uniformly bounded with respect to $ j $. By passing a subsequence, we assume that as $ j\to+\ift $, $ \vv^{(j)}\wc\vv $ weakly in $ H^1(\pa B_r^k,\R^4) $ and
$ \vv^{(j)}\to\vv $ strongly in $ L^2(\pa B_r^k,\R^4) $ for some $ \vv\in H^1(\pa B_r^k,\Ss^3) $. Define $ \wh{\PP}_b=p(\vv) $. In view of \eqref{universal}, one can deduce through simple calculations that $ \PP_b^{(j)}\wc\wh{\PP}_b $ weakly in $ H^1(\pa B_r^k,\Ss_0) $ and $ \PP_b^{(j)}\to\wh{\PP}_b $ strongly in $ L^2(\pa B_r^k,\Ss_0) $ as $ j\to+\ift $. As a result, we have $ \wh{\PP}_b=\PP_b $, $ \HH^{k-1} $-a.e. on $ \pa B_r^k $ and $ \vv\in H^1(\pa B_r^k,\Ss^3) $ satisfies $ \PP_b=p(\vv) $ and
\be
C_1\|\na_{\pa B_r^k}\PP_b\|_{L^2(\pa B_r^k)}\leq \|\na_{\pa B_r^k}\vv\|_{L^2(\pa B_r^k)}\leq C_2\|\na_{\pa B_r^k}\PP_b\|_{L^2(\pa B_r^k)},\label{vgrela}
\ee
where the inequality above follows from direct computations and \eqref{paigpaiv}. In view of Lemma \ref{Extlemball}, we can obtain $ \uu\in H^1(B_r^k,\Ss^3) $ such that $ \uu=\vv $ on $ \pa B_r^k $ and
\be
\|\na\uu\|_{L^2(B_r^k)}\leq Cr^{k/2-1/2}\|\na_{\pa B_r^k}\vv\|_{L^2(\pa B_r^k)}.\label{uvrela}
\ee
Now we define $ \PP=p(\uu) $. Similar to the derivation or \eqref{vgrela}, one can obtain 
$$
\|\na\PP\|_{L^2(B_r^k)}\leq C\|\na\uu\|_{L^2(B_r^k)}\leq Cr^{k/2-1/2}\|\na_{\pa B_r^k}\vv\|_{L^2(\pa B_r^k)}\leq Cr^{k/2-1/2}\|\na_{\pa B_r^k}\PP_b\|_{L^2(\pa B_r^k)},
$$
where for the second inequality we have used \eqref{uvrela}.
\end{proof}

\begin{proof}[Proof of Lemma \ref{k2trivial}]
Since $ H^1(\Ss^1,\cN)\hookrightarrow C^{0,1/2}(\Ss^1,\cN) $, we can easily lift $ \PP_b $ to a $ \Ss^3 $-value map, i.e., we can obtain $ \vv:\pa B_1^2\to\Ss^3 $ such that $ p(\vv)=\PP_b $. Since this covering map is locally a diffeomorphism, we obtain that $ \vv\in H^1(\pa B_1^2,\Ss^3) $ by the fact that $ \PP_b\in H^1(\pa B_1^2,\cN) $ and chain rules. Next, we can use almost the same arguments of extension in the proof Lemma \ref{k3} to complete the proof. Note that here we need to use \eqref{un2} instead of \eqref{un1} when using the extension Lemma \ref{Extlemball}. 
\end{proof}

\begin{lem}\label{k2non-trivial}
There exists a constant $ C>0 $ depending only on $ \cA $ such that for any $ 0<\va<r $, and any $ \PP_b\in H^1(\pa B_r^2,\cN) $ with $ [\PP_b]_{\cN}\neq\rH_0 $, there exists $ \PP_{\va}\in H^1(B_r^2,\Ss_0) $ such that $ \PP_{\va}=\PP_b $ on $ \pa B_r^2 $ and
$$
E_{\va}(\PP_{\va},B_r^2)\leq\cE^*([\PP_b]_{\cN})\log\f{r}{\va}+C(r\|\na_{\pa B_r^2}\PP_b\|_{L^2(\pa B_r^2)}^2+1).
$$
\end{lem}

\begin{proof}
By scaling, we can assume that $ r=1 $. We also assume that $ 0<\va<1/4 $ since for otherwise, the function $ \PP_{\va/4} $ is what we need. For later use, we define $ \eta_{\va} $ as a truncation function such that $ \eta_{\va}(\rho)=1 $ if $ \rho\geq\va $ and $ \eta_{\va}(\rho)=\rho/\va $ if $ 0\leq\rho<\va $. We also choose $ \{\PP^{(\ell)}\}_{\ell=0}^2\in H^1(\Ss^1,\cN) $ such that $ [\PP^{(\ell)}]_{\cN}=\rH_{\ell} $ for any $ \ell=0,1,2 $, and $ \{\PP^{(\ell)}\}_{\ell=1}^2 $ are minimizers of $ \{\cE(\rH_{\ell})\}_{\ell=1}^2 $. In the proof, for two point $ x,y\in\R^2 $, we use $ \ol{xy} $ to denote the directed line segment from $ x $ to $ y $. Define $ D_0:=B_1^2\backslash(\ol{B_{1/4}^2(x_1)}\cup\ol{B_{1/4}^2(x_2)}) $, where $ x_1=(1/2,0) $ and $ x_2=(-1/2,0) $. For $ \ell_1=0,1,2 $, $ \ell_2=1,2 $, define $ \PP^{(\ell_1\ell_2)}(x):=\PP^{(\ell_1)}(4(x-x_{\ell_2})) $ for $ x\in B_{1/4}^2(x_{\ell_2}) $. Define $ \W_{\va}^{(\ell_1\ell_2)}:B_{1/4}^2(x_{\ell_2})\to\Ss_0 $ such that 
$$
\W_{\va}^{(\ell_1\ell_2)}(x):=\eta_{\va}(|x-x_{\ell_2}|)\PP^{(\ell_1)}\(\f{x-x_{\ell_2}}{|x-x_{\ell_2}|}\).
$$
for $ \ell_1,\ell_2=1,2 $.
This implies $ \W_{\va}^{(\ell_1\ell_2)}|_{B_{1/4}^2(x_{\ell_2})}=\PP^{(\ell_1\ell_2)} $ for any $ \ell_1,\ell_2=1,2 $. Applying Lemma \ref{k2trivial}, we can construct $ \{\W_{\va}^{(0\ell)}\}_{\ell=1}^2\subset H^1(B_{1/4}^2(x_{\ell}),\cN) $ such that 
$ \W_{\va}^{(0\ell)}|_{\pa B_{1/4}^2(x_{\ell})}=\PP^{(0\ell)} $ for $ \ell=1,2 $ and
\be
E_{\va}(\W_{\va}^{(0\ell)},B_{1/4}^2(x_{\ell}))\leq C\|\na_{\pa B_{1/4}^2(x_{\ell})}\PP^{(0\ell)}\|_{L^2(\pa B_{1/4}^2(x_{\ell}))}\leq C.\label{Uvaell}
\ee
For $ \ell_1,\ell_2=1,2 $, and $ x\in B_{1/4}^2(x_{\ell_2}) $, we have $ x=x_{\ell_2}+(\rho\cos\theta,\rho\sin\theta) $ and
\begin{align*}
|\na\W_{\va}^{(\ell_1\ell_2)}(x)|^2&=|\pa_{\rho}\W_{\va}^{(\ell_1\ell_2)}(x)|^2+\rho^{-2}|\pa_{\theta}\W_{\va}^{(\ell_1\ell_2)}(x)|^2\leq\left\{\begin{aligned}
&C\va^{-2}&\text{ if }&\rho\leq\va,\\
&\rho^{-2}|(\PP^{(\ell_1)})'(\theta)|^2&\text{ if }&\rho>\va.
\end{aligned}\right.
\end{align*}
This, together with $ \W_{\va}^{(\ell_1\ell_2)}(x)\in\cN $ when $ x\in A_{\va,1/4}^2(x_{\ell_2}) $ for $ \ell_1,\ell_2=1,2 $, implies that
\be
\begin{aligned}
E_{\va}(\W_{\va}^{(\ell_1\ell_2)},B_{1/4}^2(x_{\ell_2}))&=E_{\va}(\W_{\va}^{(\ell_1\ell_2)},B_{\va}^2(x_{\ell_2}))+E_{\va}(\W_{\va}^{(\ell_1\ell_2)},A_{\va,1/4}^2(x_{\ell_2}))\\
&\leq\f{1}{2}\int_{\va}^{1/4}\f{\ud\rho}{\rho}\int_0^{2\pi}|(\PP^{(\ell_1)})'(\theta)|^2\ud\theta+\int_{B_{\va}^2(x_{\ell_2})}\f{C}{\va^2}\ud\HH^2\\
&\leq\kappa_*\log\f{1}{\va}+C,
\end{aligned}\label{Wvaell}
\ee
where for the first inequality, we have use the fact that $ \{\W_{\va}^{(\ell_1\ell_2)}\}_{\ell_1,\ell_2=1}^2 $ are bounded and for the second inequality we have used the definition of $ \{\PP^{(\ell)}\}_{\ell=1}^2 $, $ \cE^*(\rH_1)=\cE^*(\rH_2)=\kappa_* $ (see Lemma \ref{kappa123}). 

Choose $ \wh{x}_1=(1/4,0) $, $ \wh{x}_2=(-1/4,0) $, $ \wh{x}_3=(0,1) $, $ \al_1=\pa B_{1/4}^2(x_1) $, $ \al_2=\pa B_{1/4}^2(x_2) $, and $ \al_3=\pa B_1^2 $. Here $ \al_1,\al_3 $ are clockwise, $ \al_2 $ is anticlockwise. Define $ c_i=\ol{0\wh{x}_{\ell}} $ for $ \ell=1,2,3 $. We also choose a point $ \X_0\in\cN\backslash(\cup_{\ell=0}^2\PP^{(\ell)}(\Ss^1)) $. Define $ \{y_{\ell}\}_{\ell=1}^9\subset\R^2 $ such that $ y_{\ell}=(\cos(2\ell\pi/9),\sin(2\ell\pi/9)) $ and $ D_1=\{\sum_{\ell=1}^9t_{\ell}y_{\ell}:0<t_{\ell}<1\} $ is the interior of the convex hull of $ \{y_{\ell}\}_{\ell=1}^9 $. Also, for convenience, we have the convention $ y_{\ell}=y_{\ell_0} $ for $ \ell_0\equiv\ell $ mod $ 9 $, with $ \ell_0=1,2,...,9 $ and $ \ell\in\Z_{\geq 10}\cup\Z_{\leq 0} $. Observe that there exists $ \Phi:D_1\to D_0\backslash(\cup_{\ell=1}^3c_i) $, which is bilipschitz by identify $ \ol{y_1y_9} $ with $ \ol{y_2y_3} $, $ \ol{y_4y_3} $ with $ \ol{y_5y_6} $, and $ \ol{y_6y_7} $ with $ \ol{y_9y_8} $. Precisely, $ \Phi $ can be extended to $ \Phi_*:\ol{D_1}\to\ol{D_0} $ such that $ \Phi_*|_{D_1}=\Phi|_{D_1} $, $ \Phi(\ol{y_9y_1})=c_1 $, $ \Phi(\ol{y_1y_2})=\al_1 $, $ \Phi(\ol{y_2y_3})=\wt{c}_1 $, $ \Phi(\ol{y_3y_4})=c_2 $, $ \Phi(\ol{y_4y_5})=\al_2 $, $ \Phi(\ol{y_5y_6})=\wt{c}_2 $, $ \Phi(\ol{y_6y_7})=c_3 $, $ \Phi(\ol{y_7y_8})=\al_3 $, and $ \Phi(\ol{y_8y_9})=\wt{c}_3 $. We intend to assign the boundary datum on $ \pa D_1 $, denoted by $ \wh{\PP}_b $. Firstly, let $ \wh{\PP}_b=\PP_b $ on $ \ol{y_7y_8} $. Next, we set 
\begin{align*}
\wh{\PP}_b:=\left\{\begin{aligned}
\PP^{(11)}\text{ on }\ol{y_1y_2}\text{ and }\PP^{(02)}\text{ on }\ol{y_4y_5}&\text{ if }[\PP_b]_{\cN}=\rH_1,\\
\PP^{(01)}\text{ on }\ol{y_1y_2}\text{ and }\PP^{(22)}\text{ on }\ol{y_4y_5}&\text{ if }[\PP_b]_{\cN}=\rH_2,\\
\PP^{(11)}\text{ on }\ol{y_1y_2}\text{ and }\PP^{(22)}\text{ on }\ol{y_4y_5}&\text{ if }[\PP_b]_{\cN}=\rH_3,\\
\PP^{(11)}\text{ on }\ol{y_1y_2}\text{ and }\PP^{(12)}\text{ on }\ol{y_4y_5}&\text{ if }[\PP_b]_{\cN}=\rH_4.
\end{aligned}\right.
\end{align*}
Since $ \cN $ is path-connected, we find $ \CC_{\ell}\in C^{\infty}(\ol{y_{3\ell-3}y_{3\ell-2}},\cN) $ for $ \ell=1,2,3 $, such that $ \CC_1(y_9)=\X_0 $, $ \CC_1(y_1)=\wh{\PP}_b(y_1) $, $ \CC_2(y_3)=\X_0 $, $ \CC_2(y_4)=\wh{\PP}_b(y_4) $, $ \CC_3(y_6)=\X_0 $ and $ \CC_3(y_7)=\wh{\PP}_b(y_7) $. We can set $ \wh{\PP}_b=\CC_{\ell} $ on $ \ol{y_{3\ell-3}y_{3\ell-2}} $ and $ \wh{\PP}_b=\wt{\CC}_{\ell} $ on $ \ol{y_{3\ell-1}y_{3\ell}} $ for any $ \ell=1,2,3 $. On $ \pa D_1 $, if $ [\PP_b]_{\cN}=\rH_1 $, then $ \wh{\PP}_b=\LL_1*\LL_2*\LL_3 $ for $ \LL_{\ell}=\CC_{\ell}*(\wh{\PP}_b|_{\ol{y_{3\ell-2}y_{3\ell-1}}})*\wt{\CC}_{\ell} $ with $ \ell=1,2,3 $. By Corollary \ref{homcor}, we have 
\begin{align*}
&\vp_{\X_0}^{-1}([\LL_1]_{\cN,\X_0}),\,\,\vp_{\X_0}^{-1}([\LL_3]_{\cN,\X_0})\in\{\pm \ii\},\,\,\vp_{\X_0}^{-1}([\LL_2]_{\cN,\X_0})=1\text{ if }[\PP_b]_{\cN}=\rH_1,\\
&\vp_{\X_0}^{-1}([\LL_2]_{\cN,\X_0}),\,\,\vp_{\X_0}^{-1}([\LL_3]_{\cN,\X_0})\in\{\pm \jj\},\,\,\vp_{\X_0}^{-1}([\LL_1]_{\cN,\X_0})=1\text{ if }[\PP_b]_{\cN}=\rH_2,\\
&\vp_{\X_0}^{-1}([\LL_1]_{\cN,\X_0})\in\{\pm \ii\},\,\,\vp_{\X_0}^{-1}([\LL_2]_{\cN,\X_0})\in\{\pm \jj\},\vp_{\X_0}^{-1}([\LL_2]_{\cN,\X_0})\in\{\pm\kk\}\text{ if }[\PP_b]_{\cN}=\rH_3,\\
&\vp_{\X_0}^{-1}([\LL_1]_{\cN,\X_0}),\,\,\vp_{\X_0}^{-1}([\LL_2]_{\cN,\X_0})\in\{\pm \ii\},\,\,\vp_{\X_0}^{-1}([\LL_3]_{\cN,\X_0})=-1\text{ if }[\PP_b]_{\cN}=\rH_4,
\end{align*}
By substituting $ \{\wh{\PP}_b|_{\ol{y_{3\ell-2}y_{3\ell-1}}}\}_{\ell=1}^3 $ with the reverse of it if necessary, we can obtain $ \vp_{\X_0}^{-1}([\wh{\PP}_b]_{\cN,\X_0})=1 $. As a result, we can apply Lemma \ref{k2trivial} to find $ \V_{\va} $ such that $ \V_{\va}|_{\pa D_1}=\wh{\PP}_b $ and
\begin{align*}
\int_{D_1}|\na\V_{\va}|^2\ud\HH^2\leq C\|\na_{\pa D_1}\wh{\PP}_b\|_{L^2(\pa D_1)}^2\leq C.
\end{align*}
We note that by the choice of boundary value above, $ \V_{\va}(\Phi_*^{-1}(x)) $ is well defined for $ x\in D_0 $. If $ [\PP_b]_{\cN}=\rH_1 $, let
$$
\PP_{\va}(x):=\left\{\begin{aligned}
&\W_{\va}^{(11)}(x)&\text{ if }&x\in B_{1/4}^2(x_1),\\
&\W_{\va}^{(02)}(x)&\text{ if }&x\in B_{1/4}^2(x_2),\\
&\V_{\va}(\Phi_*^{-1}(x))&\text{ if }&x\in D_0.
\end{aligned}\right.
$$
If $ [\PP_b]_{\cN}=\rH_2 $, let
$$
\PP_{\va}(x):=\left\{\begin{aligned}
&\W_{\va}^{(01)}(x)&\text{ if }&x\in B_{1/4}^2(x_1),\\
&\W_{\va}^{(22)}(x)&\text{ if }&x\in B_{1/4}^2(x_2),\\
&\V_{\va}(\Phi_*^{-1}(x))&\text{ if }&x\in D_0.
\end{aligned}\right.
$$
If $ [\PP_b]_{\cN}=\rH_3 $, let
$$
\PP_{\va}(x):=\left\{\begin{aligned}
&\W_{\va}^{(11)}(x)&\text{ if }&x\in B_{1/4}^2(x_1),\\
&\W_{\va}^{(22)}(x)&\text{ if }&x\in B_{1/4}^2(x_2),\\
&\V_{\va}(\Phi_*^{-1}(x))&\text{ if }&x\in D_0.
\end{aligned}\right.
$$
If $ [\PP_b]_{\cN}=\rH_3 $, let
$$
\PP_{\va}(x):=\left\{\begin{aligned}
&\W_{\va}^{(11)}(x)&\text{ if }&x\in B_{1/4}^2(x_1),\\
&\W_{\va}^{(12)}(x)&\text{ if }&x\in B_{1/4}^2(x_2),\\
&\V_{\va}(\Phi_*^{-1}(x))&\text{ if }&x\in D_0.
\end{aligned}\right.
$$
Now, we have $ \PP_{\va}\in H^1(B_1^2,\Ss_0) $ and $ \PP|_{\pa B_1^2}=\PP_b $. By using \eqref{Uvaell} and \eqref{Wvaell}, we get
$$
E_{\va}(\PP_{\va},B_1^2)\leq \cE^*([\PP_b]_{\cN})\log\f{1}{\va}+C(\|\na_{\pa B_1^2}\PP_b\|_{L^2(\pa B_1^2)}+1),
$$
which completes the proof.
\end{proof}

Next, we prove an extension result on a cylinder, in dimension three. Given positive numbers $ L $ and $ r $, set $ \Lda_{r,L}=B_r^2\times(-L,L) $ and $ \Ga_{r,L}=\pa B_r^2\times(-L,L) $. Let $ \PP_b\in H^1(\Ga_{r,L},\cN) $ be a boundary datum, which is only defined on the lateral surface of the cylinder. 

\begin{lem}\label{cylinderex}
There exists a constant $ C>0 $ depending only on $ \cA $ such that the following properties hold. For any $ 0<\va<r $ and any $ \PP_b\in H^1(\Ga_{r,L},\cN) $ with non-trivial free homotopy class, there exists $ \PP_{\va}\in H^1(\Lda_{r,L},\Ss_0) $ such that $ \PP_{\va}=\PP_b $ on $ \Ga_{r,L} $. Moreover, it satisfies
$$
E_{\va}(\PP_{\va},\Lda_{r,L})\leq CL\(\f{L}{r}+\f{r}{L}\)\|\na_{\Ga_{r,L}}\PP_b\|_{L^2(\Ga)}^2+2\cE^*([\PP_b|_{\Ga_{r,L}}]_{\cN})L\log\f{r}{\va}+CL,
$$
and for $ z\in\{-L,L\} $,
$$
E_{\va}(\PP_{\va},B_r^2\times\{z\})\leq C\(\f{L}{r}+\f{r}{L}\)\|\na_{\Ga_{r,L}}\PP_b\|_{L^2(\Ga)}^2+\cE^*([\PP_b|_{\Ga_{r,L}}]_{\cN})\log\f{r}{\va}+C.
$$
\end{lem}
\begin{rem}
Here, for a.e. $ z\in[-L,L] $, $ \PP_b|_{\pa B_r^2\times\{z\}} $ has the same well defined free homotopy classes and denoted by $ [\PP_b|_{\Ga_{r,L}}]_{\cN} $. This is because the lateral surface of a cylinder is diffeomorphic to the annulus and we can apply Lemma \ref{propHomt} to obtain the result. 
\end{rem}

\begin{proof}
By using average arguments, there exists $ z_0\in[-L/4,L/4] $ such that
\be
\|\na_{\Ga_{r,L}}\PP_b\|_{L^2(\pa B_r^2\times\{z_0\})}^2\leq\f{8}{L}\|\na_{\Ga_{r,L}}\PP_b\|_{L^2(\Ga_{r,L})}\label{PPL8}
\ee
and $ [\PP_b]_{\cN}=[\PP_b(\cdot,z_0)]_{\cN} $. By a translation, we set $ z_0=0 $. Define $ \wh{\PP}:(\ol{B_r^2}\backslash B_{r/2}^2)\times[-L,L]\to\cN $ by
$$
\wh{\PP}(\rho,\theta,z)=\left\{\begin{aligned}
&\PP_b(r,\theta,z'(\rho,z))&\text{ if }&\rho_0(r,z)\leq\rho\leq r\text{ and }|z|\leq L,\\
&\PP_b(r,\theta,0)&\text{ if }&r/2\leq\rho<\rho_0(r,z)\text{ and }|z|\leq L,
\end{aligned}\right.
$$
where
$$
z'(\rho,z):=\f{2L}{r}\sgn(z)(\rho-r)+z,\,\,\rho_0(r,z):=r-\f{r}{2L}|z|.
$$
Here $ (\rho,\theta,z)\in[0,r]\times[0,2\pi]\times[-L,L] $ is the cylindrical coordinates. By simple calculations, we have
\begin{align*}
\|\na\wh{\PP}\|_{L^2(A_{r/2,r}^2\times[0,L])}&\leq r\(\f{4L^2}{r^2}+1\)\int_0^{2\pi} \int_0^L(r|\pa_z\PP_b|^2+r^{-1}|\pa_{\theta}\PP_b|^2)(r,\theta,\xi)\ud\xi\ud\theta\\
&\quad\quad+(\log 2)L\int_0^{2\pi}|\pa_{\theta}\PP_b|^2(r,\theta,0)\ud\theta.
\end{align*}
Combined with the analogous estimates for $ z\in[-L,0] $, we deduce
$$
\|\na\wh{\PP}\|_{L^2(A_{r/2,r}^2\times[-L,L])}^2\leq\(\f{4L^2}{r}+r\)\|\na_{\Ga_{r,L}}\PP_b\|_{L^2(\Ga_{r,L})}^2+(\log 2)rL\|\na_{\Ga_{r,L}}\PP_b\|_{L^2(\pa B_r^2\times\{0\})}^2.
$$
By \eqref{PPL8}, there holds
$$
\|\na\wh{\PP}\|_{L^2(A_{r/2,r}^2\times[-L,L])}^2\leq C\(\f{L^2}{r}+r\)\|\na_{\Ga_{r,L}}\PP_b\|_{L^2(\Ga_{r,L})}^2.
$$
Applying Lemma \ref{k2non-trivial} to $ \PP_b(\cdot,0) $, we can get $ \wh{\PP}_{\va}\in H^1(B_{r/2}^2,\Ss_0) $ such that
\begin{align*}
E_{\va}(\wh{\PP}_{\va},B_{r/2}^2)\leq\f{Cr}{L}\|\na_{\Ga_{r,L}}\PP_b\|_{L^2(\Ga_{r,L})}^2+\cE_*([\PP_b|_{\Ga_{r,L}}]_{\cN})\log\f{r}{\va}+C.
\end{align*}
Define
$$
\PP_{\va}(\rho,\theta,z)=\left\{\begin{aligned}
&\wh{\PP}(\rho,\theta,z)&\text{ if }&r/2<\rho\leq r,\\
&\wh{\PP}_{\va}(\rho,\theta)&\text{ if }&0<\rho<r/2.
\end{aligned}\right.
$$
Here $ \PP_{\va} $ is what we want.
\end{proof}
\subsection{Luckhaus type lemma}

In this subsection, we present the Luckhaus-type lemmas for our Landau de-Gennes model.

\begin{lem}[\cite{L88}, Lemma 1]\label{Luckhaus1}
For all $ \beta\in(1/2,1) $, there exists $ C>0 $ depending only on $ \beta $ such that the following properties hold. For any $ 0<\lda\leq 1/2 $, $ 0<\sg<1 $, and $ \U,\V\in H^1(\pa B_1,\cN) $ with
$$
K=\int_{\pa B_1}\(|\na_{\pa B_1}\U|^2+|\na_{\pa B_1}\V|^2+\f{|\U-\V|^2}{\sg^2}\)\ud\HH^2,
$$
there exists $ \W\in H^1(B_1\backslash B_{1-\lda},\Ss_0) $, satisfying
\begin{gather*}
\W(x)=\U(x)\text{ and }\W((1-\lda)x)=\V(x)\text{ for }\HH^2\text{-a.e. }x\in\pa B_1,\\
\dist(\W(x),\cN)\leq C\sg^{1-\beta}\lda^{-1/2}K^{1/2}\text{ a.e. in }B_1\backslash B_{1-\lda},
\end{gather*}
and
$$
\int_{B_1\backslash B_{1-\lda}}|\na\W|^2\ud x\leq C\lda(1+\sg^2\lda^{-2})K.
$$
\end{lem}

In the rest of this paper, we define $ h(\ol{\va}):=\ol{\va}^{1/2}\log(1/\ol{\va}) $ with $ 0<\ol{\va}<1 $.

\begin{prop}\label{Luckhaus2}
Assume $ \{\sg_{\ol{\va}}\}_{0<\ol{\va}<1} $ satisfies $ \sg_{\ol{\va}}>0 $ and $ \lim_{\ol{\va}\to 0^+}\sg_{\ol{\va}}=0 $. Let $ \U_{\ol{\va}},\V_{\ol{\va}}\in H^1(\pa B_1,\Ss_0) $. Assume that for any $ 0<\ol{\va}<1 $,
\begin{gather}
\|\U_{\ol{\va}}\|_{L^{\ift}(\pa B_1)}\leq C_0,\label{Uolvabound}\\
\V_{\ol{\va}}\in\cN,\,\,\HH^2\text{-a.e. }x\in\pa B_1,\nn\\
\int_{\pa B_1}\(|\na_{\pa B_1}\U_{\ol{\va}}|^2+\f{1}{\ol{\va}^2}f_b(\U_{\ol{\va}})+|\na_{\pa B_1}\V_{\ol{\va}}|^2+\f{|\U_{\ol{\va}}-\V_{\ol{\va}}|^2}{\sg_{\ol{\va}}^2}\)\ud\HH^2\leq C_1.\label{BoundC1Uva}
\end{gather}
for some $ C_0,C_1>0 $. Set 
$$ 
\nu_{\ol{\va}}=h(\ol{\va})+(h^{1/2}(\ol{\va})+\sg_{\ol{\va}})^{1/4}(1-h(\ol{\va})).
$$
There exist a constant $ 0<\ol{\va}_1<1 $ depending only on $ \cA,C_0,C_1 $, and for $ 0<\ol{\va}<\ol{\va}_1 $, a function $ \W_{\ol{\va}}\in H^1(B_1\backslash B_{1-\nu_{\ol{\va}}},\Ss_0) $ such that
\begin{gather}
\W_{\ol{\va}}(x)=\U_{\ol{\va}}(x)\text{ and }\W_{\ol{\va}}((1-h(\ol{\va}))x)=\V_{\ol{\va}}(x)\text{ for }\HH^2\text{-a.e. }x\in\pa B_1\\
E_{\ol{\va}}(\W_{\ol{\va}},B_1\backslash B_{1-\nu_{\ol{\va}}})\leq C\nu_{\ol{\va}},
\end{gather}
where $ C>0 $ depends only on $ \cA,C_0 $, and $ C_1 $.
\end{prop}

\begin{lem}\label{Luckhauslemma}
For all $ C_0>0 $, there exist $ 0<\eta_1,\ol{\va}_{1}<1,C>0 $ depending only on $ \cA $ and $ C_0 $ with the following property. For any $ 0<\eta<\eta_1 $, $ 0<\ol{\va}<\ol{\va}_1 $ and $ \U_{\ol{\va}}\in(H^1\cap L^{\ift})(\pa B_1,\Ss_0) $ satisfying
\begin{align}
 E_{\ol{\va}}(\U_{\ol{\va}},\pa B_1)&\leq\eta\log\f{1}{\ol{\va}},\label{log1}\\
\|\U_{\ol{\va}}\|_{L^{\ift}(\pa B_1)}&\leq C_0,\label{Boun} 
\end{align}
there exist $ \V_{\ol{\va}}\in H^1(\pa B_1,\cN) $ and $ \W_{\ol{\va}}\in H^1(B_1\backslash B_{1-h(\ol{\va})},\Ss_0) $ such that
\begin{gather}
\W_{\ol{\va}}(x)=\U_{\ol{\va}}(x)\text{ and }\W_{\ol{\va}}((1-h(\ol{\va}))x)=\V_{\ol{\va}}(x)\text{ for }\HH^2\text{-a.e. }x\in\pa B_1\label{interpoL}\\
\f{1}{2}\int_{\pa B_1}|\na_{\pa B_1}\V_{\ol{\va}}|^2\ud\HH^2\leq C E_{\ol{\va}}(\U_{\ol{\va}},\pa B_1),\label{VboundL}\\
E_{\ol{\va}}(\W_{\ol{\va}},B_1\backslash B_{1-h(\ol{\va})})\leq Ch(\ol{\va}) E_{\ol{\va}}.(\U_{\ol{\va}},\pa B_1)\label{WboundL}\\
\|\U_{\ol{\va}}-\V_{\ol{\va}}\|_{L^2(\pa B_1)}\leq Ch^{1/2}(\ol{\va}) E_{\ol{\va}}^{1/2}(\U_{\ol{\va}},\pa B_1).\label{UolvaVolvami}
\end{gather}
\end{lem}

\begin{proof}[Proof of Proposition \ref{Luckhaus2}]
Let $ 0<\eta_1<1 $ be given by Lemma \ref{Luckhauslemma} and choose $ 0<\ol{\va}_1<1 $ such that $ \eta_1\log(1/\ol{\va}_1)\geq C_1 $. By \eqref{BoundC1Uva}, we have $ E_{\ol{\va}}(\U_{\ol{\va}},\pa B_1)\leq C_1\leq\eta_1\log(1/\ol{\va}) $ for any $ \ol{\va}\in(0,\ol{\va}_1) $. In view of \eqref{Uolvabound}, we can take smaller $ 0<\ol{\va}_1<1 $ and apply Lemma \ref{Luckhauslemma} to get $ \U_{\ol{\va},1}\in H^1(\pa B_1,\cN) $ and $ \W_{\ol{\va},1}\in H^1(B_1\backslash B_{1-h(\ol{\va})},\Ss_0) $ such that
\begin{gather*}
\W_{\ol{\va},1}=\U_{\ol{\va}}(x)\text{ and }\W_{\ol{\va},1}(x-h(\ol{\va})x)=\U_{\ol{\va}}(x)\text{ for }\HH^2\text{-a.e. }x\in\pa B_1,\\
\int_{\pa B_1}|\na_{\pa B_1}\U_{\ol{\va},1}|^2\ud\HH^2\leq C E_{\ol{\va}}(\U_{\ol{\va}},\pa B_1)\leq C,\\
 E_{\ol{\va}}(\W_{\ol{\va},1},B_1\backslash B_{1-h(\ol{\va})})\leq Ch(\ol{\va}) E_{\ol{\va}}(\U_{\ol{\va}},\pa B_1)\leq Ch(\ol{\va}).
\end{gather*}
Moreover, by \eqref{UolvaVolvami}, we have $ \|\U_{\ol{\va},1}-\V_{\ol{\va}}\|_{L^2(\pa B_1)}\leq C\sg_{\ol{\va},1} $ with $ \sg_{\ol{\va},1}:=h^{1/2}(\ol{\va})+\sg(\ol{\va}) $. We get
$$
\int_{\pa B_1}\(|\na_{\pa B_1}\U_{\ol{\va},1}|^2+|\na_{\pa B_1} \V_{\ol{\va}}|^2+\f{|\U_{\ol{\va},1}-\V_{\ol{\va}}|^2}{\sg_{\ol{\va},1}^2}\)\ud\HH^2\leq C.
$$
Apply Lemma \ref{Luckhaus1} to $ \beta=3/4 $, $ \sg=\sg_{\ol{\va},1} $ and $ \lda=\sg_{\ol{\va},1}^{1/4} $. By scaling, there exists a map $ \W_{\ol{\va},2}\in H^1(B_{1-h(\ol{\va})}\backslash B_{1-\nu_{\ol{\va}}},\Ss_0) $ such that
$$
\int_{B_{1-h(\ol{\va})}\backslash B_{1-\nu_{\ol{\va}}}}|\na\W_{\ol{\va},2}|^2\ud x\leq C\sg_{\ol{\va},1}^{1/4}(1-h(\ol{\va})),\quad
\sup_{B_{1-h(\ol{\va})}\backslash B_{1-\nu_{\ol{\va}}}}\dist(\W_{\ol{\va},2},\cN)\leq C\sg_{\ol{\va},1}^{1/8}
$$
Since $ \sg_{\ol{\va},1}\to 0^+ $ when $ \ol{\va}\to 0^+ $, we can choose sufficiently small $ 0<\ol{\va}_1<1 $ such that $ \varrho\circ\W_{\va,2} $ is well defined. Now we define
$$
\W_{\ol{\va}}:=\left\{\begin{aligned}
&\W_{\ol{\va},1}&\text{ if }&x\in B_1\backslash B_{1-h(\ol{\va})},\\
&\varrho\circ\W_{\ol{\va},2}&\text{ if }&x\in B_{1-h(\ol{\va})}\backslash B_{1-\nu_{\ol{\va}}}.
\end{aligned}\right.
$$
Obviously, we have $ \W_{\ol{\va}}\in H^1(B_1\backslash B_{1-\nu_{\ol{\va}}},\Ss_0) $, $
\W_{\ol{\va}}(x)=\U_{\ol{\va}}(x) $ and $ \W_{\ol{\va}}(x-h(\ol{\va})x)=\V_{\ol{\va}}(x) $ for $ \HH^2 $-a.e. $ x\in\pa B_1 $. By simple calculations, there holds $
E_{\ol{\va}}(\W_{\ol{\va}},B_1\backslash B_{1-\nu_{\ol{\va}}})\leq C\nu_{\ol{\va}} $, which completes the proof.
\end{proof}

\begin{proof}[Proof of Lemma \ref{Luckhauslemma}]
We will divide the proof into several steps. Here the proof is given in sketch and we will give more details in the proof of Lemma \ref{Luckhauscylinder1}, which is another Luckhaus-type lemma. \smallskip

\noindent\textbf{\underline{Step 1.} Constructions of good family of grids of size $ h(\ol{\va})=\ol{\va}^{1/2}\log(1/\ol{\va}) $ on $ \pa B_1 $.} We claim that there exists $ \ol{\va}_1>0 $ and a family of good grids $ \mathcal{G}=\{\mathcal{G}^{\ol{\va}}\}_{0<\ol{\va}<\ol{\va}_1} $ with respect to $ \{\U_{\ol{\va}}\}_{0<\ol{\va}<\ol{\va}_1} $ satisfying \eqref{log1} and \eqref{Boun} with $ 0<\eta\leq 1 $ and $ C_0>0 $. Precisely speaking, there exists $ C_{\mathcal{G}} $, being an absolute constant such that
\be
\mathcal{G}^{\ol{\va}}=\{K_{i,j}^{\ol{\va}}\}_{1\leq i\leq k_j^{\ol{\va}},0\leq j\leq 2},\quad\pa B_1=\bigcup_{j=0}^2\bigcup_{i=1}^{k_j^{\ol{\va}}}K_{i,j}^{\ol{\va}}\label{G0}
\ee
and the following properties hold.
\begin{enumerate}
\item For $ 0<\ol{\va}<\ol{\va}_1 $, $ 1\leq i\leq k_j^{\ol{\va}} $ and $ j=0,1,2 $, $ K_{i,j}^{\ol{\va}} $ are mutually disjoint and there exist bilipschitz homeomorphism $ \vp_{i,j}^{\ol{\va}}:K_{i,j}^{\ol{\va}}\to B_{h(\ol{\va})}^j $ such that
\be
\|\na_{K_{i,j}^{\ol{\va}}}\vp_{i,j}^{\ol{\va}}\|_{L^{\ift}(K_{i,j}^{\ol{\va}})}+\|\na(\vp_{i,j}^{\ol{\va}})^{-1}\|_{L^{\ift}(B_{h(\ol{\va})}^j)}\leq C_{\mathcal{G}}.\label{G1}
\ee
\item For any $ p\in\{1,2,...,k_1^{\ol{\va}}\} $,
\be
\#\{q\in\left\{1,2,...,k_2^{\ol{\va}}\}:K_{p,1}^{\ol{\va}}\subset K_{q,2}^{\ol{\va}}\right\}\leq C_{\mathcal{G}}.\label{G2}
\ee
\item Define $ R_j^{\ol{\va}} $ be the $ j $-skeleton of $ \mathcal{G}^{\ol{\va}} $, i.e., $
R_j^{\ol{\va}}=\cup_{i=1}^{k_j^{\ol{\va}}}K_{i,j}^{\ol{\va}} $, $ j=0,1,2. $
There holds
\be
E_{\ol{\va}}(\U_{\ol{\va}},R_1^{\ol{\va}})\leq C_{\mathcal{G}}h^{-1}(\ol{\va}) E_{\ol{\va}}(\U_{\ol{\va}},\pa B_1).\label{G3}
\ee
\item There holds
\be
\int_{R_1^{\ol{\va}}}f_b(\U_{\ol{\va}})\ud \HH^1 \leq C_{\mathcal{G}}h^{-1}(\ol{\va})\int_{\pa B_1}f_b(\U_{\ol{\va}})\ud\HH^2.\label{G4}
\ee    
\end{enumerate}
To show the existence, we consider the grid (satisfying \eqref{G0} and \eqref{G1}) of $ \pa[0,1]^3 $ i.e, spanned by the points $
(\lceil h^{-1}(\ol{\va})\rceil^{-1} \Z^3)\cap \pa[0,1]^3 $. ($ \lceil x\rceil=\min\{k\in\Z:k\geq x\} $). Up to a bilipschitz homeomorphism, it induces a grid $ \FF^{\ol{\va}} $ on $ \pa B_1 $ such that \eqref{G1} and \eqref{G2} are true. Denote by $ T_1^{\ol{\va}} $ the $ 1 $-skeleton of $ F^{\ol{\va}} $. By average arguments, there exists $ \w\in\mathrm{SO}(3) $ such that $ \mathcal{G}^{\ol{\va}}=\{\w(K):K\in F^{\ol{\va}}\} $ satisfies \eqref{G3} and \eqref{G4}. This is exactly a good family of grids of size $ h(\ol{\va}) $ with respect to $ \U_{\ol{\va}} $. \smallskip

\noindent\textbf{\underline{Step 2.} Proving that $ \U_{\ol{\va}} $ is sufficiently close to $ \cN $ on $ R_1^{\ol{\va}} $ as $ \ol{\va}\to 0^+ $.} In view of \eqref{fBnond}, there exists $ c_1,c_2,\delta_0>0 $ and a continuous function $ \psi:[0,+\ift)\to\R $ such that
\be
\left\{\begin{aligned}
&\psi(s)=c_1s^2&\text{ for }&0\leq s<\delta_0,\\
&0<\psi(s)\leq C&\text{ for }&s\geq\delta_0,\\
&\psi(\dist(\PP,\cN))\leq f_b(\PP)&\text{ for }&\text{any }\PP\in\Ss_0.
\end{aligned}\right.\label{psifunction1}
\ee
Define $ F(s):=\int_0^s\psi^{1/6}(t)\ud t $ and $ d_{\ol{\va}}:=\dist(\U_{\ol{\va}},\cN) $. We have $ |\na_{R_1^{\ol{\va}}}d_{\ol{\va}}|\leq|\na_{R_1^{\ol{\va}}}\U_{\ol{\va}}| $ and then $ d_{\ol{\va}}\in H^1(B_1,\R_+) $. In view of \eqref{Boun}, \eqref{G3}, and the Young inequality, we get
\be
C\log\f{1}{\ol{\va}}\geq\ol{\va}^{-1/2}h(\ol{\va})\int_{R_1^{\ol{\va}}}|\na_{R_1^{\ol{\va}}}F(d_{\ol{\va}})|^{3/2}\ud\HH^1.\label{logvabou}
\ee
For $ 1 $-cell $ K_{i,1}^{\ol{\va}}\in \mathcal{G}^{\ol{\va}} $ with $ 1\leq i\leq k_1^{\ol{\va}} $, by using Sobolev embedding theorem, $
(\osc_{K_{i,1}^{\ol{\va}}}F(d_{\ol{\va}}))^{3/2}\leq C\ol{\va}^{1/4}\log(1/\ol{\va}) $. By this, we have $ \osc_{K_{i,1}^{\ol{\va}}}F(d_{\ol{\va}})\to 0^+ $ when $ \ol{\va}\to 0^+ $ and then 
\be
\lim_{\ol{\va}\to 0}\osc_{R_1^{\ol{\va}}}d_{\ol{\va}}=0.\label{R1va0}
\ee
On the other hand, by \eqref{Boun} and \eqref{G3}, we can obtain that for any $ 1 $-cell $ K_{i,1}^{\ol{\va}}\in \mathcal{G}^{\ol{\va}} $, with $ 1\leq i\leq k_1^{\ol{\va}} $,
$$
\lim_{\ol{\va}\to 0}\dashint_{K_{i,1}^{\ol{\va}}}\psi(d_{\ol{\va}})\ud\HH^1\leq\limsup_{\ol{\va}\to 0}\f{1}{h(\ol{\va})}\int_{R_1^{\ol{\va}}}f_b(U_{\ol{\va}})\ud\HH^1=0.
$$
This, together with the construction of $ \psi $ and \eqref{Boun}, implies that
$$
\lim_{\ol{\va}\to 0}\sup_{K_{i,1}^{\ol{\va}}\subset R_1^{\ol{\va}},\,\,1\leq i\leq k_1^{\ol{\va}}}\dashint_{K_{i,1}^{\ol{\va}}}d_{\ol{\va}}\ud\HH^1=0.
$$
Combining \eqref{R1va0}, we have
\be
\lim_{\ol{\va}\to 0}\sup_{x\in R_1^{\ol{\va}}}\dist(\U_{\ol{\va}},\cN)=0.\label{R1va00}
\ee
\smallskip

\noindent\textbf{\underline{Step 3.} Constructions of $ \V_{\ol{\va}} $ and $ \W_{\ol{\va}} $.} By \eqref{R1va00}, there exists $ 0<\ol{\va}_1<1 $ such that
\be
\sup_{0<\ol{\va}<\ol{\va}_1}\sup_{x\in R_1^{\ol{\va}}}\dist(\U_{\ol{\va}},\cN)<\delta_0,\label{delta0Uva}
\ee
where $ \delta_0 $ is given by Corollary \ref{fBc}. We can define $ \V_{\ol{\va}}=\varrho(\U_{\ol{\va}}) $ on $ R_1^{\ol{\va}} $. Obviously, $ \V_{\ol{\va}}\in H^1(R_1^{\ol{\va}},\Ss_0) $ and satisfies
\be
\sup_{x\in R_1^{\ol{\va}}}|\U_{\ol{\va}}(x)-\V_{\ol{\va}}(x)|\leq\delta_0.\label{R1vauv}
\ee
Next, we claim that if $ 0<\eta<1 $ in \eqref{log1} is sufficiently small, then $ \V_{\ol{\va}}|_{\pa K_{i,2}^{\ol{\va}}}:\pa K_{i,2}^{\ol{\va}}\to\cN $ is homotopically trivial for any $ K_{i,2}^{\ol{\va}} $ being a $ 2 $-cell of $ \mathcal{G}^{\ol{\va}} $ with $ 1\leq i\leq k_2^{\ol{\va}} $. Indeed, if such claim is not true, there exists $ K_{i_0,2}^{\ol{\va}}\in R_2^{\ol{\va}} $ such that $ [\varrho\circ\U|_{\pa K_{i_0,2}^{\ol{\va}}}]_{\cN}\neq\rH_0 $. Meanwhile, we have a bilipschitz homeomorphism $ \vp^{\ol{\va}}:K_{i_0,2}^{\ol{\va}}\to B_{h(\ol{\va})}^2 $ satisfying \eqref{G1}. Without loss of generality, we assume $ K_{i_0,2}^{\ol{\va}}=B_{h(\ol{\va})}^2 $. Since in \eqref{delta0Uva}, $ \delta_0 $ is sufficiently small, $ \U_{\ol{\va}}\notin\cC_1\cup\cC_2 $ on $ \pa B_{h(\ol{\va})}^2 $ and $
\phi_0(\U_{\ol{\va}},\pa B_{h(\ol{\va})}^2)>1/2 $ for $ 0<\ol{\va}\in(0,\ol{\va}_1) $. Furthermore, we can assume that $ \ol{\va}<h(\ol{\va})/80 $. It follows from Corollary \ref{lowerboundcor} that
$$
 E_{\ol{\va}}(\U_{\ol{\va}},B_{h(\ol{\va})}^2)+h(\ol{\va}) E_{\ol{\va}}(\U_{\ol{\va}},\pa B_{h(\ol{\va})}^2)\geq C_1\log\f{h(\ol{\va})}{\ol{\va}}-C_2.
$$
This, together with \eqref{G3}, implies that
$$
E_{\ol{\va}}(\U_{\ol{\va}},\pa B_1)\geq C( E_{\ol{\va}}(\U_{\ol{\va}},B_{h(\ol{\va})})+h(\ol{\va}) E_{\ol{\va}}(\U_{\ol{\va}},\pa B_{h(\ol{\va})}^2))\geq C\log\f{1}{\ol{\va}}-C'.
$$
If $ 0<\eta<1 $ is sufficiently small, i.e., there exists $ \eta_1>0 $ such that if $ 0<\eta<\eta_1 $ and \eqref{log1} is true, there is a contradiction. This completes the proof of this claim.

By this claim, for any $ K_{i,2}^{\ol{\va}}\subset R_2^{\ol{\va}} $ with $ 1\leq i\leq k_2^{\ol{\va}} $, we can use Lemma \ref{k2trivial} and \eqref{G1} to find $ \V_{\ol{\va},K_{i,2}^{\ol{\va}}}\in H^1(K_{i,2}^{\ol{\va}},\cN) $ such that $ \V_{\ol{\va},K_{i,2}^{\ol{\va}}}=\V_{\ol{\va}} $ on $ \pa K_{i,2}^{\ol{\va}} $ and
$$
\int_{K_{i,2}^{\ol{\va}}}|\na_{K_{i,2}^{\ol{\va}}}\V_{\ol{\va},K_{i,2}^{\ol{\va}}}|^2\ud\HH^2\leq Ch(\ol{\va})\int_{\pa K_{i,2}^{\ol{\va}}}|\na_{\pa K_{i,2}^{\ol{\va}}}\V_{\ol{\va}}|^2\ud\HH^1.
$$
Define $ \V_{\ol{\va}}:\pa B_1\to\cN $ by
$$ 
\V_{\ol{\va}}(x)=\left\{\begin{aligned}
&(\varrho\circ\U_{\ol{\va}})(x)&\text{ if }&x\in R_1^{\ol{\va}},\\
&\V_{\ol{\va},K_{i,2}^{\ol{\va}}}(x)&\text{ if }&x\in K_{i,2}^{\ol{\va}}\subset R_2^{\ol{\va}}.
\end{aligned}\right.
$$
By \eqref{G2}, \eqref{G3} and the definition of $ \V_{\ol{\va}} $, we have
$$
\int_{\pa B_1}|\na_{\pa B_1}\V_{\ol{\va}}|^2\ud\HH^2\leq C E_{\ol{\va}}(\U_{\va},\pa B_1).
$$
Such $ \V_{\va} $ satisfies the property in Lemma \ref{Luckhauslemma}. Now, let us construct $ \W_{\ol{\va}} $. Set $ A_{\ol{\va}}=B_1\backslash B_{1-h(\ol{\va})} $. The grid $ \mathcal{G}^{\ol{\va}} $ induces a grid $ \wh{\mathcal{G}}^{\ol{\va}} $ on $ A_{\ol{\va}} $, whose cells are
$$
\wh{K}_{i,j+1}^{\ol{\va}}=\left\{x\in\R^3:1-h(\ol{\va})\leq|x|\leq 1,\,\,\f{x}{|x|}\in K_{i,j}^{\ol{\va}}\right\}\text{ for any }K_{i,j}^{\ol{\va}}\in \mathcal{G}^{\ol{\va}}.
$$
Here $ K_{i,j}^{\ol{\va}} $ is a cell of dimension $ j $, and $ \wh{K}_{i,j+1}^{\ol{\va}} $ has dimension $ j+1 $, $ j\in\{0,1,2\} $. Define $
\wh{R}_1^{\ol{\va}}=\cup_{i=1}^{\wh{k}_{j+1}^{\ol{\va}}}\wh{K}_{i,j+1}^{\ol{\va}} $. Define $ \W_{\ol{\va}} $ as follows
\begin{enumerate}
\item If $ x\in\pa B_1\cup\pa B_{1-h(\ol{\va})} $, $ \W_{\ol{\va}} $ is given by $ \U_{\ol{\va}},\V_{\ol{\va}} $
\item For $ x\in R_1^{\ol{\va}}\cup \wh{R}_1^{\ol{\va}} $,
\be
\W_{\ol{\va}}(x)=\f{1-|x|}{h(\ol{\va})}\U_{\ol{\va}}\(\f{x}{|x|}\)+\f{h(\ol{\va})-1+|x|}{h(\ol{\va})}\V_{\ol{\va}}\(\f{x}{|x|}\),\label{homogeDef}
\ee
\item For $ 3 $-cell $ \wh{K}_{i,3}^{\ol{\va}} $ of $ \mathcal{G}^{\ol{\va}} $ with $ 1\leq i\leq \wh{k}_3^{\ol{\va}} $, the $ \W_{\ol{\va}} $ is extended homogeneously. Indeed,
$$
\W_{\ol{\va}}(x)=\W_{\ol{\va}}\circ(\Phi_{\wh{K}_{i,3}^{\ol{\va}}})^{-1}\(\f{\Phi_{\wh{K}_{i,3}^{\ol{\va}}}(x)}{|\Phi_{\wh{K}_{i,3}^{\ol{\va}}}(x)|}\),
$$
where $ \Phi^{\ol{\va}}:\wh{K}_{i,3}^{\ol{\va}}\to B_{h(\ol{\va})} $ is a bilipschitz map with respect to $ \mathcal{G}^{\ol{\va}} $. 
\end{enumerate}
Consequently, $ \W_{\ol{\va}}\in H^1(\wh{K}_{i,3}^{\ol{\va}},\Ss_0) $ for any $ 1\leq i\leq\wh{k}_3^{\ol{\va}} $ and $ \W_{\ol{\va}}\in H^1(A_{\ol{\va}},\Ss_0) $. In view of \eqref{G1}, \eqref{G2} and simple calculations, we have
$$
 E_{\ol{\va}}(\W_{\ol{\va}},A_{\ol{\va}})\leq Ch(\ol{\va})( E_{\ol{\va}}(\W_{\ol{\va}},\pa B_1)+ E_{\ol{\va}}(\W_{\ol{\va}},\pa B_{1-h(\ol{\va})})+ E_{\ol{\va}}(\W_{\ol{\va}},\wh{R}_1^{\ol{\va}})).
$$
We claim that $ E_{\ol{\va}}(\W_{\ol{\va}},\wh{R}_1^{\ol{\va}})\leq C(\ol{\va}^2h^{-2}(\ol{\va})+1) E_{\ol{\va}}(\W_{\ol{\va}},\pa B_1) $. If this claim is true, then
$$
 E_{\ol{\va}}(\W_{\ol{\va}},A_{\ol{\va}})\leq Ch(\ol{\va})(\ol{\va}^2h^{-2}(\ol{\va})+1) E_{\ol{\va}}(\U_{\ol{\va}},\pa B_1),
$$
which completes the proof of Lemma \ref{Luckhauslemma}. It remains to show the claim. The proof of this claim follows from direct computations. Indeed, by \eqref{fBconvex}, \eqref{R1vauv} and \eqref{homogeDef}, we have
\be
\int_{\wh{R}_1^{\ol{\va}}}f_b(\W_{\ol{\va}})\ud\HH^2\leq Ch(\ol{\va})\int_{\wh{R}_1^{\ol{\va}}}f_b(\U_{\ol{\va}})\ud\HH^2.\label{fW}
\ee
Again by \eqref{fBc}, \eqref{homogeDef} and \eqref{R1vauv}, one can obtain
$$
\int_{\wh{R}_1^{\ol{\va}}}|\na_{\wh{R}_1^{\ol{\va}}}\W_{\ol{\va}}|^2\ud\HH^2\leq Ch^{-1}(\ol{\va})\int_{\wh{R}_1^{\ol{\va}}}|\U_{\ol{\va}}-\V_{\ol{\va}}|^2\ud\HH^1\leq Ch^{-1}(\ol{\va})\int_{R_1^{\ol{\va}}}f_b(\U_{\ol{\va}})\ud\HH^1.
$$
This, together with \eqref{G4} and \eqref{fW}, implies that
$$
E_{\ol{\va}}(\W_{\ol{\va}},\wh{R}_1^{\ol{\va}})\leq C(h^{-1}(\ol{\va})+\ol{\va}^{-2}h(\ol{\va}))\int_{R_1^{\ol{\va}}}f_b(\U_{\ol{\va}})\ud\HH^1\leq C(h^{-2}(\ol{\va})+\ol{\va}^{-2})\int_{\pa B_1}f_b(\U_{\ol{\va}})\ud\HH^2,
$$
completes the proof of the claim.
\end{proof}

\section{Minimizers in bounded energy regime}\label{boundedenergyregime}
\subsection{Local minimizers with bounded energy} 

In this subsection, we will prove Theorem \ref{main1}. The proof will be divided into three steps as follows.\smallskip

\noindent\textbf{\underline{Step 1.} $ H^1 $ convergence and the minimizing property of the limit.} In view of \eqref{assumptionbound1}, we have, $ \{\Q_{\va}\}_{0<\va<1} $ is uniformly bounded in $ H^1(\om,\Ss_0) $. There exist $ \va_n\to 0^+ $ and $ \Q_0\in H^1(\om,\Ss_0) $ such that $
\Q_{\va_n}\wc\Q_0 $ weakly in $ H^1(\om,\Ss_0) $, $ \Q_{\va_n}\to\Q_0 $ strongly in $ L^2(\om,\Ss_0) $, and a.e. in $ \om $. By Fatou Lemma and  \eqref{assumptionbound1}, we have
$$
\int_{\om}f_b(\Q_0)\ud x\leq\liminf_{n\to+\ift}\va_n^2E_{\va_n}(\Q_{\va_n},\om)=0.
$$
This implies that $ f_b(\Q_0)=0 $ and then $ \Q_0\in H^1(\om,\cN) $. Fix $ B_r(x_0)\subset\subset\om $, we claim the following properties hold.

\begin{enumerate}
\item There is a relabeled subsequence, such that 
$$
\Q_{\va_n}\to\Q_0\text{ strongly in } H^1(B_r(x_0),\Ss_0),\,\,\lim_{n\to+\ift}\f{1}{\va_n^2}\int_{B_r(x_0)}f_b(\Q_{\va_n})\ud x=0.
$$
\item If $ \PP\in H^1(B_r(x_0),\Ss_0) $, $ \PP=\Q_0 $ on $ \pa B_r(x_0) $, then
$$
\f{1}{2}\int_{B_r(x_0)}|\na\Q_0|^2\ud x\leq\f{1}{2}\int_{B_r(x_0)}|\na\PP|^2\ud x.
$$
\end{enumerate}

The first and second results of Theorem \ref{main1} directly follows from the two properties above and basic covering arguments. Up to a translation, we assume that $ x_0=0 $. Since $ B_r\subset\subset\om $, we can take $ \mu>0 $ such that $ B_{(1+\mu)r}\subset\om $. Define
$$
D:=\left\{\rho\in(0,(1+\mu)r]:\liminf_{n\to+\ift} E_{\va_n}(\Q_{\va_n},\pa B_{\rho})>\f{2M}{\mu r}\right\}.
$$
We claim that $ \HH^1(D)\leq\mu r/2 $. Indeed, by \eqref{assumptionbound1}, we have
$$
\f{2M}{\mu r}\HH^1(D)\leq\int_r^{(1+\mu)r}\liminf_{n\to+\ift} E_{\va_n}(\Q_{\va_n},\pa B_{\rho})\ud\rho\leq\liminf_{n\to+\ift} E_{\va_n}(\Q_{\va_n},B_{(1+\mu)r}\backslash B_r)\leq M,
$$
which implies the result. By selecting a subsequence of $ \Q_{\va_n} $ without changing the notation, there must exist $ \rho\in(r,(1+\mu)r] $ such that for any $ n\in\Z_+ $, $
E_{\va_n}(\Q_{\va_n},\pa B_{\rho})\leq 2M/\mu r $. Defining $ \ol{\va}_n:=\va_n/\rho $, $ \U_n(x):=\Q_{\va_n}(\rho x) $ and $ \U_{\ift}(x):=\Q_0(\rho x) $ for $ x\in B_1 $, we have
\be
\begin{aligned}
&\U_n\wc\U_{\ift}\text{ weakly in }H^1(B_1,\Ss_0),\\
&\U_n\to\U_{\ift}\text{ strongly in }L^2(B_1,\Ss_0)\text{ and a.e. in }B_1,\\
&\U_{\ift}\in H^1(B_1,\cN),\\
& E_{\ol{\va}_n}(\U_n,\pa B_1)\leq 4C_0.
\end{aligned}\label{convegenceofUn}
\ee
By the compactness of trace operator and the convergence result above, we have
$$
\f{1}{2}\int_{\pa B_1}|\na_{\pa B_1}\U_{\ift}|^2\ud\HH^2\leq\limsup_{n\to+\ift} E_{\ol{\va}_n}(\U_n,\pa B_{\rho})\leq\f{2M}{\mu r}.
$$
Let $ \sg_n:=\|\U_n-\U_{\ift}\|_{L^2(\pa B_1)} $. We obtain that $ \sg_n\to 0^+ $, and
$$
\int_{\pa B_1}\(|\na_{\pa B_1}\U_n|^2+\f{1}{\ol{\va}_n^2}f_b(\U_n)+|\na_{\pa B_1}\U_{\ift}|^2+\f{|\U_n-\U_{\ift}|^2}{\sg_n^2}\)\ud\HH^2\leq C_0,
$$
where $ C_0>0 $ depends only on $ \mu,r $, and $ M $. We can apply Proposition \ref{Luckhaus2} to obtain $ \W_{n}\in H^1(B_1\backslash B_{1-\nu_n},\Ss_0) $ with $ \nu_n\to 0^+ $ such that $ \W_{n}(x)=\U_n(x) $, $ \W_{n}((1-\nu_n)x)=\U_{\ift}(x) $, $ \HH^2 $-a.e. on $ \pa B_1 $ and $ E_{\ol{\va}_n}(\W_n,B_1\backslash B_{1-\nu_n})\leq C\nu_n $. Choosing $ \W_* $ to be a minimizer of the problem
$$
\min\left\{\f{1}{2}\int_{B_1}|\na\U|^2\ud x:\U\in H^1(B_1,\cN),\,\,\U|_{\pa B_1}=\U_{\ift}\right\},
$$
we can define $ \W_n:B_1\to\Ss_0 $ by
$$
\W_n(x)=\left\{\begin{aligned}
&\W_n(x)&\text{ if }&x\in B_1\backslash B_{1-\nu_n},\\
&\W_*\(\f{x}{1-\nu_n}\)&\text{ if }&x\in B_{1-\nu_n}.
\end{aligned}\right.
$$
Using the local minimizing property of $ \U_n $ and the fact that $ \W_n $ is an admissible comparison map, we have $  E_{\ol{\va}_n}(\U_n,B_1)\leq  E_{\ol{\va}_n}(\W_n,B_1) $, and then
$$
E_{\ol{\va}_n}(\U_n,B_1)\leq  E_{\ol{\va}_n}(\W_n,B_1)=\f{1-\nu_n}{2}\int_{B_1}|\na\W_{*}|^2\ud x+ E_{\ol{\va}_n}(\W_n,B_1\backslash B_{1-\nu_n}).
$$
This, together with \eqref{convegenceofUn}, implies that
\be
\begin{aligned}
\f{1}{2}\int_{B_1}|\na\U_{\ift}|^2\ud x&\leq\liminf_{n\to+\ift}\f{1}{2}\int_{B_1}|\na\U_n|^2\ud x\\
&\leq\limsup_{n\to+\ift}\int_{B_1}\(\f{1}{2}|\na\U_n|^2+\f{1}{\ol{\va}_n^2}f(\U_n)\)\ud x\\
&=\f{1}{2}\int_{B_1}|\na\W_*|^2\ud x\leq\f{1}{2}\int_{B_1}|\na\U_{\ift}|^2\ud x.
\end{aligned}\label{FBconvImDe}
\ee
Combined with the choice of $ \W_* $, we obtain $
\U_n\to\U_{\ift} $ strongly in $ B_1 $,
and for any $ \V\in H^1(B_1,\cN) $ such that $ \V|_{\pa B_1}=\U_{\ift} $,
$$
\f{1}{2}\int_{B_1}|\na\U_{\ift}|^2\ud x\leq\f{1}{2}\int_{B_1}|\na\W_*|^2\ud x\leq\f{1}{2}\int_{B_1}|\na\V|^2\ud x.
$$
Moreover, by \eqref{FBconvImDe}, we also have
$$
\lim_{n\to+\ift}\f{1}{\va_n^2}\int_{B_{\rho}}f_b(\U_n)\ud x=0,
$$
By scaling $ \U_n $ and $ \U_{\ift} $ back to $ \Q_{\va_n} $ and $ \Q_0 $, we can complete the proof.\smallskip

\noindent\textbf{\underline{Step 2.} Existence of $ \cS_{\pts} $ and $ C^j $ convergence of $ \Q_{\va_n} $ in $ \om\backslash\cS_{\pts} $.} Since $ \Q_0 $ is a locally minimizing harmonic in $ \om $, we can use the famous regularity theory of harmonic maps in \cite{SU82} to obtain that there exists a locally finite set $ \cS_{\pts}\subset\om $, such that $ \Q_0 $ is smooth in $ \om\backslash\cS_{\pts} $. We will next show that up to a subsequence, $ \Q_{\va_n}\to\Q_0 $ in $ C_{\loc}^j(\om\backslash\cS_{\pts},\Ss_0) $ and satisfies \eqref{Reesti}. Firstly, we have the following lemma, which is on the uniform convergence of $ f_b(\Q_{\va_n}) $.

\begin{lem}\label{UniformfB}
For a compact set $ K\subset\om\backslash\cS_{\pts} $, there holds
$$
\lim_{\va_n\to 0^+}f_b(\Q_{\va_n})=0\text{ uniformly on }K.
$$
\end{lem}
\begin{proof}
Since $ K $ is compact, by using standard covering arguments and a translation, it is reduced to show the case that $ K=\ol{B_r} $ with $ r>0 $ and $ \ol{B_{2r}}\subset\om\backslash\cS_{\pts} $. Fix $ \eta>0 $ and $ x_1\in\ol{B_r} $, we set $ \al_n=f_b(\Q_{\va_n}(x_1)) $. In view of the definition of $ f_b $ and the fact \eqref{assumptionbound1}, there exists $ L>0 $, independent of $ \va_n $ such that
$$
|f_b(\Q_{\va_n}(x))-f_b(\Q_{\va}(y))|\leq L|\Q_{\va}(x)-\Q_{\va}(y)|\text{ for any } x,y\in B_r.
$$
By applying Corollary \ref{smooth}, we have
\begin{align*}
\al_n&\leq f_b(\Q_{\va_n}(x))+L|\Q_{\va}(x_1)-\Q_{\va}(x)|\\
&\leq f_b(\Q_{\va_n}(x))+L\|\na\Q_{\va_n}\|_{L^{\ift}(B_{3r/2})}|x-x_1|\\
&\leq f_b(\Q_{\va_n}(x))+\f{C_1L}{\va_n}|x-x_1|,
\end{align*}
for any $ x\in B_{3r/2} $, where $ C_1,L>0 $ depend only on $ \cA,M,x_0 $, and $ r $. By this, we can obtain
$$
\al_n-\f{C_1L\rho_n}{\va_n}\leq f_b(\Q_{\va_n}(x)),\text{ for all }x\in B_{\rho_n}(x_1),\,\,\rho_n<\f{r}{2},
$$
and then
\be
\f{\rho_n^3}{\va_n^2}\(\al_n-\f{C_1L\rho_n}{\va_n}\)\leq\f{1}{\va_n^2}\int_{B_{\rho_n}(x_1)}f_b(\Q_{\va_n})\ud x.\label{fBvarhok}
\ee
Since $ \Q_0 $ is smooth in $ \ol{B_{2r}} $, we can choose $ r_1>0 $ sufficiently small, such that $ B_{r_1}(x_1)\subset B_{3r/2} $ and
$$
\f{1}{r_1}\int_{B_{r_1}(x_1)}\f{1}{2}|\na\Q_0|^2\ud x<\f{\eta}{3}.
$$
This, together with the property that $ \Q_{\va_n}\to\Q_0 $ strongly in $ H_{\loc}^1(\om,\Ss_0) $, implies that if $ \va_n>0 $ is sufficiently small, then
\be
\f{1}{r_1}\int_{B_{r_1}(x_1)}\f{1}{2}|\na\Q_{\va_n}|^2\ud x<\f{2\eta}{3}.\label{23eta}
\ee
It follows from Corollary \ref{Monotonicity}, \eqref{23eta} and the first property of Theorem \ref{main1} that when $ 0<\va_n<1 $ is sufficiently small, then
\begin{align*}
\f{1}{\va_n^2}\int_{B_{\rho_n}(x_1)}f_b(\Q_{\va_n})\ud x&\leq\f{\rho_n}{r_1}E_{\va_n}(\Q_{\va_n},B_{r_1}(x_1))\\
&\leq\f{\rho_n}{r_1}\int_{B_{r_1}(x_1)}\(\f{1}{2}|\na\Q_{\va_n}|^2+\f{1}{\va_n^2}f_b(\Q_{\va_n})\)\ud x\\
&\leq\rho_n\(\f{2\eta}{3}+\f{\eta}{3}\)
\end{align*}
when $ 0<\rho_n<r_1 $. This, together with \eqref{fBvarhok}, implies that 
$$
\f{\rho_n^2}{\va_n^2}\(\al_n-\f{C_1L\rho_n}{\va_n}\)\leq\eta.
$$
If we take $ \rho_n=\al_n\va_n/(2C_1L) $, then $
\al_n^3<C\eta $, where $ C>0 $ depends only on $ C_1 $ and $ L $. By the arbitrariness of $ \eta>0 $ and $ x_1\in \ol{B_r} $, we can complete the proof. 
\end{proof}

The following lemma shows that $ \Q_{\va_n}\to\Q_0 $ in $ C_{\loc}^0(\om\backslash\cS_{\pts},\Ss_0) $.

\begin{lem}\label{UniformQva}
For a compact set $ K\subset\om\backslash\cS_{\pts} $, there holds
\be
\|\na\Q_{\va_n}\|_{L^{\ift}(K)}\leq C,\label{nablaQbound}
\ee
where $ C>0 $ depends only on $ \cA,M,K $, and $ \Q_0 $. In particular, $ \Q_{\va_n}\to\Q_0 $ in $ C^0(K,\Ss_0) $.
\end{lem}

\begin{proof}
For $ x_0\in\om\backslash\cS_{\pts} $, since $ \Q_0 $ is smooth in a small neighborhood of $ x_0 $, we can choose sufficiently small $ r>0 $ such that $ \ol{B_r(x_0)}\subset\om\backslash\cS_{\pts} $ and
$$
\f{1}{r}\int_{B_{r}(x_0)}|\na\Q_0|^2\ud x\leq\f{E_0}{2},
$$
where $ E_0 $ is given by Lemma \ref{smallreg1}. This, together with the first property of Theorem \ref{main1}, implies that when $ 0<\va_n<1 $ is sufficiently small,
$$
\f{1}{r}\int_{B_{r}(x_0)}\(\f{1}{2}|\na\Q_{\va}|^2+\f{1}{\va_n^2}f(\Q_{\va_n})\)\ud x\leq\f{1}{2r}\int_{B_r(x_0)}|\na\Q_0|^2\ud x+\f{E_0}{4}+\f{E_0}{4}\leq\f{3E_0}{4}.
$$
Moreover, since $ \ol{B_r(x_0)}\subset\om\backslash\cS_{\pts} $, by using Lemma \ref{UniformfB}, there holds
$$
\|\dist(\Q_{\va_n},\cN)\|_{L^{\ift}(B_r(x_0))}<\delta_0
$$
when $ 0<\va_n<1 $ is sufficiently small, where $ \delta_0>0 $ is also given by Lemma \ref{smallreg1}. Now, we can apply this lemma to obtain 
$$
\|\na\Q_{\va_n}\|_{L^{\ift}(B_{r/2}(x_0))}\leq C,
$$
where $ C>0 $ depends only on $ \cA,M,x_0,r $, and $ \Q_0 $. Since $ K $ is compact, the result follows from standard covering arguments.
\end{proof}

The following lemma gives the $ C^j $ convergence of $ \Q_{\va_n} $.

\begin{lem}\label{UniformQvaj}
For a compact set $ K\subset\om\backslash\cS_{\pts} $, there holds
\be
\|D^j\Q_{\va_n}\|_{L^{\ift}(K)}\leq C,\label{nablaQboundj}
\ee
where $ C>0 $ depends only on $ \cA,M,K,j $, and $ \Q_0 $. In particular, $ \Q_{\va_n}\to\Q_0 $ in $ C_{\loc}^j(\om\backslash\cS_{\pts}) $ for any $ j\in\Z_+ $.
\end{lem}
\begin{proof}
Still by covering arguments and a translation, we can assume that $ K=\ol{B_r}\subset\ol{B_{4r}}\subset\om\backslash\cS_{\pts} $. By Lemma \ref{UniformfB}, we have that when $ 0<\va_n<1 $ is sufficiently small, $ \dist(\Q_{\va_n},\cN)<\delta_0 $ in $ B_{4r} $, where $ \delta_0>0 $ is given by Lemma \ref{LiftQNll1}. It follows from this lemma and Lemma \ref{UniformQva} that
$$
\|(|\Delta\Q_{\va}|+\va^{-2}f_b^{1/2}(\Q_{\va})+|h_{\va}|+|\Y_{\va}|)\|_{L^{\ift}(B_{2r})}\leq C.
$$
Combining the results in Lemma \ref{Ckestimate}, we can obtain that for any $ j\in\Z_+ $,
$$
\|D^j\Q_{\va_n}\|_{L^{\ift}(B_r)}\leq C,
$$
where $ C>0 $ depends only on $ \cA,M,r,j $, and $ \Q_0 $. Now we can complete the proof.
\end{proof}

\noindent\textbf{\underline{Step 3.} Estimates of the remainder.} Fix $ \ol{B_r(x_0)}\subset\om\backslash\cS_{\op{line}} $ with $ r>0 $. In view of the formula \eqref{Qeq} and the definitions of $ u_{\va},v_{\va} $ given by \eqref{uvavvadef}, we have
\begin{align*}
\Delta\Q_{\va_n}&=-\f{1}{2r_*^2}|\na\Q_{\va_n}|^2\Q_{\va_n}-\f{3}{r_*^4}\tr(\na\Q_{\va_n}\na\Q_{\va_n}\Q_{\va_n})\(\Q_{\va_n}^2-\f{2r_*^2}{3}\I\)\\
&\quad\quad+2a_6v_{\va_n}\Q_{\va_n}+\f{a_2}{2r_*^2}u_{\va_n}\Q_{\va_n}+a_6'v_{\va_n}\(\Q_{\va_n}^2-\f{1}{3}(\tr\Q_{\va_n}^2)\I\)\\
&\quad\quad+\f{\va_n^2}{a_6'r_*^4}h_{\va_n}\tr(\na\Q_{\va_n}\Q_{\va_n}\Q_{\va_n})\I.
\end{align*}
Define
$$
\mathbf{R}_n:=2a_6v_{\va_n}\Q_{\va_n}+\f{a_2}{2r_*^2}u_{\va_n}\Q_{\va_n}+a_6'v_{\va_n}\(\Q_{\va_n}^2-\f{1}{3}(\tr\Q_{\va_n}^2)\I\)+\f{\va_n^2}{a_6'r_*^4}h_{\va_n}\tr(\na\Q_{\va_n}\Q_{\va_n}\Q_{\va_n})\I.
$$
Using result in Lemma \ref{uvavvvainterior}, \eqref{Reesti} follows directly.

\subsection{An example: global minimizers of Landau-de Gennes functional}

\begin{thm}\label{expansion}
Let $ \om $ be a bounded smooth domain and $ \{\Q_{\va}\}_{0<\va<1}\subset H^1(\om,\Ss_0) $ be global minimizers of \eqref{energy} in the space $ H^1(\om,\Ss_0;\Q_b) $ with $ \Q_b\in C^{\ift}(\pa\om,\cN) $. There exists a subsequence $ \va_n\to 0^+ $ and $ \Q_0\in H^1(\om,\cN) $ such that $ \Q_{\va_n}\to\Q_0 $ in $ H^1(\om,\Ss_0) $ and $ \Q_0 $ is a minimizer of the problem \eqref{limitlift}. Let $ \cS_{\pts} $ be the singular set of $ \Q_0 $. Then $ \cS_{\pts} $ is finite and
\begin{align}
\Q_{\va_n}&\to\Q_0\text{ in }C_{\loc}^{1,\al}(\ol{\om}\backslash\cS_{\pts},\Ss_0)\text{ for any }0<\al<1,\label{C1alconvergence}\\
\Q_{\va_n}&\to\Q_0\text{ in }C_{\loc}^{j}(\om\backslash\cS_{\pts},\Ss_0)\text{ for any }j\in\Z_{\geq 2}.\label{Cinftyconvergence}
\end{align}
Moreover, if the linearized operator $ \mathcal{L}_{\Q_0}:H_0^1(\om,\MM^{3\times 3})\to H^{-1}(\om,\MM^{3\times 3}) $ given by
\begin{align*}
\mathcal{L}_{\Q_0}\Q&:=\Delta\Q+\f{1}{2r_*^2}|\na\Q_0|^2\Q+\f{1}{r_*^2}(\na\Q_0:\na\Q)\Q_0\\
&\quad\quad+\f{6}{r_*^4}\tr(\na\Q\na\Q_0\Q_0)\(\Q_0^2-\f{2r_*^2}{3}\I\)\\
&\quad\quad+\f{3}{r_*^4}\tr(\na\Q_0\na\Q_0\Q)\(\Q_0^2-\f{2r_*^2}{3}\I\)\\
&\quad\quad+\f{6}{r_*^4}\tr(\na\Q_0\na\Q_0\Q_0)\Q_0\Q
\end{align*}
is bijective and $ \Q_0 $ is smooth. There exists $ \Q_{0}^{(1)}\in C^{\ift}(\om,\Ss_0)\cap H^s(\om,\Ss_0) $ such that
\begin{align}
&\va_n^{-1}(\Q_{\va_n}-\Q_0)\to\Q_{0}^{(1)}\text{ in }H^s(\om,\Ss_0)\text{ for any }0<s<1/2,\label{Hsestcon}\\
&\va_n^{-1}(\Q_{\va_n}-\Q_0)\to\Q_{0}^{(1)}\text{ in }C_{\loc}^j(\om,\Ss_0)\text{ for any }j\in\Z_{\geq 2}.\label{jZgeq2}
\end{align}
Furthermore, if we split $ \Q_0^{(1)}=\Q^{\perp}+\Q^{\parallel} $, where $ \Q^{\perp}\in (T_{\Q_0}\cN)_{\Ss_0}^{\perp} $ and $ \Q^{\parallel}\in T_{\Q_0}\cN $, then the following properties hold.
\begin{enumerate}
\item $ \Q^{\perp} $ is given by
\be
\Q^{\perp}=-\f{1}{8r_*^4(a_4+4a_6r_*^2)}|\na\Q_0|^2-\f{1}{2a_6'r_*^8}\tr(\na\Q_0\na\Q_0\Q_0)(3\Q_0^2-2r_*^2\I).\label{Qperpfor}
\ee
\item $ \Q^{\parallel} $ satisfies the equation
\be
\begin{aligned}
(\Delta\Q^{\parallel})&=-\f{1}{r_*^4}(\tr(\Q^{\parallel}\na\Q_0\na\Q_0+\na\Q^{\parallel}\na\Q_0))(3\Q_0^2-2r_*^2\I)\\
&\quad-\f{1}{r_*^2}\tr(\na\Q^{\parallel}\na\Q_0)\Q_0+\Delta\Q_0-(\Delta\Q^{\perp})^{\perp}
\end{aligned}\label{Qparallelfor}
\ee
\end{enumerate}
\end{thm}

\begin{lem}\label{Monoto2}
Assume that $ 0<\va<1 $, $ U $ is a bounded $ C^{2,1} $ domain with $ r_{U,2},M_{U,2} $, and $ \Q_b\in C^2(\pa U,\cN) $. Let $ 0<r<r_{U,2} $ and $ x_0\in\pa U $. If $ \Q_{\va}\in C^{\ift}(U,\Ss_0)\cap C^2(\ol{U},\Ss_0) $ is a solution of 
\begin{align*}
\left\{\begin{aligned}
-\va^2\Delta\Q_{\va}+\Psi(\Q_{\va})&=0&\text{ in }&U,\\
\Q_{\va}&=\Q_b&\text{ on }&\pa U,
\end{aligned}\right.
\end{align*}
satisfying
\be
E_{\va}(\Q_{\va},U)\leq C_0,\label{energyboundC1}
\ee
for some $ C_0>0 $, then for any $ 0<\rho<r $,
$$
\f{\ud}{\ud \rho}\(\f{1}{\rho}E_{\va}(\Q_{\va},U\cap B_r(x_0))\)\geq-C,
$$
where $ C>0 $ depends only on $ \cA,r_{U,2},M_{U,2},C_0 $, and $ \Q_b $.
\end{lem}

\begin{proof}
Up to a translation, we assume that $ x_0=0 $. Using Lemma \ref{lemPohozaev} to $ U\cap B_{\rho} $ with $ 0<\rho<r $, we can obtain from simple calculations that
\begin{align*}
&-\f{1}{r^2}\int_{U\cap B_{\rho}}\(\f{1}{2}|\na\Q_{\va}|^2+\f{1}{\va^2}f_b(\Q_{\va})\)\ud x+\f{1}{r}\int_{U\cap\pa B_{\rho}}\(\f{1}{2}|\na\Q_{\va}|^2+\f{1}{\va^2}f_b(\Q_{\va})\)\ud\HH^2\\
&\quad\quad=\f{2}{\va^2\rho^2}\int_{U\cap B_{\rho}}f_b(\Q_{\va})\ud x+\int_{U\cap\pa B_{\rho}}\f{1}{\rho^3}|x\cdot\na\Q_{\va}|^2\ud\HH^2\\
&\quad\quad\quad\quad+\f{1}{2\rho^2}\int_{\pa U\cap B_{\rho}}(x\cdot\nu)(|\pa_{\nu}\Q_{\va}|^2-|\na_{\pa U}\Q_b|)\ud\HH^2\\
&\quad\quad\quad\quad+\f{1}{\rho^2}\int_{\pa U\cap B_{\rho}}\pa_{\nu}\Q_{\va}:(\mathbb{P}_{\pa U}(x)\cdot\na\Q_{\va})\ud\HH^2,
\end{align*}
where we also use the fact that $ |\na\Q_{\va}|^2=|\pa_{\nu}\Q_{\va}|^2+|\na_{\pa U}\Q_b|^2 $ and $ \Q_b\in\cN $ on $ \pa U $. Using Cauchy inequality, we have
\begin{align*}
\f{\ud}{\ud\rho}\(\f{1}{\rho}E_{\va}(\Q_{\va},U\cap B_{\rho})\)&\geq\f{1}{2\rho^2}\int_{\pa U\cap B_{\rho}}(x\cdot\nu)(|\pa_{\nu}\Q_{\va}|^2-|\na_{\pa U}\Q_b|)\ud\HH^2\\
&\quad\quad+\f{1}{\rho^2}\int_{\pa U\cap B_{\rho}}\pa_{\nu}\Q_{\va}:(\mathbb{P}_{\pa U}(x)\cdot\na\Q_{\va})\ud\HH^2\\
&\geq\f{1}{2\rho^2}\int_{\pa U\cap B_{\rho}}(|(x\cdot\nu)|-C\rho^2)|\pa_{\nu}\Q_{\va}|^2\ud\HH^2\\
&\quad\quad-\f{1}{\rho^2}\int_{\pa U\cap B_{\rho}}|\na_{\pa U}\Q_b|^2\ud\HH^2-\int_{\pa U\cap B_{\rho}}|\pa_{\nu}\Q_{\va}|^2\ud\HH^2\\
&\geq-\f{1}{\rho^2}\int_{\pa U\cap B_{\rho}}|\na_{\pa U}\Q_b|^2\ud\HH^2-C\int_{\pa U\cap B_{\rho}}|\pa_{\nu}\Q_{\va}|^2\ud\HH^2\\
&\geq-C\int_{\pa U\cap B_{\rho}}|\pa_{\nu}\Q_{\va}|^2\ud\HH^2-C,
\end{align*}
where for the second inequality, we have used the fact $ x\cdot\nu\geq|x\cdot\nu|-Cr^2 $ for some $ C>0 $ depending only on $ r_{U,2} $ and $ M_{U,2} $. Now the result of this lemma follows directly from the following claim:
\be
\int_{\pa U\cap B_r}|\pa_{\nu}\Q_{\va}|^2\ud\HH^2\leq C,\label{claimQvanu}
\ee
where $ C>0 $ depends only on $ \cA,r_{U,2},M_{U,2},C_0 $, and $ \Q_b $. Choose $ \vv\in C^1(\om,\R^3) $ such that $ \vv(x)=\nu(x) $ for any $ x\in\pa U $. Such function exists since we can first construct $ \vv $ in the small neighborhood of $ \pa U $ and then use the cut off function to give the definition in the interior of $ U $. Testing $ -\va^2\Delta\Q_{\va}+\Psi(\Q_{\va})=0 $ by the function $ \vv_k\pa_k\Q $, we have
\begin{align*}
\int_U\f{1}{\va^2}(f_b)_{ij}(\Q_{\va})&\vv_k\pa_k(\Q_{\va})_{ij}=\int_U\Delta(\Q_{\va})_{ij}\vv_{\ell}\pa_{\ell}(\Q_{\va})_{ij}\ud x\\
&=\int_{\pa U}|\pa_{\nu}\Q_{\va}|^2\ud\HH^2-\int_U\pa_k(\Q_{\va})_{ij}\pa_{k}\pa_{\ell}(\Q_{\va})_{ij}\vv_{\ell}\ud x\\
&\quad\quad-\int_U\pa_k(\Q_{\va})_{ij}\pa_{\ell}(\Q_{\va})_{ij}\pa_k\vv_{\ell}\ud x\\
&=\int_{\pa U}|\pa_{\nu}\Q_{\va}|^2\ud\HH^2+\f{1}{2}\int_U|\na\Q_{\va}|^2\op{div}\vv\ud x\\
&\quad\quad-\int_{\om}\pa(\Q_{\va})_{ij}\pa_{\ell}(\Q_{\va})_{ij}\pa_k\vv_{\ell}\ud x,
\end{align*}
where for the second and last equalities, we have used the integration by parts. As a result, we have
\begin{align*}
\int_{\pa U}|\pa_{\nu}\Q_{\va}|^2\ud\HH^2&\leq C+\int_U\f{1}{\va^2}(f_b)_{ij}(\Q_{\va})\vv_k\pa_k(\Q_{\va})_{ij}\\
&=C+\f{1}{\va^2}\int_U\pa_{\ell}(f_b(\Q_{\va}))\vv_{\ell}\ud x\\
&=C-\f{1}{\va^2}\int_Uf_b(\Q_{\va})\op{div}\vv\ud x,
\end{align*}
where for the last equality, we have used the integration by parts and the fact that $ \Q_b\in\cN $ on $ \pa U $. Now the \eqref{claimQvanu} follows directly from \eqref{energyboundC1}.
\end{proof}

\begin{proof}[Proof of Theorem \ref{expansion}]
The existence of $ \va_n $, $ \Q_0 $ is ensured by Proposition \ref{logboundenergy}. By standard results on harmonic maps, we have that $ \cS_{\pts} $ is finite. Since $ \om $ is smooth, then it is also $ C^{2,1} $ with $ r_{U,2} $ and $ M_{U,2} $. Now, we claim that if $ x_0\in\pa\om $, $ 0<r<r_0/10 $ and $ (\om\cap B_{10r}(x_0))\cap\cS_{\pts}=\emptyset $, then 
$$
\lim_{n\to+\ift}\sup_{x\in\om\cap B_r(x_0)}f_b(\Q_{\va_n}(x))=0.
$$
Fix $ \eta>0 $. Firstly, we choose $ 0<r_1<\min\{\eta/10,r/10\} $, depending on $ \Q_0 $ such that for any $ x\in\pa\om $,
$$
\f{1}{r_1}\int_{B_{r_1}(x)}\f{1}{2}|\na\Q_0|^2\ud x<\f{\eta}{10}.
$$
For sufficiently large $ n\in\Z_+ $, we have
\be
\f{1}{r_1}\int_{B_{r_1}(x)}\f{1}{2}|\na\Q_{\va_n}|^2\ud x<\f{\eta}{5}.\label{r1Qvaestim}
\ee
Fix $ x_1\in\om\cap B_r(x_0) $ and let $ \al_n=f_b(\Q_{\va_n}(x_1)) $. We can assume that $ \dist(x_1,\pa\om)<r_1/10 $ since for otherwise, we use results of Lemma \ref{UniformfB}. Let $ 0<\rho_n<r_1/2 $, we have $ \om\cap B_{\rho_n}(x_1)\subset\om\cap B_{2r}(x_0) $. By using Lemma \ref{boundednessofminimizer} and Corollary \ref{smooth}, we can obtain
\begin{align*}
\al_n&\leq f_b(\Q_{\va_n}(x))+C|\Q_{\va_n}(x_1)-\Q_{\va_n}(x)|\\
&\leq f_b(\Q_{\va_n}(x))+C|\na\Q_{\va_n}|_{L^{\ift}(\cap B_{2r}(x_0))}|x-x_1|\\
&\leq f_b(\Q_{\va_n}(x))+\f{C_1}{\va_n}|x-x_1|,
\end{align*}
for any $ x\in B_{\rho_n}(x_1) $. We deduce that
\be
\f{C_2}{\va_n^2}\(\al_n-\f{C_1\rho_n}{\va_n}\)\leq\f{1}{\va_n^2}\int_{\om\cap B_{\rho_n}(x_1)}f_b(\Q_{\va_n}(x))\ud x.\label{C2rhones}
\ee
Choose $ x_2\in\pa\om $, such that $ |x_2-x_1|=\dist(x_1,\pa\om) $. Obviously, $ x_2\in\cap B_{2r}(x_0) $ and $
\om\cap B_{\rho_n}(x_1)\subset\om\cap B_{2r_2}(x_2) $. By \eqref{C2rhones}, Corollary \ref{Monotonicity}, and Lemma \ref{Monoto2}, if $ 0<\rho_n<\dist(x_1,\pa\om) $, we have
\begin{align*}
\f{C_2\rho_n^3}{\va_n^2}\(\al_n-\f{C_1\rho_n}{\va_n}\)&\leq\f{\rho_n}{r_2}E_{\va_n}(\Q_{\va_n},B_{r_2}(x_1))\\
&\leq\f{\rho_n}{r_2}E_{\va_n}(\Q_{\va_n},\om\cap B_{2r_2}(x_2))\\
&\leq\rho_n\(\f{1}{r_1}E_{\va_n}(\Q_{\va_n},\om\cap B_{r_1}(x_2))+C(r_1-2r_2)\)\\
&\leq\f{\rho_n}{r_1}E_{\va_n}(\Q_{\va_n},\om\cap B_{r_1}(x_2))+C\rho_n\eta.
\end{align*}
If $ \rho_n\geq\dist(x_1,\pa\om) $, we get $ r_2\leq\rho_n $ and
\begin{align*}
\f{C_2\rho_n^3}{\va_n^2}\(\al_n-\f{C_1\rho_n}{\va_n}\)&\leq\f{r_2}{r_2}E_{\va_n}(\Q_{\va_n},\om\cap B_{2r_2}(x_2))\\
&\leq\rho_n\(\f{1}{r_1}E_{\va_n}(\Q_{\va_n},\om\cap B_{r_1}(x_2))+C(r_1-2r_2)\)\\
&\leq\f{\rho_n}{r_1}E_{\va_n}(\Q_{\va_n},\om\cap B_{r_1}(x_2))+C\rho_n\eta.
\end{align*}
In view of \eqref{fvanlimit} and \eqref{r1Qvaestim}, we can choose sufficiently large $ n\in\Z_+ $ such that 
$$
\f{C_2\rho_n^3}{\va_n^2}\(\al_n-\f{C_1\rho_n}{\va_n}\)\leq C_3\eta.
$$
Setting $ \rho_n=\al_n\va_n/(2C_1) $ and using the arbitrariness of $ \eta,x_1 $, we can complete the proof of this claim. Combined with Lemma \ref{UniformfB}, we can obtain that $ f_b(\Q_{\va_n})\to 0^+ $ in $ C_{\loc}^0(\om\backslash\cS_{\pts}) $ as $ n\to+\ift $. Applying Lemma \ref{LiftQNll1}, \ref{smallreg1}, and \ref{smallreg2}, we prove \eqref{C1alconvergence}. Also, \eqref{Cinftyconvergence} follows from Lemma \ref{Ckestimate}. Before we show \eqref{Hsestcon} and \eqref{jZgeq2}, we firstly assume that they are true and give \eqref{Qperpfor} and \eqref{Qparallelfor}. Set $ \Q_{\va_n}=\Q_0+\va_n^2\PP_n $, we can deduce from \eqref{E-L} that
\begin{align*}
\va_n^2\Delta\PP_n&=\f{1}{\va_n^2}(-a_2+a_4(\tr\Q_0^2)+a_6(\tr\Q_0^2)^2)\Q_0\\
&\quad\quad-a_2\PP_n+a_4(\tr\Q_0^2)\PP_n+2a_4\tr(\PP_n\Q_0)(\tr\Q_0^2)\Q_0\\
&\quad\quad-\Delta\Q_0+3a_6'\tr(\PP_n\Q_0^2)\(\Q_0^2-\f{\tr(\Q_0^2)}{3}\I\)+O(\va_n^2).
\end{align*}
Using the facts that $ \tr(\Q_0^2)=2r_*^2 $, \eqref{rstar} and choosing $ \va_n\to 0^+ $, we have
$$
(2a_4+8a_6r_*^2)\tr(\PP_n\Q_0)\Q_0+3a_6'\tr(\PP_n\Q_0^2)\(\Q_0^2-\f{2r_*^2}{3}\I\)-\Delta\Q_0=0.
$$
Using \eqref{weakharmoniceq}, \eqref{Qperpfor} follows directly. For \eqref{Qparallelfor}, we firstly note that
$$
\Delta\Q^{\parallel}=(\Delta\Q^{\parallel})+(\Delta\Q_0)^{\perp}-(\Delta\Q^{\perp})^{\perp}.
$$
Moreover, by the definition of $ \Q^{\parallel} $ and Remark \ref{normalremQ}, we have $ \tr(\Q^{\parallel}\Q_0)=0 $ and $ \tr(\Q^{\parallel}\Q_0)=0 $. Taking $ \Delta $ on both sides, we can obtain \eqref{Qparallelfor}.

Finally, we will show \eqref{Hsestcon} and \eqref{jZgeq2}. Firstly, we have
$$
\Delta\Q_{\va_n}=-\f{1}{2r_*^2}|\na\Q_{\va_n}|^2\Q_{\va_n}-\f{3}{r_*^4}\tr(\na\Q_{\va_n}\na\Q_{\va_n}\Q_{\va_n})\(\Q_{\va_n}^2-\f{2r_*^2}{3}\I\)+\mathbf{R}_n,  
$$
where $ \mathbf{R}_n\in C^{\ift}(\ol{\om},\Ss_0) $. By using Lemma \ref{lemexpdestim} and Theorem \ref{main1}, $ \mathbf{R}_n $ satisfies the estimates
\be
\begin{aligned}
|\mathbf{R}_n(y)|&\leq C\left\{\f{\va_n^2}{\dist(y,\pa\om)}+\exp\(-\f{\dist(y,\pa\om)}{\va_n}\)\right\},\\
|\na^{\ell}\mathbf{R}_n(y)|&\leq\f{C\va_n^2}{\dist(y,\pa\om)^{\ell+2}},
\end{aligned}\label{Rndistesti}
\ee
for any $ y\in\om $ and $ \ell\in\Z_{\geq 0} $. Let $ \fS_n=\Q_{\va_n}-\Q_0 $, we can use the convergence of $ \Q_{\va_n} $ to $ \Q_0 $ to get
\be
\mathcal{L}(\fS_n)_{ij}:=\mathcal{L}_{\Q_0}(\fS_n)_{ij}+b_{ijk}^{(n)\al\beta}\pa_k(\fS_n)_{\al\beta}+c_{ij}^{(n)\al\beta}(\fS_n)_{\al\beta}=(\mathbf{R}_n)_{ij}\label{Sequation}
\ee
such that $ b_{ijk}^{(n)\al\beta}\in C^{0,\al}(\ol{\om})\cap C^{\ift}(\om)\cap W^{1,p}(\om) $, $ c_{ij}^{(n)\al\beta}\in C^{0,\al}(\ol{\om})\cap C^{\ift}(\om) $ for any $ \al\in(0,1) $ and $ p>1 $, and satisfy
$$
\|b^{(n)}\|_{W^{1,p}(\om)}\leq C,\,\,\lim_{n\to+\ift}\|b^{(n)}\|_{C^{0,\al}(\ol{\om})}+\|c^{(n)}\|_{C^{0,\al}(\ol{\om})}=0.
$$
If $ n\in\Z_+ $ is sufficiently large and $ t\in[0,1] $, we have
$$
\mathcal{L}_t(\fS_n)_{ij}:=\mathcal{L}_{\Q_0}(\fS_n)_{ij}+tb_{ijk}^{(n)\al\beta}\pa_k(\fS_n)_{\al\beta}+tc_{ij}^{(n)\al\beta}(\fS_n)_{\al\beta}
$$
defined from $ H_0^1(\om)\to H^{-1}(\om) $ is injective. By the stability of Fredholm index, we have that $ \mathcal{L}_t $ are bijective for any $ t\in[0,1] $. Using Lemma \ref{multione}, we have
\be
\|\fS_n\|_{L^2(\om)}\leq C\|\mathbf{R}_n\|_{(H^2(\om)\cap H_0^1(\om))^*}\leq C\va_n^2.\label{SnL2}
\ee
On the other hand, we can deduce from \eqref{Sequation} that
\begin{align*}
\|\fS_n\|_{H^s(\om)}&\leq C(\|\Delta\fS_n\|_{H^{s-2}(\om)}+\|\fS_n\|_{H^{s-2}(\om)})\\
&\leq\f{1}{2}\|\fS_n\|_{H^s(\om)}+C(\|\mathbf{R}_{n}\|_{H^{s-2}(\om)}+\|\fS_n\|_{H^{s-2}(\om)})
\end{align*}
for any $ s>0 $. Thus, by using \eqref{Rndistesti}, Lemma \ref{multione} and \ref{NZlem2}, we have $ \|\fS_n\|_{H^s(\om)}\leq C\va_n^2 $ for any $ 0<s<1/2 $. This implies \eqref{Hsestcon}. For \eqref{jZgeq2}, we can use \eqref{Rndistesti}, \eqref{SnL2}, and standard elliptic estimates to get $
\|\fS_n\|_{H^j(K)}\leq C\va_n^2 $ for any $ K\subset\subset\om $ and $ j\in\Z_+ $.
\end{proof}

\section{Minimizers in logarithmic energy regime}\label{logarithmicenergyregime}

In this section, we will prove Theorem \ref{main}.

\subsection{Existence of line defect}

Firstly, we give the clearing-out lemma as follows, we remark that such lemma is also true for more general elastic energy density satisfying some convex conditions for $ \na\Q $.

\begin{lem}\label{varegu}
Let $ \om $ be a bounded, Lipschitz domain and $ \{\Q_{\va}\}_{0<\va<1} $ is a sequence of local minimizers of \eqref{energy} satisfying \eqref{assumptionbound}. There exist $ \eta_0,C>0 $ and $ 0<\ol{\va}_0<1 $, depending only on $ \cA $ and $ M $, such that if $ x_0\in\om $, $ r>0 $, $ \ol{B_r(x_0)}\subset\om $, $ 0<\va\leq\ol{\va}_0r $ and
$$
E_{\va}(\Q_{\va},B_r(x_0))\leq\eta_0 r\log\f{r}{\va},
$$
then $ E_{\va}(\Q_{\va},B_{r/2}(x_0))\leq Cr $, where $ C>0 $ depends only on $ \cA $ and $ M $.
\end{lem}
\begin{proof}
Up to a translation, we assume that $ x_0=0 $. For $ 0<\va\leq r/4 $, we define
$$
D^{\va}:=\left\{\rho\in\(\f{r}{2},r\):E_{\va}(\Q_{\va},\pa B_{\rho})\leq 4\eta_0\log\f{r}{\va}\right\},
$$
where $ \eta_0>0 $ is to be chosen. By average arguments, we have
\be
\HH^1(D^{\va})\geq\f{r}{4}.\label{HD14}
\ee
For any $ \rho\in D^{\va} $, there holds $ r<2\rho $. Thus we get
\be
E_{\va}(\Q_{\va},\pa B_{\rho})\leq 8\eta_0\log\f{\rho}{\va},\label{Condi}
\ee
for any $ 0<\va\leq r/4 $ and $ \rho\in D^{\va} $. We claim that when $ \eta_0=\eta_1/8 $ and $ \ol{\va}_0=\min\{1/4,\ol{\va}_1/2\} $, where $ \eta_1,\ol{\va}_1 $ are given by Lemma \ref{Luckhauslemma}, there exists $ C_0>0 $ depending only on $ \cA $ and $ M $ such that for any $ 0<\va\leq\ol{\va}_0r $ and $ \rho\in D^{\va} $,
\be
E_{\va}(\Q_{\va},B_{\rho})\leq C_0r(E_{\va}^{1/2}(\Q_{\va},\pa B_{\rho})+1).\label{Iteene}
\ee
We firstly assmue that this claim is true. If for some $ \rho\in(r/2,r) $, $ E_{\va}(\Q_{\va},B_{\rho})\leq C_0r $, then there is nothing to prove. On the other hand, if $ E_{\va}(\Q_{\va},B_{\rho})>C_0r $ for any $ \rho\in(r/2,r) $, then \eqref{Iteene} implies
$$
\f{\f{\ud}{\ud\rho}E_{\va}(\Q_{\va},B_{\rho})}{(E_{\va}(\Q_{\va},B_{\rho})-C_0r)^2}\geq C_0^{-2}r^{-2}\chi_{D^{\va}}(\rho).
$$
Integrating the above inequality on $ (r/2,r) $ and applying \eqref{HD14}, we have
\begin{align*}
\f{1}{E_{\va}(\Q_{\va},B_{r/2})-C_0r}&\geq\f{1}{E_{\va}(\Q_{\va},B_r)-C_0r}+\f{1}{C_0^2r^2}\cdot\f{r}{4}\geq\f{1}{4C_0^2r},
\end{align*}
which directly implies that $
E_{\va}(\Q_{\va},B_{r/2})\leq C_0(4+C_0)r $ and then the lemma is proved. Now it remains to show \eqref{Iteene}. Set $ \ol{\va}=\va/\rho $ and $ \U_{\ol{\va}}(x)=\Q_{\va}(\rho x) $ for $ x\in B_1 $. We only need to prove
$$
E_{\ol{\va}}(\U_{\ol{\va}},B_1)\leq C( E_{\ol{\va}}^{1/2}(\U_{\ol{\va}},\pa B_1)+1).
$$
For $ \rho\in D^{\va} $, in view of \eqref{Condi} and the choices of $ \ol{\va}_0,\eta_0 $, we have $
E_{\ol{\va}}(\U_{\ol{\va}},\pa B_1)\leq\eta_1\log(1/\ol{\va}) $ and $ 0<\ol{\va}<\ol{\va}_1<1 $. Combining \eqref{assumptionbound}, we can apply Lemma \ref{Luckhauslemma} to construct $ \V_{\ol{\va}}\in H^1(\pa B_1,\cN) $, $ \W_{\ol{\va}}\in H^1(B_1\backslash B_{1-h(\ol{\va})},\Ss_0) $ satisfying
\eqref{interpoL}, \eqref{VboundL}, and \eqref{WboundL} with $ h(\ol{\va})=\ol{\va}^{1/2}\log(1/\ol{\va}) $. Using Lemma \ref{k3}, there exists $ \V_{\ol{\va}}^{(1)}\in H^1(B_1,\cN) $ such that $ \V_{\ol{\va}}^{(1)}|_{\pa B_1}=\V_{\ol{\va}} $ and 
\be
\int_{B_1}|\na\V_{\ol{\va}}^{(1)}|^2\ud x\leq C E_{\ol{\va}}^{1/2}(\U_{\ol{\va}},\pa B_1).\label{12E}
\ee
Define $ \W_{\ol{\va}}^{(1)}:B_1\to\Ss_0 $ as
$$
\W_{\ol{\va}}^{(1)}(x):=\left\{\begin{aligned}
&\W_{\ol{\va}}(x)&\text{ for }&x\in B_1\backslash B_{1-h(\ol{\va})},\\
&\V_{\ol{\va}}^{(1)}\(\f{x}{1-h(\ol{\va})}\)&\text{ for }&x\in B_{1-h(\ol{\va})}.
\end{aligned}\right.
$$
In view of \eqref{VboundL}, \eqref{WboundL}, and \eqref{12E}, we can obtain
$$
E_{\ol{\va}}(\W_{\ol{\va}}^{(1)},B_1)\leq C(E_{\ol{\va}}^{1/2}(\U_{\ol{\va}},\pa B_1)+1).
$$
Since $ \W_{\ol{\va}}^{(1)}|_{\pa B_1}=\U_{\ol{\va}}|_{\pa B_1} $, \eqref{Iteene} follows directly by the property that $ \Q_{\va} $ is a local minimizer of \eqref{energy}.
\end{proof}

Let $ 0<\va<1 $ and $ \Q_{\va}\in H^1(\om,\Ss_0) $ be a local minimizer of of \eqref{energy}. We define a positive Radon measure $ \mu_{\va} $ on $ \ol{\om} $ by
$$
\mu_{\va}(U):=\f{E_{\va}(\Q_{\va},U)}{\log\f{1}{\va}}
$$
for any Borel-measurable set $ U\in\cB(\ol{\om}) $. By the assumption \eqref{assumptionbound}, we have 
\be 
\sup_{0<\va<1/2}\mu_{\va}(\ol{\om})\leq M\(1+\f{1}{\log 2}\).\label{UniformboundMeasure}
\ee
There exist a positive Radon measure $ \mu_0\in (C(\ol{\om}))' $ and a subsequence $ \{\va_n\}_{n=1}^{+\ift} $ such that as $ n\to+\ift $, there hold $ \va_n\to 0^+ $ and 
\be 
\mu_{\va_n}\wc^*\mu_0\text{ in }(C(\ol{\om}))'\label{munconve}
\ee
Define $ \cS_{\op{line}}:=\supp(\mu_0) $ be the support of $ \mu_0 $ in $ \ol{\om} $. We obtain that $ \cS_{\op{line}} $ is a closed subset of $ \ol{\om} $. 

\begin{lem}\label{smallmu0re}
Let $ \eta_0>0 $ be given by Lemma \ref{varegu}, $ x_0\in\om $ and $ r>0 $, such that $ \ol{B_r(x_0)}\subset\om $. If $ \mu_0(\ol{B_r(x_0)})<\eta_0r $, then $ \mu_0(B_{r/2}(x_0))=0 $, i.e., $ B_{r/2}(x_0)\subset\om\backslash\cS_{\op{line}} $.
\end{lem}
\begin{proof}
Firstly, by using Proposition 4.26 of \cite{M2012}, it follows from \eqref{munconve} that
\be
\mu_{\va_n}(F)\geq\limsup_{n\to+\ift}\mu_{\va_n}(F),\quad\mu_{\va_n}(G)\leq\liminf_{n\to+\ift}\mu_{\va_n}(G),\label{weakconFG}
\ee
for any closed set $ F\subset\ol{\om} $ and open set $ G\subset\ol{\om} $. If $ \mu_0(\ol{B_r(x_0)})<\eta_0 r $, then
\begin{align*}
\limsup_{n\to+\ift}\f{ E_{\va_n}(\Q_{\va_n},B_r(x_0))}{r\log(r/\va_n)}&=\limsup_{n\to+\ift}\f{ E_{\va_n}(\Q_{\va_n},\ol{B_r(x_0)})}{r\log\f{1}{\va_n}}\f{\log\f{1}{\va_n}}{\log r+\log\f{1}{\va_n}}\\
&\leq\limsup_{n\to+\ift}\f{\log\f{1}{\va_n}}{\log r+\log\f{1}{\va_n}}\f{\mu_{\va_n}(\ol{B_r(x_0)})}{r}\\
&\leq\f{\mu_0(\ol{B_r(x_0)})}{r}<\eta_0.
\end{align*}
By using Lemma \ref{varegu}, we have $  E_{\va_n}(\Q_{\va_n},B_{r/2}(x_0))\leq Cr $ for sufficiently large $ n\in\Z_+ $. It follows from \eqref{munconve} and \eqref{weakconFG} that
$$
\mu_0(B_{r/2}(x_0))\leq\liminf_{n\to+\ift}\mu_{\va_n}(B_{r/2}(x_0))\leq\liminf_{n\to+\ift}\f{Cr}{\log\f{1}{\va_n}}=0,
$$
which completes the proof.
\end{proof}

\begin{lem}\label{monomu0}
For any $ x\in\om $, the map 
$$
r\in(0,\dist(x,\pa\om))\mapsto\f{\mu_0(\ol{B_r(x)})}{2r}
$$
is non-decreasing.
\end{lem}
\begin{proof}
For $ 0<r_1<r_2<\dist(x,\pa\om) $, there exists $ \delta>0 $ such that $ r_1+\delta<r_2 $. By using Lemma \ref{Monotonicity}, \eqref{munconve}, and \eqref{weakconFG}, we deduce
\begin{align*}
\f{\mu_0(\ol{B_{r_1}(x)})}{2r_1}&\leq\f{\mu_0(B_{r_1+\delta}(x))}{2(r_1+\delta)}\cdot\f{r_1+\delta}{r_1}\\
&\leq\liminf_{n\to+\ift}\f{ E_{\va_n}(\Q_{\va_n},B_{r_1+\delta}(x))}{2(r_1+\delta)\log\f{1}{\va_n}}\cdot\f{r_1+\delta}{r_1}\\
&\leq\limsup_{n\to+\ift}\f{ E_{\va_n}(\Q_{\va_n},\ol{B_{r_2}(x)})}{2r_2\log\f{1}{\va_n}}\cdot\f{r_1+\delta}{r_1}\\
&\leq\f{\mu_0(\ol{B_{r_2}(x)})}{2r_2}\cdot\f{r_1+\delta}{r_1}.
\end{align*}
Letting $ \delta\to 0^+ $, we can complete the proof.
\end{proof}

Now, we can define $ \Theta(\mu_0,\cdot):\om\to\R $ as
\be
\Theta(\mu_0,x):=\lim_{r\to 0^+}\f{\mu_0(\ol{B_r(x)})}{2r},\,\,x\in\om.\label{Thetadensity}
\ee
If there is no ambiguity, we use $ \Theta(x) $ to denote $ \Theta(\mu_0,x) $ for simplicity. Here we call the function $ \Theta $ the density of $ \mu_0 $. 

\begin{lem}\label{Thetamu0eta0}
If $ \eta_0>0 $ is given by Lemma \ref{varegu}, then 
$$
\cS_{\op{line}}\cap\om=\{x\in\om:\Theta(x)>0\}=\left\{x\in\om:\Theta(x)\geq\f{\eta_0}{2}\right\}, 
$$
\end{lem}
\begin{proof}
Since $ \cS_{\op{line}}=\supp(\mu_0) $, $ x\notin\om\backslash\cS_{\op{line}} $ if and only if there exists $ r>0 $ such that $ \mu_0(B_r(x))=0 $. This implies that
\be
\left\{x\in\om:\Theta(x)\geq\f{\eta_0}{2}\right\}\subset\{x\in\om:\Theta(x)>0\}\subset\cS_{\op{line}}\cap\om.\label{Slinesubset}
\ee
Moreover, if $ x\in\cS_{\op{line}}\cap\om $, then for any $ r>0 $, $ \mu_0(B_r(x))>0 $. Applying Lemma \ref{smallmu0re}, we have $ \mu_0(\ol{B_{2r}(x)})\geq 2\eta_0 r $, which shows that
$$
\cS_{\op{line}}\cap\om\subset\left\{x\in\om:\Theta(x)\geq\f{\eta_0}{2}\right\}.
$$
This, together with \eqref{Slinesubset}, completes the proof.
\end{proof}

\subsection{Further results on the singular set}\label{sectionvarifold}

Let $ \bG(3,1) $ represent the Grassmann manifold, which consists of $ 1 $-dimensional subspaces of $ \R^3 $. Moreover, It can also be characterized as 
$$
\bG(3,1)=\{\A\in\MM^{3\times 3}:\A^2=\A,\,\,\A^{\T}=\A\}.
$$
We also define $ \bG(\om)=\om\times\bG(3,1) $. It is worth noting that a $ 1 $-dimensional varifold in $ \om $ is a positive Radon measure on $ \bG_3(\om) $. Consider $ S\subset\om $, which is countably $ \HH^1 $-rectifiable, and let $ \mu\in(C(\ol{\om}))' $ be a positive Radon measure such that $ \mu\left\llcorner\right.\om=\theta_S\HH^1\left\llcorner\right. S $, where $ \theta_S\in L^1(\om) $. For $ \HH^1 $-a.e. $ x\in S $, there exists a unique approximate tangent plane $ T_xS\in\bG(3,1) $ of $ S $ at $ x $. For $ \mu $-a.e. $ x\in\om $, let $ \A(x)\in\bG(3,1) $ denote the orthogonal projection onto $ T_xS $. The varifold $ V $ associated with $ \mu\left\llcorner\right.\om $ is defined as the pushforward measure $ V=(\I,\A)_{\#}(\mu\left\llcorner\right.\om) $, where
$$
V(E)=\mu(\{x\in\om:(x,\A(x))\in E\})
$$
for any Borel-measurable set $ E\in\cB(\bG_3(\om)) $. A varifold is considered stationary if it satisfies
$$
\int_{\om}\A_{ij}(x)\pa_i\uu_j(x)\ud\mu(x)=0\text{ for any }\uu\in C_c^1(\om,\R^3).
$$

\begin{lem}\label{stationvari}
For $ \cS_{\op{line}} $, the following properties hold.
\begin{enumerate}
\item $ \cS_{\op{line}}\cap\om $ is countably $ \HH^1 $-rectifiable and $ \HH^1(\cS_{\op{line}}\cap\om)<+\ift $. 
\item For $ \mu_0 $-a.e. $ x\in\om $, let $ \A_*(x)\in\bG(3,1) $ be the orthogonal projection on to $ T_x\cS_{\op{line}} $, the tangent plane of $ S_{\op{line}} $ at $ x $. The varifold $ V_0=(\I,\A_*)_{\#}(\mu_0\left\llcorner\right.\om) $ associated with $ \mu_0\left\llcorner\right.\om $ is stationary. Moreover, $ \mu_0\left\llcorner\right.\om=\Theta(\mu_0,x)\HH^1\left\llcorner\right.(\cS_{\op{line}}\cap\om) $.
\end{enumerate}
\end{lem}
\begin{proof}
Define $ \A_{\va}:\om\to\MM^{3\times 3} $ by $ \A_{\va}:=((\A_{\va})_{ij})_{1\leq i,j\leq 3} $ with
$$
(\A_{\va})_{ij}=\f{1}{\log\f{1}{\va}}(e_{\va}(\Q_{\va})\delta_{ij}-\pa_i\Q_{\va}:\pa_j\Q_{\va})
$$
for $ i,j=1,2,3 $. From this, we have $ \A_{\va}=(\A_{\va})^{\T} $ and the following properties hold.
\be
\begin{aligned}
|\A_{\va}|&\leq \f{1}{\log\f{1}{\va}}(3\e_{\va}(\Q_{\va})^2+|\na\Q_{\va}|^4)^{1/2}\leq 2\f{\ud\mu_{\va}}{\ud x},\\
\tr\A_{\va}&=\f{1}{\log\f{1}{\va}}(3e_{\va}(\Q_{\va})-|\na\Q_{\va}|^2)\geq\f{\ud\mu_{\va}}{\ud x},
\end{aligned}\label{Avaes}
\ee
where $
\f{\ud\mu_{\va}}{\ud x}=\f{e_{\va}(\Q_{\va})}{\log\f{1}{\va}} $. For any $ \vv\in\Ss^2 $, we also get
\be
(\A_{\va})_{ij}\vv_i\vv_j=\f{1}{\log\f{1}{\va}}(e_{\va}(\Q_{\va})-|\vv_i\pa_i\Q_{\va}|^2)\leq\f{\ud\mu_{\va}}{\ud x}.\label{envaes}
\ee
In view of Lemma \ref{StressThm}, for any $ \uu\in C_c^1(\om,\R^3) $, we have
\be
\int_{\om}(\A_{\va})_{ij}(x)\pa_i\uu_j(x)\ud x=0.\label{Avaupai0}
\ee
Choosing $ \va=\va_n $, we can use \eqref{Avaes} to obtain a subsequence of $ \va_n $, without relabeling and a vector-valued Radon measure $ \A_0\in(C_c(\om,\MM^{3\times 3}))' $ such that $ \A_{\va_n}\ud x\wc^*\A_0 $ in $ C_c(\om,\MM^{3\times 3}) $ as $ n\to+\ift $. In view of \eqref{munconve} and \eqref{Avaes}, we have $ |\A_0|\leq C\mu_0\left\llcorner\right.\om $, which implies that $ \A_0 $ is absolutely continuous with respect to $ \mu_0\left\llcorner\right.\om $. Using Radon-Nikodym theorem, there exists a $ \mu_0 $-measurable function $ \fF_0\in L^1(\om,\MM^{3\times 3},\mu_0) $ (this means that $ \int_{\om}|\fF_0|\ud\mu_0<+\ift $) such that $ \A_0=\fF_0\mu_0\left\llcorner\right.\om $ in $ (C_c(\om,\MM^{3\times 3}))' $. Letting $ n\to+\ift $, we can deduce from \eqref{Avaes}, \eqref{envaes}, and \eqref{Avaupai0} that for $ \mu_0 $-a.e. $ x\in\om $, there hold
\be
(\fF_0(x))^{\T}=\fF_0(x),\,\,\tr\fF_0(x)\geq 1,\,\,\lda_i(\fF_0(x))\leq 1\text{ for any }i=1,2,3,\label{Foengenva}
\ee
and
\be
\int_{\om}\fF_0(x)\pa_i\uu_j(x)\ud\mu_0(x)=0\label{fF0station}
\ee
for any $ \uu\in C_c^1(\om,\R^3) $. By using Proposition 3.1 of \cite{A19}, we have, there exist a countably $ \HH^1 $-rectifiable set $ \om_0\subset\om $ and a map $ \B_0\in L^{\ift}(\om_0,\MM^{3\times 3}) $ such that for $ \HH^1 $-a.e. $ x\in\om_0 $, $ |\B_0(x)|=1 $ and $ \B_0(x)\in\bG(3,1) $ is the orthogonal projection matrix onto $ T_x\om_0 $, satisfying 
\be
\A_0\left\llcorner\right.\{\Theta(|\A_0|,x)>0\}=\Theta(|\A_0|,x)\B_0(x)\HH^1\left\llcorner\right.\om_0.\label{Txomcla}
\ee
Here in \eqref{Txomcla}, $ T_x\om_0 $ is the approximate tangent plane to $ \om_0 $ at $ x $ and
$$
\Theta(|\A_0|,x)=\limsup_{r\to 0^+}\f{|\A_0|(\ol{B_r(x_0)})}{2r}.
$$
Since $ \A_0=\fF_0\mu_0\left\llcorner\right.\om $, we have
\be
\fF_0\mu_0\left\llcorner\right.\om=\Theta(|\A_0|,x)\B_0(x)\HH^1\left\llcorner\right.\om_0.\label{measEq}
\ee
The choice of $ \B_0 $ implies that for $ \mu_0 $-a.e. $ x\in\om $, two eigenvalues of $ \fF_0(x) $ are zeros. Since \eqref{Foengenva} holds for $ \mu_0 $-a.e. $ x\in\om $, we can deduce that up to a permutation, the set of eigenvalues of $ \fF_0 $ is $ (1,0,0) $ and $ |\fF_0(x)|=1 $ for $ \mu_0 $-a.e. $ x\in\om $. Combined with \eqref{Thetadensity}, \eqref{measEq},  and Lemma \ref{Thetamu0eta0}, this implies that $ |\A_0|=\mu_0\left\llcorner\right.\om $, $ \Theta(|\A_0|,x)=\Theta(\mu_0,x) $, and $ \om_0=\cS_{\op{line}}\cap\om $. As a result, $ \cS_{\op{line}} $ is a countably $ \HH^1 $-rectifiable set. In particular, for $ \mu_0 $-a.e. $ x\in\om $, $ \fF_0(x)=\A_*(x) $ and 
\be
V_0(\ud\A,\ud x)=\delta_{\A_*(x)}\otimes\mu_0=\delta_{\A_*(x)}\otimes\Theta_1(\mu_0,x)\HH^1\left\llcorner\right.(\cS_{\op{line}}\cap\om).\label{V0formula}
\ee
Therefore, $ V_0 $ is a $ 1 $-dimensional stationary rectifiable varifold and we complete the proof.
\end{proof}

\begin{lem}\label{Luckhauscylinder1}
Let $ L_0\geq 2 $ and $ h\in(1/20,1/2] $. Assume that for $ 0<\wh{\va}<1 $, $ \fG_{\wh{\va}}\in H^1(\Ga_{1,L_0},\Ss_0) $ satisfies
\be
\|\fG_{\wh{\va}}\|_{L^{\ift}(\Ga_{1,L_0})}\leq C_0.\label{fGboundbound}
\ee
for some $ C_0>0 $. There exist $ 0<\wh{\va}_1,\eta_1<1 $, depending only on $ \cA $ and $ C_0 $ such that the following properties hold. If $ 0<\wh{\va}<\wh{\va}_1 $ and 
\be
E_{\wh{\va}}(\fG_{\wh{\va}},\Ga_{1,L_0})\leq\eta\log\f{1}{\wh{\va}},\label{GvaboundaryH1}
\ee
for $ 0<\eta<\eta_1 $, then there exist $ z_-\in(-L_0+h,-L_0+2h) $, $ z_+\in(L_0-h,L_0-2h) $, $ \V_{\wh{\va}}^{(h)}\in H^1(\pa B_1^2\times(z_-,z_+),\cN) $, $ \W_{\wh{\va}}^{(h)}\in H^1(A_{1-h,1}^2\times(z_-,z_+),\Ss_0) $, and $ C>0 $ depending only on $ \cA $ and $ C_0 $ such that the following assertions are true.
\begin{enumerate}
\item For $ \HH^2 $-a.e. $ x\in \pa B_1^2\times(z_-,z_+) $, 
$$
\W_{\wh{\va}}^{(h)}(x)=\fG_{\wh{\va}}(x)\text{ and }\W_{\wh{\va}}^{(h)}((1-h)x)=\V_{\wh{\va}}^{(h)}(x).
$$
\item $ \V_{\wh{\va}}^{(h)} $ and $ \W_{\wh{\va}}^{(h)} $ satisfy
\begin{align}
E_{\wh{\va}}(\W_{\wh{\va}}^{(h)},A_{1-h,1}^2\times(z_-,z_+))&\leq ChE_{\wh{\va}}(\fG_{\wh{\va}},\Ga_{1,L_0}),\label{Wvaest2}\\
E_{\wh{\va}}(\V_{\wh{\va}}^{(h)},\pa B_{1-h}^2\times(z_-,z_+))&\leq CE_{\wh{\va}}(\fG_{\wh{\va}},\Ga_{1,L_0}).\label{Vvaes2}
\end{align}
\item $ \W_{\wh{\va}}^{(h),\pm}:=\W_{\wh{\va}}^{(h)}|_{A_{1-h,1}^2\times\{z_{\pm}\}}\in H^1(A_{1-h,1}^2\times\{z_{\pm}\},\Ss_0) $ satisfies
\be
E_{\va}(\W_{\wh{\va}}^{(h),\pm},A_{1-h,1}^2\times\{z_{\pm}\})\leq C E_{\wh{\va}}(\fG_{\wh{\va}},\Ga_{1,L_0}).\label{WpluesminusC2}
\ee
\end{enumerate}
\end{lem}
\begin{proof}
The proof of this lemma is similar to Lemma \ref{Luckhauslemma} and we will give more details. We also divide it into several steps. In the constructions of good grids on a cylinder, we use the arguments in \cite{BS24}.\smallskip

\noindent\textbf{\underline{Step 1.} Choices of $ \{z_{\pm}\} $.} By using average arguments, we can find $ z_-\in(-L_0+h,-L_0+2h) $ and $ z_+\in(L_0-2h,L_0-h) $, such that for any $ z\in\{z_{\pm}\} $, $ \fG_{\wh{\va}}|_{\pa B_1^2\times\{z\}}\in H^1(\pa B_1^2\times\{z\},\Ss_0) $ and the following property holds
\be
E_{\wh{\va}}(\fG_{\wh{\va}},\pa B_1^2\times\{z\})\leq 2h^{-1}E_{\wh{\va}}(\fG_{\wh{\va}},\Ga_{1,L_0}).\label{zpluszmius}
\ee
Define $ L=(z_+-z_-)/2 $, $ \U_{\wh{\va}}(y):=\fG_{\wh{\va}}((z_+-z_-)\e^{(3)}/2+y) $ for $ y\in\Ga_{1,L} $. \smallskip

\noindent\textbf{\underline{Step 2.} Constructions of good family of grids of size $ h $ on $ \ol{\Ga_{1,L}} $.} For $ 0<\wh{\va}<\wh{\va}_1 $, where $ 0<\wh{\va}_1<1 $ is to be chosen, we claim that there exists a family of good grids of size $ h $ on $ \ol{\Ga_{1,L}} $, denoted by 
$$ 
\cG:=\{\cG^{\wh{\va},h}\}_{0<\wh{\va}<\wh{\va}_1,1/20<h\leq 1/2}=\{K_{i,j}^{\wh{\va},h}\}_{1\leq i\leq k_j^{\wh{\va},h},j=0,1,2}, 
$$
with an absolute constant $ C_{\cG}>0 $, such that the following properties hold.
\begin{enumerate}
\item $ \cG $ form a decomposition of $ \ol{\Ga_{1,L}} $, namely, $
\ol{\Ga_{1,L}}=\cup_{j=0}^2\cup_{i=1}^{k_i^{\wh{\va},h}}K_{i,j}^{\wh{\va},h} $, and the $ j $-skeleton of the grid $ \cG $ is defined by $ R_j^{\wh{\va},h}:=\cup_{i=1}^{k_j^{\wh{\va},h}}K_{i,j}^{\wh{\va},h} $ for $ j=0,1,2 $.
\item For any $ 0\leq\wh{\va}\leq\wh{\va}_1 $, $ h\in(1/20,1/2] $, $ j=1,2 $, and $ i=1,2,...,k_{j}^{\wh{\va},h} $, there exists a bilipschitz homeomorphism $ \vp_{i,j}^{\wh{\va},h}:K_{i,j}^{\wh{\va},h}\to B_{h}^j $ such that 
\be
\|\na_{K_{i,j}^{\wh{\va},h}}\vp_{i,j}^{\wh{\va},h}\|_{L^{\ift}(K_{i,j}^{\wh{\va},h})}+\|\na(\vp_{i,j}^{\wh{\va},h})^{-1}\|_{L^{\ift}(B_{h}^j)}\leq C_{\cG}.\label{G1Cylinder}
\ee
\item For all $ p\in\{1,2,...,k_1\} $,
\be
\#\{q\in\{1,2,...,k_2\}:K_{p,1}^{\wh{\va},h}\subset\pa K_{q,2}^{\wh{\va},h}\}\leq C_{\cG}.\label{G2Cylinder}
\ee
\item The following estimates of energy hold.
\begin{align}
E_{\wh{\va}}(\U_{\wh{\va}},R_1^{\wh{\va},h})&\leq C_{\cG}h^{-1}E_{\wh{\va}}(\fG_{\wh{\va}},\Ga_{1,L_0}),\label{G3Cylinder}\\
\int_{R_1^{\wh{\va},h}}f_b(\U_{\wh{\va}})\ud\HH^1&\leq C_{\cG}h^{-1}\int_{\Ga_{1,L_0}}f_b(\fG_{\wh{\va}})\ud\HH^2.
\end{align}
\end{enumerate}

Since $ L_0\geq 2 $ and $ h\in(1/20,1/2] $, we have $ L\geq 1 $. Set $ n:=\lfloor 1/h\rfloor\geq 1 $. For any $ k\in[0,n-1] $ and $ j\in[0,2\lfloor L\rfloor] $, we define 
\begin{align*}
\ell_k&:=\ol{\(\cos\f{2k\pi}{n},\sin\f{2k\pi}{n},-L\)\(\cos\f{2k\pi}{n},\sin\f{2k\pi}{n},L\)}\subset\ol{\Ga_{1,L}},\\
x_j^k&:=\(\cos\f{2k\pi}{n},\sin\f{2k\pi}{n},-L+\f{Lj}{\lfloor L\rfloor n}\)\subset\ol{\Ga_{1,L}}.
\end{align*}
For each $ k\in[0,n-1] $, $ j\in[0,2\lfloor L\rfloor n-1] $, we define $ \ga_j^k $ the union of two geodesics on $ \ol{\Ga_{1,L}} $, i.e., $ \ga_j^k=\ga_{j,1}^k\cup\ga_{j,2}^k $, where $ \ga_{j,1}^k $ connects $ x_j^k,x_{j+1}^{k+1} $, and $ \ga_{j,2}^k $ connects $ x_{j+1}^k,x_j^{k+1} $. Let $ c_{k,z} $ with $ z\in\{-L,L\} $ be arcs on $ \pa B_1^2\times\{z\} $ connecting $ x_{2\lfloor L\rfloor n}^k,x_{2\lfloor L\rfloor n}^{k+1} $ if $ z=L $, and $ x_0^k,x_0^{k+1} $ if $ z=-L $. Observe that $ \ell_k $, $ \ga_j^k $, and $ c_{k,z} $ provide a decomposition of $ \ol{\Ga_{1,L}} $ into mutually disjoint $ 0,1,2 $-dimensional cells, denoted by $
\cF:=\{\cF^{h}\}_{1/20<h\leq 1/2}=\{F_{i,j}^{h}\}_{1\leq i\leq k_j^{h},j=0,1,2} $, such that $
\ol{\Ga_{1,L}}=\cup_{i=0}^2\bigcup_{j=1}^{k_i^h}F_{i,j}^{h} $, and for any $ h\in(1/20,1/2] $, $ j=1,2 $, $ i=1,2,...,k_j^h $, there exists a bilipschitz homeomorphism $ \vp_{i,j}^h:F_{i,j}^h\to B_h^j $ satisfying
\be
\|\na_{F_{i,j}^{h}}\vp_{i,j}^{h}\|_{L^{\ift}(F_{i,j}^h)}+\|\na(\vp_{i,j}^{h})^{-1}\|_{L^{\ift}(B_h^j)}\leq C_1\label{G1Cylinder2}
\ee
for some absolute constant $ C_1>0 $. We also use $ \{R_j^h\}_{1/20<h\leq 1/2,j=0,1,2} $ to denote $ j $-skeleton. Denote $ \mathrm{SO}(2) $ as the subgroup $ \mathrm{SO}(3) $ containing rotations with respect to the axis $ \{0\}\times\R $. Thus, for each $ \w\in\mathrm{SO}(2) $, we obtain a class of decompositions of $ \ol{\Ga_{1,L}} $ given by
\be
\ol{\Ga_{1,L}}=\bigcup_{i=0}^2\bigcup_{j=1}^{k_i^h}\w(F_{i,j}^{\w,h}).\label{decomGa}
\ee
By using the integral formula
\begin{align*}
\int_{\mathrm{SO}(2)}\ud\w\int_{\w(R_1^h\cap\Ga_{1,L})}e_{\wh{\va}}(\U_{\wh{\va}},\Ga_{1,L})\ud\HH^1&=\f{\HH^1(R_1^h\cap\Ga_{1,L})}{\HH^2(\Ga_{1,L})}E_{\wh{\va}}(\U_{\wh{\va}},\Ga_{1,L}),\\
\int_{\mathrm{SO}(2)}\ud\w\int_{\w(R_1^h\cap\Ga_{1,L})}f_b(\U_{\wh{\va}})\ud\HH^1&=\f{\HH^1(R_1^h\cap\Ga_{1,L})}{\HH^2(\Ga_{1,L})}\int_{\Ga_{1,L}}f_b(\U_{\wh{\va}})\ud\HH^2,
\end{align*}
we can use the average arguments to find a $ \w=\w(\wh{\va})\in\mathrm{SO}(2) $, such that
\begin{align*}
\int_{\w(R_1^h\cap\Ga_{1,L})}e_{\wh{\va}}(\U_{\wh{\va}},\Ga_{1,L})\ud\HH^1&\leq C_2h^{-1}E_{\wh{\va}}(\fG_{\wh{\va}},\Ga_{1,L_0}),\\
\int_{\w(R_1^h\cap\Ga_{1,L})}f_b(\U_{\wh{\va}})\ud\HH^1&\leq C_2h^{-1}\int_{\Ga_{1,L_0}}f_b(\fG_{\wh{\va}})\ud\HH^2,
\end{align*}
for some absolute constant $ C_2>0 $. In view of \eqref{zpluszmius}, we define $
\cG:=\{\w(F_{i,j}^{h})\}_{1\leq i\leq k_j^{h},j=0,1,2} $ and it is a family of good grids of size $ h $ on $ \ol{\Ga_{1,L}} $. \smallskip

\noindent\textbf{\underline{Step 3.} Proving that $ \U_{\wh{\va}} $ is sufficiently close to $ \cN $ on $ R_1^{\wh{\va},h} $ as $ \wh{\va}\to 0^+ $.} Indeed, we show that for any $ h\in(1/20,1/2] $,
\be
\lim_{\wh{\va}\to 0^+}\sup_{x\in R_1^{\wh{\va},h}}\dist(\U_{\wh{\va}}(x),\cN)=0.\label{smallR1h}
\ee
Now, we choose $ \psi:[0,+\ift)\to\R $ satisfying \eqref{psifunction1}. Define $ F(s):=\int_0^s\psi^{1/6}(t)\ud t $ and $ d_{\wh{\va}}:=\dist(\U_{\wh{\va}},\cN) $. We have 
\be |\na_{R_1^{\wh{\va},h}}d_{\wh{\va}}|\leq|\na_{R_1^{\wh{\va},h}}\U_{\wh{\va}}|,\label{dvapsiva}
\ee
and then $ d_{\wh{\va}}\in H^1(\Lda_{1,L},\R_+) $. Thus, we have
\be
\begin{aligned}
\int_{R_1^{\wh{\va},h}}|\na_{R_1^{\wh{\va},h}}F(d_{\wh{\va}})|^{3/2}\ud\HH^1&\leq\int_{R_1^{\wh{\va},h}}|\na_{R_1^{\wh{\va},h}}d_{\wh{\va}}|^{3/2}\psi^{1/4}(d_{\wh{\va}})\ud\HH^1\\
&\leq\int_{R_1^{\wh{\va},h}}\(\f{\wh{\va}^{1/2}}{2}|\na_{R_1^{\wh{\va},h}}d_{\wh{\va}}|^2+\f{1}{\wh{\va}^{3/2}}\psi(d_{\wh{\va}})\)\ud\HH^1\\
&\leq\wh{\va}^{1/2}E_{\wh{\va}}(\U_{\wh{\va}},R_1^{\wh{\va},h})\\
&\leq C_{\cG}\wh{\va}^{1/2}h^{-1}E_{\wh{\va}}(\U_{\wh{\va}},\Ga_{1,L})\\
&\leq C_{\cG}\wh{\va}^{1/2}h^{-1}\log\f{1}{\wh{\va}},
\end{aligned}\label{h-1vahat}
\ee
where for the second inequality, we have used Young inequality, for the third inequality, we have used \eqref{dvapsiva}, for the forth inequality, we have used \eqref{G2Cylinder}, and for the last inequality, we have used \eqref{GvaboundaryH1} by choosing $ 0<\eta_1<1 $. 

For $ 1 $-cell $ K_{i,j}^{\wh{\va},h} $ of $ \cG^{\wh{\va},h} $, by using Sobolev embedding theorem and the fact that $ h\geq 1/20 $, we deduce from \eqref{h-1vahat} that
\be
\begin{aligned}
\(\osc_{R_1^{\wh{\va},h}}F(d_{\wh{\va}})\)^{3/2}&\leq\sup_{1\leq i\leq k_1^{\wh{\va},h}}\(\osc_{K_{i,1}^{\wh{\va},h}}F(d_{\wh{\va}})\)^{3/2}\\
&\leq\sup_{1\leq j\leq k_1^{\wh{\va},h}}\([F(d_{\wh{\va}})]_{C^{0,1/3}(K^{\wh{\va},h})}^{3/2}h^{1/2}\)\\
&\leq Ch^{1/2}\int_{R_1^{\wh{\va},h}}|\na F(d_{\wh{\va}})|^{3/2}\ud\HH^1\\
&\leq C\wh{\va}^{1/2}h^{-1/2}\log\f{1}{\wh{\va}}\to 0,
\end{aligned}\label{ocses2}
\ee
as $ \wh{\va}\to 0^+ $. Also, in view of \eqref{GvaboundaryH1} and \eqref{G3Cylinder}, we obtain 
\be
\sup_{1\leq i\leq k_1^{\wh{\va},h}}\dashint_{K_{i,1}^{\wh{\va},h}}\psi(d_{\wh{\va}})\ud\HH^1\leq\f{6}{h}\int_{R_1^{\wh{\va},h}}f_b(\U_{\wh{\va}})\ud\HH^1\leq 6C_{\cG}\wh{\va}^2h^{-2}\log\f{1}{\wh{\va}}\to 0\label{intes2psi}
\ee
as $ \wh{\va}\to 0^+ $. By \eqref{fGboundbound}, there exists $ \kappa>0 $, depending only on $ \cA $ and $ C_0 $ such that $ \|\U_{\wh{\va}}\|_{L^{\ift}(\Ga_{1,L})}\leq\kappa $. For $ \delta\in(0,\kappa) $, we choose $
\psi_*(\delta)=\inf_{\delta\leq s\leq\kappa}\psi(s) $. Using the notations above, we have, for any $ K_{i,1}^{\wh{\va},h}\subset R_1^{\wh{\va},h} $ with $ i=1,2,...,k_1^{\wh{\va},h} $, there holds
\be
\f{\HH^1(\{d_{\wh{\va}}\geq\delta\}\cap K_{i,1}^{\wh{\va},h})}{\HH^1(K_{i,1}^{\wh{\va},h})}\psi_*(\delta)\leq\f{1}{\HH^1(K_{i,1}^{\wh{\va},h})}\int_{\{d_{\wh{\va}}\geq\delta\}\cap K_{i,1}^{\wh{\va},h}}\psi(d_{\wh{\va}})\ud\HH^1\leq\dashint_{K_{i,1}^{\wh{\va},h}}\psi(d_{\wh{\va}})\ud\HH^1.\label{H11es}
\ee
Consequently, we get
\begin{align*}
\sup_{1\leq i\leq k_1^{\wh{\va},h}}\int_{K_{i,1}^{\wh{\va},h}}d_{\wh{\va}}\ud\HH^1&=\sup_{1\leq j\leq k_1^{\wh{\va},h}}\left\{\f{1}{\HH^1(K_{i,1}^{\wh{\va},h})}\(\int_{\{d_{\wh{\va}}<\delta\}\cap K_{i,1}^{\wh{\va},h}}+\int_{\{d_{\wh{\va}}\geq\delta\}\cap K_{i,1}^{\wh{\va},h}}\)\psi(d_{\wh{\va}})\ud\HH^1\right\}\\
&\leq\sup_{1\leq j\leq k_1^{\wh{\va},h}}\(\f{\HH^1(\{d_{\wh{\va}}<\delta\}\cap K_{i,1}^{\wh{\va},h})}{\HH^1(K_{i,1}^{\wh{\va},h})}\delta+\f{\HH^1(\{d_{\wh{\va}}\geq\delta\}\cap K_{i,1}^{\wh{\va},h})}{\HH^1(K_{i,1}^{\wh{\va},h})}\kappa\)\\
&\leq\delta+\f{\kappa}{\psi_*(\delta)}\dashint_{K_{i,1}^{\wh{\va},h}}\psi(d_{\wh{\va}})\ud\HH^1\\
&\leq\delta+\f{C\kappa}{\psi_*(\delta)}\wh{\va}^2h^{-2}\log\f{1}{\wh{\va}}\to 0,
\end{align*}
as $ \wh{\va}\to 0^+ $, where for the second inequality, we have used \eqref{H11es} and for the last one, we have used \eqref{intes2psi}. This, together with \eqref{ocses2} and the property that $ F $ is continuous and strictly increasing, implies that $ \lim_{\wh{\va}\to 0^+}\osc_{R_1^{\wh{\va},h}}d_{\wh{\va}}=0 $ and \eqref{smallR1h} follows directly.
\smallskip

\noindent\textbf{\underline{Step 4.} Constructions of $ \V_{\wh{\va}}^{(h)} $, $ \W_{\wh{\va}}^{(h)} $ and $ \W_{\wh{\va}}^{(h),\pm} $.} By \eqref{smallR1h}, there exists $ 0<\wh{\va}_1<1 $ such that
\be
\sup_{0<\wh{\va}<\wh{\va}_1}\sup_{x\in R_1^{\wh{\va},h}}\dist(\U_{\wh{\va}},\cN)<\delta_0,\label{delta0Uva2}
\ee
where $ \delta_0 $ is given by Corollary \ref{fBc}. We can define $ \V_{\wh{\va}}^{(h)}:=\varrho(\U_{\wh{\va}}) $ on $ R_1^{\wh{\va},h} $. By \eqref{G3Cylinder}, we have, $ \V_{\wh{\va}}^{(h)}\in H^1(R_1^{\wh{\va},h},\cN) $ and satisfies the estimate
\be
\sup_{x\in R_1^{\wh{\va},h}}|\U_{\wh{\va}}(x)-\V_{\wh{\va}}^{(h)}(x)|<\delta_0.\label{R1vauv2}
\ee
Next, we claim that if $ 0<\eta_1<1 $ is chosen sufficiently small, then for any $ i=1,2,...,k_2^{\wh{\va},h} $, $ [\V_{\wh{\va}}^{(h)}|_{\pa K_{i,2}^{\wh{\va},h}}]_{\cN}=\rH_0 $. Indeed, if such claim is not true, there exists some $ i_0=1,2,...,k_2^{\wh{\va},h} $ and a $ 2 $-cell $ K_{i_0,2}^{\ol{\va},h}\in R_2^{\wh{\va},h} $ such that $ [\varrho\circ\U_{\wh{\va}}|_{\pa K_{i_0,2}^{\ol{\va},h}}]_{\cN}\neq\rH_0 $. In view of \eqref{G1Cylinder}, there exists a bilipschitz homeomorphism $ \vp_{i_0,2}^{\ol{\va}}:K_{i_0,2}^{\wh{\va},h}\to B_{h}^2 $ satisfying \eqref{G1}. Without loss of generality, we will assume that $ K_{i_0,2}^{\wh{\va},h}=B_h^2 $. By \eqref{delta0Uva2}, if we further choose smaller $ 0<\delta_0<1 $, there hold $ \U_{\wh{\va}}\notin\cC_1\cup\cC_2 $ on $ \pa B_{h}^2 $ and moreover, $
\phi_0^2(\U_{\wh{\va}},\pa B_{h}^2)>1/2 $
for $ \wh{\va}\in(0,\wh{\va}_1) $. Since $ h\in(1/20,1/2] $, we can choose $ 0<\wh{\va}_1<1/1600 $ such that $ 0<\wh{\va}<h/80 $. Now, we can apply Corollary \ref{lowerboundcor} to obtain
$$
E_{\wh{\va}}(\U_{\wh{\va}},B_{h}^2)+h E_{\wh{\va}}(\U_{\wh{\va}},\pa B_{h}^2)\geq C\log\f{h}{\wh{\va}}-C'.
$$
As a result, \eqref{G3Cylinder} implies that
$$
E_{\wh{\va}}(\U_{\wh{\va}},\Ga_{1,L})\geq C(E_{\wh{\va}}(\U_{\wh{\va}},B_h^2)+hE_{\wh{\va}}(\U_{\wh{\va}},\pa B_h^2))\geq C|\log(h/\wh{\va})|-C'\geq C\log\f{1}{\wh{\va}}-C'.
$$
If $ \eta_1 $ is sufficiently small, there is a contradiction and we can complete the proof of this claim.

By this claim, for any $ K_{i,2}^{\wh{\va},h}\subset R_2^{\wh{\va},h} $, we can apply Lemma \ref{k2trivial} to construct $ \V_{\wh{\va},K_{i,2}^{\wh{\va},h}}\in H^1(K_{i,2}^{\wh{\va},h},\cN) $ such that $ \V_{\wh{\va},K_{i,2}^{\wh{\va},h}}=\V_{\wh{\va}}^{(h)} $ on $ \pa K_{i,2}^{\wh{\va},h} $ and
\be
\int_{K_{i,2}^{\wh{\va},h}}|\na_{K_{i,2}^{\wh{\va},h}}\V_{\wh{\va},K_{i,2}^{\wh{\va},h}}|^2\ud\HH^2\leq Ch\int_{\pa K_{i,2}^{\wh{\va},h}}|\na_{\pa K_{i,2}^{\wh{\va},h}}\V_{\wh{\va}}|^2\ud\HH^1.\label{K2jvahVva}
\ee
Define $ \V_{\wh{\va}}^{(h)}:\Ga_{1,L}\to\cN $ by
$$ 
\V_{\wh{\va}}^{(h)}(x)=\left\{\begin{aligned}
&(\varrho\circ\U_{\wh{\va}})(x)&\text{ if }&x\in R_1^{\wh{\va},h}\\
&\V_{\wh{\va},K_{i,2}^{\wh{\va},h}}(x)&\text{ if }&x\in K_{i,2}^{\wh{\va}}\subset R_2^{\wh{\va},h},\,\,i=1,2,...,k_2^{\wh{\va},h}.
\end{aligned}\right.
$$
By the construction of $ \V_{\wh{\va}}^{(h)} $, we have
\begin{align*}
\int_{\Ga_{1,L}}|\na_{\Ga_{1,L}}\V_{\wh{\va}}^{(h)}|^2\ud\HH^2&\leq\sum_{j=1}^{k_2^{\wh{\va},h}}\int_{K_{i,2}^{\wh{\va},h}}|\na_{K_{i,2}^{\wh{\va},h}}\V_{\wh{\va}}^{(h)}|^2\ud\HH^2\\
&\leq Ch\sum_{i=1}^{k_2^{\wh{\va},h}}\int_{\pa K_{i,2}^{\wh{\va},h}}|\na_{\pa K_{i,2}^{\wh{\va},h}}\V_{\wh{\va}}^{(h)}|^2\ud\HH^2\\
&\leq Ch\int_{R_1^{\wh{\va},h}}|\na_{R_1^{\wh{\va},h}}\V_{\wh{\va}}^{(h)}|^2\ud\HH^1\\
&\leq Ch\int_{R_1^{\wh{\va},h}}|\na_{R_1^{\wh{\va},h}}\U_{\wh{\va}}|^2\ud\HH^1\leq CE_{\wh{\va}}(\fG_{\va},\Ga_{1,L_0}),
\end{align*}
where for the second inequality, we have used \eqref{K2jvahVva}, for the third inequality, we have used \eqref{G2Cylinder}, for the forth inequality, we have used the construction of $ \V_{\wh{\va}}^{(h)} $, and for the last inequality, we have used \eqref{G3Cylinder}. 
Such $ \V_{\wh{\va}}^{(h)} $ satisfies \eqref{Vvaes2}.

Finally, we will construct $ \W_{\wh{\va}}^{(h)} $. For fixed $ 0<\wh{\va}<\wh{\va}_1 $ and $ h\in(1/20,1/2] $, the grid $ \cG^{\wh{\va},h} $ induces a grid $ \wh{\cG}^{\wh{\va},h,*} $ on $ A_{1,1-h}^2\times(-L,L) $, whose cells are given by
$$
K_{i,j+1}^{\wh{\va},h,*}=\left\{x\in\R^3:1-h\leq|x|\leq 1,\,\,\f{x}{|x|}\in K_{i,j}^{\wh{\va},h}\right\}\text{ for any }K_{i,j}^{\wh{\va},h}\in \cG^{\wh{\va},h},
$$
with $ j=0,1,2 $ and $ i=1,2,...,k_j^{\wh{\va},h} $. Define $
R_j^{\wh{\va},h,*}:=\cup_{i=1}^{k_{j}^{\wh{\va},h}}K_{i,j+1}^{\wh{\va},h,*} $. Now, we define $ \W_{\wh{\va}}^{(h)} $ as follows.
\begin{enumerate}
\item If $ x\in\Ga_{1,L}\cup\Ga_{1-h,L} $, $ \W_{\wh{\va}}^{(h)}(\cdot) $ is given by $ \U_{\wh{\va}}(\cdot) $ and $ \V_{\wh{\va}}^{(h)}(\cdot/(1-h)) $.
\item If $ x\in R_1^{\wh{\va},h}\cup R_1^{\wh{\va},h,*}, $
\be
\W_{\wh{\va}}^{(h)}(x)=\f{1-|x|}{h}\U_{\wh{\va}}\(\f{x}{|x|}\)+\f{h-1+|x|}{h}\V_{\wh{\va}}^{(h)}\(\f{x}{|x|}\).\label{homogeDef2}
\ee
\item For $ 3 $-cell $ K_{i,3}^{\wh{\va},h,*} $ of $ \cG^{\wh{\va},h,*} $, the $ \W_{\wh{\va}}^{(h)} $ is extended homogeneously. Indeed,
$$
\W_{\wh{\va}}^{(h)}(x)=\W_{\wh{\va}}^{(h)}\circ(\Phi_{i,3}^{\wh{\va},h,*})^{-1}\(\f{\Phi_{i,3}^{\wh{\va},h,*}(x)}{|\Phi_{i,3}^{\wh{\va},h,*}(x)|}\),
$$
where $ \Phi_{i,3}^{\wh{\va},h,*}:K_{3,i}^{\wh{\va},h,*}\to B_1 $ is a bilipschitz map.
\end{enumerate}
Consequently, $ \W_{\wh{\va}}^{(h)}\in H^1(K_{i,3}^{\wh{\va},h,*},\Ss_0) $ and $ \W_{\wh{\va}}^{(h)}\in H^1(A_{1-h,1}^2,\Ss_0) $. In view of \eqref{G1Cylinder} and \eqref{G2Cylinder}, we have
\begin{align*}
&E_{\wh{\va}}(\W_{\wh{\va}}^{(h)},A_{1-h,1}^2\times(-L,L))\leq Ch\sum_{i=1}^{k_3^{\wh{\va},h,*}}E_{\wh{\va}}(\W_{\wh{\va}}^{(h)},\pa K_{i,3}^{\wh{\va},h,*})\\
&\quad\quad\quad\quad\leq Ch(E_{\wh{\va}}(\W_{\wh{\va}}^{(h)},\Ga_{1,L})+ E_{\wh{\va}}(\W_{\wh{\va}}^{(h)},\Ga_{1-h,L})+E_{\wh{\va}}(\W_{\wh{\va}}^{(h)},R_1^{\wh{\va},h,*})).
\end{align*}
We claim that
\be
E_{\wh{\va}}(\W_{\wh{\va}}^{(h)},R_1^{\wh{\va},h,*})\leq C(\wh{\va}^2h^{-2}+1)E_{\wh{\va}}(\W_{\wh{\va}}^{(h)},\Ga_{1,L}).\label{ClaimWvah}
\ee
If this claim is true, then
\begin{align*}
E_{\wh{\va}}(\W_{\wh{\va}}^{(h)},A_{1-h,1}^2\times(-L,L))&\leq Ch(\wh{\va}^2h^{-2}+1)E_{\wh{\va}}(\U_{\wh{\va}},\Ga_{1,L})\\
&\leq Ch(\wh{\va}^2h^{-2}+1)E_{\wh{\va}}(\fG_{\wh{\va}},\Ga_{1,L_0}),
\end{align*}
which implies \eqref{Wvaest2}. Moreover, since $ A_{1-h,1}^2\times\{z_{\pm}\}\subset R_1^{\wh{\va},h,*} $, it also shows \eqref{WpluesminusC2}. It remains to show the claim \eqref{ClaimWvah}. The proof of this follows from direct computations. Indeed, by \eqref{fBconvex}, \eqref{R1vauv2} and \eqref{homogeDef2}, we have
\be
\int_{R_1^{\wh{\va},h,*}}f_b(\W_{\wh{\va}}^{(h)})\ud\HH^2\leq Ch\int_{R_1^{\wh{\va},h,*}}f_b(\U_{\wh{\va}})\ud\HH^2.\label{fW2}
\ee
Also, by \eqref{fBnond} and the constructions of $ \W_{\wh{\va}}^{(h)} $, We can obtain
$$
\int_{R_1^{\wh{\va},h,*}}|\na_{R_1^{\wh{\va},h,*}}\W_{\wh{\va}}^{(h)}|^2\ud\HH^2\leq Ch^{-1}\int_{R_1^{\wh{\va},h}}|\U_{\wh{\va}}-\V_{\wh{\va}}^{(h)}|^2\ud\HH^1\leq Ch^{-1}\int_{R_1^{\wh{\va},h}}f_b(\U_{\wh{\va}})\ud\HH^1.
$$
This, together \eqref{fW2}, implies that
$$
E_{\wh{\va}}(\W_{\wh{\va}}^{(h)},R_1^{\wh{\va},h,*})\leq C(h^{-1}+\wh{\va}^{-2}h)\int_{R_1^{\wh{\va},h}}f_b(\U_{\wh{\va}})\ud\HH^1\leq C(h^{-2}+\wh{\va}^{-2})\int_{\Ga_{1,L}}f_b(\U_{\wh{\va}})\ud\HH^2,
$$
which complete the proof of the claim.
\end{proof}

\begin{lem}\label{Luckhauscylinder2}
Let $ \delta\in(0,1/2] $ and $ r>0 $. Assume that for $ 0<\ol{\va}<1 $, $ \fG_{\delta,\ol{\va}}\in H^1(\Ga_{\delta,1},\Ss_0) $ satisfies
\be
\|\fG_{\delta,\ol{\va}}\|_{L^{\ift}(\Ga_{\delta,1})}\leq C_0\label{fGboundbound2}
\ee
for some $ C_0>0 $. There exist $ 0<\eta_1,\wh{\va}_1<1 $, depending only on $ \cA $ and $ C_0 $ such that the following properties hold. If $ 0<\ol{\va}<\wh{\va}_1\delta $, and 
\begin{align}
E_{\ol{\va}}(\fG_{\delta,\ol{\va}},\Ga_{\delta,1})&\leq\eta\log\f{\delta}{\ol{\va}},\label{GvaboundaryH12}
\end{align}
with $ 0<\eta<\eta_1 $, then there exist $ z_-\in(-1+\delta/2,-1+\delta) $, $ z_+\in(1-\delta,1-\delta/2) $, $ \V_{\delta,\ol{\va}}\in H^1(\pa B_{\delta}^2\times(z_-,z_+),\cN) $, $ \W_{\delta,\ol{\va}}\in H^1(A_{\delta/2,\delta}^2\times(z_-,z_+),\Ss_0) $, and $ C>0 $ depending only on $ \cA,C_0 $ such that the following assertions are true.
\begin{enumerate}
\item For $ \HH^2 $-a.e. $ x\in \pa B_{\delta}^2\times(z_-,z_+) $, 
$$
\W_{\delta,\ol{\va}}(x)=\fG_{\delta,\ol{\va}}(x)\text{ and }\W_{\delta,\ol{\va}}\(\f{x}{2}\)=\V_{\delta,\ol{\va}}(x).
$$
\item $ \V_{\delta,\ol{\va}} $ and $ \W_{\delta,\ol{\va}} $ satisfy
\begin{align}
E_{\ol{\va}}(\W_{\delta,\ol{\va}},A_{\delta/2,\delta}^2\times(z_-,z_+))&\leq C\delta E_{\ol{\va}}(\fG_{\delta,\ol{\va}},\Ga_{\delta,1}),\label{Wvaest22}\\
E_{\ol{\va}}(\V_{\wh{\va}},\pa B_{\delta}^2\times(z_-,z_+))&\leq CE_{\ol{\va}}(\fG_{\delta,\ol{\va}},\Ga_{\delta,1}).\label{Vvaes22}
\end{align}
\item $ \W_{\delta,\ol{\va}}^{\pm}:=\W_{\delta,\ol{\va}}|_{A_{\delta/2,\delta}^2\times\{z_{\pm}\}}\in H^1(A_{\delta/2,\delta}^2\times\{z_{\pm}\},\Ss_0) $ satisfies
\be
E_{\va}(\W_{\delta,\ol{\va}}^{\pm},A_{\delta/2,\delta}^2\times\{z_{\pm}\})\leq C E_{\ol{\va}}(\fG_{\delta, \wh{\va}},\Ga_{\delta,1}).\label{WpluesminusC22}
\ee
\end{enumerate}
\end{lem}
\begin{proof}
Define $ \wh{\va}=\ol{\va}/\delta $, $ \fG_{\wh{\va}}(x):=\fG_{\delta,\ol{\va}}(\delta x) $ for $ x\in\Ga_{1,1/\delta} $. Thus we have
$$
\|\fG_{\wh{\va}}\|_{L^{\ift}(\Ga_{1,1/\delta})}\leq C_0.
$$
Let $ 0<\eta_1,\wh{\va}_1<1 $ be given by Lemma \ref{Luckhauscylinder1}. We deduce from \eqref{fGboundbound2} and \eqref{GvaboundaryH12} that
$$
E_{\wh{\va}}(\fG_{\wh{\va}},\Ga_{1,1/\delta})=E_{\ol{\va}}(\fG_{\delta,\ol{\va}},\Ga_{\delta,1})\leq\eta\log|\wh{\va}|.
$$
By applying Lemma \ref{Luckhauscylinder1}, to $ L_0=1/\delta $, $ h=1/2 $, and $ \fG_{\wh{\va}} $, there exist $ \zeta_-\in(-1/\delta+1/2,-1/\delta+1) $, $ \zeta_+\in(1/\delta-1,1/\delta-1/2) $, $ \W_{\wh{\va}}^{(1/2)}\in H^1(A_{1/2,1}^2\times(\zeta_-,\zeta_+),\Ss_0) $, $ \V_{\wh{\va}}^{(1/2)}\in H^1(\pa B_1^2\times(\zeta_-,\zeta_+),\cN) $, $ \W_{\wh{\va}}^{(1/2),\pm}\in H^1(A_{1/2,1}^2\times\{\zeta_{\pm}\}),\Ss_0) $ satisfying \eqref{Wvaest2}, \eqref{Vvaes2}, and \eqref{WpluesminusC2}. Scaling back to $ \Ga_{\delta,1} $, namely, choosing $ z_{\pm}=\delta\zeta_{\pm} $, $ \W_{\delta,\ol{\va}}(\cdot)=\W_{\wh{\va}}^{(1/2)}(\cdot/\delta) $, $ \V_{\delta,\ol{\va}}(\cdot)=\V_{\wh{\va}}^{(1/2)}(\cdot/\delta) $, and $ \W_{\delta,\ol{\va}}^{\pm}(\cdot)=\W_{\wh{\va}}^{(1/2),\pm}(\cdot/\delta) $, we can complete the proof.
\end{proof}

\begin{lem}\label{chooseseveralcases}
Let $ \delta\in(0,1/2] $. Assume that for $ 0<\ol{\va}<1 $, $ \fG_{\delta,\ol{\va}}\in H^1(\Ga_{\delta,1},\Ss_0) $ satisfies
\begin{align}
\|\fG_{\delta,\ol{\va}}\|_{L^{\ift}(\pa\Lda_{\delta,1})}&\leq C_0,\label{fGboundbound23}\\
E_{\ol{\va}}
(\fG_{\delta,\ol{\va}},B_{\delta}^2\times\{\pm 1\})&\leq C_0\log\f{\delta}{\ol{\va}},\label{GvaboundaryH13M}
\end{align}
for some $ C_0>0 $. There exists $ 0<\eta_1,\wh{\va}_1<1 $ depending only on $ \cA $ and $ C_0 $ such that if $ 0<\ol{\va}<\wh{\va}_1\delta $,
\be
E_{\ol{\va}}(\fG_{\delta,\ol{\va}},\Ga_{\delta,1})\leq\eta\log\f{\delta}{\ol{\va}},\label{GvaboundaryH13}
\ee
with $ 0<\eta<\eta_1 $, and $ \U_{\delta,\ol{\va}} $ is a minimizer of $ E_{\ol{\va}}(\cdot,\Lda_{\delta,1}) $ in $ H^1(\Lda_{\delta,1},\Ss_0;\fG_{\delta,\ol{\va}}) $, then there exists $ i\in\{0,1,2,3,4\} $, such that
$$
\left|E_{\ol{\va}}(\U_{\delta,\ol{\va}},\Lda_{\delta,1})-2\cE^*(\rH_i)\log\f{\delta}{\ol{\va}}\right|
\leq \al(C_0,\delta,\eta)\log\f{\delta}{\ol{\va}}+C,
$$
where $ C>0 $ depends only on $ \cA,C_0 $, and
$$
\al(C_0,\delta,\eta)\leq C(C_0\delta+\delta^2\eta+\delta\eta+\eta+\delta^{-1}\eta).
$$
\end{lem}
\begin{proof}
Applying Lemma \ref{Luckhauscylinder2}, for $ 0<\eta<\eta_1 $ and $ 0<\ol{\va}<\wh{\va}_1\delta $, there exist $ z_-\in(-1+\delta/2,-1+\delta) $, $ z_+\in(1-\delta/2,1-\delta) $, $ \V_{\delta,\ol{\va}}\in H^1(\pa B_{\delta}\times(z_-,z_+),\cN) $, $ \W_{\delta,\ol{\va}}\in H^1(A_{\delta/2,\delta}^2\times(z_-,z_+),\Ss_0) $ and $ \W_{\delta,\ol{\va}}^{\pm}\in H^1(A_{\delta/2,\delta}^2\times\{z_{\pm}\},\Ss_0) $ satisfying \eqref{Wvaest22}, \eqref{Vvaes22}, and \eqref{WpluesminusC22}. Consequently, $ [\V_{\delta,\ol{\va}}]_{\cN} $ is well defined. Assume that $ [\V_{\delta,\ol{\va}}]_{\cN}=\rH_i $ for some $ i\in\{0,1,2,3,4\} $. Next, we will divide the proof into two steps.\smallskip

\noindent\textbf{\underline{Step 1.} Upper bound of $ E_{\ol{\va}}(\U_{\delta,\ol{\va}},\Lda_{\delta,1}) $.} We firstly define
\begin{align*}
\Lda_{\delta}^-&:=B_{\delta}^2\times(-1,z_-),\quad\quad
\Lda_{\delta}^+:=B_{\delta}^2\times(z_+,1),\\
\Lda_{\delta}^0&:=B_{\delta/2}^2\times(z_-,z_+),\quad\quad
E_{\delta}:=A_{\delta/2,\delta}^2\times(z_-,z_+).
\end{align*}
We intend to define a competitor $ \wh{\U}_{\delta,\ol{\va}}\in H^1(\Lda_{\delta,1},\Ss_0) $ such that $ \wh{\U}_{\delta,\ol{\va}}|_{\pa\Lda_{\delta,1}}=\fG_{\delta,\ol{\va}} $. Since $ \V_{\delta,\ol{\va}}\in H^1(\pa B_{\delta}^2\times(z_-,z_+),\cN) $, and $ 0<\ol{\va}<\wh{\va}_1\delta<\delta/2 $, we can apply Lemma \ref{cylinderex} to $ \PP_b(\cdot)=\V_{\delta,\ol{\va}}(2\cdot) $, $ r=\delta/2 $, and $ L=(z_+-z_-)/2 $. As a result, we obtain $ \wh{\U}_{\delta,\ol{\va}}^{(0)} $ such that $ \wh{\U}_{\delta,\ol{\va}}^{(0)}|_{\pa B_{\delta/2}^2\times(-1,1)}=\V_{\delta,\ol{\va}}(2\cdot) $ and satisfies
\be
\begin{aligned}
E_{\ol{\va}}(\wh{\U}_{\delta,\ol{\va}}^{(0)},\Lda_{\delta/2,1})&\leq C\(\f{z_+-z_-}{2}\)\(\f{z_+-z_-}{\delta}+\f{\delta}{z_+-z_-}\)\|\na_{\Ga_{\delta,1}}\V_{\delta,\ol{\va}}\|_{L^2(\Ga_{\delta,1})}^2\\
&\quad\quad+(z_+-z_-)\(\cE^*(\rH_i)\log\f{\delta/2}{\ol{\va}}+C\)\\
&\leq (2\cE^*(\rH_i)+C(\delta^{-1}+\delta)\eta)\log\f{\delta}{\ol{\va}}+C,
\end{aligned}\label{inteU0hat}
\ee
and
\be
\begin{aligned}
E_{\ol{\va}}(\wh{\U}_{\delta,\ol{\va}}^{(0)},B_{\delta/2}^2\times\{z_{\pm}\})&\leq C\(\f{z_+-z_-}{\delta}+\f{\delta}{z_+-z_-}\)\|\na_{\Ga_{\delta,1}}\V_{\delta,\ol{\va}}\|_{L^2(\Ga_{\delta,1})}^2\\
&\quad\quad+\f{z_+-z_-}{2}\(\cE^*(\rH_i)\log\f{\delta/2}{\ol{\va}}+C\)\\
&\leq (\cE^*(\rH_i)+C(\delta^{-1}+\delta)\eta)\log\f{\delta}{\ol{\va}}+C,
\end{aligned}\label{zplusmiunsU0}
\ee
where we have used the fact that $ z_+-z_-\leq 2 $ when $ \delta\in(0,1/2] $. Now, we define $ \fG_{\delta,\ol{\va}}^{\pm}:\pa\Lda_{\delta}^{\pm}\to\Ss_0 $ by
\begin{align*}
\fG_{\delta,\ol{\va}}^{\pm}(x):=\left\{\begin{aligned}
&\fG_{\delta,\ol{\va}}(x)&\text{ if }&x\in\pa\Lda_{\delta}^+\cap\pa\Lda_{\delta,1},\\
&\W_{\delta,\ol{\va}}^{\pm}(x)&\text{ if }&x\in A_{\delta/2,\delta}^2\times\{z_{\pm}\},\\
&\wh{\U}_{\delta,\ol{\va}}^{(0)}(x)&\text{ if }&x\in\ol{B_{\delta}^2}\times\{z_{\pm}\}.
\end{aligned}\right.
\end{align*}
In view of this construction, we have $ \fG_{\delta,\ol{\va}}^{\pm}\in H^1(\pa\Lda_{\delta}^{\pm},\Ss_0) $. By using \eqref{WpluesminusC22}, \eqref{GvaboundaryH13M}, \eqref{GvaboundaryH13}, and \eqref{zplusmiunsU0}, it follows that
\be
\begin{aligned}
E_{\ol{\va}}(\fG_{\delta,\ol{\va}}^{\pm},\pa\Lda_{\delta}^{\pm})&\leq E_{\ol{\va}}(\wh{\U}_{\delta,\ol{\va}}^{(0)},B_{\delta/2}^2\times\{z_{\pm}\})+E_{\ol{\va}}(\W_{\delta,\ol{\va}}^{\pm},A_{\delta/2,\delta}^2\times\{z_{\pm}\})\\
&\quad\quad+E_{\ol{\va}}(\fG_{\delta,\ol{\va}},\pa\Lda_{\delta}^{\pm}\cap\pa\Ga_{\delta,1})+E_{\ol{\va}}(\fG_{\delta,\ol{\va}},B_{\delta}^2\times\{\pm L\})\\
&\leq(\cE^*(\rH_i)+C(\delta^{-1}+\delta)\eta)\log\f{\delta}{\ol{\va}}+C\eta\log\f{\delta}{\ol{\va}}+C_0\log\f{\delta}{\ol{\va}}+C\\
&\leq(\cE^*(\rH_i)+C_0+C(1+\delta^{-1}+\delta)\eta)\log\f{\delta}{\ol{\va}}+C.
\end{aligned}\label{Gvadeltaplusminus}
\ee
Noticing that there exist two bilipschitz map $ \Phi^{\pm}:\ol{\Lda_{\delta}^{\pm}}\to \ol{B_{\delta}^3} $, with absolute bilipschitz constant, we can define
$$
\wh{\U}_{\delta,\ol{\va}}^{\pm}:=\fG_{\delta,\ol{\va}}^{\pm}\((\Phi^{\pm})^{-1}\(\f{\delta\Phi^{\pm}(x)}{|\Phi^{\pm}(x)|}\)\)\text{ for }x\in\Lda_{\delta}^{\pm}\backslash\{\Phi^{\pm}(0)\}.
$$
By using \eqref{Gvadeltaplusminus}, we obtain 
\be
\begin{aligned}
E_{\ol{\va}}(\U_{\delta,\ol{\va}}^{\pm},\Lda_{\delta}^{\pm})&\leq C\delta E_{\ol{\va}}(\fG_{\delta,\ol{\va}}^{\pm},\pa\Lda_{\delta}^{\pm})\\ 
&\leq\delta(\cE^*(\rH_i)+C_0+C(1+\delta^{-1}+\delta)\eta)\log\f{\delta}{\ol{\va}}+C\\
&\leq C(C_0\delta+(\delta+\delta^2+\delta^{-1})\eta)\log\f{\delta}{\ol{\va}}+C.
\end{aligned}\label{U0plusminusinmte}
\ee
Now, we define
\begin{align*}
\wh{\U}_{\delta,\ol{\va}}(x):=\left\{\begin{aligned}
&\wh{\U}_{\delta,\ol{\va}}^{(0)}(x)&\text{ if }&x\in\Lda_{\delta}^0,\\
&\W_{\delta,\ol{\va}}(x)&\text{ if }&x\in E_{\delta},\\
&\wh{\U}_{\delta,\ol{\va}}^{\pm}(x)&\text{ if }&x\in\Lda_{\delta}^{\pm}.
\end{aligned}\right.
\end{align*}
It follows from \eqref{Wvaest22}, \eqref{zplusmiunsU0}, and \eqref{U0plusminusinmte} that
$$
E_{\ol{\va}}(\wh{\U}_{\delta,\ol{\va}},\Lda_{\delta,1})\leq (2\cE^*(\rH_i)+C(C_0\delta+\delta^2\eta+\delta\eta+\eta+\delta^{-1}\eta))\log\f{\delta}{\ol{\va}}+C.
$$
The upper bound follows directly from the comparison $
E_{\ol{\va}}(\U_{\delta,\ol{\va}},\Lda_{\delta,1})\leq E_{\ol{\va}}(\wh{\U}_{\delta,\ol{\va}},\Lda_{\delta,1}) $. \smallskip

\noindent\textbf{\underline{Step 2.} Lower bound of $ E_{\ol{\va}}(\U_{\delta,\ol{\va}},\Lda_{\delta,1}) $.} For cylindrical coordinate $ (\rho,\theta,z)\in[0,\delta]\times[0,2\pi)\times[-1,1] $, we define
\begin{align*}
\wh{\U}_{\delta,\ol{\va}}^{(1)}(\rho,\theta,z):=\left\{\begin{aligned}
&\U_{\delta,\ol{\va}}(2\rho,\theta,z)&\text{ if }&\rho\in[0,\delta/2)\text{ and }z\in(z_-,z_+),\\
&\W_{\delta,\ol{\va}}\(3\delta/2-\rho,\theta,z\)&\text{ if }&\rho\in[\delta/2,\delta)\text{ and }z\in(z_-,z_+).
\end{aligned}\right.
\end{align*}
By this definition, we have $ \wh{\U}_{\delta,\ol{\va}}^{(1)}\in H^1(B_{\delta}^2\times(z_-,z_+),\Ss_0) $ and
\be
\wh{\U}_{\delta,\ol{\va}}^{(1)}(x)=\V_{\delta,\ol{\va}}(x)\text{ for }\HH^2\text{-a.e. }x\in\pa B_{\delta}^2\times(z_-,z_+).\label{boundaryconUV2}
\ee
Moreover, it follows by direct computations and \eqref{Wvaest22} that
\be
\begin{aligned}
E_{\ol{\va}}(\wh{\U}_{\delta,\ol{\va}}^{(1)},\Lda_{\delta,1})&\leq E_{\ol{\va}}(\U_{\delta,\ol{\va}},\Lda_{\delta,1})+E_{\ol{\va}}(\W_{\delta,\ol{\va}},A_{\delta/2,\delta}^2\times(z_-,z_+))\\
&\leq E_{\ol{\va}}(\U_{\delta,\ol{\va}},\Lda_{\delta,1})+C\eta\log\f{\delta}{\ol{\va}}.
\end{aligned}\label{U1Udeltaminus}
\ee
Since we can choose $ 0<\wh{\va}_1<1 $ sufficiently small, it can be assumed that $ \ol{\va}<\delta/80 $. It follows from Fubini theorem and Corollary \ref{lowerboundcor} and \eqref{boundaryconUV2} that
$$
E_{\ol{\va}}(\wh{\U}_{\delta,\ol{\va}}^{(1)},B_{\delta}^2\times\{z\})+C\delta E_{\ol{\va}}(\V_{\delta,\ol{\va}},\pa B_{\delta}^2\times\{z\})\geq\cE^*(\rH_i)\log\f{\delta}{\ol{\va}}-C.
$$
Integrating the above formula with respect to $ z\in(z_-,z_+) $, we can deduce from the fact $ z_+-z_-\geq 2-\delta $ that
$$
\E_{\ol{\va}}(\wh{\U}_{\delta,\ol{\va}}^{(1)},\Lda_{\delta,1})+C\delta\int_{\pa B_{\delta}^2\times(z_-,z_+)}|\na\V_{\delta,\ol{\va}}|^2\ud\HH^2\geq 2\cE^*(\rH_i)\log\f{\delta}{\ol{\va}}-C.
$$
This, together with \eqref{Vvaes22}, \eqref{GvaboundaryH13}, and \eqref{U1Udeltaminus}, implies that
$$
E_{\ol{\va}}(\U_{\delta,\ol{\va}},\Lda_{\delta,1})\geq (2\cE^*(\rH_i)-C(\delta+1)\eta)\log\f{\delta}{\ol{\va}}-C,
$$
which obtain the lower bound and completes the proof.
\end{proof}

\begin{lem}\label{densitydescrete}
For $ \HH^1 $-a.e. $ x\in\cS_{\op{line}}\cap\om $, there holds $ \Theta(x)\in\{\kappa_*,2\kappa_*\} $.
\end{lem}
\begin{proof}
By Lemma \ref{stationvari}, $ \mu_0 $ can be approximated by tangent line for $ \HH^1 $-a.e. $ x_0\in\cS_{\op{line}}\cap\om $. Indeed, for $ \mu_0 $-a.e. $ x\in\om $, there exists a unique $ 1 $-dimensional subspace $ L_x=T_x\cS_{\op{line}}\subset\R^3 $ such that for any $ \vp\in C_c(\R^3) $, there holds
\be
\lim_{s\to 0^+}\f{1}{s}\int_{\R^3}\vp\(\f{z-x}{s}\)\ud\mu_0(z)=\Theta(x)\int_{L_x}\vp(y)\ud\HH^1(y).\label{tangapp}
\ee
Fix such a point $ x_0\in\cS_{\op{line}}\cap\om $. By Lemma \ref{Thetamu0eta0}, $ \Theta(x_0)>0 $. We can use the formula \eqref{tangapp} to obtain $ L_{x_0}\subset\R^3 $ and a sequence of measures $ \{\nu_{s_n}\}_{0<s_n<1}\subset (C(\om))' $ with $ \nu_{s_n}(E)=s_n^{-1}\mu_0(s_nE\cap\om) $ for any $ E\in\mathcal{B}(\R^3) $ and $ \lim_{n\to+\ift}s_n=0 $, such that
\be
\nu_{s_n}\wc^*\nu_0:=\Theta(x)\HH^1\left\llcorner\right.L_{x_0}\text{ in }(C(\om))'\text{ as }n\to+\ift.\label{nu0con}
\ee
Up to a translation and a rotation, we can assume that $ x_0=0 $ and $ L_{x_0}=\{x\in\R^3:x_1=x_2=0\} $. Set $ C_0=\max\{16\sqrt{2}M/\dist(x_0,\pa\om),M\} $, where $ M>0 $ is the constant given by \eqref{assumptionbound}. Let $ 0<\eta_1,\wh{\va}_1<1 $ be given by Lemma \ref{chooseseveralcases} with respect to $ C_0 $ and $ 0<\delta<1 $ to be determined later. Let $ 0<\eta<\eta_1 $. By \eqref{nu0con}, we can choose a positive number $ 0<s_0<1 $, depending on $ \delta,\eta $, such that $ s_0<\dist(x_0,\pa\om)/(2\sqrt{2}) $ and $
\mu_0(\Ga_{s\delta,s})\leq s\eta/2 $ for any $ 0<s<s_0 $. Recall that $ \mu_{\va_n}=E_{\va_n}(\Q_{\va_n},\cdot)/\log(1/\va_n)\wc^*\mu_0 $ as $ n\to+\ift $. There exists $ n_0\in\Z_+ $, depending only on $ C_0,\eta,\delta,s $, such that for any $ n\geq n_0 $, there hold $ \ol{\va}_n:=\va_n/s\in(0,\wh{\va}_1) $ and
\be
E_{\va_n}(\Q_{\va_n},\Ga_{s\delta,s})\leq s\eta\log\f{s\delta}{\va_n}.\label{nn0geq}
\ee
Set $ r_0:=\dist(x_0,\pa\om)/2 $. Simple calculations yield that $ \Lda_{s\delta,1}\subset B_{r_0} $. Using \eqref{assumptionbound} and \eqref{Mo}, we can choose larger $ n_0\in\Z_+ $ if necessary, such that when $ n\geq n_0 $, there holds
$$
E_{\va_n}(\Q_{\va_n},\Lda_{s\delta,s})\leq\f{\sqrt{2}s}{r_0}E_{\va_n}(\Q_{\va_n},\Lda_{\delta s,s})\leq \f{C_0s}{4}\log\f{s\delta}{\va_n}.
$$
By using Fatou Lemma, we have
$$
\int_{-s}^s\liminf_{n\to+\ift}E_{\va_n}(\Q_{\va_n},B_{s\delta}^2\times\{z\})\ud z\leq \f{C_0s}{4}\log\f{s\delta}{\va_n}.
$$
Consequently, it follows from average arguments that there exists $ z_-\in[-s,-3s/4] $ and $ z_+\in[3s/4,s] $ and a subsequence of $ \va_n $ without relabeling such that
\be
E_{\va_n}(\Q_{\va_n},B_{s\delta}^2\times\{z_{\pm}\})\leq C_0\log\f{s\delta}{\va_n}.\label{upperlowb}
\ee
Without of loss of generality, we assume that $ z_{\pm}=\pm s $ and for general case one can apply almost the same arguments and a translation. Define $ \U_{\delta,\ol{\va}_n}(y):=\Q_{\va_n}(sy) $ for $ y\in\Lda_{\delta,1} $. We can deduce from \eqref{assumptionbound}, \eqref{nn0geq}, \eqref{upperlowb}, and the fact $ 0<s<1 $ that when $ n>n_0 $, the following properties hold. For $ 0<\ol{\va}_n<\wh{\va}_1\delta $, we have
\begin{align*}
\|\U_{\delta,\ol{\va}_n}\|_{L^{\ift}(\Lda_{\delta,1})}&\leq C_0,\\
E_{\ol{\va}_n}(\U_{\delta,\ol{\va}_n},\Ga_{\delta,1})&\leq\eta\log\f{\delta}{\ol{\va}_n},\\
E_{\ol{\va}_n}(\U_{\delta,\ol{\va}_n},B_{\delta}^2\times\{\pm 1\})&\leq C_0\log\f{\delta}{\ol{\va}_n}.
\end{align*}
As a result, we can apply Lemma \ref{chooseseveralcases} and the change of variables to obtain that there exists $ i\in\{0,1,2,3,4\} $ such that
$$
\left|E_{\va_n}(\Q_{\va_n},\Lda_{s\delta,s})-2s\cE^*(\rH_i)\log\f{s\delta}{\va_n}\right|
\leq s\al\log\f{s\delta}{\va_n}+C,
$$
where $ C>0 $ depends only on $ \cA,C_0 $, and
$$
\al\leq C(C_0\delta+\delta^2\eta+\delta\eta+\eta+\delta^{-1}\eta).
$$
Passing $ n\to+\ift $ and then letting $ s\to 0^+ $, we can obtain from \eqref{weakconFG} that $
|\nu_0(\Lda_{\delta,1})-2\cE^*(\rH_i)|\leq\al $. Choose $ \delta=\eta^{1/2} $ and let $ \eta\to 0^+ $, we can obtain that $ \Theta(x_0)\in\{\cE^*(\rH_i):i=0,1,2,3,4\} $. Using Lemma \ref{kappa123}, there holds $ \Theta(x_0)\in\{0,\kappa_*,2\kappa_*\} $. Since $ \Theta(x_0)>0 $, we have $ \Theta(x_0)\in\{\kappa_*,2\kappa_*\} $, which completes the proof.
\end{proof}

\begin{proof}[Proof of (\ref{pr6}) in Theorem \ref{main}]
Recall the definition of the varifold $ V_0 $ associated with $ \mu_0\left\llcorner\right.\om $, defined in Lemma \ref{stationvari}. We set $ \|V_0\| $ as the weight of it, such that 
$$
\|V_0\|(E):=V_0(E\times\bG(3,1))\text{ for any }E\in\mathcal{B}(\om).
$$
Moreover, we can deduce from the definition of $ V_0 $ that
$$
\mu_0\left\llcorner\right.\om=\Theta(\|V_0\|,x)\HH^1\left\llcorner\right.(\cS_{\op{line}}\cap\om),
$$
where by Lemma \ref{Thetamu0eta0}, there holds
$$
\Theta(\|V_0\|,x):=\lim_{r\to 0^+}\f{\|V_0\|(\ol{B_r(x)})}{2r}=\lim_{r\to 0^+}\f{\|V_0\|(\ol{B_r(x)})}{2r}=\Theta(\mu_0,x)\geq\f{\eta_0}{2}
$$
for any $ x\in\cS_{\op{line}}\cap\om $. This, together with the result of Theorem p. 89 in \cite{AA76}, implies that $
V_0=\sum_{i=1}^{+\ift}\Theta(
|V_0\|,I_i)|I_i| $, where $ \{I_i\}_{i\in\Z_+} $ is a collection of bounded open intervals (segments without endpoint), $ \Theta(\|V_0\|,I_i) $ is unique member of the range of $ \Theta(\|V_0\|,\cdot) $ restricted to $ I_i $, and $ |I_i|\in(C_c(\om\times\bG(3,1)))' $ is a finite positive Radon measure given by $
|I_i|=\delta_{\vv_i\vv_i}\otimes(\HH^1\left\llcorner\right.I_i) $, where $ \vv_i $ is the unit directing vector of $ I_i $. In view of Lemma \ref{densitydescrete}, we can apply Theorem p.89 in \cite{AA76} again to deduce that for each $ x\in\cS_{\op{line}} $, there exists $ r>0 $ such that $ \cS_{\op{line}}\cap\om\cap \ol{B_r(x)} $ is a finite union of closed straight line segments joining $ x $ with $ \pa B_r(x) $. As a result, for any $ K\subset\subset\om $, $ \cS_{\op{line}}\cap\ol{K} $ is a finite union of closed line segments, $ L_1,L_2,...,L_p $. By subdividing these segments, we can assume that for $ i,j\in\{1,2,...,p\} $ with $ i\neq j $, either $ L_i\cap L_j=\emptyset $ or they intersect to a common endpoint.

To show Property (\ref{pra}), assume that $ L $ is a straight segment in $ \cS_{\op{line}}\cap\ol{K} $, and $ B_r^2(x)\cap L=\{x\} $. Without loss of generality, we assume that $ L=\{x_1=x_2=0\} $ and $ x=0 $. Moreover, We can also assume that there exists a small neighborhood $ U $ of $ B_r^2(x) $ such that $ U\cap\cS_{\pts}=\emptyset $. This is from the property that $ \cS_{\pts} $ is locally finite and the definition of free homotopy class. To this end, we let $ B_r^2(x)\subset\{x_3=0\} $. Using the arguments of the proof for Lemma \ref{densitydescrete}, we see that the related $ \U_{\delta,\ol{\va}_n} $ satisfies
$$
\left|E_{\va_n}(\U_{\delta,\ol{\va}_n},\Lda_{\delta,1})-2\cE^*(\rH_i)\log\f{\delta}{\ol{\va}_n}\right|\leq\al\log\f{\delta}{\ol{\va}_n}+C
$$
for $ i\in\{1,2,3,4\} $ and $ C>0 $ depends only on $ \cA,M,K $. Note that in the proof of Lemma \ref{chooseseveralcases}, $ \rH_i $ represents the free homotopy class of $ \V_{\delta,\ol{\va}_n} $. Recall the construction of $ \V_{\delta,\ol{\va}_n} $, on the grids of the $ \ol{\Ga_{\delta,1}} $, $ \V_{\delta,\ol{\va}_n}=\varrho(\U_{\delta,\va_n}) $. This implies that $ [\varrho(\Q_{\va_n})|_{\Ga_{s\delta,s}}]_{\cN}\neq\rH_0 $. Since $ \cS_{\pts} $ is locally finite, we can use the uniform convergence result to obtain that $ [\Q_0|_{\pa B_r^2(x)}]=[\Q_0|_{\pa B_{\delta}^2(x)}]\neq\rH_0 $.

Finally, we prove Property (\ref{prb}). For $ x_0\in\cS_{\op{line}} $ being an endpoint of straight segments of $ \cS_{\op{line}}\cap\ol{K} $, we can choose $ r>0 $ such that $ \ol{B_r(x_0)}\subset K $ and $ \cS_{\op{line}}\cap B_r(x_0)=\cup_{j=1}^qL_{i_j} $. We denote $ \vv_j $ as the unit direction vector of $ L_{i_j} $. We claim,
\be
\sum_{j=1}^q\cE^*([\Q_0|_{\pa B_{r_j}^2(x_j)}]_{\cN})\vv_j=0,\label{claimsumldai}
\ee
where $ B_{r_j}^2(x_j) $ satisfies $ B_{r_j}^2(x_j)\cap\cS_{\op{line}}=\{x_j\} $ and $ B_{r_j}^2(x_j)\subset B_r(x_0) $ with $ x_j $ being in the relative interior of $ L_j $. We firstly assume that \eqref{claimsumldai} is true. Obviously, in view of Property (\ref{pra}), for any $ j=1,...,q $, $ \cE^*(\Q_0|_{\pa B_{r_j}^2(x_j)})>0 $, this implies that $ q\geq 2 $. If $ q=2 $, there must by $ \vv_1=-\vv_2 $ and the result follows directly. If $ q=3 $, by Lemma \ref{kappa123}, up to the relabel of $ i=1,2,3 $, there are three cases. If $ 2\vv_1+\vv_2+\vv_3=0 $, we see that $ \vv_2=\vv_3 $, which is a contradiction since we choose different $ \vv_i $ at the beginning. If $ \vv_1+\vv_2+\vv_3=0 $, there is nothing to prove. If $ \vv_1+2\vv_2+2\vv_2=0 $, up to a rotation and a translation, we assume that $ \vv_1=(0,\sqrt{15}/4,1/4) $, $ \vv_2=(0,-\sqrt{15}/4,1/4) $, and $ \vv_3=(0,0,-1) $. Let $ 0<\delta<1 $ be a small number, such that $ \Lda_{\delta,\delta/20}\subset B_r(x_0)\backslash\cS_{\pts} $ and $ L_{i_j}\cap\Lda_{\delta,\delta/20}=x_j $ for $ j=1,2,3 $. Indeed, $ x_1=(0,\sqrt{15}\delta/20,\delta/20) $, $ x_2=(0,-\sqrt{15}\delta/20,\delta/20) $, and $ x_3=(0,0,-\delta/2) $. For this case, we note that $ [\Q_0|_{\pa B_{\delta/50}^2(x_1)}]_{\cN},[\Q_0|_{\pa B_{\delta/50}^2(x_2)}]_{\cN}\in\{\rH_3,\rH_4\} $ and $ [\Q_0|_{\pa B_{\delta/50}^2(x_3)}]_{\cN}\in\{\rH_1,\rH_2\} $. Let $ x_4:=x_1-(0,\delta/50,0) $, $ x_5:=x_2+(0,\delta/50,0) $, and $ x_6:=(0,0,\delta/20) $. Using Lemma \ref{homcor}, we can obtain
\begin{align*}
&\vp_{\Q_0(x_6)}^{-1}([\Q_0|_{\ol{x_4x_6}}*\Q_0|_{\pa B_{\delta/50}^2(x_1)}*\Q_0|_{\ol{x_6x_4}}]_{\cN,\Q_0(x_6)})\in\{-1,\pm\kk\},\\
&\vp_{\Q_0(x_6)}^{-1}([\Q_0|_{\ol{x_5x_6}}*\Q_0|_{\pa B_{\delta/50}^2(x_2)}*\Q_0|_{\ol{x_6x_5}}]_{\cN,\Q_0(x_6)})\in\{-1,\pm\kk\},
\end{align*}
and then
$$
[\Q_0|_{\ol{x_5x_6}}*\Q_0|_{\pa B_{\delta/50}^2(x_2)}*\Q_0|_{\ol{x_6x_5}}*\Q_0|_{\ol{x_4x_6}}*\Q_0|_{\pa B_{\delta/50}^2(x_1)}*\Q_0|_{\ol{x_6x_4}}]_{\cN}\in\{\rH_3,\rH_4\}.
$$
Since $ \Q_0 $ itself is the free homotopy function, we have $ [\Q_0|_{\pa B_{\delta}^2(x_6)}]\in\{\rH_3,\rH_4\} $. On the other hand, there holds $ [\Q_0|_{\pa B_{\delta}^2(x_6)}]_{\cN}\in\{\rH_1,\rH_2\} $. This is a contradiction, since by the definition of $ \Lda_{\delta,\delta/20} $, $ \Q_0\in C^{\ift}(\Ga_{\delta,\delta/20},\cN) $ and then $ [\Q_0|_{\pa B_{\delta}^2(x_5)}]_{\cN}=[\Q_0|_{\pa B_{\delta}^2(x_6)}]_{\cN} $. Now it remains to show the claim \eqref{claimsumldai}. Indeed, since the first variation of $ V_0 $ is vanishing, and \eqref{V0formula} holds true, we deduce
\begin{align*}
0=\int_{\om}\A_0(x):\na\uu(x)\ud\mu_0(x)=\int_{\cS_{\op{line}}\cap B_r(x_0)}\Theta(x)\A_0(x):\na\uu(x)\ud\HH^1(x)
\end{align*}
for any $ \uu\in C_c^1(B_r(x_0),\R^3) $. Noticing that $ \cS_{\op{line}}\cap B_r(x_0)=\cup_{j=1}^qL_{i_j} $, we have
$$
0=\sum_{j=1}^q\beta_j\int_{L_{i_j}}\vv_j\vv_j:\na\uu\ud\HH^1=\sum_{j=1}^q\beta_j\int_0^{\HH^1(L_{i_j})}\f{\ud}{\ud t}\<\uu(x_0+t\vv_j),\vv_j\>\ud\HH^1=-\sum_{j=1}^q\<\chi(x_0),\beta_j\vv_j\>,
$$
where $ \beta_j=\cE^*([\Q_0|_{\pa B_{r_j}(x_j)}]_{\cN}) $. This directly implies the claim. Here we have also used the fact that $ \Theta(x_j)=\cE^*([\Q_0|_{\pa B_{r_j}(x_j)}]_{\cN}) $ for any $ j $, $ \mu_0\left\llcorner\right.B_r(x_0)=\sum_{j=1}^q\Theta(x_j)\HH^1\left\llcorner\right.L_{i_j} $, and $ \A_0\left\llcorner\right.L_{i_j}=\vv_j\vv_j $. 
\end{proof}

\appendix

\section{Some results in elliptic equations}

The following are some basic results in elliptic equations. 

\begin{lem}[\cite{HL97}, Lemma 4.1]\label{Hanarck}
Let $ B_r^k\subset\R^k $ with $ k\geq 2 $. Assume that $ u\in H^1(B_r) $ is non-negative and satisfies $ -\Delta u\leq C_0u $ in $ B_r $ for some $ C_0>0 $ in weak sense, 
i.e., for any $ 0\leq\vp\in H_0^1(B_r) $, 
$$
\int_{B_r}\na u\cdot\na\vp\ud x\leq C_0\int_{B_r}u\vp\ud x.
$$
There exists $ C>0 $ depending only on $ C_0 $ and $ k $ such that $
u(0)\leq C\dashint_{B_r}u\ud x $.
\end{lem}

\begin{lem}[\cite{BBH93}, Lemma A.1]\label{Intp}
Let $ U\subset\R^k $ be a bounded domain. Assume that $ u\in C^2(U) $ and $ f\in L^{\ift}(U) $ satisfy the equation $
-\Delta u=f $ in $ U $. There holds
$$
|\na u(x)|^2\leq C(\|f\|_{L^{\ift}(U)}\|u\|_{L^{\ift}(U)}+\dist^{-2}(x,\pa U)\|u\|_{L^{\ift}(U)}^2),
$$
for any $ x\in U $, where $ C>0 $ is a constant depending only on $ k $.
\end{lem}

\begin{lem}[\cite{NZ13}, Lemma 6]\label{au}
Let $ B_r\subset\R^3 $. Assume that $ u\in W^{1,\ift}(B_r)\cap C^0(\ol{B_r}) $ satisfies in the weak sense $
-L\Delta u+au\leq C_0 $ in $ B_r $,
where $ L,a>0 $ and $ C_0>0 $ are positive constants. For any $ \ga\geq 0 $, there holds the property that in $ B_r $,
\begin{align*}
u(x)&\leq\f{C_0}{a}+2\exp\(-\sqrt{\f{a}{L}}\f{r-|x|}{2}\)\sup_{\pa B_r}u^+\leq\f{C_0}{a}+C\min\left\{1,\f{L^{\ga}}{a^{\ga}(r-|x|)^{2\ga}}\right\}\sup _{\pa B_r}u^{+},
\end{align*}
where $ u^{+}=\max(u,0) $ and $ C>0 $ depends only on $ \ga $.
\end{lem}

\begin{cor}[\cite{NZ13}, Lemma 7]\label{aub}
Let $ U $ be a bounded Lipschitz domain with $ r_{U,0} $ and $ M_{U,0} $. Assume that $ u\in W^{1,\ift}(U\cap B_r)\cap C^0(\ol{ U\cap B_r}) $ satisfies in the weak sense $ -L\Delta u+au\leq C_0 $ in $  U\cap B_r $ and $ u=0 $ on $ \pa U\cap B_r $ with $ 0<r<r_{U,0} $, where $ L,a>0 $ and $ C_0>0 $ are positive constants. For any $ \ga\geq 0 $,
$$
u(x)\leq\f{C_0}{a}+C\min\left\{1,\f{L^{\ga}}{a^{\ga}(r-|x|)^{2\ga}}\right\}\sup_{ U\cap\pa B_r(x_0)}u^+\text{ in }\om\cap B_r,
$$
where $ C>0 $ depends only on $ \ga $.
\end{cor}

\begin{cor}\label{corau}
Let $ r>0 $. Assume that $ u\in W^{1,\ift}(B_r)\cap C^0(\ol{B_r}) $ and $ F\in L^{\ift}(B_r) $ satisfy in the weak sense $
-L\Delta u+au=F $ in $ B_r $, where $ L,a>0 $ are positive constants. For $ \ga\geq 0 $, there holds the property that in $ B_r $,
\begin{align*}
|u(x)|&\leq\f{1}{a}\|F\|_{L^{\ift}(B_r)}+2\exp\(-\sqrt{\f{a}{L}}\f{r-|x|}{2}\)\sup_{\pa B_r}\|u\|_{L^{\ift}(\pa B_r)}\\
&\leq\f{1}{a}\|F\|_{L^{\ift}(B_r)}+C\min\left\{1,\f{L^{\ga}}{a^{\ga}(r-|x|)^{2\ga}}\right\}\sup_{\pa B_r}\|u\|_{L^{\ift}(\pa B_r)},
\end{align*}
where $ C>0 $ depends only on $ \ga $.
\end{cor}

\begin{proof}
By the assumption, we can obtain 
$$ 
\max\{-L\Delta u+au,-L\Delta(-u)+a(-u)\}\leq\|f\|_{L^{\ift}(B_r)}:=C_0. 
$$
The result follows directly from Lemma \ref{au}.
\end{proof}

\begin{lem}[\cite{NZ13}, Lemma 11]\label{LempRe}
Let $ U $ be a bounded $ C^{2,1} $ domain with $ r_{U,2} $ and $ M_{U,2} $. Assume that $ 0<r<r_{U,2} $, $ F\in L^p(U\cap B_r) $, $ p>3 $, $ g\in C^2(\pa U\cap B_r) $ and $ u\in W^{2,p}(U\cap B_r) $ satisfy
$$
\left\{\begin{aligned}
-a_{ij}(x)\pa_{ij}^2u+b_i(x)\pa_iu&=F&\text{ in }& U\cap B_r,\\
u&=g&\text{ on }&\pa U\cap B_r,
\end{aligned}\right.
$$
where the coefficients $ a_{ij} $ and $ b_i $ are continuous and satisfy $ \mu^{-1}\I\leq(a_{ij})\leq\mu\I $, $ \|b_i\|_{\ift}\leq\mu $ for some $ \mu>0 $ and for any $ x,y\in U\cap B_r $,
$$
\sum_{i,j=1}^3|a_{ij}(x)-a_{ij}(y)|+\sum_{i=1}^3|b_i(x)-b_i(y)|\leq\w(|x-y|),
$$
where $ \w:[0,+\ift)\to[0,+\ift) $ is a continuous increasing function with $ \w(0)=0 $. There hold
\begin{align*}
\|\na u\|_{L^{\ift}(U\cap B_{r/2})}&\leq C\|\na_{\pa U}g\|_{L^{\ift}(\pa  U\cap B_r)}+r\|D_{\pa U}^2g\|_{L^{\ift}(\pa  U\cap B_r)}\\
&\quad\quad+Cr^{-3/2}\|\na u\|_{L^2(U\cap B_r)}+Cr^{1-3/p}\|F\|_{L^p(U\cap B_r)},
\end{align*}
where $ C>0 $ depends only on $ \mu,p,r_{U,2},M_{U,2} $ and $ \w $.
\end{lem}

\begin{lem}\label{Boundarymu}
Let $ U $ be a bounded $ C^{2,1} $ domain with $ r_{U,2} $ and $ M_{U,2} $. Assume that $ u\in C^2(\ol{U\cap B_r}) $ satisfies
$$
\left\{\begin{aligned}
|\Delta u(x)|&\leq C_0\dist^{-\mu-1}(x,\pa U)&\text{ in }& U\cap B_r,\\
u&=0&\text{ on }&\pa U\cap B_r,
\end{aligned}\right.
$$
for some $ \mu\in(0,1) $ and $ C_0>0 $. There exists $ C>0 $ depending only on $ r_{U,2},M_{U,2} $, and $ \mu $, such that
$$
|u(x)|\leq Cr^{-1}\|u\|_{L^{\ift}(U\cap B_r)}+CC_0\dist^{1-\mu}(x,\pa U)
$$
for any $ x\in  U\cap B_{r/2} $.
\end{lem}
\begin{proof}
Without loss of generality, we only need to prove
\be
u(y)\leq Cr^{-1}\|u\|_{L^{\ift}(U\cap B_r)}+CC_0\dist^{1-\mu}(y,\pa U)\label{Oneside}
\ee
for any $ y\in  U\cap B_{r/2} $. For otherwise, we apply the estimate \eqref{Oneside} to $ -u $. Also, by basic covering arguments and the change of coordinate systems, we can assume that $ r_{U,2}>0 $ and $ C_1:=\|\na\psi\|_{L^{\ift}(\R^2)}>0 $ are sufficiently small, such that
$$
\dist(y,\pa U)\leq y_3-\psi(y_1,y_2)\leq 2\dist(y,\pa U)
$$
in $  U\cap B_r $ when $ 0<r<r_{U,2} $, where $ \psi $ is given by \eqref{UcapB}. Let $ w $ be the solution of 
$$
\left\{\begin{aligned}
\Delta w&=0&\text{ in }& U\cap B_r,\\
w&=0&\text{ on }&\pa U\cap B_r,\\
w&=1&\text{ on }& U\cap\pa B_r.
\end{aligned}\right.
$$
By standard boundary $ W^{1,\ift} $ estimates and the maximum principle, we have
$$
0\leq w(y)\leq\|\na w\|_{L^{\ift}(U\cap B_{r/2})}\dist(y,\pa U)\leq Cr^{-1}\|w\|_{L^{\ift}(U\cap B_r)}\dist(y,\pa U)\leq Cr^{-1}\dist(y,\pa U)
$$
for any $  y\in U\cap B_{r/2} $. Define
$$
w_1(y):=\|u\|_{L^{\ift}(U\cap B_r)}w(y)+4C_0(\mu-\mu^2)^{-1}(y_3-\psi(y_1,y_2))^{1-\mu}.
$$
Obviously, we have
\begin{align*}
-\Delta w_1(y)&=4C_0(y_3-\psi(y_1,y_2))^{-\mu-1}(1+(\pa_1\psi)^2+(\pa_2\psi)^2)+4C_0\mu^{-1}(y_3-\psi(y_1,y_2))^{-\mu}(\pa_1^2\psi+\pa_2^2\psi)\\
&\geq 3C_0(y_3-\psi(y_1,y_2))^{-\mu-1},
\end{align*}
where for the inequality, we have used the fact that $ r_{U,2},C_1 $ are sufficiently small, (depending on $ \mu $). Thus, we have $ -\Delta(w_1-u)\geq 0 $ in $ U\cap B_r $ and $ w_1-u\geq 0 $ on $ \pa(U\cap B_r) $. The result follows directly from the maximum principle of subharmonic functions.
\end{proof}
\section{Approximation results}

Here, we will present several approximation results for $ H^1 $ maps from manifolds to manifolds. The first approximation result pertains to maps from a two-dimensional manifold to an arbitrary closed manifold.

\begin{lem}[\cite{SU83}, Proposition p.267]\label{SmoothApproximation}
Let $ \cX $ be a $ 2 $-dimensional compact Riemannian manifold with possibly non-empty $ C^1 $ boundary. Assume that $ \cY $ is a smooth compact Riemannian manifold without boundary. $ C^{\ift}(\cX,\cY) $ is dense in $ H^1(\cX,\cY) $.
\end{lem}

Here the density of $ C^{\ift}(\cX,\cY) $ is in the norm of $ H^1(\cX,\R^k) $, where $ \cY $ is isometrically embedded into $ \R^k $. Indeed, in the original paper, for any $ u\in H^1(\cX,\cY) $, the authors construct a sequence $ u_{\delta}\in C^{\ift}(\cX,\cY) $ with $ \delta\in(0,1) $ such that $
\lim_{\delta\to 0^+}\|u-u_{\delta}\|_{H^1(\cX,\R^k)}=0 $. The aforementioned smooth approximation result is not applicable when the dimension of the manifold $ \cX $ exceeds $ 2 $. Further developments on this topic can be found in \cite{BZ88}. For $ \cX $ of higher dimensions, assuming it is simply connected, there exists a result that considers weak topological approximation.

\begin{lem}[\cite{PR03}, Theorem I]\label{weakApproximation}
Let $ \cX $ and $ \cY $ be compact smooth Riemannian manifolds and assume that $ \cX $ is simply connected. $ C^{\ift}(\cX,\cY) $ is dense in $ H^1(\cX,\cY) $ for the sequential weak topology, i.e. if $ \cY $ is isometrically embedded into $ \R^k $, then for and $ u\in H^1(\cX,\cY) $, there is $ u_k\in C^{\ift}(\cX,\cY) $ such that for any $ u_k\wc u $ when $ k\to+\ift $ in $ H^1(\cX,\R^k) $.
\end{lem}

\section{Some results on the operators}

\begin{lem}[\cite{NZ13}, Lemma 17]\label{multione}
Let $ \om\subset\R^3 $ be a bounded $ C^{2,1} $ domain. Consider the elliptic operator $
\mathcal{L}u:=\Delta u+b_{i}\pa_iu+cu $, where $ u\in H_0^1(\om) $, $ b_{i}\in W^{1,p}(\om) $ for some $ p\geq 2 $, $ c\in L^{\ift}(\om) $ for $ 1\leq i\leq 3 $. If $ \mathcal{L}:H_0^1(\om)\to H^{-1}(\om) $ is bijective, then $
\|u\|_{L^2(\om)}\leq C\|\mathcal{L}u\|_{(H^2(\om)\cap H_0^1(\om))'} $ for any $ u\in H_0^1(\om) $, where $ (H^2(\om)\cap H_0^1(\om))' $ denotes the dual space of $ H^2(\om)\cap H_0^1(\om) $, and $ C>0 $ is a constant that depends only on $ \om,b_i $ and $ c $.
\end{lem}

\begin{lem}[\cite{NZ13}, Lemma 18, 19, and 20]\label{NZlem2}
Let $ \om\subset\R^3 $ be a bounded smooth domain. The following properties hold.
\begin{enumerate}
\item If $ R:\om\to\R $ satisfies $
|R(x)|\leq C_0\exp\{-C_1\dist(x,\pa\om)/\va\} $ for any $ x\in\om $ with $ 0<\va<1 $ and $ C_0,C_1>0 $, then $
\max\{\|R\|_{(H^2(\om)\cap H_0^1(\om))'},\|R\|_{H^s(\om)}\}\leq C\va^2 $ for any $ s\in[-2,-3/2) $, where $ \<R,\vp\>_{(H^2(\om)\cap H_0^1(\om))'\times(H^2(\om)\cap H_0^1(\om))}=\int_{\om}R\vp\ud x $ for any $ \vp\in H^2(\om)\cap H_0^1(\om) $ and $ C>0 $ depends only on $ \om,C_0 $, and $ C_1 $.
\item If $ R:\om\to\R $ satisfies $
|R(x)|\leq C_0/\dist(x,\pa\om) $ for any $ x\in\om $ with some $ C_0>0 $, then $
\|R\|_{H^{-1}(\om)}\leq C\va^2 $, where $ C>0 $ depends only on $ \om $ and $ C_0 $.
\end{enumerate}
\end{lem}

\section*{Acknowledgments}
The authors are partially supported by the National Key R\&D Program of China under Grant 2023YFA1008801 and NSF of China under Grant 12288101.

\end{document}